\documentclass[11pt,reqno]{amsbook}

\usepackage[T1]{fontenc}
\usepackage{slantsc}
\usepackage{bold-extra}

\usepackage{amssymb}
\usepackage{enumitem}
\usepackage{geometry}
\usepackage{accents}

\usepackage{mathtools}
\usepackage{mathrsfs}
\usepackage{xcolor}
\usepackage[colorlinks=false, linkcolor=black, citecolor=black, urlcolor=black]{hyperref}
\usepackage{extdash}

\usepackage{cleveref}

\usepackage{bookmark}
\usepackage{imakeidx}
\makeindex[name=subject,title={Subject Index}
]
\newcommand{\sid}[1]{\index[subject]{#1}}
\makeindex[name=symbol,title={Symbol Index}]
\newcommand{\syid}[1]{\index[symbol]{#1}}

\usepackage{graphicx, caption, subfig}

\counterwithin{figure}{chapter}

\usepackage{tikz}

\newtheoremstyle{thm}{}{}{\itshape}{}{\scshape\bfseries}{.}{.5em}{\thmname{#1}\thmnumber{ #2}\textnormal{\thmnote{ (\scshape#3)}}}
\newtheoremstyle{rmk}{}{}{}{}{\scshape}{.}{.5em}{\thmname{#1}\thmnumber{ #2}\textnormal{\thmnote{ (\scshape#3)}}}
\newtheoremstyle{def}{}{}{}{}{\scshape\bfseries}{.}{.5em}{\thmname{#1}\thmnumber{ #2}\textnormal{\thmnote{ (\scshape#3)}}}  

\theoremstyle{thm}
\newtheorem{thm}{Theorem}[chapter]
\newtheorem{prop}[thm]{Proposition}
\newtheorem{lem}[thm]{Lemma}
\newtheorem{cor}[thm]{Corollary}

\theoremstyle{def}
\newtheorem{definition}[thm]{Definition}
\newtheorem{example}[thm]{Example}

\theoremstyle{rmk}
\newtheorem{remark}[thm]{Remark}

\crefname{thm}{Theorem}{Theorems}
\crefname{prop}{Proposition}{Propositions}
\crefname{lem}{Lemma}{Lemmas}
\crefname{cor}{Corollary}{Corollaries}
\crefname{claim}{Claim}{Claims}
\crefname{step}{Step}{Steps}
\crefname{definition}{Definition}{Definitions}
\crefname{example}{Example}{Examples}
\crefname{remark}{Remark}{Remarks}

\numberwithin{equation}{chapter}

\newcommand*{\dt}[1]{%
	\accentset{\mbox{\large\bfseries .}}{#1}}

\newcommand{\pair}[2]{\ensuremath\langle#1,#2\rangle}
\newcommand{\norm}[1]{\ensuremath\lVert#1\rVert}
\newcommand{\compl}{^\mathsf{c}}
\newcommand{\defeq}{\vcentcolon=}

\renewcommand{\leq}{\leqslant}
\renewcommand{\geq}{\geqslant}
\newcommand{\ssubset}{%
\Subset}
\renewcommand{\phi}{\varphi}
\renewcommand{\epsilon}{\varepsilon}
\newcommand{\Fd}{$\call F$\Hyphdash*}

\newcommand{\tildee}[1]{\widetilde{#1}}

\newcommand{\ap}{\textsuperscript}

\DeclareMathOperator{\LSC}{LSC}
\DeclareMathOperator{\USC}{USC}

\DeclareMathOperator{\epi}{epi}
\DeclareMathOperator{\conv}{conv}
\DeclareMathOperator{\Diff}{Diff}
\DeclareMathOperator{\intr}{int}

\DeclareMathOperator*{\osc}{osc}
\DeclareMathOperator*{\aplimsup}{ap\,lim\,sup}
\DeclareMathOperator*{\apliminf}{ap\,lim\,inf}
\DeclareMathOperator*{\aplim}{ap\,lim}

\newcommand{\C}{\call C^+}

\newcommand{\N}{\mathbb{N}}

\newcommand{\cN}{\mathcal{N}}
\newcommand{\Sc}{\mathscr{S}}

\newcommand{\de}{\partial}
\newcommand{\di}{\mathrm{d}}
\newcommand{\R}{{\mathbb{R}}}
\newcommand{\I}{I}
\newcommand{\call}[1]{\ensuremath\mathcal{#1}}
\newcommand{\frk}[1]{\ensuremath\mathfrak{#1}}
\newcommand{\scr}[1]{\ensuremath\mathscr{#1}}
\newcommand{\barr}[1]{\ensuremath\overline{#1}}
\newcommand{\dto}{\ensuremath{\searrow}}
\newcommand{\uto}{\ensuremath{\nearrow}}

\newcommand{\Q}{Q}
\newcommand{\RR}{R}

\newcommand\restr[2]{{
  \left.\kern-\nulldelimiterspace 
  #1 
  \vphantom{|} 
  \right|_{#2} 
  }}
  
\makeatother

\setlist[enumerate]{leftmargin=2em}
\setlist[itemize]{leftmargin=2em}

\begin{document}


\date{\today} \linespread{1.1}

\let\thefootnote\relax

\title[Semiconvex functions in general potential theories]{A primer on semiconvex functions in general potential theories\footnotetext{\bf Revised version of the work \emph{A primer on quasi-convex functions in nonlinear potential theories}, with modified title, reorganization of introductory material and rewritten last chapter. Preprint version of content to be published in Springer \emph{Lecture Notes in Mathematics}, vol.\ 2371.}}
\author{Kevin R.\ Payne}
\address{Dipartimento di Matematica ``F.\ Enriques''\\ Universit\`a degli Studi di Milano\\ Via~C.~Saldini~50\\ 20133 Milano, Italy}
\email{kevin.payne@unimi.it (Kevin R. Payne)}
\author{Davide Francesco Redaelli}
\address{Dipartimento di Matematica \\ Universit\`a degli Studi di Roma  ``Tor Vergata''\\ Via della Ricerca Scientifica 1, 00133 Roma, Italy}
\email{redaelli@mat.uniroma2.it (Davide Francesco Redaelli)}



\frontmatter
\maketitle

\def\thefootnote{\arabic{footnote}}

\makeatletter
\def\l@subsection{\@tocline{2}{0pt}{2.5pc}{5pc}{}}
\makeatother

\setcounter{secnumdepth}{3}

\setcounter{tocdepth}{2}
\tableofcontents

\chapter*{Preface}

This book gives a detailed and self-contained treatment of the fundamental role that {\em semiconvex functions} play in general (nonlinear second order) potential theories. Such potential theories are concerned with the study of {\em generalized subharmonics}; that is, the class of functions whose partial derivatives up to second order are constrained to belong to a suitable closed subset (called a {\em subequation}) of the space of $2$-jets. The boundary of a subequation provides a second order partial differential equation. In any potential theory, there is a natural Dirichlet problem for {\em generalized harmonics} with prescribed boundary values. There are two natural notions of subharmonics in any potential theory: the {\em classical notion} for twice differentiable functions and the {\em viscosity notion} for upper semicontinuous functions, in which the classical second order differential at each point is replaced by the set of all {\em upper contact jets} of second order. 

The utility of the viscosity theory stems from two considerations. First, a harmonic function in a general (nonlinear) potential theory may not be twice differentiable in the classical sense. Second, the viscosity notion has many stability properties which are extremely useful in limiting constructions, such as those needed for existence results by way of Perron's method and for uniqueness results by way of the comparison principle (which is also needed for Perron's method).

Semiconvex functions are used to build a bridge between the classical and viscosity notions of solutions to the natural Dirichlet problem in any potential theory. Since semiconvex functions are (by definition) quadratic perturbations of convex functions, they inherit the differentiability properties of convex functions, such as second order differentiability (in the Peano sense) on sets of full Lebesgue measure (Alexandrov's theorem). More is true. There are deep and underused properties of their upper contact jets (essential for the viscosity theory) which makes the class of semiconvex functions a useful waystation when investigating potential theoretic results. In particular, the class of functions which are $\lambda$-semiconvex for some $\lambda \geq 0$ are the subharmonics of a certain potential theory. Moreover, as is well known, one can (locally) approximate upper semicontinuous functions which are (locally) bounded above with semicontinuous functions by way of the {\em sup-convolution}. This often allows one to transfer the wealth of information on upper contact jets for the semiconvex approximations to the upper semicontinuous limits. This deep and ``hard'' part of the viscosity theory is a central focus of this book. A brief discussion of the highlights of this analysis will be discussed in the introduction.

While general potential theories are interesting in and of themselves, the results presented here take on greater significance in light of a rapidly developing and fruitful interplay between general potential theories and the operator theory of fully nonlinear elliptic partial differential equations. Historically, research on the complex Monge-Amp\`ere equation has had a strong connection to complex pluripotential theory. Beginning with a trio of papers in 2009 by Harvey and Lawson \cite{hlpotcg, hlpluri09, hldir09}, this interplay is being applied with much success in order to embed a given partial differential operator into an associated potential theory and then to deduce properties about the operator theory from the corresponding potential theoretic properties. Moreover, these properties are shared by any operator which is {\em compatible} with the subequation that determines the potential theory. This often simplifies and unifies the operator theory. Results include, but are not limited to, the validity of the comparison principle (which implies uniqueness for the Dirichlet problem) and the application of Perron's method for existence (which makes use of the comparison principle as well as a suitable notion of {\em boundary pseudoconvexity} that has an elegant potential theoretic formulation). The reader is invited to consult \cite{chlp, cpaux, cpmain} and \cite{cprdir} as well as the organic survey presented in \cite{hpsurvey}. We will give a short history of this program at the end of the overview in the introduction. In that light, this book can be viewed as a prequel to the above mentioned works and our main aim is to provide a comprehensive presentation of the foundations on which this theory rests.

An important remark on terminology is in order. With some reservation, we have decided to use the more conventional term {\em semiconvex} for a function $u$ (defined on a convex subset of Euclidian space) such that $u + \frac{\lambda}{2} |\cdot|^2$ is a convex function  for some $\lambda \geq 0$. This is equivalent to asking that the Hessian of $u$ is bounded below (in the viscosity sense). In the works mentioned above, the less conventional term {\em quasi-convex} is used, which is consistent with the use of \emph{quasi-plurisubharmonic} in several complex analysis, as noted in~\cite{hpsurvey}. However, in convex analysis, quasi-convexity is also used to mean that inverse images of half-lines bounded above are convex sets. In an attempt to avoid confusion over a wide audience, we have adopted the more conventional term semiconvex, as used in much of the viscosity literature, such as \cite{user}. We would like to stress however that semiconvex functions are not ``half-convex'' in the same way that semicontinuous functions are ``half-continuous''. Moreover, any $C^2$ function is locally semiconvex, which shows that semiconvex functions are far from being ``half-convex''. Perhaps a different term such as ``convexable'' might be advisable, or to use \emph{Hessian bounded below} as in \cite{slod}. 

Some declarations of intent are also in order. The present work was born with the purpose of unifying and harmonizing the two articles \cite{hlqc} and \cite{hlae} of Harvey and Lawson on semiconvex functions and their use in general potential theories, which were deposited in the arXiv in 2016. These two preprints are full of foundational aspects of non-smooth analysis, some well-known, others less well-known and many new things. The authors of this work are extremely grateful to Harvey and Lawson for their encouragement to attempt this synthesis, which was the basis of the Master's Thesis of the second author, supervised by the first author. 

This book is more than a mere merging of the two preprints \cite{hlqc} and \cite{hlae}, whose combined length is less than one third of the present work. Beyond the question of length, the current presentation benefits from the numerous developments in the theory over the past decade. A discussion of what is ``new'' here can be found in the commentary section of the introduction.  

This book is aimed at a wide audience, including professional mathematicians working in fully nonlinear PDEs as well as masters and doctoral students with interest in mathematical analysis. On
the one hand, the book collects a wealth of results (some classical and some new)
on how to treat second order differentiability in a pointwise manner for functions
which are merely semicontinuous and the (viscosity) subharmonics in general potential theories. As such, we believe that much of this material should be more widely known by working mathematicians. On the other hand, much of the book has been tested in the classroom in the context of a master’s level course on the
foundations of the viscosity theory for fully nonlinear elliptic PDEs. 

We note that very little prior knowledge is necessary to read this book, which aims to be relatively self-contained. However, some familiarity with basic notions of mathematical analysis, such as the standard differential calculus of several variables, Lebesgue measure and integral, the notions of limsup and liminf are assumed. Good background references are \cite{flem} and \cite{wheedzyg}, and also \cite{evansgar} for many advanced topics. In addition, some familiarity with harmonic functions and maximum principles for elliptic equations is certainly useful, such as what one finds in \cite{evans}. For further complementary reading, other than the papers mentioned in the book, we mention the classic monographs on convexity \cite{horm} and \cite{rock}, on general potential theories \cite{chlp}, and elliptic equations \cite{gilbtrud}, \cite{cafcab} and \cite{krylov}.

The authors express their deep gratitude to Blaine Lawson and Reese Harvey for their generosity and encouragement to perform and updated and enriched synthesis of the unpublished manuscripts \cite{hlqc} and \cite{hlae}. In addition, numerous useful conversations with Reese Harvey have contributed to the development of this work. The authors are partially supported by the Gruppo Nazionale per l'Analisi
Matematica, la Probabilit\`a e le loro Applicazioni (GNAMPA) of the Istituto Nazionale di Alta Matematica
(INdAM), Italy. Redaelli is partially supported by the Fondazione Cassa di Risparmio di Padova e Rovigo, Italy and by the Italian (EU Next Generation) PRIN project 2022W58BJ5.

\vspace{\baselineskip}
\begin{flushright}\noindent
	Milano, Roma \hfill {\it Kevin R.\ Payne }\\
	March 2025\hfill {\it Davide Francesco Redaelli}\\
\end{flushright}

\chapter*{Introduction}\label{chap:intro}

As described in the preface, this book is dedicated to foundational aspects of general (nonlinear second order) potential theories and fully nonlinear elliptic PDEs. In particular, we will systematically develop the fundamental role played by semiconvex functions as a bridge between the classical and viscosity theory of generalized subharmonics determined by a given subequation (constraint set) in the bundle of $2$-jets over open subsets of Euclidian spaces. In this introduction, we will first give an overview of the book, describing the organization of the presentation, the main concepts and important results that are contained within. Then we will give a commentary, which aims to place the book in the context of the rapidly developing and fruitful interplay between the general potential theories and the operator theory of elliptic PDEs. In addition, remarks concerning what is ``new'' in this book will be presented. The reader is invited to read the preface before beginning to look at the rest of the book and may want to use this introduction as guide to consult before or after having delved into the book.

\vspace{2ex}
\begin{center}
	{\bf Overview}
\end{center}
\vspace{2ex}

The book is divided into two parts and also contains two appendices which complement the first part. In this section, we will present a synopsis of the contents. 

Part~\ref{pt1}, which consists of Chapters 1---4, is dedicated to four fundamental and complementary aspects of semiconvex analysis. The first aspect (treated in Chapter \ref{chap:diff} concerns differentiability properties of first and second order for locally semiconvex functions, including the essential  {\bf {\em Alexandrov's Theorem}}~\ref{aleks}  on the almost everywhere differentiability of second order for (locally semi-) convex functions. A complete and self-contained proof of Alexandrov's theorem is provided in Appendix~\ref{ap:legalex} and depends on two key ingredients: several important properties of the \emph{Legendre transform} and a {\em Lipschitz version of Sard's theorem}. For completeness, the highly technical proof of this last ingredient is also given in Appendix~\ref{proofsard}. 

The second aspect (treated in Chapter \ref{chap:SCUCJ}) concerns a detailed analysis of {\em upper contact jets} and {\em upper contact points} of locally semiconvex functions. This investigation culminates in two fundamental results on locally semiconvex functions at or near points of upper (quadratic) contact: the {\bf {\em Upper Contact Jet Theorem}}~\ref{ucjt} and the {\bf {\em Summand Theorem}}~\ref{pusc:sum}. The first result groups together three fundamental properties of first and second order differentiability of a locally semiconvex function at or near a point of upper (quadratic) contact. The second result produces interesting and important information about upper contact points of a pair of locally semiconvex $u$ and $v$ at or near a point of upper (quadratic) contact of the sum $w = u + v$. Interestingly, the Summand Theorem immediately yields a {\bf {\em semiconvex version of the Theorem on Sums}} (as will be noted in Corollary \ref{tosqc2}). The celebrated semicontinuous version of the Theorem on Sums is obtained after doubling variables; that is, by considering $w(x,y)\defeq  u(x) + v(y)$ and is the linchpin on which much of the conventional viscosity theory of Crandall-Ishii-Lions~\cite{user} rests. For completeness, in Section \ref{ss:tosu}, we will also give a novel proof of the semicontinuous version as a corollary of the semiconvex version and the standard device of semiconvex approximation by way of the \emph{sup-convolution} which produces a controllable displacement of the upper contact points of every given upper contact jet. Elementary properties of this approximation are discussed in Section \ref{ch:qca}. 

The third aspect (treated in Chapter \ref{chap:js}) concerns additional deep analytical results needed for a robust viscosity theory (which is based on quadratic upper test functions). One has the central problem of how to ensure \emph{a priori} the existence of a sufficient amount of upper contact points for an upper semicontinuous function. This is accomplished by combining Alexandrov's Theorem, the sup-convolution approximation of semicontinuous functions by semiconvex functions and the results presented in this chapter. The results of this Chapter ensure the existence of sets of contact points with positive Lebesgue measure nearby points of {\em strict} upper contact of locally semiconvex functions. This part of the analysis will culminate in the so-called {\bf {\em Jensen--S{\l}odkowski Theorem}}~\ref{jenslod}, which unifies and generalizes two historically distinct approaches to the central problem in a viscosity approach in potential theory and operator theory. 
The most well-known result of this type is due to Jensen \cite{jensen} and is recalled in {\bf {\em Jensen's lemma}}~\ref{u:jen}. This result says that nearby a {\em strict local maximum} point of a semiconvex function $w$ there is a set of positive measure for which small linear perturbations of $w$ still has a local maximum. An older result of S{\l}odkowski~\cite{slod}, is recalled in {\bf {\em S{\l}odkowski's density estimate}}~\ref{slodthm} and concerns a lower bound on the Lebesgue density of sublevel sets of the largest eigenvalue of the Hessian of a convex function. An important and impressive insight of Harvey and Lawson \cite{hlqc} is that if one translates these two results into the language of upper contact points, these results are equivalent. We state these reformulations as the {\bf {\em HL--Jensen Lemma}}~\ref{hl:jen} and the  {\bf {\em HL--S{\l}odkowski Lemma}}~\ref{hl:slod} and prove that they are 
equivalent in Proposition~\ref{prop:HLJS}. Moreover, these two reformulations can be merged into the aforementioned Jensen--S{\l}odkowski Theorem. This story represents one of the most interesting and unique aspects of this book. We will return to this point at the commentary section below as an illustration of parallel developments in general potential theories and nonlinear elliptic PDEs.

Chapter \ref{chap:js} also includes various consequences of the Jensen--S{\l}odkowski Theorem including the proof of the {\em partial upper semicontinuity of second derivatives} which is the most delicate part of the aforementioned Upper Contact Jet Theorem as well as a semiconvex version of {\bf {\em S{\l}odkowski's Largest Eigenvalue Theorem}}~\ref{thm:LET_qc}. Moreover, in the last section of the chapter, an alternative proof of the Jensen--S{\l}odkowski Theorem is given which only makes use of an {\bf {\em Area Theorem}}~\ref{area:cor} for gradients of locally semiconvex functions and the resulting {\bf {\em Alexandrov's maximum principle}}~\ref{alexmp}  for locally semiconvex functions.

The fourth and final apect is treated in Chapter \ref{chap:apqc} and concerns the essential role that semiconvex functions play in the viscosity theory of semicontinuous functions by way of the aforementioned {\em semiconvex approximation of upper semicontinuous functions} (see Theorem~\ref{baspropsc}). One also finds an alternative proof of the conventional {\bf {\em Theorem on Sums}} of Crandall--Ishii--Lions which makes use of the upper contact jet technology developed up to this point.

These four aspects of Part~\ref{pt1} are essential for the implementation of the viscosity theory of subharmonics and supersolutions of Part~\ref{sas}, which will be described next. We repeat the important assertion that the ingredients of Part~\ref{pt1} represent the ``hard analysis'' part of the viscosity theory. 

To facilitate the synopsis of Part~\ref{sas}, it will be useful to anticipate the definitions of a subequation (constraint set) and their subharmonics which are the basic objects in any potential theory. A {\em subequation} on an open subset $X \subset \R^n$ is a (closed) subset 
$$ \call F \subset \call J^2(X) \defeq  X \times \call J^2 \defeq  X \times \R \times \R^n \times \call S(n)
$$ 
of the bundle of $2$-jets over $X$ satisfying the properties of {\em positivity} (P), {\em negativity} (N) and {\em topological stability} (T). Properties (P) and (N) are {\em monotonicity} properties which combine to ask that for each $x \in X$ the fibers $\call F_x \defeq  \{J  \in \call J^2: (x,J) \in \call F\}$ satisfy
\begin{equation*}
J = (r,p,A) \in \call F_x \ \ \Rightarrow \ \ (r + s,p,A + P) \in \call F_x \quad \forall \, s \leq 0 \ \text{in} \ \R, \forall \, P \geq 0 \ \text{in} \ \call S(n).
\end{equation*}
When $\call F$ is defined by an operator $F \in C(\call J^2(X))$; for example, if 
$$
\call F = \{ (x,r,p,A) \in \call J^2(X): \ F(x,r,p,A) \geq 0 \},
$$
these monotonicity properties correspond to the differential operator being {\em proper elliptic}. This is the minimal monotonicity that must be assumed for the validity of the natural comparison principles associated to both $\call F$ and $F$. Property (T) helps to ensure a robust potential theory associated to a subequation $\call F$ (see Definitions~\ref{defnsubeq} and \ref{defn_prim_subeq}) including a natural notion of {\em duality} (see Chapter~\ref{chap:duality}). Given a subequation $\call F$, a function $u \in C^2(X)$ is said to be {\em \Fd subharmonic in $x \in X$} if
\begin{equation}\label{subh1_intro}
J^2_x u \defeq  (u(x), Du(x), D^2u(x)) \in \call F_x
\end{equation}
and {\em \Fd subharmonic on $X$} if \eqref{subh1_intro} holds for each $x \in X$. For merely upper semicontinuous functions $u \in \USC(X)$, the definition \eqref{subh1_intro} is taken in the {\em viscosity sense}; that is, 
\begin{equation*}
J^{2,+}_x u \subset \call F_x 
\end{equation*}
where $J^{2,+}_x u$ is the set of {\em upper contact jets for $u$ at $x$}:
$$ 
J^{2,+}_x u \defeq  \{ J^2_x \varphi: \varphi \ \text{is $C^2$ near $x$ and}\ u \leq \varphi \ \text{near $x$ with equality at $x$} \}.
$$

Part~\ref{sas}, which consists of Chapters 5---10, is dedicated to general (nonlinear) potential theories of second order and the fundamental role that (locally) semiconvex functions play in such theories. The presentation will involve four stages. Stage one consists of three chapters and is dedicated to the foundations of general potential theories. In Chapter~\ref{chap:conset} we start by deducing the minimal properties that a {\em subequation} constraint set $\call F$ in the $2$-jet bundle over a Euclidean domain needs to satisfy in order to have a well-defined and coherent potential theory of {\em \Fd subharmonic functions}. A few representative examples are given at the end of the chapter, which include both convex functions and $\lambda$-semiconvex functions. In Chapter \ref{chap:duality}, we briefly review the important notions of {\em duality} and {\em monotonicity}. Duality yields a reformulation of {\em $\call F $-superharmonics} in terms subharmonics for the dual subequation $\widetilde{\call F}$ and monotonicity is a unifying notion in the theory, with many implications. When combined, these two notions are the basis of an important method for proving the comparison principle. Finally, Chapter~\ref{chap:tools} presents basic tools for establishing that a given function is \Fd subharmonic including the {\bf {\em Bad Test Jet Lemma}}~\ref{l:btj} and the {\bf {\em Definitional Comparison Lemma}}~\ref{defcompa} as well as several often used constructions in Proposition~\ref{elemprop}.

Stage two is presented in Chapter~\ref{chap:SCSH} and treats the main theme of this work on the efficient use of semiconvex functions in nonlinear potential theories and PDEs. There are two main results which were first established in Harvey and Lawson \cite{hlae}. The first main result is the {\bf \emph{Almost Everywhere Theorem}}~\ref{aet}, which states that  for any (primitive) subequation $\call F$ on $X$, a locally semiconvex function will be \Fd subharmonic on $X$ whenever it is \Fd subharmonic on a set of full (Lebesgue) measure. When combined with Alexandrov's Theorem~\ref{aleks:qc} and the coherence property of Remark \ref{rem:coherence}, the \Fd subharmonicity for a locally semiconvex function then reduces to a classical statement almost everywhere. The second main result is the {\bf \em{Subharmonic Addition Theorem}}~\ref{add} for locally semiconvex functions. Subharmonic addition theorems are extremely useful and merit a brief explanation now. Consider a trio $\call F, \call G$ and $\call H$ of subequation constraint sets in $\call J^2(X)$ over an open set $X \subset \R^n$, which then determine spaces of subharmonic functions $\call F(X), \call G(X)$ and $\call H(X)$ on $X$. A subharmonic addition theorem concerns the following implication: from a purely algebraic statement of \emph{jet addition} in the fibers over each $x \in X$ one concludes the functional analytic statement of \emph{subharmonic addition}; that is,
\begin{equation}\label{SA1}
\call F_x + \call G_x \subset \call H_x, \ \ \forall \, x \in X \implies \call F(X)   +  \call G(X) \subset \call H(X).
\end{equation}
When the subequations have constant coefficients, the Subharmonic Addition Theorem is extended in Corollary~\ref{ccsa} to upper semicontinuous subharmonics by {\em semiconvex approximation}, which preserves subharmonicity in the constant coefficient setting (see Proposition~\ref{prop:approx}). 

Although we will not discuss this in the book, in the variable coefficient setting an additional assumption of {\em fiberegularity} of the subequations yields a suitable alternative semiconvex approximation scheme and the Subharmonic Addition Theorem extends to the variable coefficient setting (see \cite{cprdir} for details). The notion of fiberegularity is explained in the Commentary towards the end of this introduction.

Stage three is presented in Chapter~\ref{chap:comp} and concerns the use of semiconvex analysis in the proof of {\em comparison principles} for nonlinear potential theories. Two distinct settings will be discussed. The first setting (treated in Section~\ref{sec:cccc}) concerns the case of {\em sufficient monotonicity} in the sense that the subequation $\call F$ defining the potential theory admits a constant coefficient {\em monotonicity cone subequation} $\call M$. This means that there exists a subequation $\call M$, with constant fibers that are a closed convex cone in $\call J^2$ (with vertex at the origin), such that 
$$
\forall \, x \in X: \quad	\call F_x + \call M \subset \call F_x \qquad \text{($\call F$ is $\call M$-monotone)};
$$
that is, the jet addition formula of \eqref{SA1} holds with $\call G = \call M$ and $\call H = \call F$.
Duality then gives the jet addition formula
\begin{equation*}
\forall \, x \in X: \quad	\call F_x + \widetilde{\call F}_x \subset \widetilde{\call M}.
\end{equation*}
If the Subharmonic Addition Theorem holds, then the comparison principle on a domain $\Omega$ reduces to the {\em Zero Maximum Principle} for $\widetilde{\call M}$-subharmonics:
$$
w \leq 0 \ \text{on} \ \partial \Omega \implies w \leq 0 \ \text{in} \ \Omega, \quad \forall \, w \in \USC(\overline{\Omega}) \cap \widetilde{\call M}(\Omega).
$$
This is the idea of the {\em monotonicity-duality method} initiated by Harvey and Lawson in \cite{hldir09} in the constant coefficient pure second order case, which has developed into a robust and general method which can also be transported to every nonlinear degenerate elliptic PDE which is {\em compatible} with a given potential theory. The reader is invited to consult \cite{chlp, cpaux, cpmain, cprdir, hldir} for many of these developments. 

We will derive this method from first principles and then discuss the constant coefficient pure second order and gradient-free cases, which also includes the characterization of the dual cone subharmonics, which are the {\em subaffine} and {\em subaffine-plus} functions, respectively. 

The second setting, presented in Section~\ref{sec:SC}, treats the opposite situation in which there is only the {\em minimal monotonicity} that every subequation possesses. We will prove a weaker form of the comparison principle, called {\em strict comparison}. In Theorem~\ref{scqc}, we treat the variable coefficient case but require that the subharmonics are semiconvex. The proof of strict comparison in this case is a simple consequence of the Summand Theorem~\ref{pusc:sum}. On the other hand, in the constant coefficient case strict comparison also holds for semicontinuous subharmonics, as shown in Theorem~\ref{sccc}. The two approaches to strict comparison both use semiconvex approximations; one implicitly and one explicitly.

Part~\ref{sas} concludes in Chapter~\ref{sec:mflds} with the final stage which shows why the semiconvex analysis presented in the book extends from Euclidian spaces to manifolds. The key point is Theorem~\ref{BM_lqcstable} which proves that the class of locally semiconvex functions on an open subsets of $\R^n$ is stable under pull-backs via smooth changes of coordinates.

\vspace{2ex}
\begin{center}
	{\bf Commentary}
\end{center}
\vspace{2ex}

We conclude this introduction with some additional commentary of how this book fits into the rapidly developing and fruitful interplay between general potential theories and the operator theory of fully nonlinear elliptic PDEs. In particular, we will comment on and contrast the parallel developments in general potential theories and the conventional viscosity theory for fully nonlinear elliptic PDEs, including additional comments about some missed opportunities of synergy between the two realms. This will lead to a brief discussion about what is ``new'' in this primer, with respect to the two unpublished manuscripts \cite{hlqc} and \cite{hlae}. Finally, we will give a somewhat more detailed history of the progress on the interplay between general potential theories and fully nonlinear elliptic PDEs. 

As described above, the ``hard analysis part'' in treating subharmonics in general potential theories is contained in two deep results on semiconvex functions; namely Alexandrov's Theorem and the Jensen--S{\l}odkowski Theorem. These two theorems yield the Summand Theorem~\ref{pusc:sum} and the Almost Everywhere Theorem~\ref{aet}, which ensure that, for locally semiconvex functions, \emph{almost all} (Alexandrov) points admit an upper test function, and they are \emph{enough} for the viscosity-theoretic purposes. Hence, locally semiconvex functions can be seen as a potential bridge between the classical theory (twice differentiable functions) and the viscosity theory (semicontinuous functions). 

This bridge is completed by the well-known device of approximating a given upper semicontinuous function by a sequence of semiconvex functions by way of the {\em sup-convolution}, provided that one has suitable ``stability'' or ``control'' of the approximation procedure in order to close the circle and obtain results for the semicontinuous functions. This is a crucial problem, which we will address only in a few basic situations, in Part~\ref{sas}. For further investigations on the possibility of applying such a semiconvex approximation technique inside the context of a general potential theory, we invite the reader to consult ~\cite{cpaux,cpmain,cprdir}.

Of course, the conventional viscosity theory for ellipitc PDEs, as codified in \cite{user}, also makes use of the same ingredients: Alexandrov's Theorem, Jensen's Lemma and ``magical properties'' of the sup-convolution. However, the more detailed semiconvex analysis presented here can be exploited in a more efficient manner than the conventional approach, as well as providing a more regular ``test case'' for checking whether a desired result for semicontinuous function can be true. This is accomplished by first proving a (locally) semiconvex version with the help of Alexandrov's Theorem and the Almost Everywhere Theorem, and then by analyzing the sup-convolution approximation as a last step. In particular, for proving the comparison principle, one can make good use of the structure of the constraint set of a given potential theory (or structural properties of a differential operator) which need not satisfy the standard structural conditions of the conventional viscosity theory. It should be stressed that in the conventional approach to the comparison principle the most delicate step appears in the Theorem on Sums. This is a deep, important and useful general result about the local extrema of sums of upper semicontinuous functions, but in this step the structure of the PDE (or associated potential theory) is completely ignored.

Having recalled and contrasted some of the parallel developments in the viscosity theory of general potential theories and fully nonlinear elliptic PDEs, we repeat our point of view is that an increasing synergy is possible and desirable between the two realms. Historically, there are instances of lost opportunities of such a synergy. The example we will discuss at length concerns the crucial Jensen and S{\l}odkowski Lemmas. Jensen's Lemma was developed in the context of proving uniqueness for viscosity solutions of fully nonlinear elliptic PDEs~\cite{jensen}, while the earlied result of S{\l}odkowski's Lemma~\cite{slod} was developed in the context of pluripotential theory. A major aim of this book is to contribute to the many opportunities for productive interplay between these two realms. 

We now will comment on what is is ``new'' in this primer, with respect to the two unpublished manuscripts \cite{hlqc} and \cite{hlae}.  In a general sense, almost everything one finds in the two preprints of Harvey and Lawson is included here, albeit sometimes arranged in a different manner, with modified, clarified and more detailed proofs. Many additional results and remarks have been included to enrich the presentation, which also takes into account the numerous developments in general nonlinear potential theory over the past decade. 

In a more technical sense, they are many things here that are not included in \cite{hlqc} and \cite{hlae}. Some of these are well known, such as elements of convex function theory and the proof of the Lipschitz version of Sard's theorem, which have been included in order to make the primer as self-contained as reasonably possible. Others are less well known, but needed for the exposition, such as the background material on subequations and their subharmonics in Chapters \ref{chap:conset}, \ref{chap:duality} and \ref{chap:tools}. Finally, there are some genuinely ``new'' things, many of which revolve around the following consideration.

A key point in Harvey and Lawson's study of Jensen--S{\l}odkowski equivalences was to replace the spherical contact points considered by S{\l}odkowski with paraboloidal contact points which are more natural for the viscosity theory. This ``paraboloidal'' machinery is examined in detail in Section~\ref{sec:vertex_map}. In addition, using this machinery, we augment the results in \cite{hlqc}. In particular, we formulate and prove a paraboloidal version of S{\l}odkowski's upper bound on the largest eigenvalue of the Hessian (Lemma~\ref{hlKb}) and we give a paraboloidal proof of S{\l}odkowski's density estimate (Lemma~\ref{slodthm}) and S{\l}odkowski's Largest Eigenvalue Theorem~\ref{slodle}. Finally, in Section \ref{sec:JSAlex}, making use of an important observation of Harvey~\cite{har:pc}, we prove that the general Jensen--S{\l}odkowski Theorem~\ref{jenslod} can be derived by Alexandrov's maximum principle (Theorem~\ref{alexmp}), which is turn is derived from the Area formula for gradients of semiconvex functions (Theorem~\ref{area:cor}). This leads to the interesting conclusion that, in the end, the hard analysis part rests on the validity of the area formula. 

In Part~\ref{sas}, an organic presentation of the semiconvex approximation of $\call F$ subharmonics is discussed in Chapter~\ref{chap:SCSH}, which is then used to give an alternate proof of strict comparison for semicontinuous subharmonics (Theorem~\ref{sccc}) by reducing to the semiconvex version of strict comparison (Theorem~\ref{scqc}).

We conclude our commentary with a brief history of the program of using general potential theories to treat fully nonlinear elliptic PDEs. The program initiated by Harvey and Lawson was in part inspired by the geometric approach of Krylov~\cite{kr} in order to formulate a general notion of ellipticity. In the pure second order case, Krylov calls a differential operator $u \mapsto F(D^2u)$ elliptic if there exists an {\em elliptic set} for $F$; that is, an open subset $\Theta \subset \call S(n)$ such that $F$ is increasing for $A \in \Theta$ and the differential inclusion $D^2 u(x) \in \partial \Theta$ for each $x \in \Omega$  is called an \emph{elliptic branch} of the equation $F(D^2u) = 0$ if $\partial \Theta \subset \{ F = 0 \}$. One can then ``forget'' about the particular form of the operator $F$ and concentrate instead on the constraint set $\Theta$ as many operators $F$ will give rise to the same $\Theta$. Harvey and Lawson recognized similarities of this idea in pluripotential theory and starting instead from closed sets $\call F\subset \call S(n)$ (the closure of $\Theta$), formulated a suitable notion of \emph{duality} which recasts the notion of viscosity solutions in a precise topological framework and developed a \emph{monotonicity-duality} method for proving the comparison principle in general potential theories. With additional work, this yields the comparison principle for any fully nonlinear second-order partial differential equation which is compatible with a given potential theory. The comparison principle of course implies uniqueness of solutions for the associated Dirichlet problem in both contexts (potential theory and PDEs). Another major breakthrough in the Harvey-Lawson approach is to use asymptotic interiors of the constraint set to determine the correct notion of {\em boundary pseudoconvexity} for which one can construct the barriers needed to complete Perron's method for proving existence of solutions for the Dirichlet problem. 

After their pioneering paper \cite{hldir09} for pure second order constant coefficient potential theories, a sequence of ambitious papers followed, such as \cite{hldir} which studies the Dirichlet problem for fully nonlinear second-order equations on Riemannian manifolds, \cite{hlcharsmp} which characterizes the validity of the strong maximum principle, \cite{hllagma} and \cite{hlpssl} which treat the special Lagrangian potential equation, \cite{hlpotcg} on potential theory in calibrated geometry and \cite{hlidir} on the inhomogeneous Dirichlet problem for ``natural'' operators. 

However, the work of Harvey and Lawson (almost always) remains at the level of potential theory, where differential operators appear as possible generators of interesting potential theories and the operators are interesting in and of themselves when they are smooth, so that one can differentiate the differential equation. An important general question concerning the interplay between potential theory and operator theory remained; namely, how to take a given partial differential operator of interest and embed it into a potential theory. Recently, powerful versions of what we call the {\em correspondence principle} have been established. The validity of a correspondence principle involves showing that the viscosity subharmonics/superharmonics of a potential theory determined by a subequation $\call F$ correspond to the viscosity subsolutions/supersolutions of a differential operator $F$. A correspondence principle is used to embed a given (elliptic) partial differential operator into an associated potential theory and then to deduce properties of the operator theory from the corresponding potential theory. Moreover, these properties are shared by any operator which is {\em compatible} with the associated potential theory. Such correspondence principles have been established in the presence of {\em sufficient monotonicity} for compatible operator subequation pairs $(F, \call F)$ in the constant coefficient case in the monograph \cite{chlp} and for variable coefficient pairs $(F, \call F)$ under the additional assumption of {\em fiberegularity} in the paper \cite{cprdir}. For subequations $\call F$ on an open subset $X \subset \R^n$, fiberegularity just means continuity of the fiber map $\Theta : X \to \call K(X)$ defined by
$$
\Theta(x) \defeq \call F_x = \{ J \in \call J^2: (x,J) \in \call F \},
$$
where the closed subsets $\call K(X) \subset \call J^2(X)$ are equipped with the Hausdorff metric.

Inspired by \cite{hldir09}, Marco Cirant and the first author began an investigation on embedding large classes of differential operators into the potential theoretic framework of Harvey and Lawson. The variable coefficient pure second order and gradient-free cases were treated in \cite{cpaux} and \cite{cpmain} respectively, where the continuity of the fiber map $x \mapsto \mathcal{F}_x = \Theta(x)$ (now called fiberegularity) was shown to be a sufficient additional condition in order to transport the monotonicity-duality method from the constant coefficient to the variable coefficient setting. In addition, a rudimentary version of the correspondence principle was established in \cite{cpaux} and \cite{cpmain} but followed more explicitly the ideas of Krylov ((proper) elliptic maps $\Theta$ and associated branches of the PDE) as opposed to the subequation formulation. Joining forces in \cite{chlp}, a complete and robust treatment of constant coefficient potential theories and PDEs in Euclidian spaces by the monotonicity-duality method was given, including the aforementioned correspondence principle, a classification of {\em monotonicity cone subequations} and numerous illustrations of the method. Finally, the general variable coefficient case in the presence of sufficient monotonicity ({\em proper ellipticity} and {\em directionality}) and fiberegularity is treated by the authors and Cirant in \cite{cprdir}.

.

\mainmatter

\part{Semiconvex apparatus} \label{pt1}

This first part is dedicated to four fundamental and complementary aspects of semiconvex analysis for the viscosity theory of subharmonic functions in general potential theories. The first aspect concerns differentiability properties of first and second order for locally semiconvex functions, including the essential theorem of Alexandrov on the almost everywhere differentiability of second order for (locally semi-)convex functions. The second aspect concerns a detailed analysis of their upper contact jets and upper contact points, culminating in the Upper Contact Jet Theorem and the Summand Theorem. The third aspect concerns the deep analytical results which ensure the existence of sets of contact points with positive measure with the presentation of the Jensen--S{\l}odkowski Theorem and the final aspect presents the well-known device of semiconvex approximation of upper semicontinuous functions, which completes the bridge between the classical and viscosity notions of subharmonics in any potential theory. 

\vspace{2ex}

\chapter{Differentiability of convex functions} \label{chap:diff}

In this chapter, we present the well-established theory of differentiability of first and second order for convex functions and their extensions to locally semiconvex functions. This will culminate in the essential theorem of Alexandrov on second order differentiability in the Peano sense on sets of full Lebesgue measure. 

\section{First order theory}

We begin by recalling the notion of a convex function and showing that they are {\em subdifferentiable} everywhere, where subdifferentiability is a weak and, in general, multi-valued notion. Here we follow and expand upon the treatment in~\cite{hlqc}. The results that we are going to discuss actually hold in a more general context, but we limit ourselves to the context that best suits our purposes (see Remark \ref{rem:convex} below). 

\begin{definition} \label{defn:convex} \sid{convex!function}
	Let $u \colon X \to \R$ be a function defined on a convex subset $X\subset \R^n$.\syid{subset@\detokenize{$X\subset Y$}!the \emph{non-strict} inclusion of set $X$ in $Y$} We say that $u$ is \emph{convex} if its epigraph\sid{graph!epi-}\sid{epigraph|seeonly{graph, epi-}}\syid{epi@\detokenize{$\epi(u)$}!the epigraph of the function $u$} $\epi(u) \defeq \{ y \in \R : y \geq u(x) \}$ is convex.
	
	We will also say that $u$ is \emph{concave} if $-u$ is convex.\sid{concave function}
\end{definition}

\begin{remark}[Reduction to the one-dimensional notion] \label{rmk:convexpd}
	It is immediate to see that convexity of $u$ on $X \subset \R^n$ convex is equivalent to the convexity of the restriction of $u$ to every segment  \syid{segm@\detokenize{$[x,y]$}!the linear segment $\{ x+t(y-x) \}_{t \in [0,1]}$ with endpoints $x$ and $y$}
	\[
	[x_1, x_2] \defeq \{ x \in \R^n :\, x = t x_1 + (1-t)x_2,\  t\in [0,1]\}, \quad x_1, x_2 \in X.
	\]
This translates into the standard inequality formulation for convexity; that is, $u$ is convex in the sense of \Cref{defn:convex} if and only if for each $x_1, x_2 \in X$ one has
	\begin{equation}\label{Jensen_2}
	u(t x_1 + (1-t)x_2) \leq t u(x_1) + (1-t) u(x_2)\quad  \text{for each} \  t \in [0,1].
	\end{equation}
	 Moreover, the inequality \eqref{Jensen_2} is a version of {\em Jensen's inequality} for convex functions.
\end{remark}

\begin{lem}[Jensen's inequality\sid{Jensen's inequality}; finite form] \label{lem:Jensen} Let $u: X \to \R$ be convex on $X \subset \R^n$ convex. Then for each collection of points $\{x_j\}_{j=1}^N \subset X$ with $N \geq 2$ 
	\begin{equation}\label{Jensen_N}
	u \biggl(\, \sum_{j=1}^N t_j x_j \!\biggr) \leq  \sum_{j=1}^N t_j u(x_j), \quad  \text{for each}  \ \{t_j\}_{j=1}^N  \subset [0,1] \ \text{with} \ \sum_{j=1}^N t_j = 1.
	\end{equation}	
\end{lem}
The proof of Lemma \ref{lem:Jensen} is standard and typically done by induction on $N \geq 2$, where the case $N = 2$ is \eqref{Jensen_2}.

We also note that the notion of convexity can be localized. 

\begin{definition} \label{defn:locally_convex} \sid{locally convex function|seeonly{convex, function, locally}}  \sid{convex!function!locally}
	A function $u \colon X \to \R$ on an open subset $X \subset \R^n$ is said to be \emph{locally convex} if for every $x \in X$ there exists a ball $B \subset X$ centered in $x$ on which $u$ is convex. 
	\end{definition}

\begin{lem}[Localization]\label{lem:localization}
Let $u \colon X \to \R$ be a function defined on $X$ open and convex. Then
	\begin{equation*}
	\mbox{$u$ is convex on $X\ \Leftrightarrow \ u$ is locally convex on $X$}.
	\end{equation*}	
	\end{lem}
The forward implication is obvious and the reverse implication can be proved by exploiting the restriction of $u$ to arbitrary segments $[x_1,x_2] \subset X$ of \Cref{rmk:convexpd}.

We recall now a few fundamental examples of convex functions; examples which can be localized in light of \Cref{lem:localization}

\begin{remark}[Examples of convex functions]\label{rmk:obfc} One easily verifies the following facts:
		\begin{enumerate}[label=(\alph*)]
		\item a quadratic function $u$ is convex if and only if its Hessian $D^2u$ is a non-negative definite matrix; in particular, every affine function is convex;
		\item conical combinations (weighted sums with non-negative coefficients) of a finite family of convex functions are convex;
		\item the (pointwise) supremum  of any family of convex functions is convex.	
		\end{enumerate}
\end{remark}

We will see that convex functions have the following multi-valued differentiability property at each point in their domain.

\begin{definition} \label{def:subd} \sid{subdifferential}
	Let $u \colon X \to \R$ be a (not necessarily convex) function defined on a subset $X\subset\R^n$. We define its \emph{subdifferential at $x\in X$} as the set \syid{prod@\detokenize{$\langle{x,y}\rangle$}!the Euclidean inner product between vectors $x$ and $y$}
	\begin{equation} \label{defsubd} 
	\de u(x) \defeq \{ p \in \R^n :\ u(y) \geq u(x) + \pair{p}{y-x}\ \; \forall y \in X \}.
	\end{equation}
	Each $p \in \de u(x)$ will be called a \emph{subgradient} of $u$ at $x$, and we will say that {\em $u$ is subdifferentiable at $x$} if $\de u(x) \neq \emptyset$.
\end{definition}

\begin{remark} \label{geointerpsubdiff}
	Geometrically, this means that the hyperplane which is the graph of the affine function $u(x) + \pair{p}{\cdot - x}$ over $X$ is a supporting hyperplane from below at $x$ for the epigraph of $u$, for all $p \in \de u(x)$. One can also say that if $u$ is subdifferentiable at $x$, then $x$ belongs to the {\em (global) lower contact set of $u$}, using a terminology which will be introduced in \Cref{chap:SCUCJ}. Moreover, one should note the simple but useful fact that $x$ is a global minimum point for $u$ if and only if $0 \in \de u(x)$. In particular, all functions are subdifferentiable at global minima.
\end{remark}

\begin{remark} \syid{d@\detokenize{$\partial u$}!the subdifferential of the function $u$}
	We have essentially two ways to interpret the \emph{subdifferential} of $u$: as a map\syid{Part@\detokenize{$\scr P(X)$}!the power set of the set $X$}
	\[
	\de u \colon X \to \scr P(\R^n), \quad x \mapsto \de u(x),
	\]
	or as a subset of a trivial bundle over $X$ whose fiber over each $x \in X$ is $\de u(x)$; that is,
	\[
	\de u \defeq \bigsqcup_{x \in X} \de u(x) \subset X \times \R^n.
	\]
	Note that in the latter interpretation, if one considers the projection $\pi \colon \de u \to X$ which maps $\de u(x) \ni p \mapsto x$, then adopting the former interpretation means that we are calling $\pi^{-1} = \de u$.  We will mainly adopt the former interpretation; for example, when using the notation
	\[
	\de u(\Omega) = \bigcup_{x \in \Omega} \de u (x) \subset \R^n, \quad \text{$\Omega\subset X$},
	\]
	yet we will also write
	\[
	(x,p) \in \de u \quad\iff\quad p \in \de u(x),
	\]
	thus employing the latter interpretation.
\end{remark}

Convex functions $u$ on $X$ (open and convex) are {\em subdifferentiable} in the sense of \Cref{def:subd}. Moreover, this property characterizes the convexity of $u$.

\begin{prop}[Subdifferentiability of convex funtions]\label{prop:subd_convex}
	For any function $u: X \to \R$ defined on an open and convex subset $X\subset \R^n$, one has 
	\begin{equation*}
	u \ \text{is convex} \ \ \Leftrightarrow \ \ \partial u(x) \neq \emptyset \ \ \forall \, x \in X;
	\end{equation*}
	that is, $u$ is convex on $X$ if and only if $u$ is subdifferentiable on $X$.
\end{prop}

\begin{proof}
	If $u$ is convex, then by the Hahn--Banach Theorem for any $x \in X$ there exists a supporting hyperplane for the convex set $C \defeq \epi(u)$ at the point $(x,u(x))$; that is, there exists $(q,r) \in \R^n \times \R \setminus \{(0,0)\}$ such that
	\begin{equation} \label{subdiffnonempty:hb}
	\pair{q}{y-x} + r(z-u(x)) \geq 0 \quad \text{for each $(y,z) \in C$.}
	\end{equation}
	With the choices $y=x$ and $z > u(x)$, the inequality \eqref{subdiffnonempty:hb} implies that $r \geq 0$. By choosing now $z = u(x)$, one has that
	\[
	\pair{q}{y-x} + r(u(y)-u(x)) \geq 0 \quad \forall\, y \in X,
	\]
	which yields $-r^{-1}q \in \de u(x)$, provided that $r \neq 0$. To show that in fact $r\neq 0$, notice that otherwise by inequality \eqref{subdiffnonempty:hb} one would have $\pair{q}{y-x} \geq 0$ for each $y \in X$, which is impossible since $X$ is open.
	
	Conversely, if $x = \lambda x_1 + (1-\lambda) x_2$ for some $\lambda \in [0,1]$, we have
	\[
	u(x_1) \geq u(x) + (1-\lambda)\pair{p}{x_1-x_2} \quad \text{and} \quad u(x_2) \geq u(x) - \lambda \pair{p}{x_1-x_2}
	\] 
	for every $p\in \de u(x)$, thus yielding $u(x) \leq \lambda u(x_1) + (1-\lambda)u(x_2)$; that is, $u$ is convex.  
\end{proof}

Convex functions are also locally bounded.

\begin{lem}\label{lem:convex_loc_bdd} Let $u: X \to \R$ be convex on $X$ open and convex. Then 
	\begin{equation*}
	\mbox{$u$ is bounded on every compact subset $K \subset X$,}
	\end{equation*}
	or equivalently,
	\begin{equation*}
	\mbox{$u$ is bounded on some neighborhood of every $x \in X$.}
	\end{equation*}
\end{lem}

\begin{proof}
	Let $K \subset X$ be compact. By choosing any $x \in K$
	and any $p \in \partial u(x) \neq \emptyset$, one has 
	\[
	u(y) \geq u(x) + \langle p,y-x \rangle \defeq  a_x(y), \ \ \text{for every} \ y \in X,
	\]  
	where $a_x$ is affine and hence bounded from below on $K$ compact. Hence $u$ is bounded from below on $K$.
	
	For the local boundedness from above, let $x \in X$ be arbitrary and consider any compact cube $\mathcal{C}$ centered in $x$ such that $\mathcal{C} \subset X$. We denote by $\{v_j\}_{j=1}^{N}$ the collection of vertices of $\mathcal{C}$. Every closed cube is the (closed) convex hull of its vertices and hence for each $y \in \mathcal{C}$ one has
	\[
	y = \sum_{j=1}^N t_j v_j \quad \text{with} \ \ \sum_{j=1}^N t_j = 1 \ \ \text{and each} \  t_j \in [0,1].
	\]
	By Jensen's inequality \eqref{Jensen_N} we have
	\[
	u(y) = u \biggl(\, \sum_{j=1}^N t_j v_j  \biggr) \leq \sum_{j=1}^N t_j  u(v_j) \leq \max_{1 \leq j \leq N} \lvert u(v_j) \rvert < + \infty,
	\]
	and hence $u$ is bounded from above on every compact cube contained in $X$.
\end{proof}

Two interesting properties concerning the subifferential of a convex function are given in the following lemma. 

\begin{lem}\label{lem:intpropsubdiff}
	For each convex function $u: X \to \R$, one has:
	\begin{itemize}
		\item[(a)]	the fiber $\de u(x) \subset \R^n$ is non-empty, closed and convex  for each $x \in X$;
		\item[(b)]	for each pair of points $x$ and $y$ in $X$,
		\begin{equation} \label{slope}
		\pair{p}{y-x} \leq u(y) - u(x) \leq \pair{q}{y-x} \quad \forall p\in \de u(x),\ q\in \de u(y),
		\end{equation}
	\end{itemize}
\end{lem}

\begin{proof}
	For the claims of part (a), it is easy to see that
	\[
	\de u(x) = \bigcap_{y \in X} \big\{ p \in \R^n: u(x) + \pair{p}{y-x} \leq u(y) \big\} 
	\]
	is the intersection of closed affine half-spaces and therefore it is closed and convex. This follows from the fact that intersections of convex sets are clearly convex. Notice that the above argument applies to any function $u$; the convexity ensures that $\de u(x)$ is not empty.
	
	For the chain of inequalities of part (b), by \Cref{prop:subd_convex}, the subdifferentials $\partial u(x)$ and $\partial u(y)$ are non-empty. By the definition (\ref{defsubd}), for each fixed $x \in X$ one has
	\begin{equation*} 
	u(y) \geq u(x) + \pair{p}{y-x}  \quad \forall p\in \de u(x), \ \forall \, y \in X 
	\end{equation*}
	and for each fixed $y \in X$ one has
	\begin{equation*} 
	u(x) \geq u(y) + \pair{q}{x-y}  \quad \forall q\in \de u(y), \ \forall \, x \in X,
	\end{equation*}
	from which \eqref{slope} easily follows.
\end{proof}

We will see that the chain of inequalities in \eqref{slope} is extremely useful. Notice that the geometric meaning of (\ref{slope}) is that if we restrict our attention to the segment $[x,y]$ with intrinsic real coordinate such that $x\leq y$, then $\barr p \leq s \leq \barr q$, where $\barr p$, $\barr q$ are the slopes of the supporting hyperplane at $x,y$ identified by $p,q$ and $s$ is the slope of the chord connecting $(x, u(x))$ and $(y,u(y))$. An immediate consequence of \eqref{slope} is the following monotonicity property.

\begin{prop} \label{subdmon}
	The subdifferential $\de u \colon X \to \scr P(\R^n)$ is a monotone operator\sid{operator!monotone}; that is,
	\begin{equation} \label{eq:subdmon}
	\pair{q-p}{y-x} \geq 0 \qquad \forall \, (x,p), (y,q) \in \de u.
	\end{equation}
\end{prop}

This generalizes the definition of (nondecreasing) monotonicity for single-valued functions from $\R$ to $\R$, since in that case \eqref{eq:subdmon} is equivalent to $(\de u(y) - \de u(x))\cdot(y-x) \geq 0$, that is $\de u(x) \leq \de u(y)$ whenever $x\leq y$.

\begin{remark}[Monotonicity of the subdifferential] \label{equiv:subdmonconv}
	Since the chain of inequalities \eqref{slope} directly follows from the subdifferentiability of \Cref{def:subd}, \Cref{subdmon} tells us that the subdifferential $\de u$ of \emph{any} function $u \colon X \to \R$ is monotone, in the sense that \eqref{eq:subdmon} holds whenever $x,y$ are such that $\de u(x) \neq \emptyset \neq \de u(y)$. In general, this is of no use, since $\de u$ can be empty in ``too many'' (possibly all) points. However, if  $u$ is convex on $X$ open and convex then \Cref{lem:intpropsubdiff} shows that the inequality \eqref{eq:subdmon} is meaningful for each pair of points $x,y \in X$, since the domain of $\de u$ is the whole $X$; that is, $\mathrm{dom}\,\de u \defeq \{ x \in X : \de u(x) \neq \emptyset \} = X$. Furthermore, for differentiable functions $u \colon \R \to \R$ one knows that convexity is equivalent to the (nondecreasing) monotonicity of $Du = u'$, which is in fact equivalent to the monotonicity of $\de u$ in the sense of \eqref{eq:subdmon} (cf.~also  \Cref{deusingv} below). Hence, one might wonder whether a characterization of convexity for generic functions exists in term of the monotonicity of their subdifferential. We have essentially just commented on the fact that the answer is no. Indeed, all functions have monotone (possibly empty) subdifferential. Moreover, adding a hypothesis like $\mathrm{dom}\,\de u \neq \emptyset$ would also be too weak; one needs to assume that $\mathrm{dom}\,\de u = X$. In fact, consider $u$ defined by $u \defeq f$ on $X \setminus \{x_0\}$ and $u(x_0) \defeq \alpha > f(x_0)$ for some convex function $f$ on $X$. Note that $\mathrm{dom}\,\de u = X \setminus \{ x_0 \}$ but $u$ is not convex. On the other hand, the assumption $\mathrm{dom}\,\de u = X$ is too strong because it is in fact itself equivalent to the convexity of $u$. In any case, for the sake of completeness, we mention that a characterization of convexity in terms of the monotonicity of ``a subdifferential'' actually exists, and uses the so-called \emph{proximal subdifferential} (see \cite[Theorem 4.1]{csw} and the references therein).
\end{remark}

For future use, we prove a particular case of the equivalence discussed in \Cref{equiv:subdmonconv}.

\begin{prop} \label{conviffmon}
	Let $u$ be differentiable on an open convex set $X\subset \R^n$; then \syid{Der@$Du$!the first-order derivative of $u$, also the gradient/Jacobian of $u$}
	\[
	\text{$u$ is convex} \quad \iff \quad \text{$Du$ is monotone}.\footnote{That is, $\pair{Du(x) - Du(y)}{x-y} \geq 0$ for all $x,y \in X$.}
	\]
\end{prop}

\begin{proof}
	The implication ($\Longrightarrow$) immediately follows from the monotonicity of the subdifferential (\Cref{subdmon}) and the upcoming \Cref{deusingv} that characterizes convex and differentiable functions. Alternatively, a direct proof is also instructive: the convexity of $u$ is equivalent to the convexity of $u$ restricted to every segment $[x,y] \subset X$ convex with arbitrary endpoints $x,y \in X$; that is, to the convexity of all functions of the form
	\[
	\tilde u(t) \defeq u(x+t(y-x)), \quad t \in [0,1],
	\]
	with $x,y \in X$. The convexity of each $\tilde u$ is equivalent to
	\begin{equation} \label{convonline}
	\frac{\tilde u(t_2) - \tilde u(t_1)}{t_2-t_1} \leq \frac{\tilde u(t_3) - \tilde u(t_2)}{t_3-t_2} \quad \text{whenever $0 \leq t_1 < t_2 < t_3 \leq 1$}.
	\end{equation}
	Letting $t_2 \dto t_1 = 0$ one obtains
	\begin{equation} \label{monder:1}
	\pair{Du(x)}{y-x} \leq \frac{u(x+t(y-x)) - u(x)}{t} \quad \forall t \in (0,1],
	\end{equation}
	while letting $t_2 \uto t_3 = 1$ one obtains
	\begin{equation} \label{monder:2}
	\frac{u(y) - u(x+t(y-x))}{1-t} \leq \pair{Du(y)}{y-x} \quad \forall t \in [0,1).
	\end{equation}
	Combining \eqref{convonline}, \eqref{monder:1} and \eqref{monder:2} yields
	\[
	\pair{Du(x)}{y-x} \leq \pair{Du(y)}{y-x},
	\]
	which is the monotonicity of $Du$.
	
	For the reverse implication ($\Longleftarrow$), by the Mean Value Theorem, for $j=1,2$ there exist $\xi_j \in [x+t_j(y-x), x+t_{j+1}(y-x)]$ such that $\tilde u(t_{j+1}) - \tilde u(t_j) = \pair{Du(\xi_j)}{(t_{j+1} - t_j)(y-x)}$, and hence condition \eqref{convonline} holds if and only if
	\begin{equation} \label{quasimon}
	\pair{Du(\xi_1)}{y-x} \leq \pair{Du(\xi_2)}{y-x},
	\end{equation}
	for some
	\[
	\xi_1 = x + s_1(y-x), \ \ \xi_2 = x+ s_2(y-x), \quad 0 < s_1 < s_2 < 1.
	\]
	Therefore we have
	\[
	\xi_2 - \xi_1 = (s_2 - s_1)(y-x),
	\]
	and we can rewrite \eqref{quasimon} as
	\[
	(s_2-s_1)^{-1} \pair{Du(\xi_2) - Du(\xi_1)}{\xi_2 - \xi_1} \geq 0,
	\]
	which is satisfied by the monotonicity of $Du$ (since $s_2 - s_1 > 0$). This concludes the proof.
\end{proof}

Additional important consequences of the chain of inequalities \eqref{slope} concern the restrictions of convex functions to compact subsets.

\begin{thm}\label{thm:convex_Lip} Let $u: X \to \R$ be convex on $X \subset \R^n$ open and convex. Then for each $K \subset X$ compact
	\begin{itemize}
		\item[(a)] $\de u(K) = \bigcup_{x \in K} \de u (x)$ is compact in $\R^n$;
		\item[(b)] $u \in \rm{Lip}(K)$ \syid{Lip@\detokenize{$\rm{Lip}(X)$}!the space of all Lipschitz functions on the set $X$} with \syid{abs@\detokenize{$\vert x\vert $}!if $x$ is a vector, the Euclidean norm of $x$}
		\begin{equation}\label{convex_Lip}
		|u(x) - u(y)| \leq \max_{p \in \partial u(K)} \!|p| \cdot |x - y|, \ \ \forall \, x,y \in K.
		\end{equation}
		In particular, $u$ is continuous on $X$.
	\end{itemize}
\end{thm}

\begin{proof}
	First, we prove that $\de u(K)$ is bounded. Since $K$ is compact in $X$ there exists  $\delta > 0$ such that $K_{\delta}\defeq  \{ x \in X: \ d(x,K) \leq \delta \} \subset X$, where
	$K_{\delta}$ is compact. By Lemma \ref{lem:convex_loc_bdd}, $u$ is bounded on $K_{\delta}$ and by \eqref{slope} one has
	\begin{equation}\label{bounded1}
	\pair{p}{y-x} \leq u(y) - u(x) \leq 2 \max_{K_{\delta}}|u| < + \infty, \ \ \forall \, x \in K,\,  y \in K_{\delta},\,  p \in \partial u(x).
	\end{equation}
	With $y\defeq  x + \delta p/{|p|}$ for $p\neq 0$ in \eqref{bounded1}, one has
	$$
	\frac{\delta}{|p|} \pair{p}{p} \leq 2 \max_{K_{\delta}}|u|,
	$$
	and hence
	$$
	|p| \leq \frac{2}{\delta} \max_{K_{\delta}}|u| \quad \forall \, p \in \partial u(K),
	$$
	which gives $\de u(K)$ is bounded in $\R^n$.
	
	Next, we prove the Lipschitz estimate \eqref{convex_Lip}. Again using the inequality chain \eqref{slope} we have
	$$
	\forall x,y \in K: \quad |u(x) - u(y)| \leq \max \{|p|, |q| \} \, |x-y|, \quad \forall \, p \in \partial u(x), q \in \partial u(y),
	$$
	which implies \eqref{convex_Lip}, and hence the continuity of $u$ on $X$.
	
	Finally, we show that $\partial u(K) \subset \R^n$ is closed, which together with the boundedness of $\partial u(K)$ shows that $\partial u(K)$ is compact. Let $\{ p_k\}_{k \in \N}  \subset \partial u(K)$ be any convergent sequence with $p_k \to p$ in $\R^n$ for $k \to \infty$. We need to show that $p \in \partial u(K)$. For each $k \in \N$, $p_k \in \partial u(K)$ implies that there exists $x_k \in K$ such that $(x_k, p_k) \in \partial u$. Since $K$ is compact, there exist $x \in K$ and a subsequence $\{x_{k_j}\}_{j \in \N}$ such that $x_{k_j} \to x$ as $j \to \infty$. Since $p_{k_j} \in \partial u(x_{k_j})$ for each $j \in \N$, we have 
	\[
	u(y) \geq u(x_{k_j}) + \pair{ p_{k_j}}{ y - x_{k_j} } \quad \forall \, y \in X,\ \forall \, j \in \N
	\] 
	and passing to the limit $j \to \infty$ gives $u(y) \geq u(x) + \pair{p}{ y - x}$ 	for all $y \in X$, and hence $p \in \partial u(x) \subset \partial u(K)$.
\end{proof}

\begin{remark}
	In the end, for $u: X \to \R$ convex on $X$ open and convex, the following are equivalent:
	\begin{enumerate}[label=(\arabic*)]
		\item[(1)] $u$ is locally Lipschitz on $X$;
		\item[(2)] $u$ is continuous on $X$;
		\item[(3)] $u$ is continuous at some point $x \in X$;
		\item[(4)] $u$ is locally bounded on $X$;
		\item[(5)] $u$ is locally bounded from above on $X$.
	\end{enumerate}
	In fact, the implications $(1) \Rightarrow (2) \Rightarrow (3) \Rightarrow (5)$ and $(2) \Rightarrow (4) \Rightarrow (5)$ are obvious. Hence it is enough to show that $(5) \Rightarrow (1)$, which follows from the considerations above. 
\end{remark}

An important consequence on the differentiability of a convex function follows. 

\begin{cor} \label{deusingv}
	For $u: X \to \R$ convex on $X$ open and convex, one has \syid{C1@\detokenize{$C^1(X)$}!the space of all continuous functions with continuous gradient on $X$} 
	\begin{equation}\label{equivalences}
	u \ \text{is differentiable on $X$} \ \ \Leftrightarrow \ \ \de u \ \text{is single-valued on $X$}  \ \ \Leftrightarrow \ \  u\ \text{is $ C^1(X)$},
	\end{equation}
	with $\de u (x) = Du (x)$ for each $x \in X$.
\end{cor}

\begin{proof}
	We first show that if $\de u(x) = \{p\}$ then $u$ is differentiable in $x$ with $Du(x) = p$. Consider a nested sequence of compact sets $\{K_j\}_{j \in \N}$ such that $\bigcap_{j \in \N} K_j = \{x\}$, then by compactness
	\[
	\{p\} = \bigcap_{j \in \mathbb{N}} \de u(K_j),
	\]
	and thus
	\begin{equation} \label{contconv}
	p = \lim_{\substack{y\, \to\, x \\ q\, \in\, \de u(y)}} q.
	\end{equation}
	By (\ref{slope}) we have for each $x,y \in X$ with $x \neq y$:
	\[
	0 \leq \frac{ u(y)-u(x) - \pair{p}{y-x} }{ |y-x| } \leq \pair{q-p}{e},
	\]
	where $e \defeq (y-x)/|y-x| \in \mathbb{S}^{n-1}$.\syid{Sph@\detokenize{$\mathbb{S}^{n-1}$}!the unit sphere in $\R^n$} Hence for $y \to x$, \eqref{contconv} shows that $u$ is differentiable in $x$ with $Du(x) = p$. This also completes the implication that $\de u$ being single valued implies that $u$ is differentiable everywhere. 
	
	Conversely, if $u$ is differentiable at $x$, let $p \in \de u(x) \neq \emptyset$. Again by (\ref{slope}) we have $t\pair{p}{e} \leq u(x+te) -u(x)$ for all $e \in \mathbb{S}^{n-1}$ and $t$ small. This implies that $\pair pe \leq \pair{Du(x)}e$, which, by linearity, forces $p = Du(x)$. This completes the first equivalence in \eqref{equivalences}.
	
	Finally, we show that $u$ convex is differentiable on $X \ \ \Leftrightarrow \ \ u \in C^1(X)$. The implication $(\Leftarrow)$ is obvious and the implication $(\Rightarrow)$ follows from (\ref{contconv}) since it now gives
	\[
	Du(x) = \lim_{y \to x} Du(y). \qedhere
	\]
\end{proof}

Since $u$ convex is locally Lipschitz by Theorem \ref{thm:convex_Lip}(b), one can deduce first-order differentiability properties of convex functions, by exploiting the following famous result of Rademacher~\cite{rademacher}.

\begin{thm}[Rademacher] \label{rade} \sid{Theorem!Rademacher}
	Let $\Omega\subset \R^n$ be open and $G\colon \Omega \to \R^m$ be Lipschitz continuous. Then $G$ is differentiable almost everywhere in $\Omega$, with respect to the Lebesgue measure.
\end{thm}

The unidimensional version of Rademacher's theorem can be proved as a corollary of the analogous result for monotone functions, which can be extended to functions of bounded variation thanks to Jordan's decomposition, and thus to the subspace of Lipschitz continuous functions (see e.g.\ \cite{evansgar,anr}). The passage to higher dimensions essentially uses the Fubini--Tonelli theorem (see e.g.\ \cite[Theorem~3.1.6]{fed:geo} or \cite[Theorem~2.2.4]{meast}).

An immediate consequence of Theorem \ref{thm:convex_Lip}(b) (convex functions are locally Lipschitz) and Rademacher's \Cref{rade} is the following.

\begin{cor} \label{rade:conv}
	A convex function is differentiable almost everywhere, with respect to the Lebesgue measure.
\end{cor}

This is the first-order counterpart of Alexandrov's theorem on second-order differentiability of convex functions (\Cref{aleks}), which we are going to prove in \Cref{sec:so}.

\medskip
Finally, we conclude this section by highlighting two properties about sums and differences of convex functions, when one of them is quadratic. We note that it is false in general that $\de (u+v) = \de u + \de v$; nevertheless, the equality holds if, for example, $v = \phi$ is a quadratic function. 

\begin{lem} \label{lem:subdiffsum}
	Let $u$ be convex on a convex open set $X \subset \R^n$ and let $\phi$ be a quadratic convex function; then
	\begin{equation} \label{eq:equalitysubdiff}
	\de(u + \phi) = \de u + D\phi \quad \text{on $X$}.
	\end{equation}
\end{lem}

\begin{proof}
	Fix $x \in X$. By the definition \eqref{defsubd} of the subdifferential we always have $\de (u+ \phi)(x) \supset \de u(x) + \de \phi(x)$, and since $\phi$ is smooth by \Cref{deusingv} we know that $\de \phi(x) = \{D\phi(x)\}$. So let
	\begin{equation} \label{barpinde}
	\barr p \in \de(u+\phi)(x)
	\end{equation}
	(which is nonempty by \Cref{prop:subd_convex}), and set $p \defeq \barr p - D\phi(x)$. We want to show that $p \in \de u(x)$. Since one can write 
	$$
	\phi(\cdot) = \phi(x) + \pair{D\phi(x)}{\cdot - x} + \tfrac12\pair{P(\cdot-x)}{\cdot-x},
	$$
	where $P\geq 0$ is such that $D^2\phi \equiv P$ on $X$, we see that \eqref{barpinde} is equivalent to
	\[
	u(\cdot) \geq \psi(\cdot) \defeq u(x) + \pair p{\cdot - x} - \tfrac12\pair{P (\cdot -x)}{\cdot - x}.
	\]
	Hence $\epi(u) \subset \epi(\psi)$. Now, any dilation $\rho_t$ by $t>0$ centered at $(x, u(x))$ preserves this inclusion; that is,
	\begin{equation}\label{sum_quad1}
	\rho_t(\epi(u)) \subset \rho_t(\epi(\psi)), \quad \forall \, t > 0.
	\end{equation}
	By the convexity of $u$ we also have 
	\begin{equation}\label{sum_quad2}
	\epi(u) \subset \rho_t\epi(u), \quad \forall \, t \geq 1. 
	\end{equation}
	Combining \eqref{sum_quad1} and \eqref{sum_quad2}, for all $t\geq 1$ one has $\epi(u) \subset \rho_t\epi(\psi)$ where 
	$$
	\rho_t\epi(\psi) = (x,u(x)) + t\,\big\{ (z-x,y-u(x)) \in (X-x) \times \R :\ y \geq \psi(z) \big\} 
	$$
	which can be written as
	$$
	\big\{ (z',y') \in (t(X-x)+x) \times \R :\ y' \geq u(x) + \pair{p}{z'-x} - \tfrac{1}{2t}\pair{P(z'-x)}{z'-x} \big\}.
	$$
	Letting $t \to +\infty$ shows that the dilations of the epigraph of $\psi$ decrease to the half-space $H$ which is the closed epigraph of the affine function $u(x) + \pair{p}{\cdot-x}$. Hence $\epi(u) \subset H$, yielding $p \in \de u(x)$.
\end{proof}

\begin{remark}
	One might wonder if an analogous result could hold for either $u$ or $\phi$ not necessarily convex. The answer is no, in general. For example, if $u = \psi$ is a quadratic convex function and $\phi = -2\psi$, then $\de(u+\phi)(x)$ is empty for each $x \in X$, while $\de u(x) + D\phi(x)$ is not.
	
	On the other hand, the answer is yes if both $u$ and $\phi$ are not convex (though one gets the useless identity $\emptyset = \emptyset$). In that case the right-hand side of \eqref{eq:equalitysubdiff} is clearly empty, and so is the left-hand side. Indeed, if it were nonempty, then $u+\phi$ would be convex (by \Cref{prop:subd_convex}); since $\phi$ is supposed to be smooth and not convex near $x$, it is (strictly) concave near $x$, yielding $u = (u+\phi) + (-\phi)$ convex near $x$ (being it the sum of two convex functions), thus contradicting the assumption that $u$ is not convex.
\end{remark}

In addition to the result of \Cref{lem:subdiffsum} on sums of a convex function $u$ and a quadratic convex function $\phi$, one has an interesting result about differences as well. It states that convex functions $u$ are differentiable at points $x$ of {\em quadratic upper contact} and is a precursor to the results on {\em upper contact jets} in Section~\ref{qcfaj} which play a key role in all that will follow.

\begin{lem}[Differentiability at points of quadratic upper contact] \label{ch1:datucp}
	Let $u$ be convex on $X$ (open and convex) and suppose there exists a quadratic function $\phi$ such that
\begin{equation}\label{QUTF}
	(u- \phi)(y) \leq 0 \qquad \text{$\forall y$ near $x \in X$, with equality at $x$}.
	\end{equation}
	Then $u$ is differentiable at $x$ and $Du(x) = D\phi(x)$.
\end{lem}

\begin{proof}
	Since $u$ is convex, $\de u(x)$ is nonempty and hence there exists $q\in\R^n$ such that $u(x) + \pair q{\cdot-x} \leq u(\cdot)$ on $X$. Consider now the convex function $\tilde u(\cdot) \defeq u(\cdot) - u(x) - \pair q{\cdot-x}$ which satisfies $\tilde u \geq 0$ on $X$ and $\tilde u(x) = 0$.
	Hence $0\in\de\tilde u(x)$ and then, since \syid{Der2@$D^2u$! the second-order derivative of $u$, also the Hessian of $u$}
	\begin{equation} \label{phiexp}
	\phi = \phi(x) + \pair{D\phi(x)}{\cdot-x} + \tfrac12\pair{D^2\phi(x)(\cdot-x)}{\cdot-x}, \qquad \phi(x) = u(x),
	\end{equation}
	by defining $\tilde\phi(\cdot) \defeq \phi(\cdot) - \phi(x) - \pair q{\cdot-x}$ one has
	\[
	(\tilde u- \tilde \phi)(y) \leq 0 \qquad \text{$\forall y$ near $x$, with equality at $x$.}
	\] 
	It follows that
	\[
	0 \leq \tilde u(y) \leq \pair{D\phi(x)-q}{y-x} + o(|y-x|)	\qquad	\forall \, y\ \text{near}\ x,
	\]
	and thus $q=D\phi(x)$. Indeed, if that were not the case, by choosing $y=x+t(D\phi(x)-q)$, for sufficiently small $t\in \R$, one has $0 \leq t + o(t)$ for each $t$ near $0$, which is false since $t$ can change sign.
	Finally, since  
	$0\leq \tilde u(y) \leq o(|y-x|)$ for all $y$ near $x$, one sees that
	\[
	\frac{|\tilde u(y)|}{|y-x|} = o(1) 	\qquad	\text{as}\ y\to x,
	\]
	and hence $\tilde u$ is differentiable at $x$ with $D\tilde u(x) = 0$, which shows that $u$ is differentiable at $x$ with $Du(x) = D\phi(x)$. Alternatively, since our argument shows that $\de u(x) = \{ D\phi(x) \}$, one knows that $u$ is differentiable at $x$ with $Du(x) = D\phi(x)$ by \Cref{deusingv}.
\end{proof}

\begin{remark}
	Note that a consequence of \Cref{ch1:datucp} is that \emph{all} upper contact quadratic functions for a convex function at some point $x$ share the same differential at $x$. 
\end{remark}

We conclude this section with some comments concerning our choice of the setting in which we consider convex functions.

\begin{remark}[On the setting of convex functions] \label{rem:convex} We have discussed real-valued convex functions on Euclidian spaces. Much of what we have presented can be generalized to extended real-valued functions on Banach spaces. We are only interested in the finite dimensional case, where we note that Chapter \ref{sec:mflds} will show that local convexity can be defined on manifolds by using the invariance of the notion under smooth coordinate changes. It is common to consider convex functions on $X$ convex as being $(-\infty, +\infty]$-valued functions, where $+\infty$ is an allowed value and one keeps track of the domain ${\rm dom}(u) \defeq \{x \in X: u(x) \neq +\infty\}$. This extension has the advantage in simplifying many things, but, for us, will only have a substantial impact when we consider the Legendre transform in Appendix \ref{ap:legalex}. This is because the Legendre transform is more natural when defined for extended real-valued convex functions. In particular, this would avoid the pathologies that lead to the counterexample discussed in Remark \ref{rem:subdiff}. Another example of the utility of considering  $(-\infty, +\infty]$-valued convex functions concerns the validity of an identity in the subdifferential of the sum of convex functions, as discussed in Lemma \ref{lem:subdiffsum}. Under the validity of the {\em Attouch--Brezis condition} (see \cite[formula~(0.3)]{attbre}) the sum formula continues to hold.

On the other hand, having to keep track of the domain ${\rm dom}(u)$, where $u$ is finite, would complicate many of the statements we give. In particular, this would be true for Alexandrov's theorem and the discussion of upper contact jets, which are essential for the viscosity theory. In addition, we want to consider locally convex functions in relation to a general potential theory, where the generalized subharmonics are $[-\infty, +\infty)$-valued upper semicontinuous functions. The choice of this codomain is natural for the viscosity theory for two reasons. First, one maintains the Weierstrass's theorem: upper semicontinuous functions assume a real-valued maximum value on each compact subset. This is fundamental when using a contradiction argument for the comparison principle, amongst other things. Second, allowing the value $-\infty$ gives rise to the standard viscosity theory localization of subsolutions/subharmonics by extendinng the function to be $-\infty$ outside of a set of interest. This is because there are no upper contact functions and hence the extension is automaticallly a subsolution/subharmonic. We note that with this convention, an upper semicontinuous function $u$ is a viscosity subsolution of the mininal eigenvalue operator if and only if it is locally convex on the connected components where $u$ is not identically $-\infty$.
	
	\end{remark}

\section{Second order theory: Alexandrov's theorem} \label{sec:so}

We begin with the following well-known definition.

\begin{definition} \sid{differentiable (twice, in the Peano sense)}
	Let $X \subset \R^n$ be open. A function $u \colon X \to \R$ is \emph{twice differentiable (in the Peano sense) at $x \in X$} if there exist $p \in \R^n$ and a symmetric $n\times n$ real matrix $A \in \call S(n)$ \syid{Symn@$\mathcal{S}(n)$! the space of all symmetric $n\times n$ real matrices} such that
	\begin{equation} \label{2od:def}
	u(y) = u(x) + \pair{p}{y-x} + Q_A(y-x) + o(|y-x|^2) \quad \text{as $y \to x$},
	\end{equation}
	where $Q_A \defeq \frac12 \pair{A\cdot}\cdot$ is the quadratic form associated to $\frac12 A \in \call S(n)$ whose Hessian is $A$.\syid{QA@$Q_A$!the quadratic form associated to the matrix $\frac12 A \in \mathcal{S}(n)$} We will denote by \syid{Diffk@$\mathrm{Diff}^k u$!the set of points where $u$ is $k$ times ($k \in \{1,2\}$) differentiable (in the Peano sense)}
	\begin{equation*}
	\mathrm{Diff}^2 u \defeq  \{x \in X: \ \text{$u$ is twice differentiable in $x$} \}.
	\end{equation*}
\end{definition}
The following elementary lemma will be used often.

\begin{lem}
	Let $u: X \to \R$ and $x \in X$ and suppose that there exist $p \in \R^n$ and $A \in \call S(n)$ such that \eqref{2od:def} holds. Then the following hold:
	\begin{itemize}
		\item[(a)] $u$ is differentiable in $x$ with gradient $Du(x) = p$ (in particular, $p$ is unique);
		\item[(b)] $A$ is unique as well, and we denote $D^2u(x) = A$.
	\end{itemize}
\end{lem}

\begin{proof}
	It is clear that twice differentiability at $x$ implies differentiability at $x$ since \eqref{2od:def} implies
	\begin{equation}\label{Diff1}
	u(y)  = u(x) + \pair{p}{y-x} + o(|y-x|) \quad \text{as $y \to x$}.
	\end{equation}
	The uniqueness of $p$ follows easily from \eqref{Diff1}. In fact, if \eqref{Diff1} holds for $p$ and $p'$ with $p \neq p'$ one has
	$$
	\pair{p-p'}{y-x} = o(|y-x|) \quad \text{as $y \to x$},
	$$
	but choosing $y\defeq x+t(p-p')$ with $t \neq 0$ so small that $y \in X$ one obtains
	$$
	t|p - p'|^2 = |p - p'| o(t) \quad \text{as $t \to 0$},
	$$
	which gives $|p - p'| = o(1)$ for $t \to 0$, which contradicts $p-p'\neq 0$.
	
	To see that $A$ is also unique, suppose that there exist $A \neq A' \in \call S(n)$ such that \eqref{2od:def} holds with $p = Du(x)$. Hence one has
	$$
	\pair{(A - A')(y-x)}{y-x} = o(|y-x|^2) \quad \text{as $y \to x$}.
	$$
	Choosing $y = x + te$ for some eigenvector $e$ of $A-A'$ with eigenvalue $\lambda \neq 0$ one obtains a contradiction, since $\lambda t^2 = o(t^2)$ for $t \to 0$.
\end{proof}

It is well-known that a twice differentiable locally convex function $u$ can be characterized by having positive semidefinite Hessian (see the next \Cref{prop:Hpd}). In fact, this follows from the more general property that locally convex functions are those having non-negative Hessian in the viscosity sense; in other words, they are viscosity subsolutions of $\lambda_1(D^2 u) = 0$, where $\lambda_1$ denotes the minimal eigenvalue operator. This will be the content of \Cref{prop:P_Ptilde}(a), which uses the potential-theoretic formalism of \Cref{sas}.

\begin{prop} \label{prop:Hpd}
Let $u \colon X \to \R$ with $X \subset \R^n$ open. Then
\[
\mbox{$u$ locally convex in $X \quad \implies \quad D^2 u \geq 0$ on $\mathrm{Diff}^2 u$},
\]
and the converse is true if $\mathrm{Diff}^2 u = X$.
\end{prop}

\begin{proof}
It suffices to prove the proposition under the additional assumption that $X$ is convex, so that $u$ being convex is equivalent to $u$ being locally convex (by \Cref{lem:localization}). Indeed, if the proposition holds in that case, then it suffices to apply it to each open ball (hence convex) contained in $X$ to prove the general statement. 

So, suppose that $u$ is convex in $X$ (convex). We are going to prove the more general property that, for each $x \in X$, one has $D^2\phi \geq 0$ for each quadratic upper test function $\phi$ for $u$ in $x$ (that is, a quadratic function satisfying \eqref{QUTF}). Let $t \in [0,1)$ and let $v \in \R^n$ be such that $x+v \in X$. By the convexity of $u$ and the properties of $\phi$, we have
\[
\begin{split}
u(x+tv) \leq tu(x+v)+(1-t)u(x) &\leq t\phi(x+v) + (1-t)\phi(x) 
\\
&= \phi(x) + t \pair{D\phi(x)}{v} + \frac{t}{2}\pair{D^2\phi\, v}{v},
\end{split}
\]
where the last equality comes from the second order Taylor formula of $\varphi$ quadratic (see formula~\eqref{phiexp}). By \Cref{ch1:datucp} we know that $u$ is differentiable at the point $x$ of upper quadratic contact with $Du(x) = D\phi(x)$, and $\phi(x) = u(x)$, so the above inequalities give
\begin{equation} \label{D2phi>01}
t\pair{D^2\phi\,v}{v} \geq u(x + tv) -  u (x) - \pair{Du(x)}{tv} = o(t|v|).
\end{equation}
Since $X$ is open, we can choose any $v \in B_\delta(0)$ for some sufficiently small $\delta > 0$, so \eqref{D2phi>01} is equivalent to 
\[
\pair{D^2\phi\,\nu}{\nu} \geq o(1) \quad \text{as}\ t \dto 0, \quad \forall \nu \in \mathbb S^{n-1},
\]
which implies that $D^2\phi$, as desired.

Now, if $u$ is twice differentiable at $x$, then in the above argument one can choose the quadratic upper test function $\phi_{\epsilon}(\cdot) = u(x) + \pair{Du(x)}{\cdot-x} + \frac12\pair{(D^2u(x)+\epsilon I)(\cdot-x)}{\cdot-x}$, for each $\epsilon > 0$. This yields $D^2u(x) \geq -\epsilon I$ for all $\epsilon > 0$ and the conclusion follows by letting $\epsilon \dto 0$. For the converse, if $u$ is twice differentiable on $X$, then given $x,y \in X$, by the mean value theorem, there exists $z \in [x,y]$ such that
\[
\pair{Du(x)-Du(y)}{x-y} = \pair{D^2u(z)(x-y)}{x-y} \geq 0.
\]
Then it suffices to recall that the monotonicity of $Du$ characterizes the convexity of differentiable functions $u$ (\Cref{conviffmon}).
\end{proof}

We conclude this first section by stating the celebrated theorem of Alexandrov on second-order differentiability of convex functions~\cite{alex}. For the convenience of the reader, a complete proof of this crucial theorem can be found in Appendix~\ref{ap:legalex}, complemented by Appendix~\ref{proofsard}.

\begin{thm}[Alexandrov] \label{aleks} \sid{Theorem!Alexandrov}
	Let $X \subset \R^n$ be open and convex, and let $u \colon X \to \R$ be convex. Then $u$ is twice differentiable (in the Peano sense) almost everywhere in $X$, with respect to the Lebesgue measure.
\end{thm}

The proof we offer in Appendix \ref{ap:legalex} depends on two ingredients, which we can roughly describe here as a crucial property about the non-critical points of the Legendre transform of convex quadratic perturbations of a convex function (see Lemma \ref{mplt}) and a Lipschitz version of Sard's theorem (see Lemma \ref{sard}), which shows that the set of ``bad'' points (the critical values) has Lebesgue measure zero. For completeness, we also give a proof of Lemma~\ref{sard} in Appendix \ref{proofsard} which relies on the Besicovitch covering Lemma~\ref{app:bes}, whose proof is also given. The reader is invited to have a look at Appendix~\ref{ap:legalex} for a presentation of the Legendre transform and a more detailed discussion of the needed ingredients.

\chapter{Semiconvex functions and upper contact jets} \label{chap:SCUCJ}  

With Alexandrov's theorem now in hand, this chapter will discuss the next important ingredient for the theory of viscosity solutions of differential equations and differential inclusions that involve semiconvex functions. This detailed analysis concerns upper contact jets and upper contact points of (locally) semiconvex functions and culminates in two important results: the Upper Contact Jet Theorem and the Summand Theorem. The main question concerns the existence of a sufficient amount of upper (and lower) test functions. The proof of the most delicate parts of this analysis will also make use of so-called Jensen--S{\l}odkowski Theorem, which is the subject of the following chapter. In addition, in order to appreciate the notion of semiconvex/semiconcave functions, we will present a known characterization of Lipschitz functions with Lipschitz gradients. 


\section{Semiconvex functions}

The definition we adopt of a (locally) semiconvex function is the following.

\begin{definition} \label{def:qc}
	A function $u\colon C\to \R$ is said $\lambda$-\emph{semiconvex}\sid{convex!function!semi- (\emph{also} locally semi-)}\sid{semiconvex!function (\emph{also} locally)|seeonly {convex, function, semi-}} on the convex set $C \subset \R^n$ if there exists a real number $\lambda \geq 0$ such that the function $u + \frac{\lambda}{2}|\cdot|^2$ is convex on $C$.
	
	We say that $u$ is \emph{locally semiconvex} on $X$ if for every $x\in X$ there exists a ball $B\subset X$ with $x\in B$ such that $u$ is $\lambda$-semiconvex on $B$, for some nonnegative real number $\lambda = \lambda(x)$. 
\end{definition}

Note that we have included convexity into the definition when $\lambda = 0$. In many proofs involving semiconvex functions, we can assume without loss of generality that the functions are in fact convex. For instance, this is possible for results about differentiability, since the quadratic perturbation $\frac{\lambda}{2}|\cdot|^2$ is smooth. This is the case for Alexandrov's~\Cref{aleks}, which clearly extends from convex to locally semiconvex functions; that is, the following result holds.

\begin{thm}[Alexandrov: locally semiconvex version] \label{aleks:qc} \sid{Theorem!Alexandrov!locally semiconvex version}
	Let $X \subset \R^n$ be open and let $u$ be locally semiconvex on $X$. Then $u$ is twice differentiable (in the Peano sense) almost everywhere on $X$ with respect to the Lebesgue measure.
\end{thm}

\begin{proof}
	Since the Euclidean space is Lindel{\"o}f (that is, from any open covering of a set we can extract a countable subcovering), we know that there exists a countable family of balls $\{ B_i \}_{i\in\N}$ such that $u$ is $\lambda_i$-semiconvex on each $B_i \subset X$, for some $\lambda_i \geq 0$, and $X = \bigcup_{i\in\N} B_i$. Since each $v_i \defeq u + \frac{\lambda_i}2|\cdot|^2$ is convex on $B_i$, by Alexandrov's \Cref{aleks} \syid{abs@\detokenize{$\vert x\vert$}!if $x$ is a set, the Lebesgue measure of $x$}
	\[
	\big| B_i \setminus \Diff^2\! v_i \big| = 0 \quad \forall i \in \N,
	\]
	where we also note that $\Diff^2\! v_i = \Diff^2\! u$ since the norm squared is a smooth function. Hence we have, by the $\sigma$-subadditivity of the Lebesgue measure,
	\[
	\big| X \setminus \Diff^2\! u \big| \leq \sum_{i\in\N} \big| B_i \setminus \Diff^2\! u \big| = 0;
	\]
	that is, $u$ is twice differentiable almost everywhere on $X$, as desired.
\end{proof}

\begin{remark} \label{rmk:ch2:datucp}
	Another property that semiconvex functions inherit from convex functions is the differentiability at points of quadratic upper contact (\Cref{ch1:datucp}); indeed, note that if $\phi$ is an upper contact quadratic test function for a $\lambda$-semiconvex function $u$ at some point $x$, then $\tilde\phi \defeq \phi + \frac\lambda2|\cdot|^2$ is an upper contact quadratic test function for the convex function $v \defeq u + \frac\lambda2|\cdot|^2$ at $x$, and one recovers the identity $Du(x) = D\phi(x)$. This property will be used shortly in the proof of \Cref{charc11}, and it will be also paraphrased in the language of {\em upper contact jets} later on in \Cref{datucp}.
\end{remark}

\section{A semiconvexity characterization of $C^{1,1}$}

It is well known that a function is affine if and only if it is simultaneously convex and concave. 
An interesting fact which helps to understand the notion of semiconvexity is that functions which are both semiconvex and semiconcave must be differentiable with Lipschitz gradient (that is, of class $C^{1,1}$).\syid{C11@$C^{1,1}(X)$!the space of functions in $C^1(X)$ with Lipschitz gradient on $X$}
Of course, we say that $u$ is semiconcave if $-u$ is semiconvex. 

This characterization is included in \cite{hlkry} as an appendix, and was previously to be found in~\cite{hlqc}, in an appendix beginning as follows: ``It is interesting that the condition that a function be $ C^{1,1}$ is directly related to semiconvexity, in fact it is equivalent to the function being simultaneously semiconvex and semiconcave. This was probably first observed by Hiriart-Urruty and Plazanet in \cite{hiriart}. An alternate proof appeared in \cite{eber}.''

The proof exploits the fact that semiconvexity is preserved when one convolves with a \emph{mollifier},\sid{mollifier} that is a function $\eta$ such that \syid{Cinftyc@$C^{\infty}_\mathrm{c}(X)$!the space of smooth (i.e., differentiable to all orders) functions with compact support on $X$} \syid{supp@$\mathrm{supp}(u)$!the support of $u$, namely the closure of the set of points where $u$ does not vanish} 
\[
\eta \in  C_\mathrm{c}^\infty(\R^n), \quad \mathrm{supp}(\eta) \subset B_1,\quad \eta \geq 0,\quad \int_{\R^n} \eta = 1.
\]
This fact is formalized in the following lemma, where we use the standard notation \syid{epsilon@$\eta_\epsilon$ mollifier family (\emph{see also} $u_\epsilon$)!see~(\ref{etaepsdef})}
\begin{equation} \label{etaepsdef}
\eta_\epsilon \defeq \frac1{\epsilon^n}\,\eta\!\left(\frac\cdot\epsilon\right), \quad \forall \epsilon > 0,
\end{equation}
so that the collection $\{\eta_\epsilon\}_{\epsilon>0}$ is be the so-called \emph{approximation of the identity} based on $\eta$ (cf., e.g., \cite[Section~9.2]{wheedzyg}).

\begin{lem} \label{convqc}
	Suppose $u$ is $\lambda$-semiconvex on $\R^n$. Let $\eta$ be a mollifier, and let $\{\eta_\epsilon\}_{\epsilon>0}$ the approximation of the identity based on $\eta$. Then $u_{\epsilon} \defeq u \ast \eta_\epsilon$ is $\lambda$-semiconvex. \syid{uepsb@$u_\epsilon$ (\emph{see also} $\eta_\epsilon$)!see Lemma~\ref{convqc} or~(\ref{mollify})} \syid{ast@$\ast$!the integral convolution operator} \sid{approximation!smooth}
\end{lem}

\begin{proof}
	Let $f\defeq u + \frac\lambda2 |\cdot |^2$ be convex. It is well known that the mollification $f_{\epsilon} \defeq f \ast \eta_\epsilon$ is smooth and, moreover, is still convex. Indeed, since
	\[\begin{split}
	f_{\epsilon}(tx + (1-t)z) &\defeq \int_{\R^n} f(tx + (1-t)z - y) \eta_{\epsilon}(y) \, \di y 
	\\
	&= \int_{\R^n} f(t(x-y) + (1-t)(z - y)) \eta_{\epsilon}(y) \, \di y,
	\end{split}\]
by the convexity of $f$ and the nonnegativity of $\eta$ one has
$$
f_{\epsilon}(tx + (1-t)z) \leq \int_{\R^n} \bigl( tf(x-y) + (1-t)f(z - y) \bigr) \eta_{\epsilon}(y)  \, \di y \defeq  t f_{\epsilon}(x) + (1-t) f_{\epsilon}(z).
$$
In addition, by making the change of variables $z = \epsilon y$ and using that $\int_{\R^n} \eta = 1$, one has
	\[ \begin{split}
	\left(|\cdot|^2  \ast \eta_\epsilon \right)\!(x) \defeq \int_{\R^n} |x-z|^2\eta(z/\epsilon) \epsilon^{-n} \, \di z 
	= |x|^2 - 2\epsilon \pair{x}{a} + \epsilon^2 b \defeq  q_{\epsilon}(x),
	\end{split} \]
	where
	\[
	a \defeq \int_{\R^n} y\, \eta(y)\, \di y \in \R^n \quad \text{and}\quad b \defeq \int_{\R^n} |y|^2\,\eta(y)\, \di y \in \R.
	\]
	Therefore, since $f_{\epsilon}$ is smooth and convex\syid{I@$I$!the identity matrix}
	$$
		0 \leq D^2 f_{\epsilon} = D^2 \Bigl(u_{\epsilon} + \frac{\lambda}{2} q_{\epsilon}\Bigr) = D^2u_{\epsilon} + \lambda I,
		$$ 
		and the thesis follows from \Cref{prop:Hpd}.
\end{proof}

\begin{prop} \label{charc11}
	$u$ is of class $\lambda$-$ C^{1,1}$ (that is, $u \in C^{1,1}$ and the Lipschitz constant of $Du$ is $\lambda$) if and only if both $\pm u$ are $\lambda$-semiconvex.
\end{prop}

\begin{proof}
	Suppose $u$ is of class $\lambda$-$ C^{1,1}$ and let $f\defeq u + \frac\lambda2|\cdot|^2$.
	We prove that $u$ is $\lambda$-semiconvex; the proof for $-u$ is analogous. We have
	\[
	\pair{Df(x) - Df(y)}{x-y} \geq \left(\lambda|x-y| - |Du(x) - Du(y)| \right) |x-y| \geq 0, 
	\]
	thus proving (cf.~\Cref{conviffmon}) that $f$ is convex; that is, $u$ is $\lambda$-semiconvex.
	
	Conversely, suppose $\pm u$ are $\lambda$-semiconvex. We first assume that $u$ is smooth; then we have $-\lambda\I \leq D^2u(x) \leq \lambda \I$ for all $x$. By the Mean Value Theorem,
	\[
	\text{$Du(x) - Du(y) = D^2u(z) (x-y)$ \quad for some $z\in[x,y]$,}
	\]
	therefore $|Du(x) - Du(y)| \leq \lambda|x-y|$.
	
	For the general case, note that semiconvexity and semiconcavity together imply $u$ being $C^1$. Indeed, both $f_\pm \defeq \pm u + \frac\lambda2 |\cdot|^2$ have a supporting hyperplane from below at every point; this means that, for every $x$, there exist $p_\pm=p_\pm(x)$ such that the affine (and thus quadratic) functions $\phi_\pm \defeq -f_\pm(x) - \pair{p_\pm}{\cdot - x}$ are upper contact functions for $-f_\pm$ at $x$ (in the sense of \Cref{ch1:datucp}). Note now that $-f_\pm$ are both $2\lambda$-semiconvex, since both $\pm u$ are $\lambda$-semiconvex, and thus by \Cref{ch1:datucp} (cf.~also \Cref{rmk:ch2:datucp}) we know that $p=D(-f_+)(x)$, $q=D(-f_-)(x)$, yielding $(q-p)/2 = Du(x)$. That is to say, $u$ is differentiable everywhere. By \Cref{deusingv} we conclude that $u$ is $C^1$.
	
	Finally, approximate $u$ by $u_{\epsilon}$ and recall that $Du_{\epsilon} \to Du$ locally uniformly, since $u$ is $C^1$ (see, for instance, \cite[Theorems 9.3 and 9.8]{wheedzyg}). Therefore, since by \Cref{convqc} and what we showed above in the smooth case we know that $Du_{\epsilon}$ is $\lambda$-Lipschitz, so is $Du$ and we are done.
\end{proof}

\section{Upper contact jets}

The next definition is of crucial importance for all that follows. We recall that $\call S(n)$ denotes the space of symmetric $n\times n$ real matrices which is equipped with its standard partial ordering (the Loewner order) where the quadratic form associated to $A \in \call S(n)$ will often be normalized as \syid{Loewner@$\geq$!if between symmetric matrices, the Loewner order (that is, $A \geq B$ means that $A-B$ is positive semi-definite)}
\[
Q_A(y) \defeq \tfrac12\pair{Ay}y, \quad \text{so that}\ D^2Q_A \equiv A.
\]

\begin{definition} \label{def:ucp}
	We say that a point $x \in X$ is an \emph{upper contact point for $u$} \sid{contact!upper!point/jet} if there exists $(p,A)\in \R^n\times \call S(n)$ such that 
	\begin{equation}	\label{ucp}
	u(y) \leq u(x) + \pair p{y-x} + Q_A(y-x) \qquad \forall y \in X\ \text{near}\ x;
	\end{equation}
	in this case, we write $(p,A) \in J_x^{2,+}u$, \syid{J2p@\detokenize{$J_x^{2,+}u$}!the set of all upper contact jets for $u$ at $x$} and we say that $(p,A)$ is an \emph{upper contact jet} for $u$ at $x$. In addition, if \eqref{ucp} holds with strict inequality for $y\neq x$, then $x$ is called a \emph{strict upper contact point}, and $(p,A)$ is called a \emph{strict upper contact jet}, for $u$ at $x$.
\end{definition}

\begin{remark} \label{rmk:ucqf}
	Note that we always have equality in~(\ref{ucp}) for $y=x$ which justifies the use of \emph{contact} point for $u$ at $x$. Also notice that $(p,A) \in J^{2,+}_xu$ if and only if the unique quadratic function $\phi$ such that $J^2_x \phi \defeq (\phi, D\phi, D^2\phi)(x) = (u(x), p , A)$ satisfies \syid{J2@\detokenize{$J^2_xu$}!the $2$-jet associated to a twice differentiable function $u$ at $x$}
	\begin{equation} \label{eq:ucqf}
	u \leq \phi \quad \text{near $x$, with equality at $x$,}
	\end{equation} 
	that is, if and only if $\phi$ is a \emph{quadratic upper test function for $u$ at $x$} in the viscosity sense. \sid{upper test function!quadratic}
\end{remark}

\begin{remark}[Alternate spaces of upper contact jets]\label{rem:UCJ}
	In addition to the notion of strictness of upper contact jets given in \Cref{def:ucp}, there are other useful variants, including a quantified version of ``uniform'' strictness, which we call \emph{$\epsilon$-strictness}: given $\epsilon > 0$ fixed, we say that a point $x \in X$ is an \emph{$\epsilon$-strict upper contact point for $u$} 
	 if there exists $(p,A) \in \R^n \times \call S(n)$ such that
	\begin{equation*}	
	u(y) \leq u(x) + \pair p{y-x} + Q_A(y-x) - \epsilon | y - x |^2 \qquad \forall y \in X\ \text{near}\ x;
	\end{equation*}
	and, as above, $(p,A)$ will be called an \emph{$\epsilon$-strict upper contact jet for $u$ at $x$}. 

	When it comes to viscosity theory, both for subharmonics of a subequation constraint set $\call F$ and for subsolutions of an equation determined by an operator $F$, it is equivalent to work with upper contact jets or $\epsilon$-strict upper contact jets in the sense described in the next paragraph.  We discuss here the more well-known case of a continuous operator $F$, but the same considerations hold for a potential theory determined by $\call F$ which is closed.
	
	Notice that the validity of inequality $F[\phi] \defeq F(x,\phi(x),D\phi(x),D^2\phi(x)) \geq 0$ for any upper test function $\phi$ for $u$ at $x$, which defines the notion of $u$ being a subsolution of the equation $F[u] = 0$ at $x$, is equivalent to the validity, for $(p,A) = (D\phi(x), D^2\phi(x))$ for some such $\phi$, of
	\begin{equation} \label{visceq:jet}
	F(x,u(x), p, A) \geq 0.
	\end{equation}
	Such jets are not upper contact jets in general, yet they are \emph{little-o upper contact jets}
	; that is (cf.~\Cref{rmk:ucqf}),
	\begin{equation}	\label{loucp}
	u(y) \leq u(x) + \pair p{y-x} + Q_A(y-x) + o(|y-x|^2) \qquad \forall y \in X\ \text{near}\ x.
	\end{equation}
	Hence, we have five families of upper contact jets (for $u$ at $x$):
	\begin{enumerate}
		\item {$J_S(x,u)$, \emph{$\epsilon$-strict upper contact jets}, for some $\epsilon > 0$ possibly different from jet to jet};
		\item {$J_s(x,u)$, \emph{strict upper contact jets} according to \Cref{def:ucp}};
		\item {$J_Q(x,u)$, \emph{quadratic upper contact jets}; that is, $J_Q(x,u) = J^{2,+}_x u$ is the set of all the upper contact jets for $u$ at $x$ according to \Cref{def:ucp}};
		\item {$J_{ C^2}(x,u)$, \emph{$C^2$ upper contact jets} associated to upper test functions $\phi \in C^2$; the usual ones of viscosity theory};
		\item {$J_o(x,u)$, \emph{little-o upper contact jets} as in \eqref{loucp}, also often employed in viscosity theory}.
	\end{enumerate}
	It is easy to see that, for $(x,u)$ fixed,
	\begin{equation} \label{inclusionsformulations}
	J_S(x,u) \subset J_s(x,u) \subset J_Q(x,u) \subset J_{C^2}(x,u) \subset J_o(x,u),
	\end{equation}
	and one can prove (see \cite[Lemma C.1]{chlp}) that
	\[
	J_o(x,u) \subset \barr{J_S(x,u)},
	\]
	so that the closures of all sets in \eqref{inclusionsformulations} in fact coincide.\syid{barE@\detokenize{$\barr E$}!the closure of the set $E$} By the continuity of the operator $F$, this shows that, in order to prove that $u$ is a viscosity subsolution of $F[u]=0$ at $x$, one can check for \eqref{visceq:jet} to hold for all upper contact jets in any of the families above.
\end{remark}

As noted by Harvey and Lawson \cite{hlae}, there are two extreme cases where the set of upper contact jets of a function $u$ at a point $x$ are completely understood. The first extreme case is where $u$ has no upper contact jets at $x$; for instance, this is true of $u = |\cdot|$, at $x = 0$. The other extreme case is when the function is twice differentiable at $x$; by basic differential calculus one has the following result.

\begin{prop} \label{baspropucp}
	Suppose $u$ is twice differentiable at $x \in X$ (open subset); then 
	\[
	(p,A) \in J^{2,+}_x u \implies (p,A) = \big(Du(x), D^2u(x) + P\big),\ P\geq 0
	\]
	and the converse is true if $P > 0$.
\end{prop}

\begin{proof}
	Since $u$ is twice differentiable at $x$ one has the Taylor formula
	\begin{equation*} 
	u(y) = u(x) + \pair{Du(x)}{y-x} + Q_{D^2u(x)}(y-x) + o(|y-x|^2) \quad \forall\, y\ \text{near}\ x.
	\end{equation*}
	and by the definition of upper contact jet (\ref{ucp}) one sees that $(p,A) \in J^{2,+}_x u$ if and only if
	\begin{equation} \label{d-a}
	\pair{Du(x)-p}{y-x} + Q_{D^2u(x)-A}(y-x) + o(|y-x|^2) \leq 0 \quad \forall\, y\ \text{near}\ x.
	\end{equation}
	By choosing $y = x + t(Du(x) - p)$ for $t$ small, one has
	\[
	t|Du(x)-p|^2 + o{(t)} \leq 0 \quad \forall t\ \text{small},
	\]
	which forces $Du(x) = p$. At this point, recall that the Hessian of a twice differentiable function is symmetric and let $e$ be a unit eigenvector of $D^2u(x) - A$, relative to an eigenvalue $\lambda$; then by choosing $y = x+te$ with $t$ small one has
	\[
	\lambda t^2 + o(t^2) \leq 0,
	\]
	forcing $\lambda \leq 0$, and thus all the eigenvalues of $D^2u(x) - A$ have to be nonpositive, that is $A = D^2u(x) + P$ for some $P\geq 0$. Conversely, by (\ref{d-a}) it immediately follows that $(p,A) = (D^2u(x), D^2u(x) + P)$ is an upper contact jet for $u$ at $x$ if $P> 0$.
\end{proof}

A pair of remarks concerning contact functions and jets are in order.

\begin{remark}[Lower contact points and jets] 
\sid{contact!lower}
One can also define \emph{lower} contact points and jets by reversing the inequality in (\ref{def:ucp}). Hence the set $J^{2,-}_xu$ of all lower contact jets for $u$ at $x$ can be considered as a \emph{local second-order subdifferential} of $u$ at $x$; indeed, if $(p,0) \in J^{2,-}_xu$, then $p \in \de u(x)$ according to \Cref{def:subd}, provided that we restrict the domain of $u$ to a neighborhood of $x$ in which (\ref{ucp}) holds.
\end{remark}

\begin{remark}[Flat contact points and local maxima] 
As we pointed out in \Cref{geointerpsubdiff}, if $x$ is an upper (resp.~lower) contact point whose associated upper (resp.~lower) contact jets have zero matrix component, then $u$ has, for each such jet, one locally supporting hyperplane from above (resp.~below) at $x$. Hence, we will call such points \emph{flat} contact points. Also, it is worth noting that $(0,0) \in J^{2,+}_x u$ (resp.~$(0,0) \in J^{2,-}_xu$) is equivalent to $x$ being a local maximum (resp.~minimum) point for $u$. \sid{contact!flat}
\end{remark}

\section{The theorems on upper contact jets and on summands} \label{qcfaj}

The main result on upper contact jets we wish to present can be found in \cite[Theorem~4.1]{hlqc}. Two thirds of this result have been essentially proven above since these parts are the easy and natural extensions to semiconvex functions of first order properties already established for convex functions in \Cref{chap:diff}. The third part concerning the {\em partial upper semicontinuity of second derivatives} will be ``new'' also for convex functions.

\sid{upper contact|seeonly {contact, upper}}\sid{lower contact|seeonly {contact, lower}} \sid{flat contact|seeonly {contact, flat}}

\begin{thm}[Upper Contact Jet Theorem] \label{ucjt} \sid{Theorem!Upper Contact Jet} 
	Let $u$ be semiconvex on a neighborhood of $x$ and suppose $(p,A)$ is an upper contact jet for $u$ at $x$. Then the following holds:
	\begin{itemize}
		\item {\sf D\;at\;UCP}, \emph{differentiability at upper contact points}: if $(p,A)$ is an upper contact jet for $u$ at $x$, then $u$ is differentiable at $x$ and $p=Du(x)$ is unique.
	\end{itemize}
	Furthermore, for every set $E$ of full measure in a neighborhood of $x$ there exist a sequence $\{x_j\}_{j\in\N} \subset E \cap \Diff^2\! u$, with $x_j \to x$ as $j\to\infty$, and $\bar A \in \call S(n)$ such that the following hold:
	\begin{itemize}
		\item {\sf PC\;of\;FD}, \emph{partial continuity of first derivatives}: $Du(x_j) \to Du(x) = p$;
		\item {\sf PUSC\;of\;SD}, \emph{partial upper semicontinuity of second derivatives}: $D^2u(x_j) \to \bar A \leq A$.
	\end{itemize}
\end{thm}

The proof of the theorem reduces to the proof of the three Lemmas~\ref{datucp}, \ref{pcfd} and \ref{pusc} below. As suggested above, the first two lemmas follow directly from corresponding results that we proved for convex functions. In particular, the {\em differentiability at upper contact points} of Lemma \ref{datucp} is the extension to semiconvex functions (reformulated in the language of upper contact jets) of the differentiability at points of upper quadratic contact for convex functions of \Cref{ch1:datucp} and the {\em partial continuity of first derivatives} of Lemma \ref{pcfd} for semiconvex functions is immediately deduced from formula \eqref{contconv} of Corollary \ref{deusingv} which showed that convex  differentiable functions are of class $C^1$. For the sake of completeness, we will include their short proofs. 

\begin{lem}[Differentiability at upper contact points, {\sf D\;at\;UCP}] \label{datucp}
	Let $u$ be locally semiconvex. If $x$ is an upper contact point for $u$, then $u$ is differentiable at $x$. Moreover, if $(p, A)$ is an upper contact jet for $u$ at $x$, then $p = Du(x)$ is unique.
\end{lem}

\begin{proof}
	There exists a ball $B$ about $x$ and a quadratic function $\phi$ such that $u$ is $\lambda$-semiconvex on $B$ for some $\lambda>0$ and $u \leq \phi$ on $B$, with $u(x) = \phi(x)$ (that is, $\phi$ is an upper contact quadratic function for $u$ at $x$). Hence $\tilde\phi \defeq \phi + \frac\lambda2|\cdot|^2$ is an upper contact quadratic function for the convex function $\tilde u \defeq u + \frac\lambda2|\cdot|^2$ at $x$. Therefore by \Cref{ch1:datucp} $\tilde u$ is differentiable at $x$ with $D\tilde u(x) = D\tilde \phi(x)$; that is, $u$ is differentiable at $x$ with $Du(x) = D\phi(x)$. Finally, the same argument also proves that $p = Du(x)$ is unique for any $(p,A) \in J^{2,+}_xu$ (cf.~also \Cref{baspropucp}).
\end{proof}

\begin{lem}[Partial continuity of first derivatives, {\sf PC\;of\;FD}] \label{pcfd}
	Let $u$ be locally semiconvex, and $x_j \to x$. If $u$ is differentiable at each $x_j$ and at $x$, then $Du(x_j) \to Du(x)$.
\end{lem}

\begin{proof}
	As above, let $B$ be a ball about $x$ on which $\tilde u \defeq u+\frac\lambda2|\cdot|^2$ is convex. Eventually, $x_j,x \in B \cap \Diff^1\!\tilde u$, hence by the continuity property \eqref{contconv} (being $\de \tilde u(x_j) = \{ D\tilde u(x_j) \}$ and $\de \tilde u(x) = \{ D\tilde u(x) \}$), $D\tilde u(x_j) \to D\tilde u(x)$; this immediately yields $Du(x_j) \to Du(x)$, as desired.
\end{proof}

\begin{proof}[Proof of Theorem \ref{ucjt}]
	\Cref{datucp} clearly gives the first point of \Cref{ucjt}, while \Cref{datucp,pcfd} together give the second point. Indeed, since $(p,A) \in J^{2,+}_x u$, by \Cref{datucp}, $x \in \Diff^1\!u$, and by Rademacher's and Alexandrov's theorems, denoting by $\call U$ the neighborhood of $x$ which we consider, $|{\Diff^k\! u \cap \call U}| = \left|\,\call U\right|$ for $k=1,2$. Also, we know that $x$ is a limit point for $E \subset \call U$ with $|E| = |\,\call U|$ (otherwise there would exists a small ball $B_\rho(x) \subset \call U \cap E\compl$,\syid{compl@\detokenize{$E\compl$}!the complement of the set $E$} yielding $|E| < |\,\call U|$, contradiction). Analogously $x$ is a limit point for $E \cap \Diff^1\! u$ as well, hence we can conclude by invoking \Cref{pcfd}. 
	
	The third point of \Cref{ucjt} is the content of the following lemma, which states indeed the ``deepest'' of the three properties; we postpone its proof in \Cref{sec:proof:pusc}, as we will exploit the Jensen--S{\l}odkowski Theorem \ref{jenslod} in combination with Alexandrov's \Cref{aleks:qc}.
\end{proof}

\begin{lem}[Partial upper semicontinuity of second derivatives, {\sf PUSC\;of\;SD}] \label{pusc}
	Let $u$ be a locally semiconvex function, and $E$ be a set of full measure in a neighborhood of $x$. If $(p,A)$ is an upper contact jet for $u$ at $x$, then there exists a sequence of upper contact points $\{x_j\}_{j\in\N} \subset E\cap \Diff^2(u)$ such that $D^2u(x_j) \to \bar A \in \call S(n)$ with $\bar A \leq A$.
\end{lem}

\medskip
While the Upper Contact Jet \Cref{ucjt} collects the basic information about upper contact jets of a single locally convex function $u$, the following result shows that for the sum of two locally semiconvex functions one obtains useful information on the upper contact jets of the summands from each upper contact jet of the sum. Interestingly, this theorem, which is \cite[Theorem~7.1]{hlae}, can be interpreted in multiple ways:
\begin{enumerate}
	\item it obviously complements the partial upper semicontinuity of second derivatives for semiconvex functions;
	\item it easily yields a semiconvex version of the so-called {\em Theorem on Sums}, which plays a crucial role in the viscosity theory;
	\item  it conceals within a very useful \emph{addition theorem}.
\end{enumerate} 
The second interpretation is given below in \Cref{tosqc2} while the third one will be explained later on in \Cref{rmkone}.

\begin{thm}[Summand Theorem] \label{pusc:sum} \sid{Theorem!Summand (\emph{or} on Summands)}\sid{Theorem!on Summands|seeonly{Summand}}
	Suppose $u$ and $v$ are locally semiconvex and that the sum $w \defeq u+v$ has an upper contact jet $(p,A)$ at $x$. Then the following hold:
	\begin{enumerate}[label=(\roman*)]
		\item	$x$ is an upper contact point for both the summands $u$ and $v$ (which are therefore differentiable at $x$ by \Cref{datucp}), whose upper contact jets at $x$ are of the form $(Du(x), -)$ and $(Dv(x), -)$;
		\item 	for every set $E$ of full measure in a neighborhood of $x$, there exist $B,C \in \call S(n)$ and a sequence $\{x_j\}_{j\in\N}$ of upper contact points for $w$, with $x_j\in E \cap \Diff^2\! u \cap \Diff^2 \! v$ and $x_j \to x$ as $j \to \infty$, such that 
		\begin{gather*}
		Du(x_j) \to Du(x), \qquad Dv(x_j) \to Dv(x),\\
		D^2u(x_j) \to B, \quad D^2v(x_j) \to C,\quad \text{with $B+C \leq A$}.
		\end{gather*}
	\end{enumerate}
\end{thm}

\begin{proof}
	For the claim (i), we can suppose without loss of generality that $u$ and $v$ are convex. By the Hahn--Banach Theorem there exists $q\in \R^n$ such that $(-q, 0)$ is an upper contact jet for $-v$ at $x$, thus $(p-q, A)$ is an upper contact jet for $u = w-v$ at $x$. Analogously we prove that $x$ is an upper contact point for $v$ as well. 
	
	For the claim (ii), we begin by noting that the sum $w$ of locally semiconvex functions $u$ and $v$ is clearly locally semiconvex. Now, by applying \Cref{pusc} to the sum $w$ there exist a sequence $\{x_j\}\subset E \cap \Diff^2\!u \cap \Diff^2\!v$ and a matrix $\bar A \in \call S(n)$, $\bar A \leq A$, such that $x_j \to x$ and $D^2w(x_j) \to \bar A$ as $j\to\infty$.\footnote{Note that to apply the lemma we are considering $E\cap \Diff^2(u)$ as the set of full measure near $x$.} Finally, if
	\begin{equation} \label{arebounded} 
	\{ D^2u(x_j) \}_{j\in\N} \ \text{and} \ \{ D^2v(x_j) \}_{j\in\N} \ \text{are bounded in $\call S(n)$},
	\end{equation}
	then there exist $B,C \in \call S(n)$ such that, up to a subsequence, $D^2u(x_j) \to B$ and $D^2v(x_j) \to C$, where $B+C = \lim_{j\to\infty}\left( D^2u(x_j) + D^2v(x_j) \right) = \lim_{j\to\infty} D^2w(x_j) = \bar A \leq A$. In order to prove \eqref{arebounded}, let $\lambda >0$ be such that both $u$ and $v$ are $\lambda$-semiconvex in a neighborhood of $x$, and note that for any $\epsilon > 0$ there exists $j_0 = j_0(\epsilon) \in \N$ such that 
	\[
	-\lambda I \leq D^2u(x_j) = D^2w(x_j) - D^2v(x_j) \leq \bar A + \epsilon I + \lambda I \qquad \forall j \geq j_0.
	\]
	By fixing for instance $\epsilon = 1$ one sees that the tail $\{ D^2u(x_j) \}_{j\geq j_0}$ is bounded, thus the whole sequence $\{ D^2u(x_j) \}_{j\in\N}$ is bounded as well. The proof for $\{D^2v(x_j)\}_{j\in\N}$ is analogous.
\end{proof}

\begin{remark} 
	Note that since upper contact points for $w$ are also upper contact points for both $u$ and $v$, if $(p,A) \in J^{2,+}_x w$, then (cf.\ the proof above) we have $(p-q, A) \in J^{2,+}_x u$ and $(p-q', A) \in J^{2,+}_x v$, with $q+q' = p$, because one knows by \Cref{datucp} that $(p-q)+(p-q') = Du(x) + Dv(x) = Dw(x) = p$.
\end{remark}

\begin{remark} 
	We used the basic fact that each interval $[A, B] \subset (\call S(n), \leq)$ is compact  (with respect to the subspace topology inherited from $\R^{n^2}$). To prove that, let $Y \in [A,B]$ and note that by the Courant--Fischer min-max theorem one has that $\lambda_i(A) \leq \lambda_i(Y) \leq \lambda_i(B)$ for all $i = 1,\dots,n$.\syid{lambdai@$\lambda_i(A)$!the $i$-th eigenvalue of the matrix $A$, with the ordering $\lambda_i(A) \leq \lambda_{i+1}(A)$} Hence, if the eigenvalues are arranged in increasing order, the entries of $Y$ are bounded by $n\left(|\lambda_1(A)| \vee |\lambda_n(B)|\right)$, thus, since it is trivial that $[A,B]$ is closed, one concludes by invoking the Heine--Borel theorem.
\end{remark}

As we said, a version of the Theorem on Sums for semiconvex functions immediately follows if we double the variables, as it is easy to see.

\begin{cor}[Theorem on Sums for semiconvex functions] \label{tosqc2} \sid{Theorem!on Sums!for semiconvex functions}
	Let $u$ and $v$ be $\lambda$-semiconvex on the open subsets $X$ and $Y$ of $\R^n$, respectively. Define 
	\[
	w(x,y) \defeq u(x) + v(y)
	\]
	on $X\times Y \subset \R^n \times \R^n$ and suppose that $(p,A)$ is an upper contact jet for $w$ at $z = (\hat x,\hat y) $. Then $\hat x$, $\hat y$ are contact points for $u$, $v$, respectively, and for each set $E$ of full measure near $z$ there exists a sequence $(x_j, y_j) \in E \cap \left( \Diff^2(u) \times \Diff^2(v) \right)$ such that 
	\[\begin{split}
	(x_j, u(x_j), Du(x_j), D^2u(x_j)) &\to (\hat x, u(\hat x), Du(\hat x), A_1), \\
	(y_j, v(y_j), Dv(y_j), D^2v(y_j)) &\to (\hat y, u(\hat y), Du(\hat y), A_2),
	\end{split}\] 
	with
	\[
	-\lambda I \leq \left(\! \begin{array}{cc}
	A_1 &  0 \\
	0  & A_2 \\
	\end{array}\! \right)
	\leq A.
	\]
\end{cor}

By induction, one can also extend this result by adding additional variables, so that it assumes a form which more closely resembles the classical Theorem on Sums for semicontinuous functions, in the equivalent formulation one finds in~\cite[Appendix, Theorem 3.2$'$]{user}. We leave it to the reader to write an appropriate statement.

For the sake of completeness, in Section~\ref{ss:tosu}, we will give a proof of the classical Theorem on Sums for semicontinuous functions (see Theorem~\ref{tosu}). It is easily deduced from its semiconvex counterpart because, roughly speaking, it is possible to transform a function which is bounded above into a semiconvex function in such a way that it produces a regular displacement of the contact points of every given upper contact jet. More precisely, the displacement is proportional to the gradient-component of the jet, as we will see in Section~\ref{ch:qca}. In addition, in Section \ref{sec:SC} we will introduce and prove a potential theoretic version of Harvey and Lawson in Theorem \ref{tospt}, which they used in their proof of strict comparison principles with minimal monotonicity for constant coefficient potential theories.

\chapter{The lemmas of Jensen and S{\l}odkowski} \label{chap:js}

In this chapter we treat the remaining deep analytical results which are needed for a robust viscosity theory for subharmonics in general potential theories and subsolutions of fully nonlinear ellipitc PDEs. We address the central problem of how to ensure \emph{a priori} the existence of a sufficient amount of upper contact points for locally semiconvex functions which, when combined with Alexandrov's Theorem and the sup-convolution approximation of semicontinuous functions by semiconvex functions, will complete the foundations of the  ``hard analysis'' part of the viscosity theory. The analysis will culminate in the so-called Jensen--S{\l}odkowski Theorem \ref{jenslod}, which unifies and generalizes two historically distinct approaches (Jensen's Lemma and S{\l}odkowski's density estimate) to the central problem in a viscosity approach to potential theory and operator theory. Several important consequences of the Jensen--S{\l}odkowski Theorem are also given, as well as a novel proof of this fundamental theorem which is shown to follow from the area formula for gradients of locally semiconvex functions.

\section{Classical statements: a review}

In the crucial passage ``from almost everywhere to everywhere'' in the viscosity theory, two classical instruments dealing with (semi-)convex functions are {\em S{\l}odkowski's density estimate}~\cite[Theorem~3.2]{slod} and {\em Jensen's Lemma}~\cite[Lemma~3.10]{jensen}. The first instrument yields S{\l}odkowski's \emph{Largest Eigenvalue Theorem} \cite[Corollary~3.5]{slod} which, in turn, plays an essential role in Harvey and Lawson's proof of the Subaffine Theorem~\cite[Theorem~6.5]{hldir09}. The second instrument yields Jensen's maximum principle \cite[Theorem~3.1]{jensen} which, in turn, is a key ingredient in the proof of the Theorem on Sums~\cite[Theorem~3.2]{user}. These two instruments were developed independently. The first in the context of pluripotential theory and the second in the context of the machinery for proving the uniqueness of viscosity solutions. As we will see, they are in some sense equivalent. This long unrecognized equivalence is a good example of a missed opportunity for synergy between potential theory and operator theory where such a synergy is a main theme of the program initiated by Harvey and Lawson.  

S{\l}odkowski's original proof of his density estimate is based on two nontrivial properties~\cite[Proposition~3.3(iii) and Lemma~3.4]{slod}. The first property concerns the largest generalized eigenvalue of the Hessian at $x$ of a convex function $u$~\cite[Definition~3.1]{slod} and the second property concerns spheres which support from above the graph $\Gamma(u)$ of $u$ at the point $(x,u(x))$; that is, spheres in $\R^{n+1}$ lying above $\Gamma(u)$ and touching it only at $(x, u(x))$. Following Harvey and Lawson~\cite{hlqc}, our proof of S{\l}odkowski's estimate will be based on the simpler ``paraboloidal'' counterparts of such properties.

This approach will lead to an important consequence, which Harvey and Lawson call S{\l}odkowski's lemma, and which turns out to be equivalent to a reformulation of theirs of Jensen's lemma. This eventually allows one to merge the two lemmas into a \emph{Jensen--S{\l}odkowski's theorem}, and enlightens an interesting path to prove both Jensen's lemma and S{\l}odkowski's Largest Eigenvalue Theorem.

At the end of the section, a second non classical proof of Jensen's lemma will be proposed. Jensen points out that his lemma is based ideas of Pucci \cite{pucci} and gives a proof which relies on the area formula (see, e.g.,~\cite[Theorem~3.2.3]{fed:geo}); we will offer a different area-formula-based proof by reformulating Jensen's lemma \emph{à la} Harvey--Lawson~\cite{hlqc} and then following an argument by Harvey~\cite{har:pc} which passes also through Alexandrov's maximum principle (see, for example, \cite{cafcab}).

For the convenience of the reader, we recall some classical statements of Jensen's and S{\l}odkowski's lemmas which we deal with.

In \cite{slod}, S{\l}odkowski proves an estimate for the Lebesgue lower density of sub-level sets of his function $\call K(u,x)$ defining a generalized notion of the largest eigenvalue of the Hessian at $x$ of a convex function $u$. We recall, that for $u$ convex (but not necessarily twice differentiable), S{\l}odkowski defines \syid{Kux@\detokenize{$\mathcal{K}(u,x)$}!see (\ref{LEF})}
\begin{equation}\label{LEF}
\call K(u,x) \defeq \limsup_{\epsilon \to 0} 2\epsilon^{-2} \max_{{\mathbb S}^{n-1}} \! \big( u(x+\epsilon\,\cdot\,) - u(x) - \epsilon\pair{Du(x)}{\cdot\,} \big)
\end{equation}
if $x \in \Diff^1\! u$ and $\call K(u,x) \defeq +\infty$ otherwise.

\begin{remark}
	$\call K(u,x)$ indeed generalizes the concept of \emph{largest eigenvalue} of the Hessian of a convex function $u$ at $x$, since, if $u$ has second-order Peano derivatives at $x$, then $\call K(u,x)$ is the largest eigenvalue of the Hessian $D^2u(x)$:
	\[
	\call K(u,x) = \limsup_{\epsilon \to 0}\, 2 \epsilon^{-2} \displaystyle\max_{|h|=1}\! \Big(  \epsilon^2 Q_{D^2u(x)}(h) + o(\epsilon^2) \Big) = \Vert D^2u(x) \Vert = \lambda_n(D^2u(x)).
	\]
\end{remark}

S{\l}odkowski's density estimate is stated next. See  \cite[Theorem~3.2]{slod} for the original proof. We will give a novel ``paraboloidal'' proof of this result, which will be a corollary of our Lemma~\ref{hlKb}.

\begin{lem}[S{\l}odkowski's density estimate] \label{slodthm} \sid{S{\l}odkowski's density estimate}
	Let $u$ be convex near $x^* \in \R^n$, and suppose that $\call K(u,x^*) = k^* < +\infty$. Then for any $k > k^*$ the set $\{ x : \ \call K(u,x) < k \}$ is Borel and its Lebesgue lower density at $x^*$ satisfies
	\[
	\delta^-_{x^*}(\{ \call K(u,\cdot) < k\}) \geq \left(\frac{k-k^*}{2k}\right)^n.
	\]
\end{lem}
We recall that the Lebesgue lower density at $x^*$ of a measurable set $E$ is defined as \syid{deltaminusx@\detokenize{$\delta_x^-(E)$}!the (Lebesgue) lower density at $x$ of the set $E$}
\[
\delta^-_{x^*}(E) \defeq \liminf_{\epsilon \dto 0} \frac{|E \cap B_\epsilon(x^*)|}{|B_\epsilon(x^*)|} .
\]

The two nontrivial properties on which S{\l}odkowski's proof of the density estimate are recalled in the following two lemmas which correspond, respectively, to Proposition~3.3(iii) and Lemma~3.4 in \cite{slod}.

\begin{lem}[Upper bound on the largest eigenvalue of the Hessian]\label{slod(iii)}
	Let $u \colon U \to \R$ be convex on $U\subset \R^n$ open and convex. Suppose there exists a sphere $\mathbb S(c,r) \defeq \de B_r(c) \subset \R^{n+1}$ which supports the graph of $u$\syid{Gamma@$\Gamma(u)$!the graph of $u$}\sid{graph} from the above at $(x,u(x))$ with $x \in \Diff^1\!u$. Then one has
	\[
	\call K(u,x) \leq \frac{\left(1+|Du(x)|^2\right)^{\frac32}}{r}.
	\]
\end{lem}

We will formulate and prove an original paraboloidal version of this upper bound in Lemma \ref{hlKb}, which exploits the paraboloidal machinery developed by Harvey and Lawson in \cite{hlqc} and leads to the aforementioned paraboloidal proof of Lemma \ref{slodthm}.

\begin{lem}[Lower bound on the lower density of spherical contact points]\label{slodestlem}
	Let $u$ be nonnegative and convex on $B_\rho \subset \R^n$, for some $\rho > 0$, with $u(0) = 0$ and $Du(0) = 0$. Assume that there exists a ball $B_R((0,\dots,0,R)) \subset \R^{n+1}$ which intersects $\Gamma(u)$ (the graph of $u$) only at $0 \in \R^{n+1}$, and let, for $0<r<R$,
	\[
	X_r \defeq \big\{ x \in B_\rho :\ \text{$\exists\,c \in \R^{n+1}$\! s.t.\ $\mathbb S(c,r)$ supports $\Gamma(u)$ from the above at $(x,u(x))$} \big\}.
	\]
	Then the lower density of $X_r$ at $0$ satisfies
	\[
	\delta^-_0(X_r) \geq \left(\frac{R-r}{2r}\right)^n.
	\]
\end{lem}

As a corollary of his density estimate in Lemma \ref{slodthm}, S{\l}odkowski obtains the following interesting result (see \cite[Corollary~3.5]{slod}).

\begin{thm}[S{\l}odkowski's Largest Eigenvalue Theorem] \label{slodle} \sid{Theorem!S{\l}odkowski's Largest Eigenvalue}
	Let $u$ be locally convex on an open set $X$ and suppose $\mathcal K(u,x) \geq M$ for a.e.\ $x\in X$. Then $\mathcal K(u,x) \geq M$ for all $x \in X$.
\end{thm}
We will give a proof of this result under the weaker hypothesis that $u$ is just locally semiconvex in Theorem~\ref{thm:LET_qc} as a consequence of the Jensen--S{\l}odkowski Theorem~\ref{jenslod}.

We conclude this section by recalling {\em Jensen's lemma}, which we present as one finds in Lemma A.3 of the \emph{User's guide}~\cite{user}.

\begin{lem}[Jensen's Lemma] \label{u:jen} \sid{Lemma!Jensen}
	Let $w\colon \R^n \to \R$ be semiconvex and $\hat x$ be a strict local maximum point of $w$. For $p\in\R^n$, set $w_p(x) =w(x) + \pair px$. Then, for $r,\delta>0$, the set
	\[
	K\defeq\{x\in \barr B_r(\hat x):\ \text{$\exists\, p\in \barr B_\delta$ for which $w_p$ has a local maximum at $x$} \}
	\]
	has positive Lebesgue measure. 
\end{lem}

\begin{remark}
	Jensen's original statement~\cite[Lemma~3.10]{jensen} uses seemingly different hypotheses; that is, Jensen assumes that $u$ is continuous, $u$ belongs to $W^{1,\infty}$ and satisfies $D_\nu^2 u \geq -\lambda I$ in the sense of distributions, for all directions $\nu$. Nevertheless, one can see that these requirements are equivalent to asking that $u$ be $\lambda$-semiconvex (for instance, one can use \cite[Theorem~3.1]{dudley}).
\end{remark}

\section{Upper contact quadratic functions and the vertex map}\label{sec:vertex_map}

Our journey towards the reformulation of the classical results of Jensen and S{\l}odkowski begins with a presentation  the ``paraboloidal'' machinery of Harvey--Lawson~\cite{hlqc} which substitutes the original ``spherical'' one of S{\l}odkowski~\cite{slod}. This machinery concerns upper quadratic test functions of a given {\em radius}, upper contact points and jets of {\em type $A$} (meaning with fixed matrix component $A \in \call S(n)$) and the important associated \emph{upper vertex map}. In the next section, these notions will lead to estimates on the measure of contact points of fixed type.  

\begin{definition}
	We say that a function $\phi$ is a \emph{quadratic function of radius $r$} \sid{quadratic function (\emph{also} radius of and vertex point of)}\sid{vertex point (of a quadratic function)|seeonly{quadratic function}}\sid{radius (of a quadratic function)|seeonly{quadratic function}}if it is quadratic and satisfies $D^2\phi \equiv \frac1r\I$. In this case it is easy to see that $\phi$ can be written in a unique way as $\phi = \phi(v) + \frac{1}{2r}|\cdot - v|^2$; we will call $v$ the \emph{vertex point} of $\phi$. 
\end{definition}

\begin{definition} 
	Let $u \colon X \to \R$ be any function. We say that a quadratic function $\phi$ is an \emph{upper contact quadratic function for $u$ at $x\in X$} \sid{contact!upper!quadratic function} if condition \eqref{eq:ucqf} holds; that is, if 
	\begin{equation}\label{ucqf1}
	\text{$u(y) \leq \phi(y)$ for all $y \in X$ near $x$, \quad and \quad $u(x) = \phi(x)$.}
	\end{equation}
	If the inequality is strict for all $y\neq x$, then we say that $\phi$ is \emph{strict} upper contact quadratic function; if it holds for all $y\in X$, we say that $\phi$ is \emph{global} upper contact quadratic function \emph{on $X$}.
\end{definition}

We recall that in Definition \ref{def:ucp}, a point $x \in X$ for which \eqref{ucqf1} holds for some quadratic function $\varphi$ was called an upper contact point for $u$ and we presented the upper contact quadratic functions in the form $\varphi(x) = u(x) + \langle p, y - x \rangle + Q_A(y-x)$ where $Q_A(x) =  \frac{1}{2} \langle Ax,x \rangle$ and the contact condition  \eqref{ucqf1} takes the form
\begin{equation}\label{ucqf2}
u(y) \leq u(x) + \pair p{y-x} + Q_A(y-x) \qquad \forall y \in X\ \text{near}\ x.
\end{equation}
Note that $\varphi$ is the (unique) quadratic function with $J^2_x\phi = (u(x), p, A)$, and $(p,A) \in J^{2,+}_x u$ was called an upper contact jet for $u$ in $x$.

Two variants of upper contact jets will be used; namely those with fixed matrix component $A \in \call S(n)$ and then also those for which \eqref{ucqf2} holds on a fixed subset of $X$.

\begin{definition}\label{defn:ucp}
	An upper contact jet for $u$ at $x$ is said to be of \emph{type $A$} \sid{type} if it is of the form $(p,A)$ for some $p \in \R^n$, and $x \in X$ will be called an \emph{upper contact point of type $A$ for $u$}. 
	The set of all \emph{global} upper contact points of type $A$ for $u$ on $X$ will be denoted by $\C(u, X, A)$; that is, \syid{CuXa@\detokenize{$\C(u,X,A)$}!the set of all global upper contact points of type $A$ for $u$ on $X$}
	\[
	\C(u, X, A) \defeq \{ x \in X :\ \text{$\exists \, p \in \R^n$ such that \eqref{ucqf2} holds $\forall \, y \in X$} \}.
	\]
\end{definition}

Note that, according to \cref{def:ucp}, the set of the contact points for some global upper contact quadratic function of radius $r$ coincides with $\C(u, X, \lambda I)$, for $\lambda = \frac1{r}$.


\begin{remark}\syid{Brx@\detokenize{$B_\rho(x)$}!the open ball in $\R^n$ of radius $\rho$ and center $x$ (which is often omitted when $x$ is the origin)}
	We are going to focus our attention on sets of upper contact points of fixed type which are global on \emph{closed} balls centered at the origin; that is, of the form $\C(u, \barr B_\rho, A)$, up to translations that center the ball at the origin. In particular, we will be mainly interested in the Lebesgue measure of such sets, so let us also point out that, when it comes to measures, if $u$ is continuous, then one can equivalently consider open or closed balls, since
	\[
	\C(u, B_\rho, A) \subset \C(u, \barr B_\rho, A) \quad\text{and}\quad \C(u, \barr B_\rho, A) \setminus \C(u, B_\rho, A) \subset \de B_\rho.
	\]
\end{remark}

\begin{remark} 
	We are going to discuss the Lebesgue measure of the set $\C(u, X, A)$ in the special case when $u$ is locally semiconvex. For the sake of completeness, we note that $\C(u, X, A)$ is a measurable set when $u$ is locally semiconvex. Indeed, if one defines 
	$h_{y,p,A} \defeq u - u(y) + \pair p{y-\,\cdot} + Q_A(y-\,\cdot)$, then it is easy to see that
	\[
	\C(u, X, A) = \bigcup_{p\in \R^n} \bigcap_{y \in X} h_{y,p,A}^{-1}([0,+\infty)),
	\]
	where, by \Cref{datucp},
	\[
	\bigcap_{y \in X} h_{y,p,A}^{-1}([0,+\infty)) =\vcentcolon \C(u,X,A;\,p) \subset Du^{-1}(p);
	\]
	hence in fact
	\[
	\C(u,X,A) = \bigcap_{y \in X} \tilde h_{y,A}^{-1}([0,+\infty)) \subset \Diff^1\! u,
	\]
	where
	\[
	\tilde h_{y,A} \defeq u - u(y) + \pair {Du}{y-\cdot} + Q_A(y-\cdot) \quad \text{on} \ \Diff^1\! u,
	\]
	and $\Diff^1\! u$ has full measure by Rademacher's \Cref{rade} (or Alexandrov's \Cref{aleks:qc}, if one prefers). Now, clearly $\tilde h_{y,A}$ is a measurable function, so that each $\tilde h_{y,A}^{-1}([0,+\infty))$ is a measurable set; and to conclude, note that by the continuity of the map
	\[
	y \mapsto u(x) - u(y) + \pair{Du(x)}{y-x} + Q_A(y-x), \quad \text{with $x$ fixed},
	\]
	we can also write
	\[
	\C(u,X,A) = \bigcap_{y \in D} \tilde h_{y,A}^{-1}([0,+\infty)), \quad \text{for some $D \subset X$ countable and dense},
	\]
	so that one finally sees that $\C(u,X,A)$ is measurable.
\end{remark}

The following result, which is to be found in \cite[Lemma~5.1]{hlqc}, is an interesting exercise that shows an elementary geometric property of the convex hull of two open paraboloids (namely the strict epigraphs of two quadratic functions).\sid{paraboloid} It will play a key role in Harvey and Lawson's paraboloidal treatment of S{\l}odkowski's Lemma~\ref{slodthm}. As before, we will denote the graph of $\phi$ by $\Gamma(\phi)$, indicate its \emph{strict} epigraph by $\epi_S(\phi)$\syid{epis@\detokenize{$\epi_S(u)$}!the strict epigraph of $u$}\sid{graph!epi-!strict}, and the convex hull of a subset $E \subset \R^{n+1}$ by $\conv(E)$\syid{conv@\detokenize{$\conv(X)$}!the convex hull of the set $X$}\sid{convex!hull}. Recall that these objects are defined as follows: $\Gamma(\phi)\defeq\{(x,y)\in\R^{n+1}:\,y=\phi(x)\}$, $\epi_S(\phi)\defeq \epi(\phi) \setminus \Gamma(\phi)$, and $\conv(E) \defeq \bigcap_{E \subset C\, \text{convex}} C$.

\begin{lem}[Slubbed hull property] \label{shp} \sid{property!slubbed hull}
	Let $\frk C \defeq \conv(\epi_S(\phi_1) \cup \epi_S(\phi_2))$. There is an open vertical slab $\frk S\subset \R^n\times\R$, written as the intersection $\call{H}_1 \cap \call{H}_2$ of two parallel vertical open half-spaces $\call H_1$ and $\call H_2$, with the following property:
	\[
	\Gamma(\phi_j) \cap \frk C =  \Gamma(\phi_j)\cap \call H_j \qquad \text{for}\ j=1,2.
	\]
	Moreover, if $v_j$ is the vertex point of $\phi_j$, for $j=1,2$, then the width of $\frk S$ is $|v_1 - v_2|$.
\end{lem}

\begin{remark}[Tangent plane to the union of two paraboloids] \label{parab:cases}
	Let
	\[
	z\in \de\big(\!\epi_S(\phi_1)\cup\epi_S(\phi_2)\big) \setminus \big(\Gamma(\phi_1) \cap \Gamma(\phi_2)\big) =\vcentcolon \frk B
	\] and let $H_z$ be the tangent hyperplane to $\epi_S(\phi_1)\cup\epi_S(\phi_2)$ at $z$. Then one of the following is verified (see \Cref{fig:parabcases}):
	\begin{enumerate}[label=\textit{(\roman*)}]
		\item	$H_z \cap \de(\epi_S(\phi_1) \cup \epi_S(\phi_2))$ is a singleton;
		\item	$H_z \cap \de(\epi_S(\phi_1) \cup \epi_S(\phi_2))$ is a doubleton;
		\item	$H_z \cap \de(\epi_S(\phi_1) \cup \epi_S(\phi_2))$ is infinite.
	\end{enumerate}
	In the cases \emph{(i)} and \emph{(ii)} we have that $z \in \mathrm{ext}\,\barr{\frk C}$ (the \emph{extreme points} of the closure of the convex hull defined in \Cref{shp}),\syid{ext@\detokenize{$\mathrm{ext}\,E$}!\detokenize{the extreme points of $E$, namely those $p \in E$ such that if $p \in [x,y]$ with $x,y \in E$, then $p \in \{x,y\}$}} and in the case \emph{(iii)} we have $z \in \frk B \setminus \frk C$. Also, if $\frk S = \call H_1 \cap \call H_2$ satisfies
	\[
	\de \frk S \cap \de\!\bigcup_{j=1,2}\!\epi_S(\phi_j) = \{ z: H_z\ \text{satisfies \emph{(ii)}} \} = \bigcup_{j=1,2}(\de \call H_j \cap \Gamma(\phi_j)),
	\]
	then $\call H_j \subset (\mathrm{ext}\,\barr{\frk C} \cap \Gamma(\phi_j))\compl$, for $j=1,2$, in such a way that $\de\frk C \setminus \frk S = \mathrm{ext}\,\barr{\frk C}$. We leave the details to the interested reader. 
\end{remark}

\begin{figure}[ht]
	\renewcommand\thesubfigure{\emph{(\roman{subfigure})}}
	\centering
	\subfloat[][\centering The hyperplane touches only one paraboloid.]
	{\scalebox{0.69}{\includegraphics{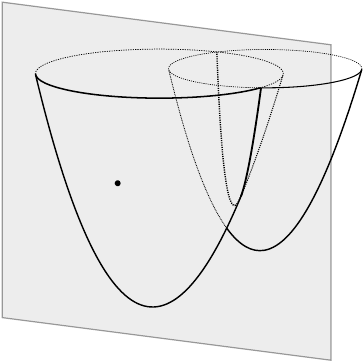}}} \qquad\quad 
	\subfloat[][\centering The hyperplane is tangent to both paraboloids.]
	{\scalebox{0.69}{\includegraphics{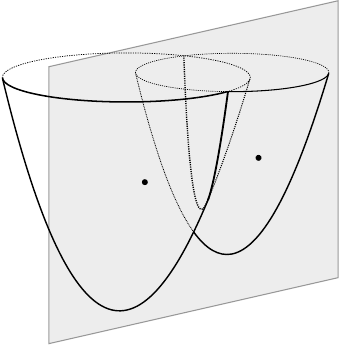}}} \qquad\quad
	\subfloat[][\centering The hyperplane cuts the other paraboloid.] 
	{\scalebox{0.69}{\includegraphics{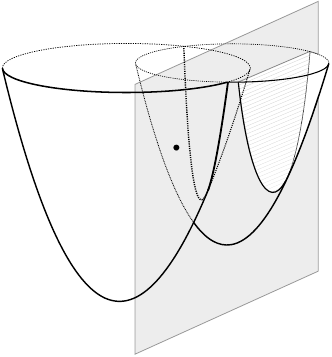}}} 
	\caption{The three situations described in \Cref{parab:cases}.}
	\label{fig:parabcases}
\end{figure}

\begin{proof}[Proof of \Cref{shp}]
	Set 
	\[
	e \defeq \frac{v_2 - v_1}{|v_2-v_1|} \quad \text{and} \quad m \defeq \frac{\phi_2(v_2)-\phi_1(v_1)}{|v_2-v_1|}
	\]
	and consider $z_1\defeq(y_1, \phi_1(y_1))\in\Gamma(\phi_1)$ and $z_2\defeq(y_2, \phi_2(y_2))\in\Gamma(\phi_2)$. We want to determine a condition on $y_1$ and $y_2$ which ensures that $z_1$ and $z_2$ have a common tangent hyperplane $H$. Equating normals $(D\phi_1(y_1), -1)$ and $(D\phi_2(y_2), -1)$ yields 
	\[
	y_1-v_1 = y_2-v_2.
	\]
	Thus $y_1 = v_1 + w$ and $y_2=v_2 + w$ for some $w\in\R^n$, so that $(\frac w r, -1)$ is normal to $H$. Since $z_1, z_2 \in H$ we have
	\[
	\tfrac1r\pair{y_1}w - \phi_1(y_1) = \tfrac1r\pair{y_2}w - \phi_2(y_2), 
	\]
	therefore $r(\phi_2(v_2) - \phi_1(v_1)) = \pair{y_2 - y_1}w = \pair{v_2-v_1}w$. Furthermore, since in particular $|y_1- v_1| = |y_2 - v_2|$, we know that $\phi_2(v_2) - \phi_1(v_1) = \phi_2(y_2) - \phi_1(y_1)$, proving that $\pair ew = rm$. Hence, if we decompose $w\in\R^n =  \langle e \rangle \oplus \langle e \rangle^\perp$,\syid{oplus@$\oplus$!the direct sum (of vector spaces)} then there exists $\barr w\in\langle e \rangle^\perp$\syid{perp@\detokenize{$V^\perp$}!the orthogonal complement of the linear subspace $V$} such that
	\[
	y_j = v_j+ rme + \barr w
	\]
	for $j=1,2$.
	Define now $\call H_1$ to be the open half-space whose boundary hyperplane $\de\call H_1$ has interior normal $(e,0)$ and passes through $(v_1+rme, 0)$. Similarly, define $\call H_2$ to have interior normal $(-e, 0)$ and boundary $\de \call H_2$ passing through $(v_2 + rme, 0)$. Clearly $\frk S = \call H_1 \cap \call H_2$ has width $|v_1-v_2|$. Also, the above proves that the mapping $\barr w \mapsto z_j$ parametrizes $\Gamma(\phi_j) \cap \de\call H_j$, for $j=1,2$, and the conclusion follows.
\end{proof}

\begin{figure}[htb]
	\begin{tikzpicture}
	\pgfdeclareimage{img}{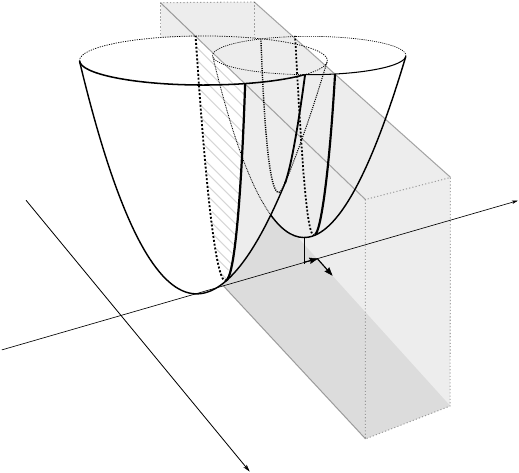}
	\node (S) at (0,0) {\pgfuseimage{img}};
	\node (ort) at (-0.75,-3.85) {\scriptsize $\langle e \rangle^\perp$};
	\node (Re) at (4.2,0.9) {\scriptsize $\langle e \rangle$};
	\node (vj) at (0.75,-0.7) {\scriptsize $v_j$};
	\node (rme) at (1.15,-0.17) {\scriptsize $rme$};
	\node (w) at (1.45,-0.7) {\scriptsize $\overline{w}$};
	\end{tikzpicture}
	\caption{The slab constructed in the proof of \Cref{shp}.}
\end{figure}
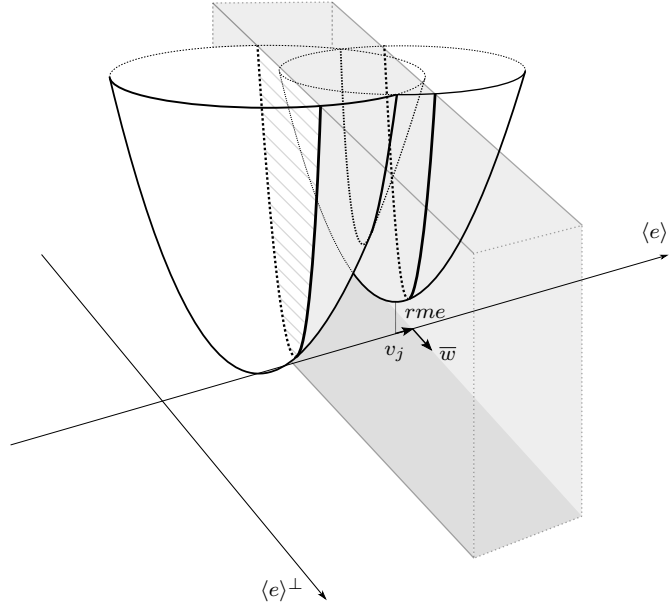

The slabbed hull property of Lemma~\ref{shp} will be used to prove a crucial property of the {\em vertex map}, which is the main tool we will need. In preparation for the definition of this map, we note that a generic upper contact quadratic function of radius $r$ for $u$ at $x$ takes the form $\phi(\cdot) = u(x) + \pair p{\cdot - x} + \frac1{2r}|\cdot -\, x|^2$ and has vertex $v=x-rp$. Moreover, if $x\in \mathrm{Diff}^1(u)$, then $p=Du(x)$ is unique. This motivates the following definition.

\begin{definition} 
	Let $u\in \USC(X)$. For each real number $r > 0$, the map\syid{Vr@\detokenize{$V_r$}!the upper vertex map of radius $r$}\sid{upper vertex map}
	\[ \begin{aligned}
	\C(u, X, \tfrac1r\I) \cap \mathrm{Diff}^1(u)\ &\xrightarrow{\,V_r\,}\ \R^n
	\\
	x \ &\longmapsto \ I - rDu(x)
	\end{aligned} \]
	is called the \emph{upper vertex map (of radius $r$) for $u$}.
\end{definition}

\begin{remark} \label{rmk:vertexc}
	Note that the domain of $V_r$ consists of those points $x \in X$ which admit a global upper contact quadratic function of radius $r$ and for which $u$ is differentiable. If $u$ is (locally semi-)convex, then the upper vertex map is well-defined on all of $\C(u, X, \tfrac1rI)$ since $\C(u, X, \frac1rI) \subset \mathrm{Diff}^1(u)$ by \Cref{datucp} (locally semiconvex functions are differentiable at all upper contact points).
	
	In particular, if $u$ is smooth, then the Jacobian of $V_r$ is $D V_r = \I - rD^2u$, and thus $0\leq DV_r \leq \I$ on $\C(u, X, \frac1r\I)$; that is, $V_r$ is nonexpansive (or 1-Lipschitz). This fact extends to $u$ which is merely convex.
\end{remark} 

\begin{prop} \label{prop:contr}
	Given a convex function $u$ defined on an open convex set $X \subset \R^n$ and given $r > 0$, the upper vertex map $V_r\colon \C(u, X, \frac1r\I) \to \R^n$ is nonexpansive.
\end{prop}

\begin{proof}
	For $j=1,2$, given $x_j \in \C(u, X, \frac1r\I)$, let $\phi_j$ be the upper contact quadratic function of radius $r$ for $u$ at $x_j$ with vertex $v_j \defeq V_r(x_j)$. With the notation in \Cref{shp}, since $u$ is convex and its graph lies below those of the $\phi_j$'s, then $\frk C \subset \epi_S(u)$. Recall that $\epi_S$ is the strict epigraph, so that $z_j\defeq(x_j, \phi_j(x_j)) = (x_j, u(x_j))\in \Gamma(\phi_j) \setminus \epi(u)$; hence $z_j\in\Gamma(\phi_j) \setminus\frk C$. Therefore, by \Cref{shp}, $z_j \notin \call H_j$. It follows that the $z_j$'s lie on opposite sides of $\frk S$, therefore $|x_1 - x_2| \geq |v_1 - v_2|$.
\end{proof}

We also have a condition under which $V_r(x)$ cannot be too far from $x$; it is a consequence of the following result, which guarantees a wealth of upper contact points for any upper semicontinuous function which is bounded from below.

\begin{prop}	\label{p:uvm}
	Suppose that $u\in\USC(K)$ is bounded from below on the compact subset $K \subset \R^n$ and let $\osc_K u \defeq \sup_K u - \inf_K u$ denote the oscillation of $u$ on $K$.\syid{osc@\detokenize{$\osc_X u$}!the oscillation of $u$ on $X$, namely $\sup_X u - \inf_X u$} Set 
	\[
	K^\delta \defeq\{y\in K:\ d(y, \de K) > \delta\}, \quad \text{with} \ \delta \defeq \sqrt{2r\osc_K u},
	\]
	and let $\C(u, K, \tfrac1r\I, v)$ be the set of the points of $K$ at which $u$ has an upper contact quadratic function of radius $r$ and vertex point $v$.\syid{CuXrv@\detokenize{$\C(u, X, \tfrac1r\I, v)$}!the set of the points of $X$ at which $u$ has an upper contact quadratic function of radius $r$ and vertex point $v$}
	Then, for each $v\in K_\delta$, the set $\C(u, K, \tfrac1r\I, v)$ is a nonempty and compact subset of $\barr B_\delta(v) \subset K$.
\end{prop}

\begin{proof}
	Set 
	\[
	c\defeq \sup_{y\in X} \left( u(y) - \tfrac1{2r}|y-v|^2 \right),
	\]
	where $c$ is finite since $u\in\USC(K)$. Then $\phi _v= \phi_v(v) + \frac1{2r}|\cdot -\, v|^2$ is an upper contact quadratic function of radius $r$ for $u$ at $x$ with vertex point $v$ if and only if $\phi_v(v) = c$ and the supremum is attained at $x$. Thus,
	\[
	\C(u, K, \tfrac1r\I, v) = \{x\in K: u(x) = \phi_v(x) \}
	\]
	is nonempty and compact.  Now, since $u$ is bounded below, its oscillation is finite. Suppose that $\delta$ is small enough to ensure that $K^\delta$ is not empty and fix $v\in K^\delta$. Clearly, if $x\in\C(u,K,\frac1r\I, v)$, then
	\[
	\tfrac1{2r}|v-x|^2 = \phi_v(x) - \phi_v(v) = u(x) - c \leq u(x) - u(v) \leq  \osc_K u.
	\]
	This proves that $\C(u, K, \frac1r\I, v) \subset \barr B_\delta(v)$.
\end{proof}

\section{A paraboloidal proof of S{\l}odkowski's density estimate}

The tools above allow one to prove the following ``paraboloidal'' counterpart of \Cref{slodestlem} which one finds in~\cite{hlqc} and which gives a lower bound on the measure of global upper contact points of type $\frac{1}{r}I$ on closed balls as defined in  Definition \ref{defn:ucp}.

\begin{lem}[Lower bound on the measure of the set of upper contact points] \label{jen:den}
	Let $u$ be a convex function on $\barr B_{\barr\rho}$ such that, for some $R>0$,
	\[
	0\leq u(y) < \frac{|y|^2}{2R} \quad \text{if $y\neq 0$}.
	\]
	Then
	\[
	\frac{|\C(u, \barr B_\rho, \tfrac1r\I)|}{|B_\rho|} > \left(1-\sqrt{\frac rR} \,\right)^n
	\]
	for all $r \in (0,R\,]$ and $\rho \in (0,\barr\rho\,]$.
\end{lem}

\begin{proof}
	Apply \Cref{p:uvm} with $K= \barr B_\rho$, $\rho\leq\barr\rho$, and $r\leq R$. Since by our assumptions $\osc_K u = M(\rho) \defeq \sup_{B_\rho} u$, we have $\delta = \delta(\rho) = \sqrt{2rM(\rho)}$ and for every $v\in \barr B_{\rho-\delta}$, the contact set $\C(u, \barr B_\rho, \frac1r\I, v)$ is a nonempty subset of $B_\rho$. Note that, by convexity, if $x$ is a point in the contact set of type $\frac1r I$, then the upper vertex map $V_r$ is well-defined at $x$. Hence, for every $v\in \barr B_{\rho-\delta}$ we have that $v= V_r(x)$ for some $x\in \C(u, \barr B_\rho, \frac1r\I, v) \subset \C(u, \barr B_\rho, \frac1r\I)$; that is,
	\[
	\barr B_{\rho-\delta} \subset V_r\!\left( \C(u, \barr B_\rho, \tfrac1r\I) \right).
	\]
	Since $V_r$ is a contraction, one obtains that
	\begin{equation*} 
	|B_{\rho-\delta}| \leq | \C(u, \barr B_\rho, \tfrac1r\I) |.
	\end{equation*}
	One concludes by noting that Bauer's maximum principle (see~\cite{bau}) says that $M(\rho) = \max_{\de B_\rho} u$, and thus $M(\rho) < \rho^2/(2R)$, whence $\delta < \rho\sqrt{r/R}$. Therefore
	\begin{equation*} 
	\rho - \delta > \rho \left(1-\sqrt{\frac rR}\, \right) \geq 0
	\end{equation*}
	and the desired conclusion easily follows.
\end{proof}

A novel paraboloidal version of \Cref{slod(iii)} is now easy to formulate and prove.

\begin{lem}[Upper bound on the largest eigenvalue of the Hessian; paraboloidal version] \label{hlKb}
	Let $u \colon U \to \R$ be convex on $U\subset \R^n$ open and convex. Suppose that there exists a paraboloid of radius $r$ which supports $\Gamma(u)$ from above at $(x,u(x))$; that is, there exits a quadratic function $\varphi$  of radius $r$ such that $\epi(\phi)$ lies above the graph $\Gamma(u)$ of $u$ and touches $\Gamma(u)$ at $(x,u(x))$. Then one has
	\[
	\call K(u,x) \leq \frac{1}{r} \,.
	\]
\end{lem}

\begin{proof}
	By \Cref{ucjt}({\sf D\;at\;UCP}) one knows that $x \in \Diff^1\! u$, and
	\[
	u(x+\cdot\,) \leq u(x) + \pair{Du(x)}{\cdot\,} + \frac1{2r}|\cdot|^2 \quad \text{near $0$}.
	\]
	The inequality $\call K(u,x) \leq \frac1r$ then immediately follows from the definition of $\call K$.
\end{proof}

These two paraboloidal instruments of Harvey and Lawson are sufficient to prove S{\l}odkowski's density estimate by following the same basic reasoning of S{\l}odkowski with his ``spherical'' instruments.

\begin{proof}[Proof of \Cref{slodthm}]
	Without loss of generality, suppose that $x^* = 0$, $u(0) = 0$ and $Du(0) = 0$. These reductions imply that $u \geq 0$ since $u$ is convex. Take $k^* < K < k$, and note that by the definition of $\call K$ there exists some $\bar\rho > 0$ such that $u < \frac K2 |\cdot|^2$ on $B_{\bar\rho}$.
	Now, for any $r > 0$ such that $k^{-1} < r < K^{-1}$, by \Cref{hlKb}, one has
	\[
	\C(u,\bar B_\rho, \tfrac 1r I) \subset \{ x :\ \call K(u,x) < k \} \qquad \forall \rho \in (0,\bar\rho).
	\]
	Hence, by \Cref{jen:den} with $R = K^{-1}$, one has
	\[
	\delta^-_{0}(\{ \call K(u,\cdot) < k \}) \geq \liminf_{\rho \dto 0} \frac{|\C(u,\bar B_\rho, \frac1r I)|}{|B_\rho|} \geq \left( 1-\sqrt{\frac{r}{R}}\, \right)^n.
	\]
	Letting $K \dto k^*$ (that is, $R \uto (k^*)^{-1}$) and $r \dto k^{-1}$, one finds that
	\[
	\delta^-_{0}(\{ \call K(u,\cdot) < k \}) \geq \bigg( 1-\sqrt{\frac{k^*}{k}} \,\bigg)^n = \left({\frac{k-k^*}{\sqrt{k}(\sqrt{k} + \sqrt{k^*})}} \right)^n \geq \left(\frac{k-k^*}{2k}\right)^n,
	\]
	as desired.
\end{proof}

\section{The Jensen--S{\l}odkowski theorem}

 We are now ready to present the reformulations of Harvey and Lawson of the classical lemmas of S{\l}odkowski and Jensen. There will be a Harvey--Lawson (HL) version for both lemmas, which will be shown to be equivalent. Moreover, the HL versions of the Jensen lemma will be shown to be equivalent to the original formulation. Finally, the two versions will be merged into a more general result in the Jensen--S{\l}odkowski Theorem~\ref{jenslod}.
 
 We begin by using the lower bound on the measure of the set of lower contact points in \Cref{jen:den} to prove the HL version of the S{\l}odkowski lemma as done in~\cite{hlqc}.
 
 \begin{lem}[HL--S{\l}odkowski lemma] \label{hl:slod} \sid{Lemma!S{\l}odkwoski|seeonly{S{\l}odkwoski's density estimate}} \sid{Lemma!S{\l}odkwoski!HL version (HL--S{\l}odkowski)}
 	Suppose that $u$ is a locally convex function on $X\subset \R^n$ with strict upper contact jet $(0,\lambda \I)$ at $x$. Then there exists $\barr\rho > 0$ such that
 	\[
 	|\C(u,\barr B_\rho(x),\lambda \I)| > 0 \qquad \forall \rho \in (0,\barr\rho\,].
 	\]
 \end{lem}
 
 \begin{proof}
 	Since $(0,\lambda \I)$ is strict upper contact jet at $x$ for $u$ convex, one has $\lambda \geq 0$. The case $\lambda = 0$ happens only when $u$ is constant near $x$ and the thesis is trivial in this case. Hence,  we may assume that $\lambda >0$ and, without loss of generality, that $x=0$ and $u(x) = 0$. Then the hypothesis of $(0, \lambda\I)$ being a strict upper contact jet for $u$ at $x=0$ is equivalent to the statement that there exists $\barr\rho >0$ such that
 	\[
 	0\leq u(y) < \frac1{2R}|y|^2 \qquad \text{for}\ 0<|y|\leq \barr\rho,
 	\]
 	where $R=1/\lambda$. Indeed, if $(0, \lambda\I)$ is a strict upper contact jet at $0$, then by \Cref{datucp} of differentiability at upper contact points, $Du(0) = 0$ and by convexity $u \geq 0$; the converse implication is trivial. The thesis now follows from \Cref{jen:den} with $r = R$.
 \end{proof}

 This lemma assures that if $u$ is convex then the set of contact points for some global upper contact quadratic function of radius $r$ on a ball $B_\rho(x)$ cannot be too small (in the sense of the Lebesgue measure), provided that $x$ is the vertex point of a strict upper contact quadratic function of radius $r$ for $u$ at $x$ and $\rho$ is sufficiently small. Indeed, the requirement that $x$ is both a contact point of such a quadratic function and that $x$ is its vertex point is equivalent to the statement that $(0, \lambda I)$ is a strict upper contact jet for $u$ at $x$.
 
 Furthermore, note that the HL--S{\l}odkowski lemma can be paraphrased by saying that if $u$ is locally convex and its quadratic perturbation $u-\frac{\lambda}{2}|\cdot-\,x|^2$ has a strict local maximum at $x$, then the set of global upper contact points of type $\lambda I$ for $u$ on each small ball about $x$ has positive measure. Analogously, Jensen's \Cref{u:jen} considers a semiconvex function with a strict local maximum at $\hat x$ and states that the set of the upper contact points near $\hat x$ whose associated upper contact jets are of the form $(p, 0)$, with $p \in \R^n$ small ($|p| \leq \delta$), has positive measure. 
 
 With the intuition that a bridge between convex and semiconvex functions can relate the two lemmas, Harvey and Lawson~\cite{hlqc} prove that the HL--S{\l}odkowski lemma is equivalent to the following result, which they simply call \emph{Jensen's lemma}. It roughly states that if $u$ is locally semiconvex, then the set of contact points of locally supporting hyperplanes from above near a strict local maximum point cannot be too small.
 
 \begin{lem}[HL--Jensen lemma]  \label{hl:jen} \sid{Lemma!Jensen!HL version (HL--Jensen)}
 	Suppose that $w$ is a locally semiconvex function on $X\subset \R^n$ with strict upper contact jet $(0, 0)$ at $x$. Then there exists $\barr\rho > 0$ such that
 	\[
 	|\C(w,\barr B_\rho(x),0)| > 0	\qquad  \forall \rho \in (0,\barr\rho\,].
 	\]
 \end{lem}
 
 \begin{prop}\label{prop:HLJS}
 	The HL--S{\l}odkowski and the HL--Jensen lemmas are equivalent.
 \end{prop}
 
 \begin{proof}
 	Consider $u$ and $w$ related by the identity $u = w + \frac\lambda 2|\cdot -\, x|^2$. Then, on $B_\rho(x)$, $u$ is convex if and only if $w$ is $\lambda$-semiconvex; note that by~\cref{def:qc}, the parameter of semiconvexity $\lambda(x)$ is locally constant, therefore we may assume that the locally semiconvex function $w$ is $\lambda$-semiconvex on $B_\rho(x)$ if $\rho>0$ small enough. Furthermore, $(0,\lambda\I)$ is a strict upper contact jet for $u$ at $x$ if and only if $(0,0)$ is for $w$. This shows that $\C(w,\barr B_\rho(x),0) = \C(u,\barr B_\rho(x),\lambda \I)$, and completes the proof of the desired equivalence.
 \end{proof}
 
 Next we will show that the HL--Jensen Lemma \ref{hl:jen} is indeed a reformulation of Jensen's \Cref{u:jen}. In view of our discussion above, note that one only needs to show that \emph{all} global upper contact jets of type $0 \in \call S(n)$ for $w$ whose associated upper contact points are close enough to $x$ are of the form $(p, 0)$ with $|p|\leq \delta$.
 
 \begin{prop} \label{cl:j-hlj}
 	The Jensen and the HL--Jensen lemmas are equivalent.
 \end{prop}
 
 \begin{proof}
 	Note that the request that $w$ has a strict local maximum at $\hat x$ is equivalent to the request that $(0,0)$ is a strict upper contact jet for $w$ at $\hat x$, which for $w$ $\lambda$-semiconvex is equivalent to the requirement that $(0, \lambda I)$ is a strict upper contact jet for $u \defeq w + \frac\lambda2|\cdot-\,\hat x|^2$ at $\hat x$. Furthermore, $w_p$ has a local maximum at $x$ if and only if $(-p, 0)$ is an upper contact jet for $w$ at $x$; hence
 	\[
 	K = \bigcup_{|p| \leq \delta} \C(w,\barr B_r(\hat x),0\,;\, p),
 	\]
 	where the set $\C(w,\barr B_r(\hat x),0\,;\, p)$ denotes the set of the global upper contact points on $\barr B_r(\hat x)$ for $w$ associated to the jet $(p,0)$.
 	It is easy to see that $\C(w, \barr B_r(\hat x), 0\,;\, p)$ is the set of those points $x\in \barr B_r(\hat x)$ such that $(p + \lambda(x-\hat x), \lambda I)$ is a global upper contact jet for $u$ at $x$. By~\Cref{datucp}, $Du(x) = p + \lambda(x-\hat x)$ and thus
 	\[
 	V_{\frac{1}{\lambda}}(x) \defeq  \Bigl(I - \frac{1}{\lambda}Du\Bigr)(x) = \hat x - \frac p\lambda \quad \implies\quad  |p| = \lambda\,|V_{\frac{1}{\lambda}}(x) - \hat x|,
 	\]
 	where $V_{\frac{1}{\lambda}}\colon \C(u, \barr B_r(\hat x), \lambda I) \to \R^n$ is the upper vertex map of radius $r = \frac1\lambda$ associated to $u$.
 	Since $V_{\frac{1}{\lambda}}$ is nonexpansive (\Cref{prop:contr}) and $V_{\frac{1}{\lambda}}(\hat x) = \hat x$,  we see that
 	\[
 	|V_{\frac{1}{\lambda}}(x) - \hat x| \leq |V_{\frac{1}{\lambda}}(x) - V_{\frac{1}{\lambda}}(\hat x)| + |V_{\frac{1}{\lambda}}(\hat x)- \hat x| \leq |x-\hat x|,
 	\]
 	whence $|p| \leq \lambda r$. Therefore, for $r \leq \delta/\lambda$, we have
 	\begin{equation*} \label{fatica}
 	K = \bigcup_{p \in \R^n} \C(w,\barr B_r(\hat x),0\,;\, p) = \C(w,\barr B_r(\hat x), 0).\qedhere
 	\end{equation*}
 \end{proof}
 
 \begin{remark}
 	We implicitly assumed that in the statement of Jensen's \Cref{u:jen}, by \emph{local} one means \emph{global on $\barr B_r(\hat x)$}. In this case, the thesis of Jensen's lemma holds for every $\delta >0$ and sufficiently small $r>0$, and the statements of the Jensen and HL--Jensen lemmas are equivalent. If instead, one requires that the thesis of Jensen's lemma must hold for all $r>0$, then, in the definition of $K$, $w_p$ having a local maximum at $x$ means that there exists $\eta=\eta(x)>0$ such that $x\in \C(w, \barr B_\eta(x), 0)$, so that the set $K$ is \emph{a priori} larger than $\C(w, \barr B_\rho(\hat x), 0)$ for any $\rho\leq r$. 
 \end{remark}
 
 \begin{remark}
 	The proof of Proposition \ref{cl:j-hlj} also completes an alternate proof of Jensen's \Cref{u:jen} which exploits Harvey and Lawson's strategy of focusing on upper contact quadratic functions.
 \end{remark}
 
 We now merge the two equivalent Lemmas~\ref{hl:slod} and \ref{hl:jen} into a theorem which, despite being another equivalent reformulation of the previous results, unveils their full generality. This is also done in~\cite{hlqc}.
 
 \begin{thm}[Jensen--S{\l}odkowski Theorem] \label{jenslod} \sid{Theorem!Jensen--S{\l}odkowski}
 	Suppose that $w$ is a locally semiconvex function with strict upper contact jet $(p, A)$ at $x$. Then there exists $\barr\rho > 0$ such that
 	\[
 	|\C(w,\barr B_\rho(x),A)| > 0	\qquad  \forall \rho \in (0,\barr\rho\,].
 	\]
 \end{thm}
 
 \begin{proof}
 	
 	Take 
 	\[
 	\phi \defeq -\pair p{\cdot - x} - Q_A(\,\cdot - x) + \frac\lambda2 |\cdot-\,x|^2,
 	\]
 	for some $\lambda >0$ to be determined. It is easy to see that $(0, \lambda\I)$ is a strict upper contact jet for $u\defeq w+\phi$ at $x$, and $\C(w, \barr B_\rho(x), A) = \C(u, \barr B_\rho(x), \lambda\I)$. Moreover, let $\alpha, \beta >0$ such that $w$ is $\alpha$-semiconvex and $A\leq \beta\I$; then $u$ is convex if $\lambda \geq \alpha + \beta$. The thesis now follows from \Cref{hl:slod}.
 \end{proof}
 
 \begin{remark}\label{rem:HLSJ}
 	We have just proved that the HL--S{\l}odkowski lemma implies the Jensen--S{\l}odkowski theorem, while it is clear that the converse is also true. Hence, the two results (and the HL--Jensen lemma as well) are in fact equivalent.
 \end{remark}

 \section{Consequences of the Jensen--S{\l}odkowski Theorem} \label{sec:proof:pusc}
 
 As previously mentioned, the main consequence of the Jensen--S{\l}odkowski Theorem~\ref{jenslod} that concerns us in this work is the property of partial upper semicontinuity of second derivatives of a locally semiconvex function stated in \Cref{pusc}({\sf PUSC\;of\;SD}). We will now give the proof, which also competes the proof of the Upper Contact Jet Theorem~\ref{ucjt}.

\begin{proof}[Proof of \Cref{pusc}]
	Given a sequence $\epsilon_k\dto 0$ of positive numbers, pick 
	\[
	x_k \in E \cap \Diff^2(u) \cap \C(u, \barr B_{\epsilon_k}(x), A+ \epsilon_k I).
	\]
	Note that this is certainly possible for every sufficiently small $\epsilon>0$ since $E \cap \Diff^2(u)$ has full measure in a neighborhood of $x$ by Alexandrov's \Cref{aleks:qc} and $(p,A+\epsilon I)$ is a strict upper contact jet for $u$ at $x$, therefore $\C(u, \barr B_{\epsilon}(x), A+ \epsilon I)$ has positive measure for $\epsilon$ small by the Jensen--S{\l}odkowski Theorem~\ref{jenslod}.
	We know that since $u$ is twice differentiable at $x_k$, then $D^2u(x_k) \leq A+\epsilon_k I$. Furthermore, by semiconvexity, $D^2u(x_k) + \lambda I \geq 0$ for some $\lambda\geq0$. By compactness, there exists a subsequence $\{x_j\}$ such that $D^2u(x_j) \to \bar A$ for some $\bar A \in [-\lambda I, A]$.
\end{proof} 

We also wish to note that, as the HL--S{\l}odkowski and the HL--Jensen lemmas are equivalent, despite dealing with convex and semiconvex functions, respectively, one should expect results coming from the former lemma to have a counterpart holding for semiconvex functions as well. This is the case, for instance, of S{\l}odkowski's Largest Eigenvalue \Cref{slodle}. In fact, if $u$ is $\lambda$-semiconvex near $x^*$ with $\call K(u,x) = k^*$, then $\tilde u \defeq u + \frac\lambda2|\cdot|^2$ is convex with $\call K(\tilde u, x^*) = \call K(u,x^*) + \lambda = k^* + \lambda$. Then \Cref{slodthm} tells that, for $k > k^*$,
\[
\delta^-_{x^*}(\{ \call K(u, \cdot) < k \}) = \delta^-_{x^*}(\{ \call K(\tilde u, \cdot) < k + \lambda \}) \geq \left(\frac{k-k^*}{2(k+\lambda)}\right)^n.
\]
Exploiting this density estimate for semiconvex functions one can deduce a semiconvex version of S{\l}odkowski's Largest Eigenvalue Theorem. Alternatively, it is possible to prove it directly from the Jensen--S{\l}odkowski \Cref{jenslod}, as we do here. 

\begin{thm}[S{\l}odkowski's Largest Eigenvalue Theorem: semiconvex version] \label{thm:LET_qc} \sid{Theorem!S{\l}odkowski's Largest Eigenvalue!semiconvex version}
	Let $u$ be locally semiconvex on an open set $X$ and suppose $\call K(u,x) \geq M$ for a.e.\ $x\in X$. Then $\call K(u,x) \geq M$ for all $x \in X$.
\end{thm}
\begin{proof}
	Suppose that $u$ is locally semiconvex on $X$ open and that $\call K(u,x) \geq M$ for a.e.\ $x \in X$. Seeking a contradiction to $\call K(u,x) \geq M$ for every $x \in X$, suppose that there exists $\hat x\in X$ such that $\call K(u, \hat x) < M$; say $\call K(u, \hat x) = M - 3\delta$ for some $\delta > 0$. By the definition~\eqref{LEF} of $\call K$, there exists $\barr\epsilon >0$ such that $\barr B_{\barr\epsilon}(\hat x) \subset X$ and
	\[
	2 \epsilon^{-2} \displaystyle\max_{|h|=1} \big( u(\hat x+\epsilon h) - u(\hat x) - \epsilon\pair{Du(\hat x)}{h} \big) \leq M-2\delta \qquad \forall\, \epsilon \in (0, \barr \epsilon),
	\]
	which is equivalent to
	\[
	u(\hat x + z) \leq  u(\hat x) + \pair{Du(\hat x)}{z} + \tfrac12(M-2\delta)|z|^2 \qquad \forall\, z\in B_{\barr\epsilon}(0).
	\]
	This last inequality is equivalent to $(Du(\hat x), (M-2\delta)I)$ being an upper contact jet for $u$ at $\hat x$. Hence, $(Du(\hat x), (M-\delta)I)$ is strict and by the Jensen--S{\l}odkowski ~\Cref{jenslod} we have 
	\[
	|\C(u, \barr B_\rho(\hat x), (M-\delta)I)| > 0
	\]
	for some $\rho > 0$. Since, by definition,
	\[
	u(y) \leq u(x) + \pair{Du(x)}{y-x} + \tfrac12(M-\delta)|y-x|^2
	\]
	for all $x\in \C(u, \barr B_\rho(\hat x), (M-\delta)I)$ and every $y \in \barr B_\rho(\hat x)$, this implies that
	\[
	\C(u, \barr B_\rho(\hat x), (M-\delta)I) \subset \{ \call K(u,\cdot) \geq M \}\compl,
	\]
	thus contradicting the hypothesis that $\call K(u,\cdot) \geq M$ almost everywhere.
\end{proof}

\section{An area formula proof of the Jensen--S{\l}odkowski theorem}\label{sec:JSAlex}

We conclude this chapter with a very interesting fact observed by Reese Harvey~\cite{har:pc}; namely, that the Jensen--S{\l}odkowski theorem can be seen as an immediate consequence of Alexandrov's maximum principle, a short proof of which can be in turn built around the well-known area formula (see \cite[Area Theorem~3.2.3]{fed:geo}).  Therefore, if one also considers that the Lipschitz version of Sard's theorem used in the proof of Alexandrov's theorem can be seen as a corollary of the area formula as well (cf.~\cite[Notes on the appendix]{user}),\footnote{The proof of Sard's theorem we propose in Appendix~\ref{proofsard}, based on Besicovitch's covering theorem, could seem at first sight more involuted than just invoking the area formula, yet the reader should be aware that classical proofs of the area formula itself (see, e.g.~\cite{evansgar}) rely, among other instruments, on the Radon--Nikodym theorem, which in turn also requires Besicovitch's result.} then the Area Theorem arises as the pillar of the theory we discussed so far.

We recall the following definitions of \emph{approximate (upper/lower) limit} and \emph{approximate differential}; for further details, see \cite[p.~212]{fed:geo}. 

\begin{definition}
	Let $g$ be a measurable function in a neighborhood of a point $x$. The \emph{approximate upper limit}\sid{approximate!(upper/lower) limit} of $g$ at $x$ is the quantity
	\[
	\aplimsup_{y\to x} g(y) \defeq \inf \left\{ t \in \R\ :\ \lim_{r \dto 0} \frac{|g^{-1}((t,+\infty)) \cap B_r(x)|}{|B_r(x)|}= 0 \right\};
	\]
	symmetrically, the \emph{approximate lower limit} is
	\[
	\apliminf_{y\to x} g(y) \defeq \sup \left\{ t \in \R\ :\ \lim_{r \dto 0} \frac{|g^{-1}((-\infty,t) \cap B_r(x)|}{|B_r(x)|}= 0 \right\}.
	\]
	Then, the \emph{approximate limit}, $\aplim_{y \to x} g(y)$, is the common value of approximate upper and lower limits, if it exists.
\end{definition}

\begin{definition}\label{defn:AD}
	Let $f\colon A \to \R^n$ with $A$ and $f$ measurable, $A\subset \R^n$. The \emph{approximate differential}\sid{approximate!differential} of $f$ at $a \in A$ is the unique linear map $\mathrm{ap}\,Df(a) \colon \R^n \to \R^n$ such that 
	\[
	\aplim_{x\to a} \frac{|f(x) - f(a) - \mathrm{ap}\, Df(a) (x-a)|}{|x-a|} = 0.
	\]
\end{definition}

If, for $f$ as in Definition~\ref{defn:AD}, one defines the {\em Stepanov set of $f$}\syid{Su@$S(u)$!the Stepanov set of $u$} by
\[ 
S(f) \defeq \left\{ x \in A:\ \aplimsup_{y\to x} \frac{|f(y) - f(x)|}{|y-x|} < \infty \right\},
\]
then by Federer's extension of Stepanov's theorem \cite[Lemma~3.1.7 and Theorem~3.1.8]{fed:geo}, $f$ has an approximate differential almost everywhere in $S(f)$; also, as noted in the beginning of~\cite[Section~3.2]{fed:geo}, the Area Theorem can be stated as follows.

\begin{thm}[Area Theorem]
\sid{Theorem!Area}
	For each measurable subset $E\subset S(f)$,
	\[
	|f(E)| \leq \int_{\R^n} N(\restr{f}{E}, y)\, \di y = \int_E |\det \mathrm{ap}\,Df(x)|\, \di x,
	\]
	where $N(\restr{f}{E}, y) \equiv \#\{x\in E:\ f(x) = y\}$.
\end{thm}

Let $u\colon X \to \R$ be locally semiconvex on $X$ open in $\R^n$ and consider the gradient map $Du \colon \Diff^1(u) \to \R^n$. We know that both $\Diff^1(u)$ and $\Diff^2(u)$ have full measure in $X$, and hence the next result follows easily.

\begin{lem}
	Let $u$ be as above. Then $\Diff^2(u) \subset S(Du)$.
\end{lem}

\begin{proof}
	Let $f= Du$. Fix $x\in \Diff^2(u)$ and set
	\[
	E_t \defeq \left\{ y\in \Diff^1(u):\  g(y)\defeq\frac{|f(y) - f(x)|}{|y-x|} > t \right\}.
	\]
	Since $|Du(y) - Du(x)| \leq |D^2u(x)|\,|y-x| + o(|y-x|)$ for $y \to x$, let $t>|D^2u(x)|$ and $\epsilon > 0$ be such that $g(y) \leq t$ for all $y\in B_\epsilon(x)\cap \Diff^1(u)$. Hence $|E_t \cap B_\epsilon(x)| = 0$,  yielding $\aplimsup_{y\to x} g(y) \leq |D^2u(x)|$ and the desired conclusion follows.
\end{proof}

Hence, we have the following version of the Area Theorem, which is more suitable for our purposes. 

\begin{thm}[Area Theorem for gradients of locally semiconvex functions] \label{area:cor} \sid{Theorem!Area!for gradients of locally semiconvex functions}
	Given $u \colon X \to \R$ a locally semiconvex function, for every measurable set $E\subset \Diff^2(u)$ one has
	\begin{equation}\label{AT_lqc}
	|Du(E)| \leq \int_E |\det D^2u(x)|\, \di x.
	\end{equation}
\end{thm}

Next, we use this Area Theorem \ref{area:cor} for locally semiconvex functions to prove a locally semiconvex version of \emph{Alexandrov's maximum principle}, which takes the form of the following pointwise \emph{a priori} estimate.

\begin{thm}[Alexandrov's maximum principle for locally semiconvex functions] \label{alexmp} \sid{Alexandrov's maximum principle|seeonly{maximum principle, Alexandrov}}\sid{maximum!principle!Alexandrov}
	Let $\Omega \subset \R^n$ be open and bounded. If \syid{diam@\detokenize{$\mathrm{diam}(E)$}!the diameter of the set $E$, namely \detokenize{$\sup\{\vert x-y\vert :\, x,y \in E\}$}}$u\in\USC(\overline\Omega)$ is locally semiconvex, then
	\begin{equation} \label{alexineq}
	\sup_{\overline{\Omega}} u \leq \sup_{\de\Omega} u + \frac{\mathrm{diam}(\Omega)}{\omega_n^{1/n}} \biggl(\int_{E(u)} |\det D^2u(x)|\, \di x\biggr)^{1/n},
	\end{equation}
	where $\omega_n \defeq  |B_1|$ is the Lebesgue measure of the unit ball in $\R^n$ and $E(u)$ is the set of all flat global upper contact points at which $u$ is twice differentiable; that is,
	\begin{equation}\label{E}
	E(u) \defeq  \C(u, \Omega, 0) \cap \Diff^2(u).
	\end{equation}
\end{thm}

\begin{proof} It suffices to prove the weaker inequality\footnote{It is weaker since $\sup_{\de\Omega} u \leq \sup_{\de\Omega} u^+$.}
	\begin{equation} \label{alexineq2}
	\sup_{\overline{\Omega}} u \leq \sup_{\de\Omega} u^+  +  \frac{\mathrm{diam}(\Omega)}{\omega_n^{1/n}} \biggl(\int_{E(u)} |\det D^2u(x)|\, \di x\biggr)^{1/n}.
	\end{equation}	
	Indeed, the translated function $\bar{u}\defeq  u - \sup_{\de\Omega} u$,
	where $\sup_{\de\Omega} u < + \infty$ since $u \in \USC(\overline{\Omega})$, also satisfies $\bar{u} \in \USC(\overline{\Omega})$ and $\bar{u}$ is locally semiconvex. Moreover, one clearly has
	\begin{equation}\label{u_bar}
	\bar{u}^+ = 0 \ \text{on} \ \partial \Omega \quad \text{and} \quad D^2 \bar{u} = D^2 u \ \text{on} \ E \defeq E(\bar{u}) = E(u).
	\end{equation}
	Applying the weaker inequality \eqref{alexineq2} to $\bar{u}$ yields
	$$
	\sup_{\overline{\Omega}} \bar{u}  \leq  \sup_{\de\Omega} \bar{u}^+ + \frac{\mathrm{diam}(\Omega)}{\omega_n^{1/n}} \Bigg(\int_{E(\bar{u})} |\det D^2\bar{u}(x)|\, dx\Bigg)^{1/n}
	$$
	which, using the definition of $\bar{u}$ and the properties \eqref{u_bar}, is equivalent to 
	$$
	\sup_{\overline{\Omega}} \Bigl( u - \sup_{\de\Omega} u \Bigr) \leq  \frac{\mathrm{diam}(\Omega)}{\omega_n^{1/n}} \Bigg(\int_{E(u)} |\det D^2u(x)|\, dx\Bigg)^{1/n},
	$$
	and hence \eqref{alexineq} for $u$.

	The idea of the proof of \eqref{alexineq2} is to show that the image of the gradient map $Du(E)$ (or some superset of it with the same measure) contains an open ball $B_{\delta}$ of a suitable radius $\delta$; in particular, it suffices to choose
	\begin{equation} \label{alex:delta}
	\delta \defeq \frac{\sup_{\overline\Omega} u - \sup_{\de\Omega} u^+}{\mathrm{diam}(\Omega)}\,.
	\end{equation}
	Now, if one can show that $B_{\delta} \subset Du(E)$, then 
	\begin{equation} \label{alexball}
	\omega_n \delta^n = |B_\delta| \leq |Du(E)| \leq \int_{E} |\det D^2u(x) |\, \di x,
	\end{equation}
	where the second inequality is the area formula \eqref{AT_lqc}. Solving the inequality \eqref{alexball} for $\delta$ defined by \eqref{alex:delta} gives the desired inequality \eqref{alexineq}.
	
	Notice that, without loss of generality, we can assume that $\delta$ defined by \eqref{alex:delta} is positive because if $\delta \leq 0$ then $|B_{\delta}| = 0$ and the inequality \eqref{alexineq} is trivial. This also underlines the meaning of the estimate \eqref{alexineq}; one has an upper bound on the value of positive interior maximum for $u$.
	
	Exploiting the argument above, it suffices to show that $B_{\delta} \subset  Du(E)$. We will, in fact, show a weaker inclusion, but one which is still sufficient to use the argument. More precisely, consider the \emph{a priori} enlargement of $E$ defined by \eqref{E} by eliminating the request that $u$ be twice differentiable; that is,
	\[
	\widetilde E \defeq \C(u, \Omega, 0).
	\]
	The map $Du$ is well-defined on $\widetilde E$ since $\C(u, \Omega, 0) \subset \Diff^1(u)$ by property ({\sf D\;at\;UCP}) of the Upper Contact Jet \Cref{ucjt}. With $u, \delta, E$ and $\tilde E$ as above, in the following two lemmas, we will show that
	\begin{equation}\label{E_Etilde}
	|Du(\widetilde E) \setminus Du(E)| = 0 \quad \text{and} \quad B_{\delta} \subset Du(\tilde E),
	\end{equation}
	which clearly yields the desired estimate \eqref{alex:delta} since $|B_{\delta}| \leq |Du(\widetilde E)| = |Du(E)|$ and then one can use the area formula as in \eqref{alexball}.
\end{proof}
As noted, the proof of Theorem \ref{alexmp} has been reduced to the claims in \eqref{E_Etilde}, which we now prove. 

\begin{lem}
	Let $u$, $E$ and $\widetilde E$ be as above. Then $|Du(\tilde E) \setminus Du(E)| = 0$.
\end{lem}

\begin{proof}
	We will prove that the larger set $Du(\widetilde E \setminus E)$ is null. First, note that  $u$ being locally semiconvex implies that $\widetilde E\setminus E$ is null by Alexandrov's \Cref{aleks:qc}. This suggests using the well-known fact that Lipschitz maps on $\R^n$ send null sets to null sets. To this end, the local semiconvexity of $u$ also implies that there is an exhaustion of $\Omega$ by compact convex sets  $\{K_i\}_{i\in\N}$ such that $u$ is $\lambda_i$-semiconvex on $K_i$ for some $\lambda_i > 0$. For each $i \in \N$, set 
	\[
	E_i \defeq \widetilde E \cap K_i \subset \C(u, K_i, 0) \quad  \text{and} \quad u_i \defeq u + \frac{\lambda_i}2|\cdot|^2,
	\]
	which is convex on $K_i$. 
	Since $\C(u, K_i, 0) = \C(u_i, K_i, \lambda_i\I)$, we know that the upper vertex map $V_{1/\lambda_i}$ for $u_i$ is well-defined and $1$-Lipschitz on $E_i$ (by \Cref{rmk:vertexc} and \Cref{prop:contr}); that is,
	\[
	\Bigl| \Bigl( I - \frac{1}{\lambda_i}Du_i\Bigr)(p) - \Bigl( I - \frac{1}{\lambda_i}Du_i\Bigr)(q) \Bigr| \leq |p - q|, \quad \forall\, p,q \in \R^n.
	\]
	This implies that $Du_i$ is $2\lambda_i$-Lipschitz, and thus $Du$ is $3\lambda_i$-Lipschitz, on $E_i$. We can now conclude that $Du(E_i \setminus E)$ is null for every $i\in\N$, by using the fact noted above that Lipschitz functions map null sets into null sets (see, e.g., \cite[Theorem~2.8(i)]{evansgar}). The claim now follows from the countable subadditivity of the Lebesgue measure.
\end{proof}

\begin{lem}
	Let $u, \widetilde{E}$ and $\delta$ be as above. Then $B_\delta \subset Du(\tilde E)$.
\end{lem}

\begin{proof}
	As noted above, we may assume that $\delta$ defined by \eqref{alex:delta} is positive since the desired inclusion is trivial for $\delta \leq 0$. This means that $u$ has a positive interior maximum which is attained only at interior points.

	For each $p\in B_\delta$, consider $\max_{\overline{\Omega}} \left( u(\cdot) - \langle p, \cdot \rangle \right)$ which is finite since $u(\cdot) - \langle p, \cdot \rangle \in \USC(\overline{\Omega})$, and let $x_p \in \overline{\Omega}$ be a point realizes this maximum; that is
	\begin{equation}\label{x_p}
	u(y) \leq u(x_p) + \langle p, y -  x_p \rangle \quad \forall \, y \in \overline{\Omega}.
	\end{equation}
	If we can show that
	\begin{equation}\label{x_p_interior}
	x_p \in \Omega \quad  \text{($x_p$ is an interior point)}
	\end{equation}
	then \eqref{x_p} says that $x_p$ a contact point of type $0$ for $u$ on $\Omega$; that is, $x_p \in \widetilde E = \C(u, \Omega, 0)$. Hence, by \Cref{ucjt}({\sf D\;at\;UCP}), $u$ is differentiable at $x_p$ with $Du(x_p) = p$ so that $p \in Du(\widetilde E)$, which proves the lemma since $p \in B_{\delta}$ was arbitrary.
	
	To see that \eqref{x_p_interior} holds we use an argument by contradiction. Suppose that $x_p\in\de\Omega$ and pick $y\in\overline\Omega$ with $u(y) = \sup_{\overline\Omega} u$ (in fact, $y \in \Omega$ in light of the assumption $\delta > 0$). Then from \eqref{x_p} one has
	\[
	\sup_{\overline\Omega} u = u(y) \leq u(x_p) + \pair p{y-x_p} \leq \sup_{\partial \Omega} u^+ + |p| \, {\rm diam}(\Omega),
	\]
	and hence $|p| \geq ( \sup_{\overline\Omega} u - \sup_{\partial \Omega} u^+)/{\rm diam}(\Omega) =\vcentcolon \delta$, which contradicts $p\in B_\delta$.
\end{proof}

A proof of the Jensen--S{\l}odkowski Theorem \ref{jenslod} using Alexandrov's maximum principle (Theorem \ref{alexmp}) now proceeds as follows.

\begin{proof}[An area formula proof of \Cref{jenslod}]
	We recall that Remark~\ref{rem:HLSJ} noted that the HL--Jensen \Cref{hl:jen} is an equivalent formulation of the Jensen--S{\l}odkowski Theorem~\ref{jenslod}. Hence, we will prove that for all $\rho > 0$ small, $|\C(u, \barr B_\rho(x), 0)| > 0$, assuming that $u$ is a locally semiconvex with strict upper contact jet $(0, 0)$ at $x$. By translating $u$ we can assume that $u(x) > 0$ and hence, for all $\rho > 0$ small, the radius $\delta$ defined by~(\ref{alex:delta}) is positive. By (\ref{alexball}) with $\Omega =  B_\rho(x)$ (recall the definition~\eqref{E} of $E(u)$), for all sufficiently small $\rho > 0$ one has
	\[
	0 < |B_\delta| \leq \int_{E} |\det D^2 u(y)|\, \di y \quad \text{with} \ E \defeq  \C(u, B_\rho(x), 0) \cap \Diff^2(u),
	\]
	which forces $\C(u, \barr B_\rho(x), 0)$ to have positive measure. 
\end{proof}

The ``paraboloidal'' proof of the Jensen--S{\l}odkowski theorem passed through a density estimate for a certain set of global upper contact points for a convex function (\Cref{jen:den}); likewise, we can deduce a similar estimate, for the set of flat global upper contact points for a semiconvex function on small balls around a strict local maximum point, using Alexandrov's maximum principle.

\begin{cor}
	Let $u \colon X \to \R^n$ be $\lambda$-semiconvex, with strict upper contact jet $(0,0)$ at $x$. Then there exists $\bar\rho > 0$ such that
	\[
	\frac{|\C(u, \barr B_\rho(x), 0)|}{|B_\rho(x)|} \geq  \left(\frac{\sup_{B_\rho(x)} u - \sup_{\de B_\rho(x)} u^+ }{2\lambda\rho^2}\right)^n \qquad \forall \, \rho \in (0,\bar\rho\,].
	\]
\end{cor}

\begin{proof}
	Consider $y\in E\defeq   \C(u, B_\rho(x), 0) \cap \Diff^2(u)$. By the semiconvexity of $u$ we have $D^2u(y) \geq - \lambda \I$, for some $\lambda>0$, and since $y$ is a maximum point of $u + \pair p\cdot$ for some $p\in \R^n$, we have $D^2u(y) \leq 0$. Hence $|\det D^2 u| \leq \lambda^n$ on $E$, and, letting $\delta$ be as in \eqref{alex:delta}, \Cref{alexmp} with $\Omega = B_\rho(x)$ (along with \Cref{aleks:qc}) yields $|B_\delta| \leq \lambda^n |\widetilde E|$, with $\widetilde E = \C(u, \barr B_\rho(x), 0)$. The desired conclusion now is obtained by noting that
	\[
	|B_\delta|= \left(\frac{\sup_{B_\rho(x)} u - \sup_{\de B_\rho(x)} u^+ }{2\rho^2}\right)^n |B_\rho(x)|. \qedhere
	\]
\end{proof}

\chapter{Semiconvex approximation of upper semicontinuous functions} \label{chap:apqc}

In this chapter we present the well-known semiconvex approximation of upper semicontinuous functions by way of the {\em sup-convolution}. This approximation is of crucial importance for both general potential theories and the viscosity theory for PDEs since it allows one to use information on upper contact jets along the approximating sequences. In addition, we will also give  an alternative proof of the conventional Theorem on Sums of Crandall--Ishii--Lions which makes use of the upper contact jet technology developed up to this point.

\section{Elementary properties of semicontinuous functions}

We begin by recalling the definitions and some well-known and useful properties of semicontinuous functions. Here and below $\dt{B}_{\rho}(x_0) \defeq  B_{\rho}(x_0) \setminus \{x_0\}$ denotes the punctured open ball of radius $\rho > 0$ and center $x_0$.\syid{Brxd@\detokenize{$\dt{B}_{\rho}(x)$}!the punctured open ball \detokenize{$B_{\rho}(x) \setminus \{x\}$}}

\begin{definition}\label{defn:usc_lsc} \sid{semicontinuous}
	Let $X \subset \R^n$.
	\begin{itemize}
		\item[(a)] A function $u \colon X \to  [-\infty, +\infty)$ is \emph{upper semicontinuous at $x_0 \in X$} if
		\begin{equation}\label{usc1}
		u(x_0) \geq \limsup_{X \ni x \to x_0} u(x) \defeq \lim_{\rho \dto 0} \sup_{x \in \dt{B}_{\rho}(x_0) \cap X} u(x) = \inf_{\rho > 0} \sup_{x \in \dt{B}_{\rho}(x_0) \cap X} u(x)
		\end{equation}
		If \eqref{usc1} holds for each limit point $x_0 \in X$,  $u$ is  \emph{upper semicontinuous on $X$} which is denoted by $u \in \USC(X)$.\syid{USC@$\USC(X)$!the space of all upper semicontinuous functions on $X$}
		\item[(b)] A function $u \colon X \to (-\infty, +\infty]$ is \emph{lower semicontinuous at $x_0 \in X$} if
		\begin{equation}\label{lsc1}
		u(x_0) \leq \liminf_{X \ni x \to x_0} u(x) \defeq \lim_{\rho \dto 0} \sup_{x \in \dt{B}_{\rho}(x_0) \cap X} u(x) = \sup_{\rho > 0} \sup_{x \in \dt{B}_{\rho}(x_0) \cap X} u(x).
		\end{equation}
		If \eqref{lsc1} holds for each $x_0 \in X$, $u$  is \emph{lower semicontinuous on $X$} which is denoted by $u \in \LSC(X)$.\syid{LSC@$\LSC(X)$!the space of all lower semicontinuous functions on $X$}
	\end{itemize}
\end{definition}
Before recalling some well-known properties, a few remarks are in order. First, any function $u$ is both upper and lower semicontinuous at isolated points $x_0 \in X$ since, by definition $\sup \emptyset = - \infty$ and $\inf \emptyset = + \infty$. Second, clearly $u$ is lower semicontinuous at $x_0$ if and only if $-u$ is upper semicontinuous at $x_0$ and hence one can (and we will) focus primarily on upper semicontinuous functions.  Third, there are other equivalent definitions of upper semicontinuity in terms of the topology of upper and lower level sets of $u$; for example, $u \in \USC(X)$ in the sense of Definition \ref{defn:usc_lsc} if and only if
$$
\mbox{$\{ x \in X: \ u(x) < \alpha \}$ is (relatively) open in $X$ for each $\alpha \in \R$.}
$$
Fourth, semicontinuous functions which take values in $\R$ or $[-\infty, +\infty]$ are also of use, but the choices made in Defintion \ref{defn:usc_lsc} will be useful for the viscosity calculus of Part~\ref{sas}. This choice is explained in the following remark.

\begin{remark}[Choice of codomain] \label{rmk:codom} 
Given $u \colon X \to [-\infty, +\infty)$ upper semicontinuous on $X$, we have the following:
	\begin{itemize}
		\item[(a)] if $X$ is closed, then 
		\[
		\hat u(x) \defeq \begin{cases}
		u(x) & \text{on $X$} \\
		-\infty & \text{outside $X$}
		\end{cases}
		\]
		defines an upper semicontinuous extension of $u$ to all of $\R^n$;
		\item[(b)] for each $K \subset X$ compact, if $u \not\equiv -\infty$, then it admits a finite maximum on $K$ (see Proposition~\ref{prop:usc}(c));
	\end{itemize}
	Analogous facts hold for  $u: X \to (-\infty, +\infty]$ which is lower semicontinuous on $X$.
\end{remark}

\begin{prop}[Elementary properties of semicontinuous functions]\label{prop:usc}  \sid{property!elementary!of semicontinuous functions}
Let $X \subset \R^n$. The following hold.
	\begin{itemize}
		\item[(a)] A real-valued function on $X$ is continuous at $x_0 \in X$ if and only if it is both upper and lower semicontinuous at $x_0$. In particular, $\USC(X) \cap \LSC(X) = C(X)$.\syid{C@$C(X)$!the space of all continuous functions on $X$}
		\item[(b)] The sum of two $[-\infty, +\infty)$-valued functions on $X$ which are both upper semicontinuous at $x_0 \in X$ is upper semicontinuous at $x_0$. In particular, $\USC(X) + \USC(X) \subset \USC(X)$.
		\item[(c)] For every compact subset $K \subset X$, any upper semicontinuous function on $K$ admits a maximum value which is finite if $u \not\equiv -\infty$ on $K$.
		\item[(d)] Suppose that $X = \R^n$ and $u \in \USC(\R^n)$ is \emph{anti-coercive}; that is, $u(x) \to - \infty$ as $|x| \to \infty$. Then $u$ is bounded above and the supremum $\sup_{\R^n} u$ is attained; that is, there exists $\bar x \in \R^n$ such that $\sup_{\R^n} u = \max_{\R^n} u = u(\bar x)$.
		\item[(e)] The pointwise infimum of any family of upper semicontinuous functions on $X$ is upper semicontinuous; that is, $\inf_{i \in I} u_i \in \USC(X)$ if $\{u_i\}_{i \in I} \subset \USC(X)$,
		where the index set $I$ has arbitrary cardinality. On the other hand, the pointwise supremum $\sup_{i \in I} u_i = \max_{i \in I} u_i \in \USC(X)$, provided that $I$ if finite.	
	\end{itemize}
\end{prop}
\begin{proof}
	\underline{{(a)}} \ It is enough to note that the continuity of $u$ in $x_0$, written as $\lim_{X \ni x \to x_0} u(x)$,
	is clearly equivalent to the two inequalities \eqref{usc1} and \eqref{lsc1}. 
	
	\noindent\underline{{(b)}} \ It is enough to note that the sup operation is sublinear
	$$
	\sup_{x \in \dt{B}_{\rho}(x_0) \cap X}(u(x) + v(x)) \leq  \sup_{x \in \dt{B}_{\rho}(x_0) \cap X}u(x) +  \sup_{x \in \dt{B}_{\rho}(x_0) \cap X} v(x)
	$$
	for any pair of functions $u$ and $v$ on $X$, and then one passes to the limit as $\rho \dto 0$ if both $u$ and $v$ are upper semicontinuous at $x_0$.
	
	\noindent\underline{{(c)}} \ We can assume that $K$ is not a singleton $\{x_0\}$, since the conclusion is trivial in this case. Then, if $u \in \USC(K)$ is not identically $-\infty$, consider a maximizing sequence $\{ x_j \}_{j\in\N} \subset K$; that is,
	$u(x_j) \to \sup_K u$ as $j\to\infty$.
	Since $K$ is sequentially compact, there exists a point $x_0 \in K$ and a subsequence $x_{j_k} \to x_0$ as $k \to + \infty$. Since $u \not\equiv -\infty$, by the choice of the codomain and by the semicontinuity of $u$, one has 
	\begin{equation*} 
	{+\infty} > u(x_0) \geq \limsup_{k\to\infty} u(x_{j_k}) = \lim_{k \to +\infty} u(x_{j_k}) = \sup_K u,
	\end{equation*}
	so that $x_0 \in K$ realizes $\sup_K u$.
	
	\noindent\underline{{(d)}} \ Let $M \defeq \sup_{\R^n} u$. By the anti-coercivity of $u$, there exists an open ball $B$, centered at the origin, such that $\restr{u}{B\compl} < M$. Therefore $\sup_{\R^n} u = \sup_{\barr B} u$, and the conclusion follows from the previous point.
	
	\noindent\underline{{(e)}}\ For the first claim, if $v\defeq  \inf_{i \in I} u$ then one has
	\begin{equation}\label{inf1}
	v(x) \leq u_i(x) \quad \forall \, x \in X,\, \forall \, i \in I.
	\end{equation}
	Now, for each $x_0 \in X$ by upper semicontinuity and \eqref{inf1} one finds
	\begin{equation*}
	u_i(x_0) \geq \limsup_{X \ni x \to x_0} u_i(x) \geq \limsup_{X \ni x \to x_0} v(x) \quad \forall \, i \in I,
	\end{equation*}
	and taking the infimum over $i \in I$ yields $v(x_0) \geq \limsup_{X \ni x \to x_0} v(x)$.
	The second claim of part (d) is left to the reader.
\end{proof}

\begin{remark}\label{rem:usc_ext_X_compact}
	As a consequence of Proposition~\ref{prop:usc}(c), the extension $\hat{u}$ of Remark~\ref{rmk:codom}(a) is bounded above if $X$ is compact.
\end{remark}

\section{Semiconvex approximation via sup-convolution} \label{ch:qca}

One reason why semiconvex functions are so interesting in general potential theories and the viscosity theory of nonlinear PDEs is the widely used fact that any upper semicontinuous $u$ function (which is bounded above) is a decreasing limit of a sequence of semiconvex functions. The convergence of the semiconvex approximation is pointwise and can be locally uniform if  $u$ is more regular. This fact is well known and for further details see, for example, \cite{lioham, crankoc, hldir09}. The semiconvex approximation of $u$ is given by \emph{sup-convolutions} of $u$ which are defined next.

\begin{definition} \label{def:supconv} \sid{sup-convolution}
	Let $X \subset \R^n$ and let $u \in \USC(X)$ be bounded above. For each $\epsilon >0$ the $\epsilon$-\emph{sup-convolution} of $u$ the function defined by\syid{ueps@\detokenize{$u^\epsilon$}!the $\epsilon$-sup-convolution of $u$}
	\begin{equation} \label{supconv}
	u^\epsilon(x) \defeq \sup_{y \in X} \Big( u(y) - \frac1{2\epsilon} |y-x|^2 \Big), \quad \text{for each} \ x \in  \R^n.
	\end{equation}
\end{definition}

Before proving the results we need, here are four noteworthy facts.

\begin{remark}
	One can visualize how a function is transformed by the sup-convolution operation by noticing that $u^\epsilon$ is the upper envelope of all quadratic functions of radius $-2\epsilon$ with vertex at some point of the graph of $u$.\footnote{By \emph{upper envelope} of a family of functions $\scr F$ we mean the function $g \defeq \sup_{f \in \scr F} f$. Also note that a quadratic function with negative radius is well-defined: it just opens downwards.} Also, notice that, even if we decided to ``preserve'' the domain of $u$ in \Cref{def:supconv} by defining its sup-convolution only on $X$, $u^\epsilon$ is in fact well-defined by \eqref{supconv} on the whole space $\R^n$.
\end{remark}

\begin{remark}
	The use of the term \emph{convolution} for $u^\epsilon$ defined above is justified by the fact that it is obtained by a transformation which is an actual convolution. Indeed, if we consider the functions to be valued in the so-called \emph{tropical semiring} $(\R \cup \{-\infty\}, \vee, +)$, we see that $u^\epsilon = u \ast (-\frac1{2\epsilon}|\cdot|^2)$.
	
	Moreover, it is  not difficult to show that the sup-convolution is the pointwise limit of integal convolutions (mollifications). The proof of this fact is a direct application of a general large deviations result, due to Varadhan; it is given in \cite{capdol} for the analogous result concerning the {\em inf-convolution}, whose definition is given in the following remark.   
\end{remark}

\begin{remark}
	In \Cref{baspropsc} below, we prove that sup-convolutions are semiconvex functions. In the obvious manner, one can also define the $\epsilon$-\emph{inf-convolution} 
	\begin{equation} \label{infconv}
u_\epsilon(x) \defeq \inf_{y \in X} \Big( u(y) + \frac1{2\epsilon} |y-x|^2 \Big), \quad \text{for each} \ x \in  \R^n.
\end{equation}	
and prove that it is semiconcave. Moreover, it is not difficult to see that if $\epsilon$ is sufficiently small, then the $\epsilon$-sup-convolution (resp.\ $\epsilon$-inf-convolution) of a semiconcave (resp.\ semiconvex) function remains semiconcave (resp.\ semiconvex). This fact, along with the characterization given in \Cref{charc11},  proves a part of a well-known result of Lasry and Lions~\cite{laslio} on \emph{sup-inf-convolutions} being $C^{1,1}$.
\end{remark}

\begin{remark}
The sup-convolution and inf-convolution are also used to give representation formulas for generalized solutions to Cauchy problems for Hamilton-Jacobi equations on $\R^n$ by way of the Hopf--Lax formula. This was first shown in \cite{hopf} and extended to the viscosity setting in \cite{barev}, with many subsequent developments. The simplest example is that, when $X = \R^n$, the function $u(x, \epsilon) \defeq u_\epsilon(x)$ defined in \eqref{infconv} is the solution at time $t =\epsilon$ of 
\[
\partial_t u + \frac{1}{2} |Du|^2 = 0 \quad \text{with} \quad u(x,0) = u(x).
\]
This helps to understand the regularizing effect of the Hamilton--Jacobi equation, since the inf-convolution of a semicontinuous initial datum is semiconcave. 
\end{remark}

Four important properties of the sup-convolution are gathered in the following theorem. We have already partially revealed the first one.

\begin{thm} \label{baspropsc}
	Let $u \in \USC(X)$ be bounded above, let $u^\epsilon$ be its $\epsilon$-sup-convolution. The following hold:
	\begin{enumerate}[label=(\roman*)]
		\item	$u^\epsilon$ is $\frac1\epsilon$-semiconvex on $\R^n$;
		\item	if $X$ is closed then the supremum in (\ref{supconv}) is attained; that is, for every $x \in \R^n$ there exists $\xi = \xi(\epsilon,x) \in \R^n$ such that $u^\epsilon(x) = u(\xi) - \frac1{2\epsilon}|\xi-x|^2$;
		\item	$u^\epsilon$  decreases pointwise to $u$ on $X$ as $\epsilon\dto 0$; also, if $x \in X$ is a local maximum point for $u$, then $u^\epsilon(x) = u(x)$ for any $\epsilon$ sufficiently small;
		\item	if $u$ is also bounded from below, and thus $|u| \leq M$ for some $M>0$, then
		\begin{equation} \label{supcball}
		u^\epsilon = \max_{z \in \barr B_\delta} \Big( u(\cdot-z) - \frac1{2\epsilon} |z|^2 \Big), \quad \text{where $\delta \defeq 2\sqrt{\epsilon M}$},
		\end{equation}
		on $X^\delta \defeq \{ x \in X :\ d(x,\de X) > \delta \}$.
	\end{enumerate}
\end{thm}

\begin{proof}
	\underline{{(i)}} \ It suffices to note that
	\[
	u^\epsilon + \frac1{2\epsilon}|\cdot|^2 = \sup_{y \in \R^n} \bigg( u(y) -  \frac1{2\epsilon}|y|^2 +  \frac1{\epsilon}\pair y\cdot \bigg)
	\]
	is the supremum of a family of affine functions, hence it is convex.
	
	\noindent\underline{{(ii)}} \ For each $x \in X$ fixed, one has
	\begin{equation}\label{max1}
	u^{\epsilon}(x) = \sup_{y \in X} v_x(y), \quad \text{where} \quad  v_x(y)\defeq  u(y) - \frac1{2\epsilon} |y-x|^2.
	\end{equation}
	Each function $v_x$ is upper semicontinuous on $X$ since it is the sum of upper semicontinuous functions. As $X$ is closed, extend each $v_x$ to be the upper semicontinuous function $\hat{v}_x$ defined on all of $\R^n$ by assigning the value $-\infty$ on $X\compl$ (see Remark~\ref{rmk:codom}(a)). Notice that with this extension, one can rewrite \eqref{max1} as
	\begin{equation}\label{max2}
	u^{\epsilon}(x) = \sup_{y \in \R^n}\hat{v}_x(y) \quad \forall \, x \in \R^n.
	\end{equation}
	Moreover, since  $\hat{v}_x \equiv -\infty$ on $X\compl$ and is defined by \eqref{max1} on $X$ (where $u$ is bounded above), one has that $\hat{v}_x$ is anti-coercive. Hence, for each $x \in \R^n$ and $\epsilon > 0$ fixed, applying Proposition~\ref{prop:usc}{(d)} to $\hat{v}_x$ gives the existence of a point $\xi = \xi(x, \epsilon) \in \R^n$ which realizes $u^{\epsilon}(x)$ given by \eqref{max2}.
	
	\noindent\underline{{(iii)}} \ Since $y = x \in X$ competes in the definition \eqref{supconv} of $u^{\epsilon}(x)$ and since the penalization term in \eqref{supconv} is non positive, one  has
	\begin{equation} \label{uepsdecr}
	u(x)\leq u^{\epsilon'}(x) \leq u^\epsilon(x) \quad \text{whenever $\epsilon' \leq \epsilon$}, \quad \forall \, x \in X.
	\end{equation}
	The argument of point {(ii)} (where each $\hat{v}_x$ is anti-coercive) shows that for each $x \in \R^n$ there exists a ball $B_{\rho(x, \epsilon)}(x)$ about $x$ such that
	\begin{equation} \label{questa}
	u^\epsilon (x) = \max_{\barr B_{\rho(x, \epsilon)}(x)} \left(u - \frac1{2\epsilon} | \cdot - \, x|^2\right) \leq \max_{\barr B_{\rho(x, \epsilon)}(x)} u,
	\end{equation}
	where we can write ``$\max$'' in place of ``$\sup$'' since both $u$ and its quadratic penalizations are upper semicontinuous and the closed balls are compact. It is easy to see that for each fixed $x$ we have $\rho(x,\epsilon) \dto 0$ as $\epsilon \dto 0$; indeed, if not, then for all $\epsilon$ sufficiently small, $\arg\max_{y \in X} \left( u - \frac1{2\epsilon} |\cdot -\, x|^2\right) \subset B_r(x)\compl$, for some $r = r(x) > 0$ independent of $\epsilon$, yielding $u^\epsilon(x) \leq \sup_{B_r(x)} u - \frac{r^2}{2\epsilon} < u(x)$ if $\epsilon$ is sufficiently small, which contradicts \eqref{uepsdecr}. Hence, the rightmost term in \eqref{questa} converges to $u(x)$ as $\epsilon \dto 0$,\footnote{The term converges to $\max\{ u(x), \limsup_{y\to x} u(y) \}$, where $ \limsup_{y\to x} u(y) \leq u(x)$ by the upper semicontinuity of $u$.} yielding $u^\epsilon(x) \dto$ $u(x)$.
	
	By \eqref{questa} we also see that if $x$ is a local maximum point for $u$, then for $\epsilon$ so small that $x$ is a global maximum on $\barr B_{\rho(x)}(x)$ we have $u^\epsilon(x) \leq u(x)$ and hence the equality.
	
	\noindent\underline{{(iv)}} \ If $u$ is bounded with $|u| \leq M$ then for each $y$ such that $|y-x|^2 > 4\epsilon M =\vcentcolon \delta^2$ one has
	\[
	u(y) - \frac{1}{2\epsilon} |y-x|^2 - u(x) < M - \frac{1}{2\epsilon} 4 \epsilon M + M = 0;
	\] 
	that is, 
	$$
	\sup_{y \in X \cap B_\delta(x)\compl} \Bigr( u(y) - \frac{1}{2\epsilon} |y-x|^2 \Bigr) \leq u(x) \leq u^\epsilon(x).
	$$ 
	Hence, another expression for the sup-convolution is given by  
	$$
	u^\epsilon(x) = \sup_{y \in X \cap \bar B_\delta(x)} \left( u(y) - \frac{1}{2\epsilon} |y-x|^2 \right),$$
	which is \eqref{supcball} after the change of variables $x-y= z$, provided that $B_\delta(x) \ssubset X$ (which is true if $x \in X_\delta$). \syid{subsetcpt@\detokenize{$X \ssubset Y$}!\detokenize{$\barr X$ is compact in $Y$}}
\end{proof}

\begin{remark}
	Property {(iv)} is crucial for the semiconvex approximation technique described in the Part~\ref{sas} on nonlinear potential theories determined by a subequation $\call F$. In particular, property {(iv)} ensures that the sup-convolution of an \Fd subharmonic function remains \Fd subharmonic on a suitably shrunk domain (cf.~\Cref{prop:approx}).
\end{remark}

A finer result on the correspondence between upper contact points and associated upper contact jets of a function $u$ with those of its  sup-convolutions $u^{\epsilon}$ is the following, which Crandall--Ishii--Lions~\cite[Lemma~A.5]{user} call \emph{magical properties} of the sup-convolution. \sid{property!magical!of the sup-convolution}

\begin{prop} \label{magprop}
	Let $u \in \USC(\R^n)$ be bounded above and, for $\epsilon >0$, let $u^\epsilon$ be its $\epsilon$-sup-convolution. If $(p,A)$ is an upper contact jet for $u^\epsilon$ at $x$, then $(p,A)$ is also an upper contact jet for $u$ at $x+ \epsilon p$. 
	
	Furthermore, if $(0,A)$ is the limit as $j \to +\infty$ of a sequence of upper contact jets $(p_j, A_j)$ for $u^\epsilon$ at points $x_j \to 0$, then the corresponding sequence of upper contact points for $u$ satisfies $x_j + \epsilon p_j \to 0$ and $u$ is continuous at $x = 0$ along this sequence contact points.
\end{prop}

\begin{proof}
	Let $\xi\in \R^n$ such that
	\begin{equation} \label{eq0:magprop}
	u(\xi) - \frac1{2\epsilon} |\xi-x|^2 = u^\epsilon(x);
	\end{equation}
	note that such a point $\xi$ exists by part (b) of \Cref{baspropsc}. For every $z\in \R^n$ and every $y \in B_{\rho}(x)$, with $\rho > 0$ sufficiently small, we have
	\begin{equation} \label{eq1:magprop}
	u(z) - \tfrac1{2\epsilon}|y-z|^2 \leq  u^{\epsilon}(y) \leq u^{\epsilon}(x) + \pair{p}{y-x}+ Q_A(y-x),
	\end{equation}
	where the first inequality follows from the definition of $u^\epsilon$ and the second inequality holds since $(p,A)$ is an upper contact jet for $u^\epsilon$ at $x$. Now using 
	the choice of $\xi$ in \eqref{eq0:magprop} we can rewrite \eqref{eq1:magprop} as
	\begin{equation} \label{eq2:magprop}
	u(z) - \tfrac1{2\epsilon}|y-z|^2 \leq  u(\xi)- \tfrac1{2\epsilon}|\xi-x|^2 + \pair{p}{y-x}+ Q_A(y-x),
	\end{equation}
	which holds for every $z \in \R^n$ and every  $y \in B_{\rho}(x)$. Then, by choosing $z=y-x+\xi$ in \eqref{eq2:magprop} and using the fact that $y \in B_\rho(x)$ is arbitrary, one has
	\[
	u(z) \leq u(\xi) + \pair{p}{z-\xi} + Q_A(z-\xi), \quad \forall \, z \in B_\rho(\xi);
	\]
	that is, $(p,A)$ is an upper contact jet for $u$ at $\xi$. 
	
	For the first part of the proposition, it remains only to show $\xi = x + \epsilon p$. One makes the choices 
	\[
	z = \xi \quad \text{and} \quad y=x + \frac{t}{\epsilon}(x +\epsilon p -\xi )
	\]
	with
	\[
	|t| < \frac{\rho}{\epsilon |x  + \epsilon p - \xi|} \quad \text{if $\xi \neq x+\epsilon p$}, \qquad \text{$t$ arbitrary if $\xi = x+\epsilon p$},
	\]
	so that $y \in B_{\rho}(x)$ and one can again apply ~\eqref{eq2:magprop} (multiplied by $\epsilon^2$) to find that as $t \to 0$,
	\[
	- \pair{x-\xi}{t(x + \epsilon p -\xi)} + o(t) \leq \epsilon \pair{p}{t(x+\epsilon p-\xi)} + o(t);
	\]
	that is,
	\begin{equation} \label{abs:magprop}
	t|x + \epsilon p - \xi |^2 + o(t) \geq 0.
	\end{equation}
	Since along a sequence $0 > t_n \to 0^-$ the left-hand side of the above inequality~\eqref{abs:magprop} is eventually negative when $\xi \neq x+\epsilon p$, the only possibility for that to hold is that
	\begin{equation} \label{xi=:magprop}
	\xi = x + \epsilon p.
	\end{equation}
	
	It is now immediate to see that if $(p_j, A_j) \to (0, A)$ are upper contact jets for $u^\epsilon$ at $x_j \to 0$, then they are upper contact jets for $u$ as well, at $x_j + \epsilon p_j \to 0$. Finally, note that plugging \eqref{xi=:magprop} into \eqref{eq0:magprop}, one has $u^\epsilon(x) = u(x+\epsilon p) - \frac\epsilon2|p|^2$. Since  $u \leq u^\epsilon$, $u^\epsilon$ is continuous, and $u$ is upper semicontinuous, it follows that
	\[
	\begin{split}
	u(0) \leq u^\epsilon(0) &= \lim_{j\to\infty} \!\big( u^\epsilon(x_j) + \tfrac\epsilon2|p_j|^2 \big) \\&=  \liminf_{j\to\infty} u(x_j + \epsilon p_j) \leq \limsup_{j\to\infty} u(x_j + \epsilon p_j) \leq u(0);
	\end{split}
	\]
	this implies that $u(x_j + \epsilon p_j) \to u(0)$, which completes the proof.
\end{proof}

\section{The Crandall--Ishii--Lions Theorem on Sums} \label{ss:tosu}

The properties of the semiconvex approximation via sup-convolution contained in \Cref{baspropsc} will be used extensively in Part~\ref{sas}, when treating subharmonics in general potential theories. Moreover, as is well known, the semiconvex approximation of semicontinuous functions is also a key ingredient in the theory of viscosity solutions for nonlinear PDEs. In particular, together with Jensen's \Cref{u:jen}, one is able to prove the important Theorem on Sums (see the Appendix of~\cite{user}). As an application of what has been developed thus far, we give a short proof of that celebrated theorem. From our point of view, the idea is to use \Cref{magprop} to extend the semiconvex version of the Theorem on Sums (as discussed in \Cref{tosqc2}) to the conventional upper semicontinuous version.

We begin by recalling the statement of the general version of the conventional result.

\begin{thm}[On Sums; {\cite[Theorem 3.2]{user}}] \label{tosu} \sid{Theorem!on Sums!original version (Crandall--Ishii--Lions)}
	Let $\scr O_i$ be a locally compact subset of $\R^{N_i}$ for $i = 1, \dots, k$,
	\[
	\scr O \defeq \scr O_1 \times \cdots \times \scr O_k \subset \R^N, \quad N \defeq N_1 + \cdots + N_k,
	\]
	$u_i \in \USC(\scr O_i)$, and $\phi$ be continuously differentiable in a neighborhood of $\scr O$. Set
	\[
	w(x) \defeq u_1(x_1) + \cdots + u_k(x_k) \quad \text{for}\ x = (x_1,\dots,x_k) \in \scr O,
	\]
	and suppose that $\hat x = (\hat x_1, \dots, \hat x_k) \in \scr O$ is a local maximum of $w - \phi$ relative to $\scr O$. Then for each $\epsilon > 0$ there exists $X_i \in \call S(N_i)$ such that
	\[
	(D_i \phi(\hat x), X_i) \in \bar J^{2,+}_{\scr O_i} u_i(\hat x_i) \quad \text{for}\ i = 1, \dots, k,
	\]
	and the block diagonal matrix $X$ with entries $X_i$ satisfies \syid{normM@$\norm{A}$!the operator norm of the matrix $A$}
	\begin{equation} \label{matineq_tos}
	-\bigg(\frac1\epsilon + \norm{D^2\phi(\hat x)}\bigg) I \leq X \leq D^2\phi(\hat x) + \epsilon D^2\phi(\hat x)^2.
	\end{equation}
\end{thm}

\begin{remark}
	The ``closure'' $\bar J^{2,+}_{\scr O_i} u_i(\hat x_i)$ appearing above is, with the terminology we have used so far, the set of jets $(p,A)$ such that there exists a sequence $(y_n, u(y_n), p_n, A_n) \to (\hat x_i, u_i(\hat x_i), p, A)$ where $(p_n, A_n)$ are global upper contact jets for $u_i$ at $y_n$ on $\scr O_i$.
\end{remark}

\begin{remark}
	As noted in \cite{user}, some straightforward reductions can be made without loss of generality:
	\begin{itemize}
		\item due to the local nature of upper contact jets, the local compactness of $\scr O_i$, and the possibility to trivially extend upper semicontinuous functions outside compact sets (cf.~Remarks~\ref{rmk:codom}(a) and~\ref{rem:usc_ext_X_compact}), one may suppose that $\scr O_i = \R^{N_i}$;
		\item up to translations, one may suppose that $\hat x = 0$, $u_i(0) = 0$, and $\phi(0) = 0$;
		\item by replacing $\phi$ and $u_i$ with $\phi - \pair{D\phi(0)}{\cdot\,}_{\R^N}$ and $u_i - \pair{D_i\phi(0)}{\cdot\,}_{\R^{N_i}}$, respectively, one may suppose that $D\phi(0) = 0$.
	\end{itemize}
	With these reductions, $w(x) \leq Q_{D^2\phi(0)}(x) + o(|x|^2)$ for $x \to 0$, hence by redefining $u_i \defeq -\infty$ outside a small ball about $0 \in \R^{N_i}$ one has that $w - Q_{D^2\phi(0)+\eta I}$, for any $\eta>0$ small, has a global maximum at $0 \in \R^N$; therefore, one last reduction is possible:
	\begin{itemize}
		\item one may suppose that $\phi$ is quadratic.
	\end{itemize}
	Indeed, if the inequality \eqref{matineq_tos} holds with $D^2\phi(0) + \eta I$ in place of $D^2\phi(0)$, it suffices to let $\eta \dto 0$ in order to have the desired conclusion.
\end{remark}

Finally, in addition to the reductions listed above, for the proof (of the theorem with these reductions) it suffices to consider a pair of functions; that is, $k = 2$, as noted in \cite{user} when they declare that Theorem~\ref{tosu} is {\em ``formulated in a useful but distracting generality''}. Taking these considerations together, one arrives at the following statement, which we also rephrased in the formalism of upper contact jets.

\begin{thm}[On Sums, reduced] \label{tos} \sid{Theorem!on Sums!reduced version}
	Let $u,v \in \USC(\R^n)$ be bounded above and with $u(0) = 0 = v(0)$. Define 
	\[
	w(x,y) \defeq u(x) + v(y)
	\]
	on $\R^n \times \R^n$ and suppose that $(0,A)$ is an upper contact jet for $w$ at $0 \in \R^{2n}$. Then $0 \in \R^n$ is a contact point for both $u$ and $v$, and for every $\epsilon >0$ there exist $A_1, A_2\in \call S(n)$ such that $(0,A_1), (0,A_2)$ are limits of upper contact jets for $u, v$, respectively, at points converging to $0$,
	with
	\begin{equation} \label{tosineq}
	-\Big(\,\frac1\epsilon + \Vert A\Vert \Big) \,I \leq \left(\! \begin{array}{cc}
	A_1 &  0 \\
	0  & A_2 \\
	\end{array}\! \right)
	\leq A + \epsilon A^2.
	\end{equation}
	Also, $u$ and $v$ are continuous along the respective sequences of contact points.
\end{thm}

\begin{proof}
	Fix $\epsilon >0$. Easy calculations exploiting the Cauchy--Schwarz inequality yield 
	\[\begin{split}
	0 &\leq \Big(\sqrt{\epsilon} |A \xi| - \frac1{\sqrt\epsilon}|z-\xi|\Big)^2 \\
	&\leq \epsilon |A \xi|^2 - 2\pair{A\xi}{z-\xi} + \frac1\epsilon|z-\xi|^2 \\
	&= \pair{A\xi}{\xi} + \epsilon \pair{A\xi}{A\xi} + \frac1\epsilon|z-\xi|^2 + \pair{A(z-\xi)}{z-\xi} - \pair{Az}{z} \\
	&\leq \pair{(A+\epsilon A^2)\xi}{\xi} + \Big(\frac1\epsilon + \norm{A}\Big) |z-\xi|^2 - \pair{Az}{z};
	\end{split}\]
	that is,
	\[
	\pair{Az}{z} \leq \pair{(A+\epsilon A^2)}{\xi} + \lambda|z-\xi|^2 \qquad \forall\, z,\xi \in \R^{2n}, \ \ \text{with}\ \lambda \defeq \frac1\epsilon + \Vert A\Vert.
	\]
	Therefore, we have
	\[
	w(z) - \frac\lambda2|z-\xi|^2 \leq Q_{A+\epsilon A^2}(\xi).
	\]
	Now, by taking the supremum over $z\in \R^{2n}$ and by writing $\xi=(x,y)$, we have that $(0, A+\epsilon A^2)$ is an upper contact jet at $0$ for the sup-convolution
	\[
	w^{1/\lambda}(x,y) =  u^{1/\lambda}(x) +  v^{1/\lambda}(y).
	\]
	Indeed, $u^{1/\lambda}(0) \geq u(0) = 0$ by definition, and the same holds for $v$, while $w^{1/\lambda}(0) \leq 0$, and thus $u^{1/\lambda}(0) = 0 = v^{1/\lambda}(0)$. Now, by \Cref{tosqc2} we know that there exist jets $(0, A_i)$ which are the limits of upper contact jets for the $\lambda$-semiconvex functions $u^{1/\lambda}, v^{1/\lambda}$ at some points (of second-order differentiability) converging to $0$ and that the $A_i$'s satisfy~(\ref{tosineq}). The conclusion now follows from \Cref{magprop}.
\end{proof}

\begin{remark}
	In view of \Cref{tosqc2} and the proof of \Cref{magprop}, we can also say that one may choose the converging sequences of upper contact jets in \Cref{tos} to be $(Du^{1/\lambda}(x_j), D^2u^{1/\lambda}(x_j))$ and $(Dv^{1/\lambda}(y_j), D^2v^{1/\lambda}(y_j))$, for some sequences $\{x_j\}_{j \in \N}, \{y_j\}_{j \in \N}$ of points at which the sup-convolutions $u^{1/\lambda}, v^{1/\lambda}$ are twice differentiable.
\end{remark}

\part{General potential--theoretic analysis}\label{sas}

This second part consists of four chapters and is dedicated to general nonlinear second order potential theories and the fundamental role that semiconvex functions play in such theories. The presentation will involve four stages which include a derivation of the minimal conditions that a subequation constraint set should satisfy in order to have a robust potential theory of subharmonic functions. The fundamental roles of duality and monotonicity will be discussed, along with the main tools in the viscosity theory of subharmonics and the essential role that semiconvex functions play in that theory. The validity of the comparison principle will be examined in two distinct regimes of sufficient monotonicity and minimal monotonicity. The final brief chapter shows that the notion of semiconvexity is preserved by smooth changes of coordinates, which is the key point in order to transport semiconvex analysis to manifolds. It is included here because we make use of potential theoretic arguments.

\vspace{2ex}

\chapter{General potential theories} \label{chap:conset}

In this chapter we will discuss the main objects in general potential theories; that is, subequation constraint sets and their subharmonics. Both classical and viscosity formulations will be discussed as well as the relations between the two notions. A few representative potential theories will be presented; namely classical subharmonics, convex and semiconvex functions.

\section{Subequation constraint sets and their subharmonics}

We start with a discussion of the axioms which should be placed on a {\em subequation} constraint set $\call F$ in the space of  $2$-jets in order to have a well-defined and coherent potential theory of \Fd subharmonic functions. We will see that a monotonicity property with respect to the $\call S(n)$ variable is needed. In order to have a useful notion of duality, a certain topological stability axiom is useful, and leads to a natural notion of \Fd superharmonics by duality. Finally, in order to prevent certain obvious counterexamples to the validity of the comparison principle, another monotonicity property with respect to the $\R$ variable is needed. We recall that the comparison principle is an essential tool for treating the natural Dirichlet problem, since it ensures uniqueness of solutions and plays a key role in existence by Perron's method.

Potential theories associated to subequations are interesting from two different points of view. First, there are many interesting geometric contexts with naturally associated potential theories as reviewed in \cite[Section~1.1]{hpsurvey}. Second, subequation constraint sets and their associated potential theory can arise naturally in the study of a given (weakly elliptic) PDE, which we recall now for the benefit of those familiar with the standard viscosity theory of PDEs but not with the corresponding general potential theoretic context.

Suppose we are given an operator $F \in C(X \times \R \times \R^n \times \call S(n))$\sid{operator!general differential} on some open set $X\subset \R^n$ which determines a PDE
\begin{equation} \label{Fxududdu=0}
F(x, u(x), Du(x), D^2u(x)) = 0, \quad x \in X.
\end{equation} 
A classical \emph{subsolution}\sid{subsolution|seeonly{solution, sub-}}\sid{solution!sub-!classical} of (\ref{Fxududdu=0}) is a function $u\in C^2(X)$ which satisfies
\begin{equation} \label{Fxududdu>0}
F(x, u(x), Du(x), D^2u(x)) \geq 0 \quad \forall x \in X.
\end{equation}
The inequality \eqref{Fxududdu>0} identifies a closed subset $\call F \subset \call J^2(X) = X \times \R \times \R^n \times \call S(n)$,\footnote{Note that in fact this Cartesian product represents for us the bundle $\call J^2(X)$ of the $2$-jets of functions on $X$; indeed, for an $m$-dimensional Riemannian manifold $M$ one has the following isomorphism for the bundle of the $2$-jets of functions on $M$:
	\[
	\call J^2(M) \simeq M\times \R \underset{M}{\oplus} T^*M \underset{M}{\oplus} \call S(T^*M),
	\]
	where $\call S(T^*M)$ denotes the space of all symmetric $(0,2)$-tensor fields on $M$.}\syid{J2X@\detokenize{$\mathcal{J}^2(X)$}!\detokenize{the bundle of $2$-jets over $X$; if $X \subset \R^n$, $\call J^2(X) = X \times \R \times \R^n \times \call S(n)$}}
\begin{equation}\label{compatibility}
\call F \defeq \{(x,p,r,A) \in \call J^2(X) : F(x, r, p, A) \geq 0 \},
\end{equation}
which is a natural constraint set associated to classical subharmonics where one says that $u \in C^2(X)$ is \emph{\Fd subharmonic on $X$} if 
\begin{equation*}
J^2_x u \defeq (u(x), Du(x), D^2u(x))\in \call F_x \qquad \forall x \in X,
\end{equation*}
where
\[
\call F_x \defeq \left\{(r,p,A)\in \R \times \R^n \times \call S(n):\ (x,r,p,A) \in \call F \right\}
\] 
is the fiber of $\call F$ over $x$.\syid{Fx@\detokenize{$\call F_x$}!the fiber of $\call F$ over $x$}\sid{fiber (of a subequation)} One extends the notions of subsolutions and subharmonics to upper semicontinuous functions by using a pointwise viscosity formulation in terms of $2$-jets of spaces of upper contact test functions; that is, spaces of upper contact jets. Under minimal conditions on $\call F$, there are many equivalent formulations as will be discussed below in Remark \ref{rem:equiv_FSH}. 

\begin{definition}\label{defn:FSH} Given a function $u\in\USC(X)$:
	\begin{itemize}
		\item[(a)] a function $\varphi$ which is $C^2$ near $x \in X$ is said to be a {\em $C^2$ upper test function for $u$ at $x$} if\sid{upper test function!$C^2$}
		\begin{equation}\label{TF}
		u-\varphi \leq \ 0  \ \ \text{near}\ x \ \ \ \text{and} \ \ \  
		u-\varphi = 0  \ \ \text{at} \ x;  
		\end{equation}	
		\item[(b)] the function $u$ is said to be {\em \Fd subharmonic at $x$}\sid{subharmonic|seeonly{harmonic, sub-}}\sid{harmonic (\emph{or} \Fd harmonic)!sub- (\emph{or} \Fd sub-)} if \begin{equation}\label{FSubx}
		\mbox{$J^2_{x} \varphi \in \call F_x$ for all $C^2$ upper test functions $\varphi$ for $u$ at $x$};
		\end{equation} 
		\item[(c)] the function $u$ is said to be {\em \Fd subharmonic on $X$} if $u$ is \Fd subharmonic at each $x \in X$. 
	\end{itemize} 
	
	The set of all \Fd subharmonic functions on $X$ will be denoted by $\call F(X)$.\syid{FofX@$\call F(X)$!the set of all \Fd subharmonic functions on $X$}
\end{definition}

The corresponding (and standard) notion of $u\in\USC(X)$ being a {\em viscosity subsolution of the PDE \eqref{Fxududdu=0} in $x \in X$} is that\sid{solution!sub-!viscosity}
\begin{equation}\label{Vsubx}
F(x, J^2_{x} \varphi) \geq 0 \quad \text{for all upper $C^2$ test functions $\varphi$ for $u$ at $x$}.
\end{equation}
The analogous notion of $u\in\LSC(X)$ being a {\em viscosity supersolution of the PDE \eqref{Fxududdu=0} in $x \in X$} is that\sid{solution!viscosity super-}\sid{supersolution|seeonly{solution, super-}}
\begin{equation}\label{Vsuperx}
F(x, J^2_{x} \varphi) \leq 0 \quad \text{for all lower $C^2$ test functions $\varphi$ for $u$ at $x$},
\end{equation}
where one merely inverts the inequalities in \eqref{TF} and \eqref{Vsubx}. 

In the case that a constraint $\call F \subset \call J^2(X)$ is determined by a differential operator $F \in C(\call J^2(X))$ by \eqref{compatibility}, the definition of supersolution \eqref{Vsuperx} suggests defining $u \in \LSC(X)$ to be {\em \Fd superharmonic in $x \in X$} if\syid{int@$\intr E$!the interior of the set $E$}\sid{superharmonic|seeonly{harmonic, super-}}\sid{harmonic (\emph{or} \Fd harmonic)!super- (\emph{or} \Fd super-)}
\begin{equation}\label{FSuperx}
\mbox{$J^2_{x} \varphi \not\in \intr \call F_x$ for all $C^2$ lower test functions $\varphi$ for $u$ at $x$.}
\end{equation}
Finally, a function $u \in C(X)$ is said to be {\em \Fd harmonic in $x \in X$}\sid{harmonic (\emph{or} \Fd harmonic)} if both \eqref{FSubx} and \eqref{FSuperx} hold. In the next chapter, we will use a natural notion of {\em duality} to reduce superharmonics to subharmonics of a dual constraint set. 

\begin{remark}[Weak ellipticity]\label{rem:weakly_elliptic} 
	In order to be meaningful, the viscosity formulation of \Fd subharmonic functions of Definition~\ref{defn:FSH} requires making a monotonicity assumption on the constraint set $\call F$ with respect to the $\call S(n)$ variable. More precisely one needs to assume that the constraint set $\call F$ fiberwise satisfies the {\em positivity} property:\sid{property!positivity (P)}\sid{ellipticity!weak}
	\begin{equation}\label{property_P}
	(r,p,A) \in \call F_x \implies (r,p,A+P) \in \call F_x \quad \forall P \geq 0 \ \text{in} \ \call S(n), \quad \forall x \in X.
	\end{equation}
	Indeed, if \eqref{property_P} fails then there exist $x_0 \in X$ and $P_0 > 0$ such that
	\begin{equation}\label{P1}
	\mbox{$(r,p,A) \in \call F_{x_0}$ \quad but \quad $(r,p,A + P_0) \not\in \call F_{x_0}$.}
	\end{equation}
	Let $u \in \call F(X)$ with $u(x_0) \neq -\infty$ and suppose that $\varphi$ is an upper test function for $u$ at $x_0$. Since $u$ is \Fd subharmonic in $x_0$ we have $(r,p,A)\defeq  J^2_{x_0} \varphi \in \call F_{x_0}$ and the function 
	\[
	\psi\defeq  \varphi + \frac{1}{2} \langle P_0(x - x_0), x - x_0 \rangle
	\]
	is also an upper test function for $u$ at $x_0$, but, by \eqref{P1}, it satisfies
	\[
	J^2_{x_0} \psi = J^2_{x_0} \varphi + P_0 \not\in \call F_{x_0},
	\]
	which contradicts $u$ being \Fd subharmonic in $x_0$. 
	
	When $\call F$ is {\em compatibile} with an operator $F \in C(\call J^2(X))$ in the sense that\sid{compatibility}
	\begin{equation*}
	\call F = \{(x,r,p,A) \in \call J^2(X): \ F(x,r,p,A) \geq 0\},
	\end{equation*}
	positivity of $\call F$ is implied by $F$ being {\em degenerate ellipitc}, i.e.\ \sid{ellipticity!degenerate|seeonly{weak}}
	\begin{equation*}
	F(x,r,p,A) \leq F(x,r,p,A+P) \quad \forall P \geq 0,
	\end{equation*}
	while positivity of $\call F$ implies that $F$ satisfies
	\[
	F(x,r,p,A)\geq 0 \quad \implies \quad
	F(x,r,p,A+P)\geq 0 \quad \forall P\geq 0.
	\] 
	This monotonicity should be viewed as the weakest possible notion of ellipticity suitable for a meaningful viscosity theory in both contexts.
\end{remark}

A widely used consequence of the positivity condition \eqref{property_P} should be noted. 

\begin{remark}[Positivity and coherence]\label{rem:coherence}
	Of fundamental importance in many constructions is to have the following \emph{coherence property}\sid{property!coherence} for \Fd subharmonics in the sense of Definition \ref{defn:FSH} with $\call F \neq \emptyset$: if $u$ is $C^2$ near $x_0 \in X$ then
	\begin{equation}\label{coherence}
	\mbox{$u$ is \Fd subharmonic at $x_0 \ \Leftrightarrow \ J^2_{x_0}u \in \call F_{x_0}$.}
	\end{equation}
	The implication $(\Rightarrow)$ is obvious since $u$ is its own upper test function at $x_0$. The reverse implication $(\Leftarrow)$ requires the positivity condition \eqref{property_P}. Indeed, suppose that $J^2_{x_0}u \in \call F_{x_0}$ and that $\varphi$ is $C^2$ upper test function for $u$ at $x_0$. Since $u - \varphi$ has a local maximum value of $0$ in $x_0$, by elementary calculus, one has
	\[
	\varphi(x_0) = u(x_0), \quad D\varphi(x_0) = Du(x_0), \quad \text{and} \quad D^2\varphi(x_0) \geq Du(x_0),
	\]
	hence, denoting by $P\defeq  D^2\varphi(x_0) - Du(x_0) \geq 0$ in $\call S(n)$,
	\[
	J^2_{x_0} \varphi = J^2_{x_0} u + P \quad \text{with} \ P \geq 0,
	\]
	which yields $J^2_{x_0} \varphi \in \call F_{x_0}$ if and only if the positivity property \eqref{property_P} holds.
\end{remark}
The same equivalence holds if $u$ is merely twice differentiable in $x_0$ by using the {\em little-o} formulation of upper test jets and second order Taylor expansions, provided that the fibers $\call F_x$ are closed. See Proposition \ref{elemprop}(iii).

\begin{remark}[Equivalent notions of \Fd subharmonicity]\label{rem:equiv_FSH} In the notation of Remark~\ref{rem:UCJ}, Definition~\ref{defn:FSH} says that
	\begin{equation}\label{equiv1}
	\mbox{$u$ is \Fd subharmonic in $x \ \Leftrightarrow \ (u(x), p, A) \in \call F_x \quad  \forall \, (p,A) \in J_{C^2}(x,u)$}.
	\end{equation}
	As noted in Remark~\ref{rem:UCJ}, we have the chain of inclusions \eqref{inclusionsformulations}
	and \cite[Lemma C.1]{chlp} shows that all of those spaces have the same closure in $\call J^2$ for each $x \in X$. Hence, if the fiber $\call F_x$ is closed in $\call J^2$,
	then in \eqref{equiv1} one can replace the space of upper contact jets $J_{C^2}(x,u)$ with any of the others in the chain \eqref{inclusionsformulations}. Moreover, if $\call F$ is closed in $\call J^2(X)$,
	then one can also replace $J_{C^2}(x,u)$ with a limiting space $\overline{J}_{C^2}(x,u)$ consisting of those $(p,A) \in \R^n \times \call S(n)$ for which there are sequences $\{x_k\}_{k \in \N} \subset X$ and $\{(p_k, A_k)\}_{k \in \N} \subset J_{C^2}(x,u)$ for which $(x_k, u(x_k),p_k,A_k) \to (x, u(x), p, A)$ as $k \to \infty$,
	and similarly for the other spaces of upper test jets. 
\end{remark}

The considerations above lead us to a first class of constraint sets $\call F$ which have a meaningful and flexible notion of upper semicontinuous \Fd subharmonic functions.

\begin{definition}\label{defn_prim_subeq}
	A subset $\call F \subset \call J^2(X)$ with $X$ open is called a \emph{primitive subequation (constraint set)}\sid{subequation (constraint set)!primitive} if\sid{property!closedness (C)}
	\begin{equation} \tag{C} \label{(C)}
	\mbox{$\call F$ is closed in $\call J^2(X)$,}
	\end{equation}
	and $\call F$ satisfies fiberwise the \emph{positivity} condition\sid{property!positivity (P)}
	\begin{equation} \tag{P} \label{(P)}
	(r,p,A) \in \call F_x \implies (r,p,A+P) \in \call F_x \quad \forall P> 0, \quad \forall x \in X.
	\end{equation}
\end{definition}

As noted in Remarks~\ref{rem:weakly_elliptic} and \ref{rem:coherence}, the positivity condition~\eqref{(P)} is essential for a meaningful notion of upper semicontinuous \Fd subharmonics and in order to have $C^2$-coherence, while the closedness condition~\eqref{(C)} ensures the equivalent formulations of Remark \ref{rem:equiv_FSH}. Condition~\eqref{(C)} also ensures that the elementary properties of $\call F(X)$ involving limits and upper envelopes hold (see \cite[Theorem~2.6]{hldir}; the same properties are also listed, for nonconstant coefficient subequations, in \Cref{elemprop}).

The two minimal conditions \eqref{(C)} and \eqref{(P)}, however, are not sufficient for all of our purposes. First, as already mentioned, in order to treat the natural Dirichlet problem for \Fd harmonic functions, we will need to prove {\em comparison}; that is, to show that
\begin{equation}\label{intro_CP}
u\leq v\ \text{on $\de X$} \implies u\leq v\ \text{on $X$}
\end{equation}
whenever $u$ and $v$ are \Fd subharmonic and \Fd superharmonic, respectively (or a sub/super-solution pair for a degenerate elliptic PDE). In order to avoid obvious counterexamples to comparison \eqref{intro_CP}, one needs to require that $\call F$ has another fiberwise monotonicity property with respect to $\R$ variable: that is, that $\call F$ satisfies the \emph{negativity} property\sid{property!negativity (N)}
\begin{equation} \tag{N} \label{(N)}
(r,p,A) \in \call F_x \implies (r+s, p, A) \in \call F_x \quad \forall s < 0, \quad \forall x \in X.
\end{equation}
The simplest counterexample is the following.
\begin{example}
	Consider, for $n=1$, the ordinary differential equation $u''+u = 0$ on $\R$, that is $F(u, u'') = 0$, with $F(r,A) = A+ r$. It is clear that the associated $\call F$ does not satisfy \eqref{(N)} and comparison does not hold on arbitrary bounded domains, but only on intervals of length less than $\pi$.
\end{example}

Property~\eqref{(N)} will have additional implications and utility, including the fact that it guarantees a \emph{sliding property} for subharmonics: if $u$ is a subharmonic, then so is $u-m$ for each constant $m \geq 0$. Indeed, if $(p,A)$ is an upper contact jet for $u$ at $x$, so it will be for $u-m$. This sliding property will play a key role in some upcoming constructions and is included, along with many elementary properties mentioned above, in \Cref{elemprop}. Hence, in general, one also wants property~\eqref{(N)} for a constraint set $\call F$. However, there are important exceptions to this, including when treating {\em generalized principal eigenvalues} and associated eigendirections where typically one has the opposite monotonicity in the $r$ variable (see \cite{BP21}, for example).

\begin{remark}[Proper ellipticity]
\sid{ellipticity!proper}
The combined monotonicity properties of positivity \eqref{(P)} and negativity \eqref{(N)} for a constraint $\call F$ determined by an operator $F$ with fibers
	\[
	\call F_x \defeq \left\{(r,p,A) \in \call J^2:\ F(x,r,p,A) \geq 0\right\}
	\]
	is implied by the operator $F$ being \emph{proper elliptic}\footnote{In the standard viscosity theory, $F$ is often said to be merely {\em proper}. We prefer the term proper elliptic to recall the crucial importance of weak (degenerate) ellipticity required for the viscosity theory.}; that is, for each $(x,r,p,A) \in \call J^2(X)$,
	\begin{equation*} 
	F(x,r+s,p,A+P) \geq F(x,r,p,A) \quad  \forall \,  s \leq 0 \ \text{in} \ \R,\ \ P \geq 0 \ \text{in} \ \call S(n).
	\end{equation*}
	
\end{remark}

Two final considerations are needed before arriving at the desired class of constraint sets.

\begin{remark}[Proper fibers and topological stability]
	Note that if some fiber $\call F_x$ of a constraint set were to be empty, then $X$ might not admit an \Fd subharmonic. To see this, the same argument we used at the beginning of the proof of \Cref{p:uvm} shows that the space of upper quadratic test jets $J^{2,+}_xu$ for $u$ at $x$ is non-empty for some $x \in X$ and hence for each $(p,A) \in J^{2,+}_xu$ the condition $(u(x), p, A) \in \call F_x$ fails if $\call F_x$ is empty.
	
	Moreover, if some fiber $\call F_x$ were to be equal to the entire space of $2$-jets $\call J^2$, the property of being \Fd subharmonic in $x$ is satisfied by all functions which are upper semicontinuous in some neighborhood of $x$. Hence one should consider only constraint sets $\call F$ such that, for all $x \in X$, the fiber $\call F_x$ is a closed, non-empty and proper subset of $\call J^2$.
	
	Finally, one obtains a well behaved and useful theory of constraint sets $\widetilde{\call F}$ which are {\em dual} to a given constraint set $\call F$ {provided} that one also assumes the following three conditions of {\em topological stability}:\sid{property!topological stability (T)}
	\begin{equation} \tag{T} \label{(T)}
	\call F = \barr{\intr{\call F}}, \qquad (\intr{\call F})_x = \intr(\call F_x) \quad \forall  x \in X, \qquad \call F_x = \barr{\intr{\call F}_x}\quad \forall  x \in X .
	\end{equation} 
	Notice that the first condition above implies the closedness condition \eqref{(C)} and the third condition above implies that all fibers are closed.
\end{remark}

The desired class of constraint sets $\call F$ with a robust associated potential theory is now at hand.

\begin{definition} \label{defnsubeq} \sid{subequation (constraint set)}
	A {\em subequation (constraint set)} on an open set $X \subset \R^n$ is a subset $\call F \subset \call J^2(X)$ with nontrivial fibers (that is, proper and non-empty fibers) which satisfies the {\em positivity property}
	\begin{equation} \tag{P} \label{(P)'} \sid{property!positivity (P)}
	(r,p,A) \in \call F_x \implies (r,p,A+P) \in \call F_x \quad \forall P> 0, \quad \forall x \in X,
	\end{equation}
	the {\em negativity property}
	\begin{equation} \tag{N} \label{(N)'} \sid{property!negativity (N)}
	(r,p,A) \in \call F_x \implies (r+s, p, A) \in \call F_x \quad \forall s < 0, \quad \forall x \in X,
	\end{equation}
	and the {\em topological property}\sid{property!topological stability (T)}
	\begin{equation} \tag{T} \label{(T)'}
	\call F = \barr{\intr{\call F}}, \qquad (\intr{\call F})_x = \intr(\call F_x) \quad \forall  x \in X \qquad \call F_x = \barr{\intr{\call F}_x} \quad \forall x \in X .
	\end{equation}
\end{definition}

In addition to ensuring a well-behaved duality including the {\em reflexivity} property (see \Cref{sec:duality1}) property \eqref{(T)'} also ensures the \emph{local existence of strict classical subharmonics}, when we adopt the following notion of strictness.

\begin{definition} \label{def:strict}
	A function $u\in C^2(X)$ is said to be \emph{strictly} \Fd subharmonic\sid{harmonic (\emph{or} \Fd harmonic)!sub- (\emph{or} \Fd sub-)!strictly} at $x$ if
	\[
	J^2_x u \in \intr{\call F}_x
	\]
	and {strictly} \Fd subharmonic on $X$ if this inclusion holds for each $x \in X$.
\end{definition}

\begin{remark}[Local existence of strict classical subharmonics] \label{exquadsub}
	Using the topological condition \eqref{(T)'}, one can prove that for each $x \in X$  there exists a quadratic function which is strictly \Fd subharmonic on a neighborhood of $x$. Indeed, by \eqref{(T)'} we know that $\intr{\call F}_x \neq \emptyset$ for all $x \in X$; hence, let $J = (r, p, A) \in \intr{\call F}_x$ and consider the quadratic function $\phi_{J}$ associated to $J$, namely
	\[
	\phi_{J} = r + \pair{p}{\cdot -\, x} + Q_A(\,\cdot - x),
	\]
	so that $J^2_x \phi_J= J$. Since $\phi_J$ is $C^2$ and $(x, J) \in \intr{\call F}$, we have that $J^2_y \phi_J \in \intr{\call F}_y$ for all $y$ in some neighborhood $U$ of $x$, hence $\phi_J \in \call F(U)$.
\end{remark}

Two interesting remarks that one finds in \cite{hlae} are worth noting. They play an important role in the {\em Almost Everywhere Theorem} (\Cref{aet}).

\begin{remark}\label{subinclusion}
	The notion of being \Fd subharmonic can be expressed in terms of an inclusion between subequations. Indeed, each upper semicontinuous function $w$ on $X$ determines a minimal primitive subequation $\call G$ such that $w$ is $\call G$-subharmonic, namely $\call G=\barr{\call J^{2,+}w}$, where\syid{J2pxu@\detokenize{$\call J^{2,+}u$}!see~(\ref{J2+w})}
	\begin{equation}\label{J2+w}
	\call J^{2,+}w\defeq  \{ (x,w(x),p,A): \ x \in X,\ (p,A) \in J^{2,+}_xw \}
	\end{equation}
	Therefore $w$ is \Fd subharmonic if and only if $\call J^{2,+}w \subset \call F$ (if and only if $\call G \subset \call F$, since $\call F$ is closed).
\end{remark}

\begin{remark} \label{ae?}
	We already found a subset which surely contains $\call J^{2,+}w$. Indeed, if we set \syid{JwE@$\call J(w,E)$!see~(\ref{JwE})}
	\begin{equation}\label{JwE}
	\call J(w,E) \defeq \big\{ (x, w(x), Dw(x), D^2w(x) + P):\ x \in E \cap \Diff^2(w),\ P \geq 0 \big\},
	\end{equation}
	then \Cref{pusc} says that if $w$ is locally semiconvex and $E$ has full measure, then $\call J^{2,+}w \subset \barr{\call J(w,E)}$.
\end{remark}

We conclude this section with a brief discussion of an important class of subequations that has been extensively studied in the monograph \cite{chlp}. 

\begin{definition}
	A (primitive) subequation $\call F$ is said a \emph{constant coefficient}\sid{subequation (constraint set)!constant coefficient}\sid{subequation (constraint set)!translation invariant|seeonly{constant coefficient}}, or \emph{translation invariant}, (primitive) subequation if the mapping $x \mapsto \call F_x$ is constant; that is if the fibers do not depend on the point $x$.
\end{definition}

\begin{remark} \label{rmkccs}
	Since our definition of subequation requires that all the fibers are nontrivial, it is important to specify the open set $X$ on which we are considering a subequation $\call F$. Nevertheless, if $\call F$ is a constant coefficient subequation on $X$, it can be naturally extended to the whole $\R^n$ by setting, for some $x_0 \in X$ fixed, $\call F_y = \call F_{x_0}$ for all $y \in \R^n$. Hence any constant coefficient subequation can be thought of as defined on all of $\R^n$ and we can identify it with a nontrivial subset of $\call J^2 \defeq \R \times \R^n \times \call S(n)$\syid{J2@$\call J^2$!the space of $2$-jets $\R \times \R^n \times \call S(n)$} that satisfies the conditions \eqref{(P)'}, \eqref{(N)'} and \eqref{(T)'}. Notice that in this case, the property \eqref{(T)'} reduces to the single condition $\call F = \overline{\intr \call F}$ in $\call J^2$. On the other hand, obviously, when speaking about \Fd subharmonics/superharmonics/harmonics, we must specify on which open set $X$ we are working. 
\end{remark}

\section{Some simple classical examples}

Finally, we present a few representative examples of constant coefficient subequations on an arbitrary open set $X \subset \R^n$, starting from the classical case of Laplacian subharmonics.

\begin{example}[The Laplacian subequation]
\sid{subequation (constraint set)!Laplacian}
The {\em Laplacian subequation} is $\call F = X \times \call F_{\Delta}$ where
	\[
	\call	F_{\Delta}\defeq  \{(r,p,A) \in \call J^2: \ {\rm tr}(A) \geq 0\}
	\]
	and its subarmonics are characterized by 
	\begin{equation*}
	\mbox{$\displaystyle{u \in \call F_{\Delta}(X) \ \Leftrightarrow \ u(x_0) \leq \frac{1}{|B_{\rho}(x_0)|} \int_{B_{\rho}(x_0)} u(x) \, dx}$ \quad for each $B_{\rho}(x_0) \Subset X$.}
	\end{equation*}
	This subequation is {\em self-dual} ($\widetilde{\call F} = \call F$) in terms of the duality operation discussed in the following chapter.
\end{example}

\begin{example}[The convexity subequation]\label{exe:P}
\sid{subequation (constraint set)!convexity}
 The {\em convexity subequation} is $\call F =  X \times \call M(\call P)$ where
	\[
	\call M(\call P) \defeq  \R \times \R^n \times \call P\defeq  \{(r,p,A) \in \call J^2: \ A \geq 0 \} 
	\]
	and its subharmonics are characterized by 
	\begin{equation*}
	\mbox{$u \in \call F (X) \ \Leftrightarrow \ u$ is locally convex } 
	\end{equation*}
	(away from the connected components on which $u \equiv - \infty$). In particular, if $X$ is also convex and $u$ is real valued then $u \in  \call F (X)$ if and only if $u$ is convex on $X$. On the other hand, the subharmonics of the dual subequation are the so-called {\em subaffine}\sid{subaffine!function} functions; that is, they satisfy the comparison principle with respect to all affine functions.
The characterizations of these two fundamental examples of constant coefficient pure second order suabequations will be briefly discussed in \Cref{sec:cccc} (see \Cref{prop:P_Ptilde}).
	
	The convexity subequation is the most basic example of a {\em monotonicity cone subequation}\sid{subequation (constraint set)!monotonicity cone} (a constant coefficient subequation which is also a closed convex cone with vertex at the origin; see \Cref{sec:duality2}). This particular cone is used in the proof of the comparison principle for all pure second order subequations and degenerate elliptic PDEs, as will be discussed in Chapter~\ref{chap:comp}. A {\em fundamental family}\sid{subequation (constraint set)!monotonicity cone!fundamental family}\sid{fundamental family|seeonly{subequation, monotonicity cone, fundamental family}} of such cones can be found in \cite[Chapter~5]{chlp}; see also \Cref{rem:CP_cc}.
\end{example}

Related examples are fundamental to our discussion.

\begin{example}[Semiconvex functions]\label{exe:qc}\sid{subequation (constraint set)!semiconvexity|seeonly{convexity, semi-}}\sid{subequation (constraint set)!convexity!semi-}
 For $\lambda >  0$, the {\em $\lambda$-semiconvexity subequation} is  $\call F_{\lambda} =  X \times \call M(\call P_{\lambda})$ where
	\[
	\call M(\call P_{\lambda}) \defeq  \R \times \R^n \times \call P_{\lambda}\defeq  \{(r,p,A) \in \call J^2: \ A + \lambda I \geq 0 \} 
	\]
	and its subharmonics are characterized by
	\begin{equation*}
	\mbox{$u \in \call F_{\lambda} (X) \ \Leftrightarrow \ u$ is locally $\lambda$-semiconvex} 
	\end{equation*}
	(away from the connected components on which $u \equiv - \infty$). In particular, if $X$ is also convex and $u$ is real valued then $u \in  \call F_{\lambda}(X)$ if and only if $u$ is $\lambda$-semiconvex on $X$.
	
	Moreover, it is clear that the notion of being \Fd subharmonic is local since the notion of upper contact jet is local (see Proposition \ref{elemprop}(i)). Hence one can characterize the space $ \mathsf{sc}_\mathrm{loc}(X)$ of {\em locally semiconvex functions on $X$}\syid{sclocX@\detokenize{$ \mathsf{sc}_\mathrm{loc}(X)$}!the space of locally semiconvex functions on $X$} as those $u \in \USC(X)$ such that
	\begin{equation}\label{qc_loc}
	\text{for each} \ x_0 \in X \ \text{there exist} \ \rho_0 > 0 \ \text{and} \ \lambda_0 > 0 \ \text{such that} \ u \in \call F_{\lambda_0}(B_{\rho_0}(x_0)).
	\end{equation}
\end{example}

\chapter{Duality and monotonicity in general potential theories} \label{chap:duality}

This chapter is dedicated to two fundamental notions in general second order potential theories. Duality is reviewed first with its important consequence of reformulating \Fd superharmonics as subharmonics for a dual subequation $\widetilde{\call F}$. Next, the unifying notion of monotonicity is reviewed. Together, they give rise to the duality-monotonicity method for proving the validity of the comparison principle, which is a key ingredient in the treatment of existence and uniqueness of solutions to the Dirichlet problem for \Fd harmonic functions.

\section{Duality}\label{sec:duality1}

A subequation $\call F$ naturally gives rise to a dual subequation $\widetilde{\call F}$, which will play a crucial role in all that follows. Consider the following notion of duality, introduced by Harvey and Lawson~\cite{hldir09, hldir} and defined in a purely set-theoretic way.

\begin{definition} \label{def:dual}
	Given a proper subset $\call F \subset \call J^2(X)$, we define its \emph{Dirichlet dual}\sid{Dirichlet dual|seeonly{subequation, dual}}\sid{subequation (constraint set)!dual (\emph{or} Dirichlet dual)}
	\[
	\tildee{\call F} \defeq  (-\intr{\call F})\compl = - (\intr{\call F})\compl,
	\]
	where the complement is relative to $\call J^2(X)$.\syid{tildeF@\detokenize{$\tildee{\call F}$}!the Dirichlet dual of the subequation $\call F$, also called the dual subequation of $\call F$}
\end{definition}

It turns out that if $\call F$ is a subequation, then $\tildee{\call F}$ is a subequation as well; this is a consequence of the following proposition, which collects some elementary properties of the Dirichlet dual. These properties are to be found in \cite[Section 4]{hldir09} in a pure second-order context and in \cite[Section 3]{hldir} for subequations on Riemannian manifolds, as well as  in \cite[Proposition~3.2]{chlp} for constant coefficient subequations.  We essentially reproduce the proof of this last proposition in order to stress that it also works well if one has non-constant coefficients. This happens because of a consequence of the topological condition \eqref{(T)'}, as noted in \cite{hldir}: if we define the dual of the fiber $\call F_x \subset \R \times \R^n \times \call S(n)$ relative to the ambient space $\R \times \R^n \times \call S(n)$, one can compute the dual of $\call F$ fiberwise; that is,
\[
\tildee{\call F} = \bigsqcup_{x\in X}  (-\intr{\call F}_x)\compl,
\]
where the complement is relative to $\R \times \R^n \times \call S(n)$, so there is no ambiguity in the notation $\tildee{\call F}_x$.

\begin{prop} \label{propdual}
	Let $\call F$ and $\call G$ be proper subsets of $\call J^2(X)$ and let $J \in \R \times \R^n \times \call S(n)$. Then the following hold:
	\begin{enumerate}[label={\it(\arabic*)}, topsep=7pt, itemsep=4pt, partopsep=1.1pt, parsep=1.1pt]
		\item\ $\call F_x \subset \call G_x \implies \tildee{\call G}_x \subset \tildee{\call F}_x$;
		\item\ $\call F_x + J \subset \call F_x \implies \tildee{\call F}_x + J \subset \tildee{\call F}_x$;
		\item\ $\tildee{\call F_x + J} = \tildee{\call F}_x - J$;
		\item\ $\call F$ satisfies \eqref{(P)'} $\implies$ $\tildee{\call F}$ satisfies \eqref{(P)'};
		\item\ $\call F$ satisfies \eqref{(N)'} $\implies$ $\tildee{\call F}$ satisfies \eqref{(N)'};
		\item\ $\call F =\barr{\intr{\call F}}$ $\iff$ $\hspace{2.5pt}\tildee{\hspace{-2.5pt}\tildee{\call F}} = \call F$;
		\item\ $\call F$ satisfies \eqref{(T)'} $\implies$ $\tildee{\call F}$ satisfies \eqref{(T)'};
		\item\ $\call F$ is a subequation $\implies$ $\tildee{\call F}$ is a subequation.
	\end{enumerate}
\end{prop}

\begin{proof}
	Property~{(1)} follows from \Cref{def:dual}. Properties~{(4)} and~{(5)} follow from~{(2)}, and property~{(8)} follows from properties~{(4)}, {(5)} and~{(7)}. 
	
	For property~{(2)}, note that $\call F_x + J \subset \call F_x$ implies that $\intr{\call F}_x + J$ is an open subset of $\call F$, hence $\intr{\call F}_x + J \subset \intr{\call F}_x$, which yields, taking the complements, $\tildee{\call F}_x \subset \tildee{\call F}_x -J$, as desired. 
	Similarly, since $\intr{\call F}_x + J$ is the interior of $\call F_x + J$, we have $\tildee{\call F_x + J} = \tildee{\call F}_x - J$, which is property~{(3)}.
	
	Note now that by \Cref{def:dual}
	\begin{equation} \label{bidual}
	\hspace{2.5pt}\tildee{\hspace{-2.5pt}\tildee {\call F}} = \Bigl(\intr\bigl((\intr\call F)\compl\bigr)\Bigr)\compl;
	\end{equation}
	since for any subset $\call S \subset \call J^2(X)$ one has $\intr{\call S} = {\barr{\call S\compl}\,}\compl$, equality~(\ref{bidual}) yields $\hspace{2.5pt}\tildee{\hspace{-2.5pt}\tildee {\call F}} = \barr{\intr{\call F}}$, thus proving property~{(6)}.
	Similarly, if $\call F$ satisfies \eqref{(T)'}, consider the following chain of equalities:
	\[
	\barr{\intr{\tildee{\call F}}} = \barr{{\barr{\tildee{\call F}\,{\vphantom{I}}\compl}\,}\compl} = \barr{{\barr{ - \intr{\call F}}\,}\compl} = \barr{ - \call F\compl} = (-\intr{\call F})\compl = \tildee{\call F};
	\]
	this shows the first and the last of the three conditions in \eqref{(T)}, since it also holds with $\call F_x$ instead of $\call F$. Furthermore, we also read in it that
	\[
	(\intr{\tildee{\call F}})_x = -( \call F\,\compl)_x =  - (\call F_x)\compl= \intr(\tildee{\call F}_x)
	\]
	and hence the second condition is satisfied, thus proving property~{(7)}.
\end{proof}

\begin{remark} 
	Property~{(6)} is the aforementioned \emph{reflexivity} of $\call F$,\sid{property!reflexivity} which together with property~{(8)} shows that the $\sim$ operation gives a true duality.  Note also that this implies that we can reverse all the implications in \Cref{propdual}.
\end{remark}

By making use of duality, we can reformulate the notions of being \Fd subharmonic and \Fd harmonic, which we record in the following observation.

\begin{remark}[{\Fd superharmonics by duality}] \sid{harmonic (\emph{or} \Fd harmonic)!super- (\emph{or} \Fd super-)!by duality}
	For $v\in \LSC(X)$, one can define $v$ to be \Fd superharmonic on $X$ if  $-v$ is $\tildee{\call F}$-subharmonic on $X$. This is equivalent to \eqref{FSuperx}. Indeed, \eqref{FSuperx} for $v$ is  
	\[
	\mbox{$(v(x), q, B) \notin \intr{\call F}_x$ \quad for each $(q,B) \in J^{2,-}_x v$,}
	\]
	which is equivalent to 
	\[
	\mbox{$(-v(x), p, A) \notin -\intr{\call F}_x$ \quad for each $(p,A) = (-q,-B) \in J^{2,+}_x (-v)$,}
	\]
	because $- J^{2,-}_x v = J^{2,+}_x (-v)$ (see \Cref{def:ucp}). This means that $v$ is \Fd superharmonic at $x$ if and only if $-v$ is $\tildee{\call F}$-subharmonic at $x$, because 
	\[
	(-v(x), p, A) \notin -\intr{\call F}_x \iff (-v(x), p, A) \in \tildee{\call F}_x,
	\]
	by the definition on the dual of $\call F$.
	
	Finally, this yields the following equivalent definition of \Fd harmonic functions (in the viscosity sense):\sid{harmonic (\emph{or} \Fd harmonic)!by duality} for $u\in C(X)$ 
	$$
	\mbox{\emph{$u$ is \Fd harmonic on $X$} \quad if and only if \quad $u\in \call F(X)$ and $-u\in\tildee{\call F}(X)$.}
	$$
\end{remark}

\section{Monotonicity}\label{sec:duality2}

We now turn to the notion of monotonicity.\sid{monotone!operator|seeonly{operator, monotone}}\sid{monotone!subequation|seeonly{subequation, monotone}}\sid{subequation (constraint set)!monotone (\emph{or} $\call M$-monotone)} This fundamental notion appears in various guises. It is a useful and unifying concept. One says that a subequation $\call F$ is {\em $\call M$-monotone} for some subset $\call M \subset \call J^2(X)$ if 
\begin{equation}\label{monotone}
\call F_x + \call M_x \subset \call F_x \quad \forall x \in X.
\end{equation}
For simplicity, we will restrict attention to (constant coefficient) {\em monotonicity cones}; that is, monotonicity sets $\call M$ for $\call F$ which have constant fibers which are closed cones with vertex at the origin. 

First and foremost, the properties \eqref{(P)'} and \eqref{(N)'} are monotonicity properties. As noted previously, property \eqref{(P)'} for subequations $\call F$ corresponds to {\em degenerate elliptic} operators $F$ and properties \eqref{(P)'} and \eqref{(N)'} together correspond to {\em proper elliptic} operators. They can be expressed as the single monotonicity property 
\begin{equation*}
\call F_x + \call M_0 \subset \call F_x   \quad \forall x \in X
\end{equation*}
where\syid{M0@$\call M_0$!the minimal monotonicity cone in $\call J^2$, namely $\call N \times \{0\} \times \call P$}
\begin{equation}\label{MMC}
\call M_0 \defeq  \call N \times \{0\} \times \call P \subset \call J^2 = \R \times \R^n \times \call S(n)
\end{equation}
with\syid{N@$\call N$!the negativity cone in $\R$, namely $\{r \in \R:\, r \leq 0\}$}\syid{P@$\call P$!the positivity cone in $\call S(n)$, namely $\{A \in \call S(n):\, A \geq 0\}$}
\begin{equation} \label{def:np}
\call N \defeq  \{ r \in \R : \ r \leq 0 \} \quad \text{and} \quad  \call P \defeq  \{ P \in \call S(n) : \ P \geq 0 \} .
\end{equation}
Hence $\call M_0$ will be referred to as the {\em minimal monotonicity cone} in $\call J^2$.\sid{monotonicity cone!minimal} However,  it is important to remember that $\call M_0 \subset \call J^2$ is {not} a subequation since it has empty interior, so that property \eqref{(T)'} fails. 

Monotonicity is also used to formulate {\em reductions} when certain jet variables are ``silent'' in the subequation constraint $ \call F$. See \cite[Section~2.2]{hpsurvey} for details, as well as the examples in \Cref{sec:cccc}.\sid{subequation (constraint set)!reduced}

Most importantly, monotonicity, when combined with duality and {\em fiberegularity} in the variable coefficient case (the natural continuity of the fiber map $x \mapsto \call F_x$), yields a very general, flexible and elegant geometrical approach to comparison when a subequation $\call F$ admits a constant coefficient monotonicity cone subequation $\call M$. This will be described in \Cref{chap:comp}.

\chapter{Basic tools in nonlinear potential theory} \label{chap:tools}

In this chapter we present various results that are included in the ``basic tool kit of viscosity solution techniques'' in \cite{chlp}: the Bad Test Jet \Cref{l:btj}, the Definitional Comparison \Cref{defcompa}, and \Cref{elemprop} which gathers various widely used constructions of \Fd subharmarmonic functions obtained by algebraic operations, optimization procedures and limits. 

\section{The Bad Test Jet and Definitional Comparison Lemmas}

We begin with the first useful tool which allows one to check the validity of subharmonicity at a point by a contradiction argument. More precisely, if $u$ fails to be subharmonic at a given point, then one must have the existence of a \emph{bad test jet} at that point, as stated in the following lemma. This criterion is essentially the contrapositive of the definition of viscosity subsolution, when one takes \emph{strict} upper contact quadratic functions as upper test functions (see \cite[Lemmas~2.8 and C.1]{chlp}). Nevertheless, for the benefit of the reader, we provide a brief proof based on \Cref{defn:FSH}. 

\begin{lem}[Bad Test Jet Lemma] \label{l:btj} \sid{Lemma!Bad Test Jet}
	Given $u \in \USC(X)$, $x \in X$ and $\call F_x \neq \emptyset$ closed, suppose $u$ is not \Fd subharmonic at $x$. Then there exists $\epsilon > 0$ and a $2$-jet $J \notin \call F_x$ such that the quadratic function $\phi_J$ with $J^2_x\phi_J = J$ is an upper test function for $u$ at $x$ in the following $\epsilon$-strict sense:
	\begin{equation} \label{btj:i}
	u(y) - \phi_J(y) \leq -\epsilon|y-x|^2 \quad \text{$\forall y$ near $x$ (with equality at $x$)}.
	\end{equation}
\end{lem}

\begin{proof}
	If $u$ fails to be \Fd subharmonic at $x$, then according to \Cref{defn:FSH} and \Cref{rem:equiv_FSH} there exists a quadratic upper test function $\varphi_{J'}$ for $u$ at $x$ such that $J' \defeq J^2_x \varphi_{J'} \not\in \call F_x$. Defining $J \defeq J'+(0,0, 2\epsilon I)$, this means that $\phi_{J} \defeq \phi_{J'} + \epsilon|\cdot - \,x|^2$ satisfies \eqref{btj:i}.
	The conclusion follows by noting that, since the complement of $\call F_{x}$ in $\call J^2$ is open, we have that $J \notin \call F_{x}$ for $\epsilon$ small.
\end{proof}

The second tool is a \emph{comparison principle} whose validity characterizes the \Fd subharmonic functions for a given subequation $\call F$. We begin by recalling two equivalent forms of potential theoretic comparison which will play an important role in all of our proofs of comparison principles.

\begin{definition}
\sid{comparison (principle)}
	Given a subequation $\call F$, we say that \emph{comparison holds for $\call F$ on a domain $\Omega \subset \R^n$} if
	\begin{equation} \tag{CP} \label{(CP)}
	u \leq w\ \text{on $\de\Omega$} \implies u\leq w\ \text{on $\Omega$}
	\end{equation}
	for all $u\in \USC(\barr \Omega)$, $w \in \LSC(\barr \Omega)$ which are \Fd subharmonic, \Fd superharmonic on $\Omega$, respectively. 
\end{definition}
Making use of the duality reformulation with $v \defeq -w$, comparison \eqref{(CP)} is equivalent to
\begin{equation} \tag{CP$'$} \label{CP'}
u +v\leq 0\ \text{on $\de\Omega$} \implies u+v\leq 0\ \text{on $\Omega$}
\end{equation}
for all $u,v \in \USC(\barr\Omega)$ which are, respectively, \Fd , $\tildee{\call F}$-subharmonic on $\Omega$.
This form \eqref{CP'} is in turn is equivalent to the following \emph{zero maximum principle}:\sid{zero maximum principle (ZMP)|seeonly{maximum principle, zero}}\sid{maximum!principle!zero}
\begin{equation} \tag{ZMP} \label{(ZMP)}
z \leq 0\ \text{on $\de\Omega$} \implies z\leq 0\ \text{on $\Omega$}
\end{equation}
for all $z \in \USC(\barr \Omega) \cap \bigl(\call F(\Omega) + \tildee{\call F}(\Omega)\bigr)$.

We are now ready for the lemma, which states that comparison holds if the function $z$ in \eqref{(ZMP)} is the sum of a \Fd subharmonic and a $C^2$ strict $\tildee{\call F}$-subharmonic.\sid{Lemma!Definitional Comparison}\sid{comparison (principle)!definitional|seeonly{Lemma, Definitional Comparison}}\sid{definitional comparison|seeonly{Lemma, Definitional Comparison}}
It is called \emph{definitional} comparison because it relies only upon the ``good'' definitions we gave for subequations $\call F$ and \Fd subharmonic functions. It was first stated and proven in \cite[Lemma~3.14]{chlp} for constant coefficient subequations; we reproduce here its proof to highlight that it holds in the nonconstant coefficient case, too.

\begin{lem}[Definitional Comparison] \label{defcompa}
	Let $\call F$ be a subequation and $u \in \USC(X)$.
	\begin{enumerate}[label=\it(\roman*)]
		\item	If $u$ is \Fd subharmonic on $X$, then the following form of the comparison principle holds for each bounded domain $\Omega \ssubset X$:
		\[
		u+v \leq 0\ \text{on}\ \de \Omega \implies u+v \leq 0\ \text{on}\  \Omega
		\]
		whenever $v \in \USC(\barr \Omega) \cap C^2(\Omega)$ is strictly $\tildee{\call F}$-subharmonic on $\Omega$.
		\item	Conversely, suppose that for each $x\in X$ there is a neighborhood $\Omega \ssubset X$ of $x$ where the above form of comparison holds. Then $u$ is \Fd subharmonic on $X$. 
	\end{enumerate}
\end{lem}

\begin{proof}
	If the form of comparison in {(i)} fails for some domain $\Omega \ssubset X$ and some regular strict $\tildee{\call F}$-subharmonic function $v$, then $u+v \in \USC(\barr \Omega)$ will have a positive maximum value $m>0$ at an interior point $x_0 \in \Omega$ and hence $\phi\defeq -v+m$ is $C^2$ near $x_0$ and satisfies $u-\phi \leq 0$ near $x_0$, with equality at $x_0$.
	Since $u$ is \Fd subharmonic at $x_0$, this implies that
	\[
	J^2_{x_0}(-v+m) = - J^2_{x_0}v + (m, 0,0) \in \call F_{x_0},
	\]
	which contradicts the property \eqref{(N)'} for $\call F$ since $m>0$ and $J^2_{x_0}v \in \intr{\tildee{\call F}}_{\!x_0} = (-\call F_{x_0})\compl$.
	
	Conversely, suppose that $u$ fails to be \Fd subharmonic at some $x_0\in X$. By the Bad Test Jet \Cref{l:btj}, there exist $\rho, \epsilon > 0$ and an upper contact jet $(p,A) \in J^{2,+}_{x_0} u$ such that $J=(u(x_0), p, A) \notin \call F_{x_0}$ and
	\begin{equation} \label{btj}
	u(x) - \phi_J(x) \leq -\epsilon|x-x_0|^2 \quad \text{on $B_\rho(x_0)$, \ with equality at $x_0$},
	\end{equation}
	where $\phi_J$ is the upper contact quadratic function for $u$ at $x_0$ with $ J^2_{x_0}\phi_J = J$. Consequently, the function $-\phi_J$ is smooth and strictly $\tildee{\call F}$-subharmonic at $x_0$ and therefore, by choosing $\rho$ sufficiently small, the function $\tilde v \defeq -\phi_J + \epsilon \rho^2$ will be smooth and strictly $\tildee{\call F}$-subharmonic on $B_\rho(x_0)$.\footnote{Notice that one can make such a choice since there exist $r,R>0$ such that
		\[
		\call B_R\big(J^2_{x}(-\phi_J)\big) \subset \tildee{\call F}_{x} \quad \forall x \in B_r(x_0),
		\]
		and thus $J^2_{x}(-\phi_J+\epsilon\rho^2) \in \intr \tildee{\call F}_x$ for each $x \in B_\rho(x_0)$, provided that $\rho < \min \big\{ r, \sqrt{R/\epsilon} \big\}$.} Since reducing $\rho$ preserves the validity of (\ref{btj}), the form of comparison in {(i)} fails for $\tilde v$ on $B_\rho(x_0)$ since
	\[
	u+\tilde v = 0 \ \text{on $\de B_\rho(x_0)$} \quad \text{but}\quad u(x_0)+\tilde v(x_0) = \epsilon \rho^2 > 0. \qedhere
	\]
\end{proof}

\begin{remark}[Applying definitional comparison] 
	Sometimes it is useful to prove the contrapositive of the form of comparison in part {(i)} of \Cref{defcompa} in order to conclude \Fd subharmonicity. That is to say, in order to show by {(ii)} that $u$ is  \Fd subharmonic on $X$ one proves that, for each $x \in X$ there is a neighborhood $\Omega \ssubset X$ of $x$ where
	\begin{equation} \label{appdefcompa}
	(u+v)(x_0) > 0 \ \text{for some $x_0 \in \Omega$} \ \implies \ (u+v)(y_0) > 0 \ \text{for some $y_0 \in \de\Omega$}
	\end{equation}
	for every $v \in \USC(\barr \Omega) \cap C^2(\Omega)$ which is strictly $\tildee{\call F}$-subharmonic on $\Omega$. Conversely, one can also infer that the implication (\ref{appdefcompa}) holds whenever one knows that $u$ is  \Fd subharmonic on $X$. In situations where we are interested in proving the  \Fd subharmonicity of a function which is somehow related to a given  \Fd subharmonic, this helps to close the circle (for example, see the proof of upcoming \Cref{elemprop}).
\end{remark}

\section{Fundamental properties of subharmonics}

Finally, we give a collection of tools which represents elementary properties shared by functions in $\call F(X)$, the set of \Fd subharmonics on $X$.\sid{property!elementary!of \Fd subharmonics}\sid{property!fundamental (of subharmonics)|seeonly{elementary}} They are to be found in \cite[Section~4]{hldir09} for pure second-order subequations, in \cite[Theorem 2.6]{hldir} for subequations on Riemannian manifolds, in \cite[Proposition D.1]{chlp} for constant-coefficient subequations. We place it here because by invoking the definitional comparison lemma one can perform most of the proofs along the lines of those in \cite{hldir09}; indeed, we are going to use the definitional comparison in order to make up for the lack, for arbitrary subequations, of a result like \cite[Lemma 4.6]{hldir09}.

\begin{prop}[Elementary properties of $\call F(X)$] \label{elemprop}
	Let $X \subset \R^n$ be open, and let $\call F$ be a subequation on $X$. Then the following properties hold:
	\begin{enumerate}[left=\parindent, align=left, itemsep=1.1pt, label=(\roman*)] 
		\item {\bf local}: given $u \in \USC(X)$,\sid{property!local}
		\[
		\mbox{$u$ is locally \Fd subharmonic $\iff$ $u \in \call F(X)$;}
		\]
		\item {\bf maximum}:	$u,v \in \call F(X)$ $\implies$ $\max\{u,v\}\in \call F(X)$;\sid{property!maximum} \sid{maximum!property|seeonly{property, maximum}}
		\item {\bf coherence}: if $u \in \USC(X)$ is twice differentiable at $x_0\in X$, then\sid{property!coherence}
		\[
		\text{$u$ is \Fd subharmonic at $x_0$}\ \iff\ \text{$J^2_{x_0} u \in \call F_{x_0}$;}
		\]
		\item {\bf sliding}: $u\in \call F(X)$ $\implies$ $u-m \in \call F(X)$ for any $m >0$;\sid{property!sliding}
		\item {\bf decreasing sequences}: given a (pointwise) decreasing sequence $\{u_k\}_{k\in \mathbb{N}} \subset \call F(X)$, one has $ {\lim_{k\to\infty} u_k \in \call F(X)}$;\sid{property!decreasing sequence(s)}
		\item {\bf uniform limits}: if $\{u_k\}_{k\in \mathbb{N}} \subset \call F(X)$  converges to $u$ locally uniformly on $X$, then $u \in \call F(X)$;\sid{property!uniform limit(s)}
		\item {\bf families locally bounded above}: if $\scr F \subset \call F(X)$ is a family of functions which are locally uniformly bounded above,\sid{property!families locally bounded above} and $u \defeq \sup_{w \in \scr F} w$ is the associated Perron function, then $u^* \in \call F(X)$, where $u^*$ denotes the upper semicontinuous envelope of $u$.\footnote{Recall that the \emph{upper semicontinuous envelope} of a function $g$ is defined as the function
			\[
			g^\ast (x) \defeq \lim_{r\dto 0}\sup_{y \in B_r(x)} g(y).
			\]
			It is immediate to see that the \emph{upper semicontinuous envelope operator} ${}^* \colon g \mapsto g^*$ is the identity on the set of all upper semicontinuous functions. Also, we called \emph{Perron function} the upper envelope of the family $\scr F$, since $\scr F$ is a family of subharmonics.}
	\end{enumerate}
	Furthermore, if $\call F$ has constant coefficients, the following property also holds:
	\begin{enumerate}[resume, left=\parindent, align=left, itemsep=1.1pt, label=(\roman*)]
		\item {\bf translation}:\sid{property!translation} $u \in \call F(X)$ $\iff$ $u_y\defeq u(\,\cdot - y) \in \call F(X + y)$, for any $y \in \R^n$.
	\end{enumerate}
\end{prop}

\begin{proof}
	Property {(i)} follows immediately from the local nature of the notion of \Fd subharmonicity. 
	For property (ii), let $w\defeq \max\{u,v\}$ and note that if $(p,A) \in J^{2,+}_x w$, then either $(p,A) \in J^{2,+}_x u$ or  $(p,A) \in J^{2,+}_x v$ since $w(x)$ is either $u(x)$ or $v(x)$ and, in general, $u(y), v(y) \leq w(y)$; therefore, in either case, $(w(x),p,A) \in \call F_x$.
	Property (iii) follows from the positivity condition \eqref{(P)'}, thanks to the elementary fact that, for $u$ which is at least twice differentiable at $x$, if $(p,A)$ is an upper contact jet for $u$ at $x$, then $(p,A) = (Du(x), D^2u(x) + P)$ for some $P\geq 0$.
	Property (iv) follows immediately from the negativity condition \eqref{(N)'}.
	
	To prove property (v), we are going to use the {definitional comparison}. Since the decreasing limit $u$ of $\{u_k\}_{k\in \mathbb N} \subset \USC(X)$ belongs to $\USC(X)$, by \Cref{defcompa} it suffices to show that for each $\Omega \ssubset X$, one has $u+v \leq 0$ on $\Omega$ for each $v \in \USC(\barr\Omega) \cap C^2(\Omega)$ which is strictly $\tildee{\call F}$-subharmonic on $\Omega$ with $u+v \leq 0$ on $\de\Omega$.
	For each $\epsilon >0$ fixed, consider the sets
	\[
	\mathfrak E_k \defeq (u_k+v)^{-1}([\epsilon, + \infty)) \cap \de\Omega.
	\]
	By the upper semicontinuity of each $u_k$ (and thus of each $u_k+v$), $\mathfrak E_k$ are closed subsets of the compact set $\de \Omega$ and hence each $\mathfrak E_k$ is compact; since $\{u_k\}_{k \in \mathbb N}$ is decreasing, they are nested (that is, $\mathfrak E_{k+1} \subset \mathfrak E_k$ for all $k \in \mathbb N$), and since $u+v \leq 0$ on $\de \Omega$, they have empty intersection, which means that
	\[
	\bigcap_{k\in\mathbb{N}} \mathfrak E_k = (u+v)^{-1}([\epsilon, +\infty)) \cap \de \Omega = \emptyset.
	\]
	Therefore, by the well-known property of nested families of compact sets, $\mathfrak E_k$ must be empty for all large $k$. This means that $u_k+ v < \epsilon$ on $\de \Omega$ for $k$ large and since $v - \epsilon$ is still smooth and a strict $\tildee{\call F}$-subharmonic (by the sliding property (iv)), the definitional comparison lemma yields $u_k + v \leq \epsilon$ on $ \Omega$. Hence $u+v \leq u_k + v \leq \epsilon$ on $\Omega$, for every $\epsilon >0$. Taking the limit $\epsilon \dto 0$ gives $u+v \leq 0$ on $\Omega$ and thus $u\in \call F(X)$, again by definitional comparison.
	
	As for property~(vi), note that if $u_k$ converges to $u$ uniformly on the compact set $\barr\Omega$ and if we suppose $u+v \leq 0$ on $\de \Omega$ as above, then, for any $\epsilon > 0$ fixed, $u_k + v < \epsilon$ on $\de\Omega$ for $k$ large. The definitional comparison lemma yields $u_k + v \leq \epsilon$ on $\Omega$, hence $u+v = (u-u_k) + (u_k +v) < 2\epsilon$ on $\Omega$ for $k$ large, where one exploits the uniform convergence to say that $u-u_k < \epsilon$ on $\Omega$ if $k$ is large. Taking the limit $\epsilon \dto 0$ gives $u+v \leq 0$ on $\Omega$ and thus, by the converse implication of the definitional comparison, $u \in \call F(X)$.
	
	Property~(vii) is also proved by using the definitional comparison. If $u^\ast +v \leq 0$ on $\de \Omega$, then $w+ v \leq 0$ on $\de \Omega$ for all $w \in \scr F$; then, by the implication {(i)} of the Definitional Comparison \Cref{defcompa}, $w+ v \leq 0$ on $\Omega$ for all $w \in \scr F$ and thus $u+v \leq 0$ on $\Omega$. Since $v$ is smooth and since the upper semicontinuous envelope operator preserves inequalities, we conclude that $u^* + v = (u+v)^* \leq 0$ on $\Omega$ and the thesis follows from the implication {(ii)} of \Cref{defcompa}.
	
	Finally, we observe that property (viii) holds only in a constant coefficient context, as its validity for every $y\in \R^n$ is actually equivalent to the fiber $\call F_x$ being independent of the point $x\in X$. Indeed, since by the respective definitions $u_y(x+y) = u(x)$ and  $J^{2,+}_x u= J^{2,+}_{x+y} u_y$, one has, for $z \in X+y$, that $J^{2,+}_z u_y \subset \call F_{z-y}$, and thus property (viii) holds if and only if $\call F_{z-y} = \call F_z$.
\end{proof}

As a final remark, the coherence property {(iii)} says that an upper semicontinuous function which is twice differentiable and classically \Fd subharmonic on $X$ (that is, it satisfies the condition~(\ref{coherence} for all $x\in X$), it is \Fd subharmonic in the viscosity sense. One might hope to have the same conclusion even if~(\ref{coherence}) is satisfied almost everywhere (with respect to the Lebesgue measure).
This is true if $u$ is locally semiconvex and is known as the {\em Almost Everywhere Theorem} from \cite{hlae}. This deep result, which was first established in the pure second-order case in \cite[Corollary 7.5]{hldir09}, will be discussed in the next chapter.

\chapter{Semiconvex functions and subharmonics}\label{chap:SCSH}

In this chapter we treat the main theme of this work on the efficient use of semiconvex functions in general potential theories. We begin with the Almost Everywhere Theorem which says that in the potential theory defined by a given subequation constraint set $\call F$,  a locally semiconvex function is \Fd subharmonic on an open set if it is \Fd subharmonic on a set of full (Lebesgue) measure. Moreover, by Alexandrov's Theorem and coherence, \Fd subharmonicity reduces to a classical statement almost everywhere. Next we discuss the Subharmonic Addition Theorem which gives conditions under which the algebraic sum formula of jet addition implies the functional analytic sum formula of subharmonic addition for the associated subharmonics. Combined with monotonicity and duality this leads to a robust method for establishing comparison principles, which will be discussed in the following chapter.

\section{The Almost Everywhere Theorem}

The first essential result is the following theorem.

\begin{thm}[Almost Everywhere Theorem]\label{aet} \sid{Theorem!Almost Everywhere}
	Suppose that $\call F$ is a primitive subequation on an open set $X\subset \R^n$ and
	$u \colon X \to \R$ is a locally semiconvex function. Then
	\[
	J_x^2u \in \call F_x\ \text{for almost all $x\in X$}\quad \iff\quad u\in \call F(X).
	\]
\end{thm}

\begin{proof}
	By \Cref{ae?} with $E \defeq \{x \in X:\,  J^2_xu \in \call F_x\}$, we know that
	\[
	\call J^{2,+}u \subset \barr{\call J(u,E)},
	\]
	where $\call J^{2,+}u$ and $\call J(u,E)$ are defined in \eqref{J2+w} and \eqref{JwE}, respectively. 
	By property \eqref{(P)} one has $\call J(u,E) \subset \call F$ and by property \eqref{(C)} also $\barr{\call J(u,E)} \subset \call F$. Hence $u \in \call F(X)$ by \Cref{subinclusion}.
	
	Conversely, for every $P>0$, we know that $(Du(x), D^2u(x) + P)$ is an upper contact jet for $u$ at $x \in \Diff^2(u)$. Therefore, in particular, $J^2_x u + (0,0,\epsilon I) \in \call F_x$ for all $\epsilon >0$. By taking $\epsilon \dto 0$ and using the fact that $\call F$ is closed, we get $J_x^2u \in \call F_x$ for all $x\in \Diff^2 u$ and hence the conclusion follows from Alexandrov's \Cref{aleks:qc}.
\end{proof}

Before moving on to the next fundamental result, a technical remark is in order.

\begin{remark}[on the proof of \Cref{aet}]
The proof above relies heavily on Remarks \ref{subinclusion} and \ref{ae?}, which, in some sense, ``hide the hard analysis'' of the combined use of the Jensen--S{\l}odkowski \Cref{jenslod} and of Alexandrov's \Cref{aleks:qc}. For somewhat more explicit proofs one can consult \cite{hlae}, or \cite[Lemma 2.10]{chlp} where the Bad Test Jet Lemma is also used.
\end{remark}

\section{The Subharmonic Addition Theorem}

We are now ready for the second fundamental result. As noted in \cite{hlae}, an immediate consequence of the Almost Everywhere \Cref{aet} is a {\em subharmonic addition theorem} for locally semiconvex functions. That is, by restricting attention to subharmonics which are locally semiconvex, a purely algebraic statement of \emph{jet addition} always implies the functional analytic statement of \emph{subharmonic addition}. This can be written as\sid{subharmonic addition|seeonly{Theorem, Subharmonic Addition}}
\begin{equation}\label{SAT}
\call F_x + \call G_x \subset \call H_x \ \ \forall \, x \in X \ \ \implies\ \ (\call F(X) \cap  \mathsf{sc}_\mathrm{loc}(X)) +  (\call G(X) \cap  \mathsf{sc}_\mathrm{loc}(X)) \subset \call H(X),
\end{equation}
where $ \mathsf{sc}_\mathrm{loc}(X)$ is the space of all locally semiconvex functions on $X$, which were characterized in formula \eqref{qc_loc} of \Cref{exe:qc}. The result is formalized in the following way.

\begin{thm}[Subharmonic addition for semiconvex functions]\label{add}\sid{Theorem!Subharmonic Addition!for semiconvex functions}
	Let $\call F$ and $\call G$ be primitive subequations on an open set $X\subset \R^n$ and define $\call K \defeq \barr{\call F + \call G}$; let $u\in \call F(X)$ and $v\in \call G(X)$ be locally semiconvex functions on $X$. Then $u+v \in \call K(X)$.
\end{thm}

\begin{proof}
	Since $u \in \call F(X)$ and $v \in \call G(X)$ are locally semiconvex, by the reverse implication in the Almost Everywhere \Cref{aet} we have $J^2_x u \in \call F_x$ and $J^2_x v \in \call G_x$ for almost every $x \in X$. Hence $J^2_x(u+v) = J^2_x u + J^2_x v \in \call K_x$ almost everywhere, with $u + v$ locally semiconvex. The forward implication of the Almost Everywhere \Cref{aet} now yields $u+v \in \call K(X)$. 
\end{proof}

\begin{remark}
	Theorem \ref{add} implies subharmonic addition in the form \eqref{SAT} even if $\call H \supset \call F + \call G$ is not a (primitive) subequation. This is because $\call K$ is the minimal primitive subequation containing $\call F + \call G$. Indeed, the (fiberwise) sum satisfies \eqref{(P)}, therefore so does its closure. 
\end{remark}

\begin{remark}
	Taking the closure of the sum is necessary in order to assure that property \eqref{(C)} holds. For example, consider $\call F$ and $\call G$ induced by the two equations $F(u) = 0$ and $G(u) = 0$ on $\R^n$, with 
	\[
	F(r) \defeq \chi_{[0, +\infty)}(r) \sin^2(\pi r^2) - \tfrac34 \quad \text{and} \quad G(r) \defeq \chi_{(-\infty, 0]}(r) \sin^2(\pi (r+\alpha)^2)  - \tfrac34,
	\]
	for some $\alpha \in (\frac13, \frac23)$, where $\chi_E$ denotes the characteristic function of the set $E$ (that is, $\chi(r) = 1$ if $r \in E$ and $\chi(r) = 0$ if $r \notin E$). It is easy to see that $(0,0,0,0) \notin \call F + \call G$, yet it is a limit point, since for every positive integer $k$, we have that  $\sqrt{1/3 + k} - \sqrt{2/3 - \alpha + k} = O\big(1/\sqrt k\big)$ belongs to the $\R$ component of each jet in $\call F + \call G$.
\end{remark}

With respect to \Cref{add} on subharmonic addition, two observations are in order. The first observation contains a reformulation of the Summand Theorem \ref{pusc:sum} which provides a stronger version of \Cref{add} while the second observation notes, as in \cite{hlae}, that the semiconvexity assumption on the subharmonics is not needed in the constant coefficient case.

\begin{remark}[A stronger form of subharmonic addition] \label{rmkone} \sid{Theorem!Subharmonic Addition!stronger form}
	The Summand Theorem~\ref{pusc:sum} leads to a slightly stronger form of \Cref{add}.
	Indeed, it tells us that if $(p,A)$ is an upper contact jet for the sum $w = u+v$ at $x$, then there exist two jets 
	\[
	(x, u(x), Du(x), B') \in \barr{\call J^{2,+}u}, \quad (x, v(x), Dv(x), C) \in \barr{\call J^{2,+}v}
	\]
	which sum to $(w(x), p, A)$. To see this, note that, considering the sequence $\{x_j\}_{j \in \N}$ in \Cref{pusc:sum}(ii), it is sufficient to take $B' = A + P$, where  $P = A - B - C \geq 0$, or, more succinctly, note that 
	\[
	\call J^{2,+}u + \call J^{2,+}v \subset \call J^{2,+}(u+v) \subset \barr{\call J^{2,+}u} + \barr{\call J^{2,+}v},
	\]
	where the former inclusion is trivial and the latter is given by \Cref{pusc:sum}.
	
	In other words,  \Cref{pusc:sum} assures that every upper contact $2$-jet of the sum $u+v$ of two semiconvex functions $u \in \call F(X)$ and $v \in \call G(X)$ can be represented as the sum of two $2$-jets of $\call F$ and $\call G$, respectively, hence $u+v \in (\call F + \call G)(X)$. This will be used in the proof of the strict comparison for semiconvex functions (\Cref{scqc}).
\end{remark}

\begin{remark}
	If $\call F$ and $\call G$ are constant coefficients subequations, then the assumption that $u$ and $v$ are semiconvex can be dropped; therefore, we have 
	\[
	\call F(X) + \call G(X) \subset \barr{\call F + \call G}\,(X).
	\]
	This will be proved next and follows from the semiconvex approximation technique that was introduced in \cite{hldir09} in a context of pure second-order subequations and that works fine also for general constant coefficient subequations.
\end{remark}

As we anticipated in the previous part, the sup-convolution of Definition \ref{def:supconv} is known to provide an effective way to approximate upper semicontinuous functions by a semiconvex \emph{regularization} of them (recall the four basic properties of the sup-convolution collected in \Cref{baspropsc}).
The next proposition, which imitates \cite[Theorem~8.2]{hldir09}, shows that, for constant coefficent subequations $\call F$, the sup-convolution of a bounded \Fd subharmonic remains \Fd subharmonic, provided that we are far enough from the boundary of the domain $X$. 

\begin{prop}[Semiconvex approximation of \Fd subharmonics] \label{prop:approx} \sid{semiconvex!approximation|seeonly{approximation, semiconvex}}\sid{approximation!semiconvex}
	Let $\call F$ be a constant coefficient subequation and suppose that $ u \in \call F(X)$ is bounded. For $\epsilon >0$, let $\delta \defeq 2\sqrt{\epsilon \sup_X\! |u|}$, and $X_\delta \defeq \{ y \in X:\ d(y, \de X) > \delta\}$. Then the $\epsilon$-sup-convolution of $u$ is \Fd subharmonic on $X_\delta$; in symbols, $u^\epsilon \in \call F(X_\delta)$. \syid{dist@$d(x,E)$!the distance between the point $x$ and the set $E \subset \R^n$, namely \detokenize{$\inf\{\vert x-y\vert :\, y \in E\}$}} \syid{Xdelta@$X_\delta$!the $\delta$-enlargement (or $\delta$-neighborhood) of the set $X$}
\end{prop}

\begin{proof}
	Since $\call F$ is constant coefficient, we have that $u_z \defeq u(\cdot - z) \in \call F(X_\delta)$ for all $z \in B_\delta$. Indeed, it is easy to see that $J^{2,+}_x u_z \subset \call F_{x-z} = \call F_x$ for every $x \in X_\delta$. If we now take
	\[
	\scr F \defeq \bigg\{ u_z - \frac{1}{2\epsilon}|z|^2:\ z \in B_\delta \bigg\},
	\]
	then $\scr F \subset \call F(X_\delta)$ by the negativity condition \eqref{(N)'} and, by (\ref{supcball}), its upper envelope is $u^\epsilon$, and it coincides with $(u^\epsilon)^\ast$ because it is semiconvex, hence upper semicontinuous. Therefore, we have $u^\epsilon \in \call F(X_\delta)$, by property (vii) of \Cref{elemprop}.
\end{proof}

The final result of this section is an important application of the semiconvex approximation technique. We are going to deduce a Subharmonic Addition Theorem for merely upper semicontinuous subharmonics for constant coefficient subequations from its semiconvex counterpart (\Cref{add}). This was first stated in \cite[Theorem~7.1]{chlp} and we essentially reproduce the original proof. 

\begin{cor}[Subharmonic addition for constant coefficient subequations] \label{ccsa}\sid{Theorem!Subharmonic Addition!for constant coefficient subequations}
	Let $\call F$, $\call G$ and $\call H$ be constant coefficient subequations. Then, for any open set $X\subset \R^n$, jet addition implies subharmonic addition; that is,
	\[
	\call F + \call G \subset \call H \implies \call F(X) + \call G(X) \subset \call H(X).
	\]
\end{cor}

\begin{proof}
	Let $u\in \call F(X)$, $v \in \call G(X)$. Since the result is local, it is sufficient to show that, for each $x\in X$, one has $u+v \in \call H( U)$, where $ U$ is any neighborhood of $x$. Hence by upper semicontinuity we may assume that $u$ and $v$ are bounded above. Also, we may assume that $u$ and $v$ are bounded below, since it suffices to consider, for any positive integer $m$, the subharmonics $\max\{u, \phi - m\} \in \call F( U)$ and $\max\{v, \psi- m\} \in \call G( U)$, for some bounded quadratic functions $\phi \in \call F( U)$ and $\psi \in \call G( U)$, and then take the limit $m\to \infty$. Recall that we have already highlighted in \Cref{exquadsub} that the existence of such quadratic functions is guaranteed by the topological property \eqref{(T)'}, and that their translations by negative constants are still subsolutions by the \emph{sliding property} (\Cref{elemprop}(iv)) yielded by the properness condition \eqref{(N)'}, while the \emph{truncated} functions are subharmonic thanks to the \emph{maximum property} (\Cref{elemprop}(ii)). With these assumptions, by \Cref{prop:approx} there exist sequences $\{u_j\}_{j\in\N}$ and $\{v_j\}_{j\in\N}$ of semiconvex subharmonics near $x$ that converge monotonically downward to $u$ and $v$, respectively. By \Cref{add} we know that $u_j + v_j \in \call H(X)$ and by the \emph{decreasing sequence property}  (\Cref{elemprop}(v)) we conclude that $u+v \in \call H(X)$. 
\end{proof}

We conclude this section with some remarks on the utility, limitations and generalizations of the subharmonic addition theorem for constant coefficient subequations in \Cref{ccsa}. First, we give one important use of it, which concerns the validity of the comparison principle and will be the subject of \Cref{chap:comp}.

\begin{remark}[Subharmonic addition and the comparison principle]\label{rem:SAT_CC}\sid{monotonicity-duality method}
	If a subequation constraint set $\call F$ on $X$ (not necessarily with constant coefficients) is $\call M$-monotone for some constant coefficient monotonicity cone subequation $\call M$ as in \eqref{monotone},
	then duality gives the jet addition formula
	\[
	\call F_x + \widetilde{\call F}_x \subset \widetilde{\call M} \quad \forall\, x \in X;
	\]
	indeed, as for any $J \in \call F_x$ we have $J + \call M \subset \call F_x$, taking the duals and recalling the properties (1) and (3) of \Cref{propdual}, we deduce that $J + \tildee{\call F}_x \subset \tildee{\call M}$.
	 
	Using then $\call G = \widetilde{\call F}$ and $\call H = \widetilde{\call M}$, the validity of a subharmonic addition theorem would yield
	\begin{equation}\label{monodual}
	\call F (X) + \widetilde{\call F} (X) \subset \widetilde{\call M}(X),
	\end{equation}
	and then the validity of the comparison principle on $\Omega \ssubset X$ can be reduced to the validity of the {\em Zero Maximum Principle} \eqref{(ZMP)} for $z \in \USC(\overline{\Omega}) \cap  \widetilde{\call M}(\Omega)$.
	
	Corollary \ref{ccsa} says that if $\call F$ has constant coefficients, then \eqref{monodual} indeed holds. This is the {\em monotonicty-duality} method for proving comparison, which is extensively treated in the constant coefficient setting in \cite{chlp} including a classification of monotonicity cone subequations $\call M$ and the validity of \eqref{(ZMP)} for the dual cone subharmonics. 
\end{remark}

Next, we discuss the limitations and generalizations of Corollary \ref{ccsa}.

\begin{remark}[Subharmonic addition for variable coefficient subequations]
\sid{fiberegular|seeonly{property, fiberegularity}}\sid{property!fiberegularity}
	In the constant coefficient setting, the semiconvex approximation preserves \Fd subharmonicity. This is essentially due to validity of the translation property (viii) of Proposition~\ref{elemprop}, which one loses when passing to nonconstant coefficients (as seen at the end of the proof of that proposition). In the variable coefficient setting, one can extend the semiconvex approximation technique to obtain the subharmonic addition theorem and its consequences provided that the variable coefficient subequation $\call F$ is both {\em $\call M$-monotone} (for some monotonicty cone subequation $\call M$) and {\em fiberegular}. 
	
	The $\call M$-monotonicty of $\call F$ always holds for variable coefficient pure second order and gradient-free subequations, but when there is gradient dependence, one also needs {\em directionality}, as studied in \cite{cprdir}. \sid{property!directionality} Moreover, in this case, the assumption of $\call M$-monotonicity is also needed for the monotonicity-duality method for proving the comparison principle, as noted in Remark \ref{rem:SAT_CC}. 
	
	The fiberegularity of $\call F$ is just the continuity (with respect to the Hausdorff metric) of the fiber map $\Theta: X \to \scr P(\call J^2)$ defined by $\Theta_x\defeq  \call F_x$ for each $x \in X$. Using this regularity property and the $\call M$-monotonicity, one can replace the translation property by a {\em uniform translation property},\sid{property!translation!uniform} which roughly speaking means that \emph{\em if $u \in \call F(\Omega)$, then there are small $C^2$ strictly \Fd subharmonic perturbations of {all small translates} of  $u$ which belong to  $\call F(\Omega_{\delta})$}, where  $\Omega_{\delta}\defeq  \{ x \in \Omega: d(x, \partial \Omega) > \delta \}$ (see \cite[Theorem~3.3]{cprdir} for a precise statement and the proof). This uniform translation property was developed first in the variable coefficient pure second order case in \cite{cpaux} and generalized to the variable gradient-free case in \cite{cpmain}, where we stress again that the needed $\call M$-monotonicity holds in these two gradient-free cases.
\end{remark}

\chapter{Comparison principles}\label{chap:comp}

This chapter is dedicated to the proof of comparison principles in general potential theories; that is,  the proof that if a subharmonic-superharmonic pair for a subequation $\call F$ are ordered on the boundary of a bounded domain, then they are ordered in the interior of the domain. Two distinct situations will be presented. The first situation concerns the use of the monotonicity-duality method which works when $\call F$ has sufficient monotoncity. The second situation concerns the extreme case when $\call F$ has only the minimal monotonicity that any subequation must have. In this situation, a result on strict comparison is presented for semiconvex subharmonics and dual subharmonics.

\section{Preliminaries}

We begin by recalling that, before stating the definitional comparison, we rewrote the comparison principle on $\Omega \ssubset X$\sid{comparison (principle)}
\begin{equation*} 
\begin{cases}
\text{$u$ \Fd subharmonic, $v$ \Fd superharmonic on $X$} \\
u \leq v\ \text{on $\de\Omega$}
\end{cases}
\implies u\leq v\ \text{on $\Omega$}
\end{equation*}
as the zero maximum principle for sums\sid{maximum!principle!zero}
\begin{equation} \label{eq:ZMP} 
\begin{cases}
w \in \call F(X) + \tildee{\call F}(X) \\
w \leq 0\ \text{on $\de\Omega$}
\end{cases}
\implies w\leq 0\ \text{on $\Omega$},
\end{equation}
which merely exploits the fact that $v$ is \Fd superharmonic if and only if $-v \in \tildee{\call F}(X)$.
This suggests that a key role in the proof of comparison could be played by a subharmonic addition theorem; that is, one would like to know if 
\[
\call F_x + \tildee{\call F}_x \subset \call G_x \quad \forall x \in X \quad\implies\quad \call F(X) + \tildee{\call F}(X) \subset \call G(X).
\]
This would be particularly useful if two things were true. First, that $\call G$ does not depend on the specific form of $\call F$, but, for example, only on the monotonicity properties of $\call F$. Second, if we are able to say that comparison holds for $\call G$-subharmonics, for example, by way of a suitable characterization of $\call G(X)$.

In order to find a suitable $\call G$, recall that we noted that 
\begin{equation*}\label{jetadd:f}
\call F_x + \tildee{\call F}_x \subset \tildee{\call M}_x \quad \forall \, x \in X
\end{equation*}
whenever $\call F$ is $\call M$-monotone in the sense of \eqref{monotone}.
Therefore, in this case one may choose $\call G = \tildee{\call M}$.

Moreover, one knows that such a subset $\call M$ always exists. Indeed, regardless of the subequation $\call F$, thanks to the conditions \eqref{(P)'} and \eqref{(N)'} a set (with constant fibers) for which \ref{monotone} holds is 
\[
\call M = \call N \times \{0\} \times \call P,
\]
with $\call N$ and $\call P$ defined as in \eqref{def:np}; that is, \eqref{monotone} always holds for $\call M$ being the minimal monotonicity cone $\call M_0$ defined in \eqref{MMC}.
Note that $\call M_0$ is a constant coefficient primitive subequation (it is closed and satisfies property \eqref{(P)}) and it also satisfies the negativity property \eqref{(N)'}; however, we already noted that it is not a subequation because, having empty interior due to the factor $\{0\}$, it fails to satisfy the topological property \eqref{(T)'}. 

Assuming that $\call M$ can be enlarged to be a \emph{monotonicity cone subequation} for which \eqref{monotone} holds, the only missing piece would be to prove that the zero maximum principle holds for functions in $\tildee{\call M}(X)$.  Notice that choosing $\call M$ as large as possible facilitates the proof of \eqref{(ZMP)} for $\tildee{\call M}(X)$, which becomes smaller as $\call M$ increases. This represents the main observation of the monotonicity-duality method of Harvey and Lawson, as discussed in \cite{chlp}.\sid{monotonicity-duality method} 

Having derived the monotonicity-duality method, we will present a few illustrative applications of it.
First, in Section \ref{sec:cccc}, a pair of constant coefficient situations coming from \cite{chlp,hldir09} for which comparison holds on all bounded domains and the needed \eqref{(ZMP)} for the $\widetilde{\call M}$-subharmonics is accomplished by characterizing them. Then, in Section \ref{sec:SC}, we will discuss {\em strict comparison} for semiconvex functions in the variable coefficient setting and for semicontinuous functions in the constant coefficient setting. 

\section{Comparison with constant coefficients and sufficient monotonicity} \label{sec:cccc} 

We will present two basic examples of classes of $\call M$-monotone subequations for which comparison holds on all bounded domains. The classes are determined by the two most elementary  monotonicity cone subequations $\call M$, which have interesting characterizations of the dual subharmonics $\tildee{\call M}(X)$, as known from \cite{chlp,hldir09}. These characterizations easily lead to a direct proof that the $\tildee{\call M}$-subharmonics satisfy \eqref{(ZMP)} on all bounded domains, and hence the comparison principle in all $\call M$-monotone potential theories. The two classes discussed here can be included in a much more general result that will be recalled in \Cref{rem:CP_cc} below.

\begin{definition} 
\sid{subequation (constraint set)!constant coefficient!pure second-order}
	A \emph{pure second-order constant coefficient subequation} is a  closed subset of $\call J^2$ of the form  $\call F = \R \times \R^n \times \call A$ with $\call A \subset \call S(n)$ satisfying the positivity property ($\call P$-monotonicity)
	\begin{equation}\label{pos_PSO}
	\call A  + \call P \subset \call A, 
	\end{equation}
	where $\call P = \{ P \in \call S(n): P \geq 0\}$.
\end{definition}

\begin{remark} \label{identify}
	We have used \Cref{rmkccs} to identify a constant coefficient subequation $\call F \subset \call J^2(\R^n)$ with a subset $\R \times \R^n \times \call S(n)$. Notice that such an $\call F$ is indeed a subequation because the negativity property \eqref{(N)} is automatic and the topological property \eqref{(T)} follows from the $\call P$-monotonicity, which implies the monotonicity property $\call F + (\R \times \R^n \times \call P) \subset \call F$,
	where the monotonicity $\call F + (\R \times \R^n \times \{0\}) \subset \call F$ is an elegant way to say that $\call F$ is of pure second order.
	In addition, since the only component of a jet that is actually relevant to define a subsolution is the one living in $\call S(n)$, we will often identify $\call F$ with $\call A \subset \call S(n)$ that satisfies the positivity condition \eqref{pos_PSO}. Hence, we can say that a  pure second-order constant coefficient subequation is a closed subset $\call A \subset \call S(n)$ which is $\call P$-monotone. 
\end{remark}

The most basic constant coefficient pure second order subequations can be defined in terms of the ordered eigenvalues
\[
\lambda_1(A) \leq \cdots \leq \lambda_n(A)
\]
of $A \in \call S(n)$.
Of particular importance are the \emph{convexity subequation}
\[
\call P \defeq \{ A \in \call S(n):\ A \geq 0 \} = \{ A \in \call S(n):\ \lambda_1(A) \geq 0 \},
\]
(see (\ref{def:np}) and \Cref{exe:P}) and its dual subequation 
\[
\tildee{\call P} \defeq \{ A \in \call S(n):\ -A \not\in \intr \, \mathcal{P} \} = \{ A \in \call S(n):\ \lambda_n(A) \geq 0 \},
\]
which is known as  and the \emph{subaffine subequation},\sid{subaffine!subequation|seeonly{subequation, subaffine}}\sid{subequation (constraint set)!subaffine} \sid{subequation (constraint set)!co-convex|seeonly{subaffine}} for reasons which we now review. 

\begin{definition}\label{defn:SA}
	A function $w\in \USC(X)$ is called \emph{subaffine on $X$}, and we write $w \in \mathrm{SA}(X)$,\syid{SAX@$\mathrm{SA}(X)$!the set of all subaffine functions on $X$} if, for every open subset $\Omega \ssubset X$ and each affine function $a \in \mathrm{Aff}$,\sid{subaffine!function}
	\begin{equation} \label{affcomp}
	w \leq a\ \text{on $\de \Omega$} \implies  w \leq a\ \text{on $\Omega$}.
	\end{equation}
\end{definition}

\begin{remark}
	We have denoted the space of all affine functions by $\mathrm{Aff}$,\syid{Aff@$\mathrm{Aff}$!the space of all affine functions} without specifying their domain because it is known that any $a$ which is affine on an open subset $\Omega$ has a unique extension to the whole space $\R^n$ (see, e.g., \cite{vesely}).
\end{remark}
 
The following characterizations of $\call P$ and $\tildee{\call P}$ subharmonics are the content of \cite[Proposition~4.5]{hldir09} (see also \cite[Theorem~10.7]{chlp}).

\begin{prop}\label{prop:P_Ptilde}
Let $X \subset \R^n$ be open and let $u \in \USC(X)$. Then one has the following:
\begin{enumerate}[label=(\alph*)]
	\item $u \in \call P (X)$ $\iff$ $u$ is locally convex (away from the connected components on which $u \equiv - \infty$);
	\item $u \in \tildee{\call P}(X)$ $\iff$ $u$ is subaffine on $X$ (that is, $\tildee{\call P}(X) = \mathrm{SA}(X)$).
\end{enumerate}	
\end{prop}

Notice that, for real-valued $u$, \Cref{prop:P_Ptilde}(a) yields
\[
u \ \text{locally convex in $X$} \iff u \in \call P(X),
\]
which is the characterization of locally convex functions as those functions having non-negative Hessian in the viscosity sense, as we anticipated the discussion preceding \Cref{prop:Hpd}. 

We refer the reader to \cite{hldir09} for the complete proof of \Cref{prop:P_Ptilde},\footnote{Note that \cite{hldir09} the \Fd subharmonic functions were said to be \emph{of type $\call F$}.} which follows from two interesting facts: first,
\begin{equation*}
v \in \mathrm{SA}(X) \iff u + v \in \mathrm{SA}(X) \ \ \forall \, u \ \text{locally convex (and $\R$-valued)};
\end{equation*}
second, denoting by $\barr{\mathrm{C}}\mathrm{onvex}(X)$ the set of upper semicontinuous functions which are locally either convex or identically $-\infty$, \begin{equation*}
u \in \barr{\mathrm{C}}\mathrm{onvex}(X)  \iff u + v \in \mathrm{SA}(X) \ \ \forall \, v \in \mathrm{SA}(X).
\end{equation*}
These are the content of \cite[Propositions~2.5 and 2.6]{hldir09}, respectively. Moreover, they are, in the end, consequences of the jet-addition formula
\[
\call P + \tildee{\call P} \subset \tildee{\call P},
\]
which comes from the fact that $\call P$ is a monotonicity cone for itself and hence also for the dual $\tildee{\call P}$. More precisely, they represent consequences of the most elementary example of the Subaffine Theorem~\ref{t:subaff} below, in the case where $\call F = \call P$.


Indeed, note that $\call P$ is a monotonicity cone subequation for all pure second-order constant coefficient subequations (and it is minimal in this sense), hence a straightforward application of the Subharmonic Addition Theorem (\Cref{ccsa}) yields the following result (while for a proof that uses S{\l}odkowski's Largest Eigenvalue Theorem and semiconvex approximation, the reader can instead refer to \cite{hldir09}).

\begin{thm}[{Subaffine Theorem; \cite[Theorem 6.5]{hldir09}}] \label{t:subaff} \sid{subaffine!theorem|seeonly{Theorem, Subaffine}} \sid{Theorem!Subaffine}
	Let $\call F$ be a pure second-order constant coefficient subequation. Then, for any open set $X \subset \R^n$,
	\[
	\call F(X) + \tildee{\call F}(X) \subset \mathrm{SA}(X).
	\]
\end{thm}

At this point it is evident that comparison holds, provided that one has the following strengthening of the property of subaffinity.

\begin{lem} \label{str:sub}
	Let $\Omega \subset \R^n$ be a bounded domain and $w\in \USC(\barr\Omega)\cap \mathrm{SA}(\Omega)$. Then \eqref{affcomp} holds for all $a \in \mathrm{Aff}$.
\end{lem}

\begin{proof}
	Assume without loss of generality that $a=0$. We consider an exhaustion of  $\Omega$ by an increasing sequence of compact sets $\{ K_j\}_{j \in \N}$ and set $U_\delta \defeq \{ x\in \barr \Omega:\, w(x) < \sup_{\de\Omega} w + \delta\}$, for any $\delta > 0$. Since $w\in \USC(\barr\Omega)$, we know that $U_\delta$ is an open neighborhood of $\de\Omega$ in $\barr \Omega$. Therefore $\de K_j \subset U_\delta$ for all  $j$ sufficiently large and, since $w$ is subaffine, $\sup_{K_j} \leq \sup_{\de \Omega} w + \delta$.
	This proves that $\sup_{\Omega} w \leq \sup_{\de \Omega} w + \delta$; since $\delta$ is arbitrary, $w\leq 0$ on $\de\Omega$ and $w$ is continuous, we conclude that $w \leq 0$ in $\barr\Omega$, which is the desired conclusion.
\end{proof}

\begin{thm}[Comparison for pure second-order constant coefficient subequations]\label{thm:C_pso}\sid{comparison (principle)!for pure second order constant coefficient subequations} \sid{Theorem!Comparison|seeonly{comparison (principle)}}
	Let $\call F$ be a pure second-order constant coefficient subequation. Then comparison holds on every domain $\Omega \ssubset \R^n$:
	\[
	u+v \leq 0\ \text{on $\de \Omega$} \implies u+v \leq 0\ \text{on $ \Omega$}
	\]
	if $u \in \USC(\barr \Omega) \cap \call F(\Omega)$ and $v \in \USC(\barr \Omega) \cap \tildee{\call F}(\Omega)$.
\end{thm}

\begin{proof}
	It follows from~\Cref{str:sub} with $a=0$, since $u+v$ is a subaffine function by the Subaffine \Cref{t:subaff}.
\end{proof}

The case of pure second-order constant coefficient subequations is a subset of the one we are going to discuss now, yet we considered historically significant to introduce them separately. Again, we will show that comparison holds on all bounded domains.

\begin{definition} 
\sid{subequation (constraint set)!constant coefficient!gradient-free}
	A \emph{gradient-free constant coefficient subequation} is a closed subset of $\R \times \R^n \times \call S(n)$ of the form $\call F = \call R \times \R^n \times \call A$ with $\call R \subset \R$ and $\call A \subset \call S(n)$ satisfying the conditions of negativity and positivity, respectively; that is, 
	\begin{equation*}
	\call R  + \call N \subset \call R \quad \text{and} \quad \call A  + \call P \subset \call A,
	\end{equation*}
	where $\call N = \{ r \in \R: r \leq 0\}$ and $\call P = \{ P \in \call S(n): P \geq 0\}$.
\end{definition}

\begin{remark}
	Just as in \Cref{identify}, we have identified the constant coefficient $\call F$ with a subset of $\call J^2$ and it is indeed a subequation, since the remaining topological property \eqref{(T)'} again follows by the $(\call N \times \R^n \times \call P)$-monotonicity of $\call F$. 
	Moreover, one can again {\em reduce} the subequation by suppressing the ``silent'' gradient variable,\sid{subequation (constraint set)!reduced} which then allows one to further identify $\call F$ with a subset of $\R \times \call S(n)$. Hence, one can say that a gradient-free constant coefficient subequation is a closed subset $\call F \subset \R \times \call S(n)$ which is $\call Q$-monotone where
	\begin{equation}\label{def_Q}
	\call Q\defeq  \call N \times \call P \subset \call \R \times \call S(n).
	\end{equation}
\end{remark}

The most basic example of a gradient-free constant coefficient subequation is the (reduced) monotonicity cone subequation $\call Q$ defined in \eqref{def_Q}. Its dual subequation is the so-called \emph{subaffine-plus subequation}\sid{subaffine!subequation!-plus|seeonly{subequation, subaffine, -plus}}\sid{subequation (constraint set)!subaffine!-plus}\sid{subaffine-plus|seeonly{subaffine, \dots, -plus}}
\begin{equation*}
\tildee{\call Q} = \{ (r,A) \in \R \times \call S(n): \ r \leq 0 \  \text{or} \ A \geq 0 \}.
\end{equation*}
Also in this case the name comes from a characterization of dual subharmonics: by \cite[Theorem 10.7]{chlp}, an upper semicontinuous function on $X$ is $\tildee{\call Q}$-subharmonic if and only if it is subaffine-plus on $X$; that is,
\[
\tildee{\call Q}(X) = \mathrm{SA}^+(X),
\] 
where $\mathrm{SA}^+(X)$ is the set of all subaffine-plus functions on $X$, which are defined as follows.\syid{SAXp@$\mathrm{SA}^+(X)$!the set of all subaffine-plus functions on $X$}

\begin{definition}
	For $K\subset \R^n$ compact, we denote by \syid{Affp@$\mathrm{Aff}^+(X)$!the set of all affine-plus functions on $X$}
	\[
	\mathrm{Aff}^+(K) \defeq \{ \restr{a}{K}:\ a \in \mathrm{Aff},\ a \geq 0\ \text{on $K$} \}
	\]
	the set of \emph{affine-plus functions on $K$}.
\end{definition}

\begin{definition} \sid{subaffine!function!-plus}
	A function $w\in \USC(X)$ is called \emph{subaffine-plus on $X$}, and we write $w \in \mathrm{SA}^+(X)$,  if \eqref{affcomp} holds for every open subset $\Omega\ssubset X$ and each $a \in \mathrm{Aff}^+(\barr\Omega)$.
	\end{definition}

\begin{remark}
As shown in \cite[Theorem~10.7]{chlp}, one has the following interesting characterization of the subaffine-plus functions on open subsets $X$ of $\R^n$:
	\[
	\mathrm{SA}^+(X) = \{ u \in \USC(X): \ u^+ \in \mathrm{SA}(X)\},
	\]
	so that the $\widetilde{\call Q}$-subharmonics on $X$ are those upper semicontinuous functions whose positive part is subaffine on $X$; note that this is equivalent to 
	\[
	\tildee{\call Q}(X) = \{ u \in \USC(X): \ u^+ \in \tildee{\call P}(X)\}.
	\]
	\end{remark}

In analogy with the pure second-order case, $\call Q$ is a monotonicity cone subequation for all gradient-free subequations (and it is minimal in this sense), hence it is now immediate to see that, by~\Cref{ccsa}, we have an analogue of the Subaffine Theorem.

\begin{thm}[{Subaffine-Plus Theorem; \cite[Theorem 10.8]{chlp}}] \sid{subaffine!theorem!-plus|seeonly{Theorem, Subaffine, -Plus}}\sid{Theorem!Subaffine!-Plus}
	If $\call F$ is a gradient-free constant coefficient subequation, then for any open set $X \subset \R^n$
	\[
	\call F(X) + \tildee{\call F}(X) \subset \mathrm{SA}^+(X).
	\]
\end{thm}

Also, since $0\in \mathrm{A}^+(\R^n)$, we have comparison. We omit the proof, which is essentially the same as in the case of pure second-order subequations. 

\begin{thm}[Comparison for gradient-free constant coefficient subequations]\label{thm:C_gf} \sid{comparison (principle)!for gradient-free constant coefficient subequations} 
	Let $\call F$ be a gradient-free constant coefficient subequation. Then comparison holds on every domain $\Omega \ssubset \R^n$:
	\[
	u+v \leq 0\ \text{on $\de \Omega$} \implies u+v \leq 0\ \text{on $ \Omega$}
	\]
	if $u \in \USC(\barr \Omega) \cap \call F(\Omega)$ and $v \in \USC(\barr \Omega) \cap \tildee{\call F}(\Omega)$.
\end{thm}

We conclude this section by noting that the results presented here are part of a wider program, which exploits a general method for determining the validity of the zero maximum principle for dual cone subharmonics.

\begin{remark}[A general comparison theorem for constant coefficient potential theories]\label{rem:CP_cc} \sid{comparison (principle)!general for constant coefficient potential theories}
	The comparison principles of \Cref{thm:C_pso,thm:C_gf} are also corollaries of the general comparison principle for constant coefficient potential theories contained in \cite[Theorem~7.5]{chlp}. This result states that comparison for a constant coefficient subequation $\call F$ on a domain $\Omega$ holds if $\call F$ is $\call M$-monotone for a monotonicity cone subequation $\call M$ for which there exists a $C^2$-strictly $\call M$-subharmonic function $\psi$ on $\Omega$, as described in \Cref{rem:SAT_CC} above. The point is that the existence of $\psi$ ensures the validity of \eqref{(ZMP)} for the $\widetilde{\call M}$-subharmonics on $\Omega$.
	
	With the notation of \cite{chlp}, \Cref{thm:C_pso,thm:C_gf} correspond to the monotonicity cone subequations
	\begin{equation}\label{elem_cones}
	\call M(\call P) = \R \times \R^n \times \call P \quad \text{and} \quad \call M(\call N, \call P) = \call M(\call N) \cap \call M(\call P) = \call N \times \R^n \times \call P,
	\end{equation}
	respectively. These two cones do admit a (quadratic) $C^2$-strictly $\call M$-subharmonic function $\psi$ on every bounded domain $\Omega$, as shown in  \cite[Theorem~6.6]{chlp}. Moreover, the two cones in \eqref{elem_cones} are two members of a {\em fundamental family} of monotonicity cone subequations which are generated by double and triple intersections of $\call M(\call P)$ above, $	\call M(\call N) \defeq  \call N \times \R^n \times \call S(n)$ and the three one-parameter cones\sid{subequation (constraint set)!monotonicity cone!fundamental family}
	\[
	\begin{gathered}
	\call M(\gamma) \defeq  \{ (r,p,A) \in \call J^2: \ r \leq - \gamma |p| \}, \quad \gamma \in (0, +\infty),
	\\
	\call M(R) \defeq  \left\{ (r,p,A) \in \call J^2: \  A \geq \frac{|p|}{R} I  \right\}, \quad R \in (0, +\infty),
	\\
	\call M(\call D) \defeq  \R \times \call D \times \call S(n), \quad \call D \subset \R^n \ \text{is a {\em directional cone}},
	\end{gathered}
	\]
	a directional cone being a closed convex cone with vertex at the origin and non-empty interior. Notice that $\call M(\call P)$ and $\call M(\call N)$ are the ``limits'' of $\call M(R)$ and $\call M(\gamma)$ for $R \nearrow + \infty$ and $\gamma \searrow 0$, respectively.
	
	For all of the cones in the fundamental family, the validity of the zero maximum principle for the $\widetilde{\call M}$-subharmonics on $\Omega$, and hence the validity of comparison for any $\call M$-monotone potential theory, is known from the detailed analysis in \cite{chlp}. There is a simple dichotomy: comparison holds on every bounded $\Omega$ except for the case when the cone has a generator $\call M(R)$ with $R$ finite, and in this case comparison holds on all domains $\Omega$ which are contained in a translate of the truncated cone $\call D \cap B_R(0)$.
	
	The family above is fundamental in the sense that any monotonicity cone subequation contains one of the fundamental cones, and hence for every constant coefficient $\call M$-monotone potential theory, the comparison principle holds (with possibly a restriction on the size of the domain $\Omega$). 
\end{remark}

As a final consideration, the analysis discussed in Remark~\ref{rem:CP_cc} has two additional extensions. First, by exploiting the {\em correspondence principle} of \cite[Theorem~11.13]{chlp}\sid{correspondence principle} one deduces the validity of the comparison principle for every constant coefficient and $\call M$-monotone differential operator $F$ which is {\em compatible} with an $\call M$-monotone constant coefficient subequation $\call F$. Many applications of this are discussed in \cite[Part~IV]{chlp}. Second, in the variable coefficient case, if $\call F$ is also {\em fiberegular}, then the analysis above extends to the potential theory defined by $\call F$ as well as to associated PDE for any operator $F$ which is compatible with $\call F$. This program was carried out in \cite{cprdir}, which also gives characterizations of $\widetilde{\call M}$-subharmonics for the other monotonicty cones $\call M$ in the fundamental family, extending the characterizations above for the cones $\call M(\call P)$ and $\call M(\call N, \call P)$; see \cite[Section~6]{cprdir}.

The following section treats some complementary results to \cite{chlp} and \cite{cprdir}. 

\section{Strict comparison with minimal monotonicity and semiconvex functions}\label{sec:SC}

In this section, we present some comparison principles which are not included in the aforementioned \cite{chlp} and \cite{cprdir}. These results give additional situations in which one has sufficient monotonicity in order to prove a form of comparison. In fact, the minimal monotonicity present in every subequation becomes sufficient if one weakens comparison to a notion of {\em strict comparison}. In the variable coefficient setting, this form of comparison is shown to hold for semiconvex functions, while, in the constant coefficient setting, it is shown to hold for functions which are merely semicontinuous. Two approaches for the semicontinuous version will be presented, both of which use semiconvex approximation; one in an implicit sense and one in an explicit sense, which will be compared in Remark~\ref{rem:SC}.

We begin by recalling that, by properties \eqref{(P)} and \eqref{(N)}, every (variable coefficient) subequation $\call F$ has a constant coefficient monotonicity cone; namely the \emph{minimal monotonicity cone} \sid{monotonicity cone!subequation|seeonly{subequation, monotonicity cone}}\sid{monotonicity cone!minimal}
$$
\call M_0 =\call N \times \{0\} \times \call P,
$$
which, \emph{a priori}, might be the maximal monotonicity cone for $\call F$. As noted previously, $\call M_0$ is not a subequation because it has empty interior and hence does not satisfy the topological property \eqref{(T)}. In such a situation, the methods described in the previous section do not apply. However, one can ask under what situations $\call M_0$-monotonicity for $\call F$ still yields comparison in the potential theory determined by $\call F$. One way is to focus on \emph{semiconvex} subharmonics and to add a \emph{strictness} assumption; to that end, we give the following notion of \emph{strong strictness}.

\begin{definition} \label{def:strictv}
	We say that $u \in \USC(X)$ is \emph{strongly strictly \Fd subharmonic on $X$} if $u \in \call G(X)$ for some subequation $\call G \subset \intr{\call F}$. We will write $u \in \call F_{\rm strict}(X)$.\sid{harmonic (\emph{or} \Fd harmonic)!sub- (\emph{or} \Fd sub-)!strongly strictly}\syid{FstrictX@$\call F_{\rm strict}(X)$!the set of all strongly strict \Fd subharmonics on $X$}
\end{definition}

Note that this notion is, in general, stronger than the one of \Cref{def:strict}, even for classical subsolutions.

\begin{example}
	For $n=1$ and $X=(0,1)$, consider $F(x,p) = x-p$. Then $u(x) = \frac{x^3}3$ is a strict smooth subsolution on $X$, according to \Cref{def:strict}, but it is not strongly strict in the sense of \Cref{def:strictv}. 
\end{example}

Harvey and Lawson~\cite{hlae} showed that one always has the form \eqref{eq:ZMP} of comparison if the subsolutions are semiconvex and one of them is strongly strict, just like the Definitional Comparison \Cref{defcompa} tells that \eqref{eq:ZMP} holds, even without the semiconvexity assumption, if one of the subsolutions is strict and smooth.

\begin{thm}[Strict comparison for semiconvex functions] \label{scqc} \sid{comparison (principle)!strict!for semiconvex functions}
	Let $\Omega \subset \R^n$ be a bounded domain and let $\call F$ be a subequation on $\Omega$.  Then strict comparison holds for semiconvex functions; that is,
	\[
	\begin{dcases}
	u \in \USC(\barr\Omega) \cap \call F_{\rm strict}(\Omega) \cap  \mathsf{sc}_\mathrm{loc}(\Omega)
	\\
	v \in \USC(\barr\Omega) \cap \tildee{\call F}(\Omega) \cap  \mathsf{sc}_\mathrm{loc}(\Omega)
	\\
	u+v \leq 0\ \text{on $\de\Omega$}
	\end{dcases} 
	\quad
	\implies
	\quad
	u+v \leq 0\ \text{on $\Omega$}.
	\]
\end{thm}

\begin{proof}
	Suppose, by contradiction, that strict comparison fails on $\Omega \ssubset \R^n$. Then $w= u+v$ must have a positive interior maximum $m\defeq w(x_0) > 0$ at $x_0 \in \Omega$. Hence $w$ has $(0,0)$ as an upper contact jet at $x_0$ and Theorem \ref{pusc:sum} on summands yields
	\[
	(u(x_0), p, A) \in \call G_{x_0}, \qquad (v(x_0), q, B) \in \tildee{\call F}_{x_0}, 
	\]
	with $p+q = 0$ and $A+B = -P \leq 0$. This contradicts $\call G_{x_0} \subset \intr{\call F}_{x_0}$, since by positivity and negativity
	\[
	-(v(x_0), q, B) = (u(x_0)-m, p, A+P)  \in \call G_{x_0},
	\]
	but $-(v(x_0), q, B) \notin \intr{\call F}_{x_0}$ by the definition of the dual subequation.
\end{proof}

\begin{example}
	\Cref{scqc} is indeed useful if we can assure \emph{a priori} that the subharmonics and dual subharmonics of a given subequation are semiconvex. For example, consider the subequation associated to the inhomogeneous Monge--Amp\`{e}re equation
	\begin{equation}\label{maex}
	\det D^2 u = f
	\end{equation}
	for some $f\in C(\Omega)$ which is positive and bounded from above. If we look for convex solutions $u$, we immediately see that they are strongly convex, in the sense that $\lambda_1({D^2u}) \geq C$ on $\Omega$, for some $C>0$; indeed by semicontinuity and compactness we know that $f \geq c$ on $\Omega$, for some $c > 0$, while $\det D^2u \leq M$ for some upper bound $M>0$ on $f$. Therefore, if we restrict ourselves to strongly convex sub-- and supersolutions, we trivially have that subsolutions are semiconvex. Also, since the fibers of the subequation associated to (\ref{maex}) are
	\[
	\call F_x = \R \times \R^n \times \big\{ A \in \Sc(n):\ \det A - f(x) \geq 0 \big\}, 
	\]
	we have that
	\[
	\tildee{\call F}_x = \R \times \R^n \times \big\{ B \in \Sc(n):\ \det (-B) - f(x) \leq 0 \big\}
	\]
	and the strong convexity assumption on the supersolutions leads us to require that the matrices $B$ are negative definite with their maximum eigenvalue $\lambda_n (B ) \leq \ell$ for some $\ell < 0$, independent of $x$. Therefore, we see that
	\[
	-\lambda_1(B)\cdot(-\ell)^{n-1} \leq \det(-B) \leq M \quad\implies\quad  \lambda_1(B) \geq -\Lambda \defeq \frac{-M}{(-\ell)^{n-1}},
	\]
	hence the dual subsolutions we are interested in are $\Lambda$-semiconvex.
\end{example}

The assumption that $u$ and $v$ are semiconvex can be dropped if $\call F$ and $\call G$ are constant coefficient subequations. The following result is \cite[Corollary~C.3]{hldir}. 
\begin{thm}[Strict comparison for constant coefficient subequations] \label{sccc} \sid{comparison (principle)!strict!for constant coefficient subequations}
	Let $\call F$ be a constant coefficient subequation.  Then strict comparison holds on every bounded domain $\Omega \subset \R^n$; that is,
	\[
	\begin{dcases}
	u \in \USC(\barr\Omega) \cap \call F_{\rm strict}(\Omega) 
	\\
	v \in \USC(\barr\Omega) \cap \tildee{\call F}(\Omega)
	\\
	u+v \leq 0\ \text{on $\de\Omega$}
	\end{dcases} 
	\quad
	\implies
	\quad
	u+v \leq 0\ \text{on $\Omega$}.
	\]
\end{thm}

We will omit the short proof, which one finds in \cite{hldir}, noting only that it follows the same argument of its semiconvex counterpart (\Cref{scqc}), where the Theorem in Summands~\ref{pusc:sum} for locally semiconvex functions is replaced by a potential-theoretic version of the Theorem on Sums~\ref{tosu} for semicontinuous functions.  Even though we will sketch an alternate proof of Theorem \ref{sccc}, for the benefit of the reader and in order to comment on the two approaches, we present Harvey and Lawson's potential-theoretic version of the Theorem on Sums \cite[Theorem~C.1]{hldir}. 

\begin{thm}[On Sums; simplified potential theoretic version] \label{tospt} \sid{Theorem!on Sums!simplified potential theoretic version}
	Let $\call F, \call G \subset \call J^2(X)$ be two subequations and let $u \in \call F(X)$, $v \in \call G(X)$. Suppose that $u+v$ does not satisfies the Zero Maximum Principle on $\Omega \ssubset X$. Then there exist a point $x_0 \in \Omega$ where $u(x_0) + v(x_0) >0$, a net $\{(x_\epsilon, y_\epsilon)\}_{\epsilon > 0} \subset \Omega \times \Omega$ converging to $(x_0, x_0)$, and matrices $A_\epsilon, B_\epsilon \in \call S(n)$ such that
	\begin{gather} \label{tosptinc}
	\Big(u(x_\epsilon), \frac{x_\epsilon-y_\epsilon}\epsilon, A_\epsilon\Big) \in \call F_{x_\epsilon}, \quad \Big(v(y_\epsilon), -\frac{x_\epsilon-y_\epsilon}\epsilon, B_\epsilon\Big) \in \call G_{y_\epsilon},
	\\[3pt] \notag
	u(x_\epsilon) + v(y_\epsilon) \dto u(x_0) + v(x_0),
	\\[3pt] \notag
	|x_\epsilon - y_\epsilon| = o(\sqrt\epsilon\,),
	\\[3pt] \label{tosptineq}
	-\frac3\epsilon \left(\! \begin{array}{cc}
	I &  0 \\
	0  & I \\
	\end{array}\! \right) \leq \left(\! \begin{array}{cc}
	A_\epsilon &  0 \\
	0  & B_\epsilon \\
	\end{array}\! \right) \leq \frac3\epsilon
	\left(\! \begin{array}{cc}
	I &  -I \\
	-I  & I \\
	\end{array}\! \right).
	\end{gather} 
\end{thm}

\begin{proof}
	Let $w(x,y) \defeq u(x) + v(y)$. For any $\epsilon > 0$, let $(x_\epsilon, y_\epsilon)$ be a maximum of $w(x,y) - Q_{\epsilon^{-1} I}(x-y)$ on $\barr\Omega \times \barr\Omega$;  by \cite[Proposition~3.7]{user}, by taking $\epsilon$ small enough, we may suppose that $(x_\epsilon, y_\epsilon) \in \Omega \times \Omega$. Then by \Cref{tosu} there exist $A_\epsilon, B_\epsilon$ such that \eqref{tosptinc} holds (since $u \in \call F(X)$, $v \in \call G(X)$ and the fibers of subequations are closed) and
	\[
	-\frac1\epsilon \bigl( 1 + \Vert J\Vert \bigr) \,I \leq \left(\! \begin{array}{cc}
	A_\epsilon &  0 \\
	0  & B_\epsilon \\
	\end{array}\! \right)
	\leq \frac1\epsilon \bigl( J + J^2 \bigr), 
	\qquad J \defeq \left(\! \begin{array}{cc}
	I &  -I \\
	-I  & I \\
	\end{array}\! \right),
	\]
	which is easily seen to be \eqref{tosptineq}. Finally, all the remaining properties to prove easily follow again from \cite[Proposition~3.7]{user}.
\end{proof}

We conclude with some reflections on the role of semiconvex approximation in strict comparison under minimal monotonicity.

\begin{remark}[Strict comparison and semiconvex approximation]\label{rem:SC} \sid{semiconvex!approximation}
	The proof of Theorem \ref{sccc} above is an ``almost conventional'' viscosity theory proof in that it exploits a potential-theoretic formulation of the conventional Theorem on Sums, based on the doubling of variables and penalization. Semiconvex approximation, by way of the sup-convolution, appears implicitly in the proof, since it plays a crucial role in the proof of the Theorem on Sums. 
	
	One might ask if an explicit use of the semiconvex approximation technique can reduce  Theorem \ref{sccc} to the semiconvex version of Theorem \ref{scqc}. If so, one would be extending the semiconvex approximation technique from sufficient monotonicity $\call M$ (as discussed in the previous section) to minimal monotonicity $\call M_0$, at least for strict comparison. We note that one encounters a possible obstruction to this intent; namely, one would like to know that the semiconvex approximations preserve subharmonicity for the subequations $\call G$ and $\widetilde{\call F}$, and \Cref{prop:approx} says that this is true, provided that such subharmonics are bounded.
	Boundedness from above on compact sets is ensured by the upper semicontinuity of the subharmonics; the problem is that they may not be bounded from below.
	
	The decreasing limits property of \Fd subharmonics in \Cref{elemprop}(v) suggests that one might first approximate an \Fd subharmonic $u$ with a deceasing sequence $\{u_m\}_{m \in \N}$ of functions which are bounded from below and \Fd subharmonic, and a simple truncation from below with $u_m \defeq  \max \{ u,- m\}$ works if each constant $m \in \N$ is an \Fd subharmonic by the maximum property of \Cref{elemprop}(ii); more generally, as seen in the proof of \Cref{ccsa}, one can take $u_m \defeq  \max \{ u, \varphi - m\}$ for any fixed \Fd subharmonic which is bounded from below. Unfortunately, the existence of such a bounded subharmonic $\varphi$ is, in general, only guaranteed locally, as discussed in Remark \ref{exquadsub}.
	
	Therefore, in a proof of strict comparison with minimal monotonicity $\call M_0$ by an explicit semiconvex approximation, one should first show that one can assume without loss of generality that the subharmonics are bounded (for example, by proving the existence of  a bounded global subharmonic for $\call G$ and $\widetilde{\call F}$) or find a way to carry the global information of the behavior of $u+v$ on $\de\Omega$ into a neighborhood of some given point (for instance, a point in $\Omega$ at which $u+v$ assumes its positive maximum, if one is proceeding by contradiction). If instead, one tries to work only locally as in the proof of \Cref{scqc}, then the most immediate path probably leads to a proof with the flavour of the Theorem on Sums, since it is likely that one eventually ends up needing results like \Cref{pusc:sum} and \Cref{magprop}.
\end{remark}

To complete the discussion, we give a proof of strict comparison for constant coefficient potential theories with minimal monotonicity $\call M_0$, provided that the following additional hypothesis is made:
\begin{equation}\label{sc_hyp}
\mbox{there exist $\Phi \in \call G(\Omega)$ and $\Psi \in \tildee{\call F}(\Omega)$ which are bounded (from below) on $\barr\Omega$.}
\end{equation}

\begin{proof}[Proof of \Cref{sccc} assuming \eqref{sc_hyp}] 
	By contradiction, suppose that strict comparison fails for some $\Omega \ssubset \R^n$; hence there exists $\barr x \in \Omega$ such that $u(\barr x) + v(\barr x) \defeq  M > 0$. For each $\eta \in (0, M)$ fixed, by the sliding property (\Cref{elemprop}(iv)), $\hat u \defeq u-\eta \in \call G(\Omega)$ and one still has $\hat u(\barr x) + v(\barr x) > 0$. 
	
	Now, with $\Phi$ and $\Psi$ satisfying \eqref{sc_hyp}, define for each for $m \in \N$
	\[
	u_m \defeq \max\{ \hat u, \Phi - m\} \quad \text{and} \quad v_m \defeq \max\{ v, \Psi - m\}.
	\] 
	Note that $u_m$ and $v_m$ are bounded on $\barr\Omega$, $u_m \dto u$ and $v_m \dto v$ pointwise on $\barr\Omega$ as $m \to +\infty$, and $u_m(\barr x) + v_m(\barr x) \geq \hat u(\barr x) + v( \barr x) > 0$ for all $m \in \N$. In addition, one has $u_m \in \call G(\Omega)$ and $v_m \in \tildee{\call F}(\Omega)$, by the maximum property (\Cref{elemprop}(ii)).
	Semiconvex approximation (\Cref{prop:approx}) yields the sup-convolutions $u_m^\epsilon \in \call G(\Omega_\delta)$ and $v_m^\epsilon \in \tildee{\call F}(\Omega_\delta)$,  where 
	\[
	\delta = \delta(\epsilon) \defeq 2\sqrt{\epsilon \max\Bigl\{\max_{\barr{\Omega}}|u_m|,\, \max_{\barr{\Omega}}|v_m|\Bigr\}}\,.
	\]
	These semiconvex approximations satisfy $u_m^\epsilon(\barr x) + v_m^\epsilon(\barr x) \geq u_m(\barr x) + v_m(\barr x) > 0$ for each small $\epsilon > 0$.
	Therefore, by the strict comparison for semiconvex functions (\Cref{scqc}), for any $\epsilon >0$ so small that $\barr x \in \Omega_\delta$, there exists a point $y_\epsilon$ such that 
	\begin{equation} \label{disban}
	u_m^\epsilon(y_\epsilon) + v_m^\epsilon(y_\epsilon) > 0,\quad y_\epsilon\in \de\Omega_\delta.
	\end{equation}
	Consider now a vanishing sequence $\epsilon_k \dto 0$. By the compactness of $\barr{\Omega}$, up to a subsequence, $y_k \defeq y_{\epsilon_k} \to \barr y \in \de \Omega$ as $k\to\infty$. Set $\delta_k \defeq  \delta(\epsilon_k)$ and note that, by (\ref{disban}), one has
	\[
	u_m^{\epsilon_k}(y_k) + v_m^{\epsilon_k}(y_k) > 0, \quad y_k \in \de\Omega_{\delta_k} \qquad \forall\, k\in \mathbb{N}.
	\]
	Also, by (\ref{supcball}), one has
	\begin{equation}\label{sc_bound}
	u_m^{\epsilon_k}(y_k) \leq \max_{\barr B_{\delta_k}(y_k)} u_m \qquad  \forall k\in \mathbb{N}.
	\end{equation}
	Now, since $\delta_k \dto 0$ and $y_k \to \barr y$, for each  $\rho > 0$ fixed one has  $B_{\delta_k}(y_k) \subset B_\rho(\barr y)$ if $k$ is large enough.  Hence, taking the limsup in \eqref{sc_bound}, one obtains
	\[
	\limsup_{k\to\infty} u_m^{\epsilon_k}(y_k) \leq \max_{\barr B_\rho( \barr y)} u_m \qquad \forall \rho > 0,
	\]
	and then taking the limit $\rho \dto 0$ yields
	\[
	\limsup_{k\to\infty} u_m^{\epsilon_k}(y_k) \leq \limsup_{z \to  \barr y} u_m(z) \leq u_m( \barr y),
	\]
	where the last inequality comes from the upper semicontinuity of $u_m$.
	The same argument applied to $v_m$ gives that $u_m( \barr y) + v_m( \barr y) \geq 0$
	and, letting $m \to +\infty$, one has $\hat u( \barr y) + v( \barr y) \geq 0$;
	that is,
	\[
	u( \barr y) + v( \barr y) \geq \eta > 0, \quad  \barr y \in \de \Omega,
	\]
	which gives the desired contradiction.
\end{proof}

\chapter{From Euclidean spaces to manifolds: a brief note}\label{sec:mflds}

In this chapter we show how the notion of semiconvexity can be transported from Euclidian spaces to manifolds. The key point will be to show that the notion of semiconvexity is preserved under smooth coordinate changes.

 \section{Objectives}
 
 In our presentation, for simplicity, we have worked only in Euclidean spaces. However, many of the results discussed here extend to manifolds; for example, the Almost Everywhere Theorem~\ref{aet}, the Subharmonic Addition Theorem~\ref{add} and the Strict Comparison Theorem~\ref{scqc} have been given in \cite{hlae} in a more general form on manifolds. We note that, roughly speaking, \emph{all local results naturally extend from Euclidean spaces to manifolds}, via a suitable use of local coordinates. 
 
 To make this passage, one needs to extend the definition of a locally semiconvex function on an open subset of $\R^n$ (see Definition \ref{def:qc}) to functions defined on a manifold $M$ of dimension $n$. The crucial observation for this extension is to prove that the class of locally semiconvex functions on an open subsets of $\R^n$ is \emph{stable under pull-backs via diffeomorphisms} (that is, smooth changes of coordinates).
 
 We will prove this claim and then formulate a definition of a locally convex function on a manifold.

 \section{Invariance of local semiconvexity under coorinate changes}
 
 The discussion above is formalized in the following statement.
 
 \begin{thm} \label{BM_lqcstable}
 	Suppose that $\Phi: X \to Y$ is a smooth diffeomorphism between open subsets $X,Y \subset \R^n$. Then $\Phi$ pulls back locally semiconvex functions on $Y$ to locally semiconvex functions on $X$; that is,
 	\begin{equation*}
 	\mbox{$u: Y \to \R$ is locally semiconvex \quad $\implies \quad u \circ \Phi: X \to \R$ is locally semiconvex.}
 	\end{equation*}
 \end{thm}
 
We are going to reduce Theorem \ref{BM_lqcstable} to Lemma \ref{lem:smooth_reduction} below. This reduction involves several simplifications, where the final step will be a potential-theoretic argument which exploits the decreasing limit property 
 	and smooth approximations (mollifications) that preserve convexity and Lipschitz constants. The properties of the mollifications that we will need are contained in the following result which augments Lemma \ref{convqc}. We make use of an even mollifier; that is,\sid{mollifier}
 	\begin{equation*}
 	\mbox{ $\eta \in C^{\infty}(\R^n)$, \quad $\eta$ non-negative, \quad ${\rm supp}(\eta) \subset B_1(0)$, \quad $\int_{\R^n} \eta = 1$},
 	\end{equation*}
 	and
 	\begin{equation}\label{mollifier2}
 	\eta(-z) = \eta(z) \quad \forall z \in \R^n. 
 	\end{equation}
 	
 	\begin{lem}\label{lem:reg_convex} Suppose that $u$ is convex and $K$-Lipschitz on a convex neighborhood $\cN(\overline{Y})$  of $\overline{Y}$ convex. For each $\varepsilon > 0$, consider the mollified function \sid{approximation!smooth}\sid{mollification|seeonly{approximation, smooth}}\syid{uepsb@$u_\epsilon$ (\emph{see also} $\eta_\epsilon$)!see Lemma~\ref{convqc} or~(\ref{mollify})}
 		\begin{equation}\label{mollify}
 		u_{\varepsilon}(y)\defeq  (u \ast \eta_{\varepsilon})(y) = \int_{\R^n} u(y-z) \eta_{\varepsilon}(z) \, \di z = \int_{\R^n} u(y-z) \eta(z/\varepsilon) \, \varepsilon^{-n} \,\di z,
 		\end{equation}
 		for $y \in Y^{\varepsilon} \defeq  \{ y \in \cN(\overline{Y}): \ {\rm dist}(y, \partial \cN(\overline{Y})) > \varepsilon\}$ is convex and contains $\overline{Y}$ for all sufficiently small $\varepsilon$. The following hold: 
 		\begin{itemize}
 			\item[(a)] $u_{\varepsilon}$ is smooth and convex on $Y^{\varepsilon}$;
 			\item[(b)] $u_{\varepsilon}$ is $K$-Lipschitz on $Y^{\varepsilon}$;
 			\item[(c)] $u_{\varepsilon} \searrow u$ on compact subsets of $\cN(\overline{Y})$ as $\varepsilon \searrow 0$ 
 		\end{itemize}
 	\end{lem}
	
		\begin{proof}
		 By the condition on the support of $\eta$, $u_{\varepsilon}$ defined by \eqref{mollify} can also be expressed as
 		\begin{equation}\label{mollif2}
 		u_{\varepsilon}(y)\defeq  \int_{B_{\varepsilon}(0)} u(y-z) \eta_{\varepsilon}(z) \, \di z
 		\end{equation}
 		and hence is clearly well-defined on $Y^{\varepsilon}$, which contains $\overline{Y}$ for all small $\varepsilon$. The smoothness of $u_{\varepsilon}$ on $Y^{\varepsilon}$ is a standard consequence of the dominated convergence theorem for $u$ continuous. The convexity of $u_{\varepsilon}$ follows from
 		by the same argument used for $f$ convex in Lemma~\ref{convqc} by replacing $\R^n$ with $Y^{\varepsilon}$, which completes part (a).
 		
 		For property (b), note that for each $y_1, y_2$ in $Y^{\varepsilon}$ and each $\varepsilon$ small we have
 		\[
 		|u_{\varepsilon}(y_1) - u_{\varepsilon}(y_2)| \leq \int_{\R^n} K|y_1 - y_2| \, \eta_{\varepsilon}(z) \, \di z = K|y_1 - y_2|,
 		\]
 		where we have used that $u$ is $K$-Lipschitz and $\eta$ is non-negative with integral equal to $1$.
 		
 		It remains only to show the decreasing limit property of part (c), which is a consequence the convexity of $u$ and the properties of the mollifier $\eta$. First note that, since $u$ is convex and $\eta$ is even (that is, \eqref{mollifier2} holds),
 		\[ \begin{split}
 		u(y) = \int_{\R^n} u(y) \eta_{\varepsilon}(z) \, \di z & \leq \frac{1}{2} \int_{\R^n} u(y -z) \eta_{\varepsilon}(z) \, \di z +  \frac{1}{2} \int_{\R^n} u(y + z) \eta_{\varepsilon}(z) \, \di z \\
 		& \leq \frac{1}{2} \int_{\R^n} u(y -z) \eta_{\varepsilon}(z) \, \di z +  \frac{1}{2} \int_{\R^n} u(y - z) \eta_{\varepsilon}(-z) \, \di z = u_{\varepsilon}(y),
 		\end{split}
 		\]
 		and hence $u \leq u_{\varepsilon}$ on $Y^{\varepsilon}$. Given the pointwise (locally uniform) convergence of $u_\epsilon$ to $u$, it remains only to show the monotonicity
 		\begin{equation}\label{eps_delta}
 		0 < \delta < \varepsilon \quad \implies \quad u_{\delta}(y) \leq u_{\varepsilon}(y) \quad \forall \, y \in Y^{\varepsilon} \subset Y^{\delta}.  
 		\end{equation}
 		The claim \eqref{eps_delta} is a simple consequence of the following {\em four-point inequality} for the convex function $u$: if $0 < \delta < \varepsilon$ then
 		\begin{equation}\label{4point}
 		u(y - \delta z) + 	u(y + \delta z) \leq 	u(y - \varepsilon z) + 	u(y + \varepsilon z),
 		\end{equation}
 		which holds for each $y \in Y^{\varepsilon}$ if $\varepsilon$ is sufficiently small. To prove \eqref{4point}, it is enough to note that with 
 		\[
 		t= \frac{\varepsilon + \delta}{2 \varepsilon} \in (0,1) \quad \text{and} \quad 1- t = \frac{\varepsilon - \delta}{2 \varepsilon} \in (0,1)
 		\]
 		one has
 		\[
 		y - \delta z = t(y - \varepsilon z) + (1-t)(y + \varepsilon z) \quad \text{and} \quad 	y + \delta z = t(y + \varepsilon z) + (1-t)(y - \varepsilon z).
 		\]
 		For the decreasing property \eqref{eps_delta}, starting from \eqref{mollif2}, changing variables and using that $\eta$ is even, one finds that
 		\[ \begin{split}
 		u_{\varepsilon}(y) & = \int_{B_{\varepsilon}(0)} u(y - z) \eta\!\left(\frac{z}{\varepsilon}\right) \varepsilon^{-n} \, \di z = \int_{B_{1}(0)} u(y - \varepsilon z)  \eta(z)  \, \di z \\
 		& =  \int_{B_{1}(0) \cap \{z_1 > 0\}} \bigl(u(y - \varepsilon z) + u(y + \varepsilon z) \bigr) \,  \eta(z)  \, \di z;
 		\end{split}
 		\]
 		using this formula also for $u_{\delta}$, the four-point inequality \eqref{4point} gives \eqref{eps_delta}.
 	\end{proof}
 
 	We proceed now with the simplifications. First, note that the result of \Cref{BM_lqcstable} is local. If one were to use directly the Definition~\ref{def:qc}, the proof would entail showing that for each pair of points $(x,y) \in X \times Y$ with $y = \Phi(x)$ and $u$ being $\lambda$-semiconvex on some ball $B_{\rho}(y) \Subset Y$ for some $\lambda > 0$, one has $v \defeq u \circ \Phi$ is $\Lambda$-semiconvex on some ball $B_{\sigma}(x) \subset \Phi^{-1}(B_{\rho}(y)) \Subset X$ for some $\Lambda > 0$. Instead, we exploit this local nature to make the first two simplifying assumptions in the statement of the theorem. By replacing $Y$ with some smaller ball than $B_{\rho}(y)$, we can assume that
 	\begin{equation}\label{simple1}
 	\mbox{$\Phi: \cN(\overline{X})  \to \cN(\overline{Y}) $ is a diffeomorphism},
 	\end{equation}
 	where  $\cN(\overline{X})$ and $\cN(\overline{X})$ are open neighborhoods of $\overline{X}$ and $\overline{Y}$ both compact and
 	\begin{equation}\label{simple2}
 	\mbox{$u$ is $\lambda$-semiconvex on $\cN(\overline{Y})$ convex.}
 	\end{equation}
 	With these reductions, it is enough to prove that
 	\begin{equation}\label{reduction}
 	\mbox{$v = u \circ \Phi$ is $\Lambda$-semiconvex on $X$, for some $\Lambda > 0$.}
 	\end{equation}
	
 	We can further simplify by assuming that
 	\begin{equation}\label{simple3}
 	\mbox{$u$ is convex on $\cN(\overline{Y})$ convex.}
 	\end{equation}
 	To see this, replace $u$ by $u_{\lambda}(y)\defeq  u(y) + \frac{\lambda}{2} |y|^2$, which is convex on $\overline{Y}$ by \eqref{simple2}. Since the pullback function
 	$$
 	v_{\lambda}(x)\defeq  u_{\lambda}(\Phi(x)) = u(\Phi(x)) + \frac{\lambda}{2} |\Phi(x)|^2 = v(x) + \frac{\lambda}{2} |\Phi(x)|^2,
 	$$
 	differs from $v$ by the smooth function $\varphi(x)\defeq  \frac{\lambda}{2} |\Phi(x)|^2$, we can make use of the standard fact that if $\varphi \in C^2(X)$
 	\begin{equation*}
 	\mbox{$v$ is locally-semiconvex on $X \ \ \Leftrightarrow \ \ v + \varphi$ is locally-semiconvex on $X$,}
 	\end{equation*}
 	which follows from the local boundedness of the Hessian of $\varphi \in C^2(X)$.
 	
 	Next, since $u$ is convex on the convex neighborhood $\cN(\overline{Y})$, by replacing $\cN(\overline{Y})$ with a smaller convex neighborhood of $\overline{Y}$ compact, we can assume that
 	\begin{equation}\label{simple5}
 	\mbox{$u$ is $K$-Lipschitz on $\cN(\overline{Y})$, for some $K > 0$},
 	\end{equation}
 	which follows from Theorem \ref{thm:convex_Lip}(b).
 	
 	With these reductions, it suffices to show \eqref{reduction}, assuming \eqref{simple1}, \eqref{simple3} and \eqref{simple5}; that is, assuming that $\Phi: \cN(\overline{X}) \to \cN(\overline{Y})$ is a diffeomorphism and
 	\begin{equation}\label{simple6}
 	\mbox{$u$ is convex and $K$-Lipschitz on $\cN(\overline{Y})$ convex}.
 	\end{equation}

 	We are now ready for the last reduction, in which we can also assume that
	\begin{equation} \label{simple7}
	\mbox{$u$ is smooth in a neighborhood of $\overline{Y}$}.
	\end{equation}
	To do that, consider the potential theory determined by the following subequation:
 	\begin{equation*}
 	\call F = \{(r,p,A) \in \R \times \R^n \times \call S(n): \ |p| \leq K, \ A \geq 0\}.
 	\end{equation*}
 	Any function $u$ satisfying \eqref{simple6} will be \Fd subharmonic $\cN(\overline{Y})$ in the viscosity sense of Definition \ref{defn:FSH}. Lemma \ref{lem:reg_convex} shows that the regularizations $u_{\varepsilon}$ are smooth \Fd subharmonics on $\cN(\overline{Y})$ for each $\varepsilon$ sufficiently small and decrease to $u$ as $\varepsilon \searrow 0$. Hence, for this constant coefficient convex subequation, any $u$ which is \Fd subharmonic on a neighborhood of $\overline{Y}$ is the decreasing limit of smooth \Fd subharmonics. 
 	Now, the pullback approximations $v_{\varepsilon} \defeq  u_{\varepsilon}\circ \Phi = (u \ast \eta_{\varepsilon}) \circ \Phi$ will then decrease pointwise to $v = u \circ \Phi$ as $\varepsilon \searrow 0$. Hence, in order to show \eqref{reduction}, it is enough to show that each pullback regularization $v_{\varepsilon}$ is $\Lambda$-semiconvex on $X$ for some $\Lambda > 0$. This is because decreasing limits $v_{\varepsilon} \searrow v$ of $\Lambda$-semiconvex functions are $\Lambda$-semiconvex (on connected components where $v$ is not identically $-\infty$, which happens only where $u$ has this property). This completes the justification of the last reduction \eqref{simple7}.
	
	Considering the simplifications provided by \eqref{simple6} and \eqref{simple7}, proving \Cref{BM_lqcstable} now amounts to showing that the following result holds. 
 	
 	\begin{lem}\label{lem:smooth_reduction} Suppose that $\Phi : \cN(\overline{X}) \to \cN(\overline{Y})$ is a diffeomorphism between open neighborhoods of $\overline{X}$ and $\overline{Y}$, where $X$ and $Y$ are open and relatively compact in $\R^n$. Assume \eqref{simple6} and \eqref{simple7}; that is,
 		\begin{equation*}
 		\mbox{$u$ is smooth, convex and $K$-Lipschitz on $\cN(\overline{Y})$ convex.}
 		\end{equation*} 
 		Define the constant\footnote{Note that
		\[
		\cN(\overline{X}) \ni x \mapsto \pair{D^2 \Phi(x)e}{e} = \sum_{i,j=1}^n \frac{\de^2\Phi(x)}{\de x_{i}\de x_j}\,e_ie_j \in \R^n,
		\]
		being $\Phi = (\Phi_1,\dots,\Phi_n)$.}
 		\begin{equation}\label{Hessian_Phi}
 		C\defeq  \sup \bigl\{ \bigl| \langle D^2 \Phi(x) e, e \rangle \bigr|: \ x \in \overline{X}, \ |e| = 1 \bigr\}.
 		\end{equation}
 		Then $v\defeq  u \circ \Phi$ is $\Lambda$-semiconvex on $X$ for $\Lambda = KC$.
 	\end{lem}
 	
 	\begin{proof}
 		Since $v$ is smooth, it is enough to show that $D^2v(x) + \Lambda I \geq 0$ on $X$, which holds if we can show that
 		\begin{equation*}
 		\pair{D^2v(x) e}{e} \geq -KC \quad \forall  x \in \overline{X},\ \forall  e \in \mathbb S^{n-1}.
 		\end{equation*}	
 		By applying twice the chain rule to $v = u \circ \Phi$ one finds
 		\[
 		\begin{split}
 		\pair{D^2v(x) e}{e} & = \sum_{i,j = 1}^n \frac{\partial^2 v}{\partial x_i\partial x_j}(x)e_i e_j \\ 
 		& = \pair{(D^2 u(\Phi(x))(D\Phi(x) e)}{D\Phi(x) e} + \sum_{k = 1}^n  \frac{\partial u}{\partial y_k}(\Phi(x)) \pair{D^2\Phi_k(x) e}{e}.
 		\end{split}
 		\]
 		Since $u$ is convex, the first term is non-negative. It remains to show that the second term is bounded below by $-KC$. This follows from the Cauchy--Schwartz inequality, since $u$ is $K$-Lipschitz and recalling the definition of $C$ (cf.~\eqref{Hessian_Phi}).
 	\end{proof}
	
	Hence, Theorem \ref{BM_lqcstable} is also proven. As a consequence of it, one can give the following definition.
	\begin{definition}
 	Let $M$ be an $n$-dimensional differentiable manifold, and let $u \colon M \to \R$. We say that $u$ is locally semiconvex on $M$ if $u \circ \phi^{-1}$ is locally semiconvex on $\phi(U) \subset \R^n$ for any chart $(U,\phi)$ of $M$.
	 \end{definition}
 Notice that \Cref{BM_lqcstable} guarantees that such a definition is chart-independent; indeed, given two charts $(U,\phi)$ and $(V,\psi)$ with $U \cap V \neq \emptyset$, then $u \circ \phi^{-1}|_{\phi(U \cap V)}$ is locally semiconvex if and only if $u \circ \psi^{-1}|_{\psi(U\cap V)}$ is locally semiconvex, since $\Phi \defeq \phi \circ \psi^{-1}|_{\psi(U\cap V)} \colon \psi(U \cap V) \to \phi(U \cap V)$  is a diffeomorphism.

\appendix

\chapter{The Legendre transform and Alexandrov's theorem}\label{ap:legalex}

This appendix, together with the next appendix, provides a self-contained proof of Alexandrov's~\Cref{aleks} on second order differentiability of convex functions. The proof is based on two ingredients. The first ingredient concerns properties of the Legendre transform, which is well known to be an important tool in the study of convex functions. Roughly speaking, this transform provides a means of passing from first-order differentiability to second-order differentiability for convex functions. The second ingredient is a Lipschitz version of Sard's theorem (\Cref{sard}).

Our proof of Alexandrov's theorem follows the main steps of Harvey and Lawson's proof in \cite{hlqc}, which are basically the same as those proposed by Crandall, Ishii, and Lions in~\cite{user}. A different proof exploiting Lebesgue's Differentiation Theorem and mollifications is given in the monograph of Evans and Gariepy~\cite{evansgar}.

We begin with the definition of Legendre transform. Given $f\colon X \to \R$, with $X \subset \R^n$, recall the \Cref{def:subd} of its subdifferential $\de f$.

\begin{definition}
	The {\em Legendre transform of $f$} is the function $g \colon \de f(X) \to \R$ defined by the relation\sid{Legendre transform}\syid{Leg@$\scr Lu$!the Legendre transform of $u$}
	\begin{equation}	\label{leg:def}
	f(x) + g(y) = \pair xy \qquad \forall (x,y)\in\de f;
	\end{equation}
	that is, one defines
	\begin{equation}	\label{leg:def2}
	g(y) \defeq \pair xy - f(x) \quad \forall\, y \in \de f(x), \ x \in X
	\end{equation}
	We will also use the notation $g = \scr Lf$.
\end{definition}

\begin{remark}[Alternate definition of the Legendre transform] \label{Lt:aboveremark}
	Notice that the definition \eqref{leg:def2} is somewhat implicit; that is, in order to define $g(y) = \scr L f(y)$ for a given $y \in \de f(X) = {\rm dom}(g)$, one must identify a fiber $\de f(x)$ in which $y$ lives. A more explicit way to define the Legendre transform $g$ of $f$ involves an optimization procedure based on the observation that
	\[
	(x, y) \in \de f \quad \iff \quad
	f(z) - \pair{y}{z} \geq f(x) - \pair{y}{x} \quad \forall z\in X;
	\]
	hence $(x,y) \in \de f$ if and only if the function $f - \pair y\cdot$ assumes its minimum value on $X$ in the point $x$. That minimum value defines $-g(y)$; that is,
	\[
	-g(y) \defeq \inf_{z \in X} \big(f(z) - \pair{y}{z} \big) \quad \forall y \in \de f(X),
	\]
	or, equivalently,
	\begin{equation}\label{def:leg_sup}
	g(y) \defeq \sup_{z \in X} \big(\pair{y}{z} - f(z) \big) \quad \forall y \in \de f(X).
	\end{equation}
	We will see that these equivalent formulations \eqref{leg:def2} and \eqref{def:leg_sup} are both useful.
\end{remark}

\begin{remark}
If we consider $(-\infty,+\infty]$-valued functions (which are not identically equal to $+\infty$), then the identity in \eqref{def:leg_sup} provides a definition of $g = \scr L f$ for all $y \in \R^n$, with possibly $g(y) = +\infty$ for $y \notin \de f(X)$. This corresponds to the \emph{convex conjugate} of the extension $\tilde f$ of $f$ defined by letting $\tilde f = +\infty$ outside $X$.
\end{remark}

The first result we prove is that the Legendre transform is an involution that produces a convex function whose subdifferential is the inverse of the subdifferential of $f$, in the sense of multivalued maps.

\begin{prop} \label{leg:gen}
	Let $f: X \to \R$, with $X \subset \R^n$, and suppose that $\de f(X) \neq \emptyset$. Then the following hold:
	\begin{itemize}
		\item[(a)] $y \in \de f(X)$ if and only if 
		\begin{equation} \label{leg:sup}
		\scr L f (y) = \sup_{z\in X} \left( \pair yz - f(z) \right) < +\infty;
		\end{equation}
		\item[(b)] $\scr Lf: \de f(X) \to \R$ is convex;
		\item[(c)] for each $x\in X$ such that $\de f(x) \neq \emptyset$ and $y\in \de f(X)$,
		\[
		(x,y) \in \de f \quad\iff \quad (y,x) \in \de (\scr L f);
		\]
		\item[(d)] $\tilde X \defeq \{ x \in X:\, \de f(x) \neq \emptyset \} \subset \de(\scr Lf)(\de f(X))$ and $\restr{\scr L(\scr L f)}{\tilde X} = f$.
	\end{itemize}
\end{prop}

\begin{proof}
	The equivalence of part (a) has been presented in \Cref{Lt:aboveremark}. Part (b) follows from the fact that $\scr L f$ is the supremum of the family of affine functions $\{\pair\cdot z - f(z)\}_{z\in X}$; indeed, the convexity of $\scr Lf$ comes from the property (c) recalled in \Cref{rmk:obfc} (also note that $\de f(X)$ is convex by \Cref{lem:intpropsubdiff}(a), since $\de f(x)$ is convex even when $f$ is not, as noted in the proof).
	
	For the implication $(\Longrightarrow)$ of part (c), let $(x,y) \in X \times \de f(X)$. Suppose that $(x,y)\in\de f$; that is, $y \in \de f(x)$. In order to show that $(y,x) \in \de g$, where $g = \scr L f$, one needs to show that $y \in \de f(x)$; that is,
	\begin{equation}\label{dual_leg1}
	g(y) + \pair{x}{w-y} \leq g(w), \quad \forall \, w \in {\rm dom}(g) = \de f(X).
	\end{equation}
	For arbitrary $w\in \de f(X)$ we have $w \in \de f(z)$ for some $z\in X$ and hence
	\begin{equation}\label{dual_leg2}
	f(z) + \pair{w}{\hat{x}-z} \leq f(\hat{x}), \quad \forall \, \hat{x} \in X;
	\end{equation}
	by the definition \eqref{leg:def2} of the Legendre transform we have $g(w)\defeq  \pair{z}{w} - f(z)$ and $g(y)\defeq  \pair{x}{y} - f(x)$,
	which imply
	\begin{equation*}
	g(y) + \pair{x}{w-y} = \pair{x}{w} - f(x) = g(w) + \pair{x-z}{w} + f(z) - f(x) \leq g(w),
	\end{equation*}
	where the last inequality uses \eqref{dual_leg2} in $\hat{x} = x$, thus giving \eqref{dual_leg1}. Notice that we also have $\tilde X \defeq \{ x \in X:\, \de f(x) \neq \emptyset \} \subset \de g(\de f(X))$.
	
	For the involution claim of part (d), notice that we have shown that $(x,y) \in \de f$ implies that $(y,x) \in \de g$, where $g$ is convex. Hence, for each $x \in \tilde X$, by the definition \eqref{leg:def} we have
	\begin{equation*} \label{eq:Linvol}
	\scr L(\scr Lf)(x) = \scr Lg(x) \defeq  \pair xy - g(y) \defeq  \pair xy - \pair yx + f(x) = f(x), 
	\end{equation*}
	which proves {(d)}.
	
	Finally, we need to prove the implication $(\Longleftarrow)$ of part (c).  Let $(x,y) \in \tilde X \times \de f(X)$ and suppose that $(y,x) \in \de g$. We now have by \eqref{leg:def2} and the involution property of part (d)
	\[
	\pair{x}{y} =  g (y) + \scr{L}g (x) =  \scr L f (y) + f(x),
	\]
	and hence by \eqref{leg:sup}
	\[
	\pair xy = \sup_{z\in X} \bigl( \pair{y}{z} - f(z) \bigr) + f(x),
	\]
	yielding $f(z) \geq  f(x) + \pair y{z-x}$ for each $z \in X$; that is, $(x,y) \in \de f$.
\end{proof}

\begin{remark}\label{rem:subdiff}
	The properties (c) and (d) of \Cref{leg:gen} simplify when $f$ is convex and $X$ is open, in which case $\tilde X = X$. However, it remains false, in general, that
	\[
	(x,y) \in \de f \ \iff \ (y,x) \in \de g,
	\]
	without further assumptions on $x$ and $y$.
	This is because the implication $(\Longleftarrow)$ requires \emph{a priori} that $x \in X$. In other words, $(y,x) \in \de g$ does not imply that $(x,y) \in \de f$ if $x$ is only assumed to belong to $\de g(\de f(X)) \supset X$, which is all that is needed for the definition of the subdifferential of $g$.
	
	It can in fact happen that $\de g(\de f(X)) \supsetneqq X$, which also implies that $\scr L\!\scr L f \neq f$, in the sense that the domain of the former is larger than that of the latter. For example, consider $f = |\cdot|$, defined on $X=(-1,1) \subset \R$; it is easy to see that
	\[
	\de f(x) = \begin{cases}
	-1 & \text{if $x \in (-1,0)$} \\
	[-1,1] & \text{if $x = 0$} \\
	1 & \text{if $x \in (0,1)$},
	\end{cases}
	\]
	so that, according to \eqref{leg:def}, $g \defeq \scr Lf \equiv 0$ on $\de f(X) = [-1,1]$. Therefore,
	\[
	\de g(y) = \begin{cases}
	(-\infty, 0] & \text{if $y = -1$} \\
	0 & \text{if $y \in (-1,1)$} \\
	[0,+\infty) & \text{if $y = 1$}
	\end{cases}
	\]
	and $\scr L g = |\cdot|$ on $\de g(\de f(X)) = \de g([-1,1]) = \R \supsetneqq X$.
	
	For the sake of completeness, we mention that this example is based on the fact that the function $\tilde f$ defined on $\R$ by $\tilde f = f$ in $(-1,1)$ and $f \equiv +\infty$ in $\R \setminus (-1,1)$ is not lower semicontinuous function. Indeed, the Fenchel--Moreau Theorem states that given $f \colon \R^n \to (-\infty,+\infty]$ proper (that is, $\mathrm{dom}(f) \defeq \{x \in \R^n:\, f(x) < +\infty \} \neq \emptyset$), the identity $\scr L\!\scr L f = f$ is equivalent to $f$ being lower semicontinuous and convex. In other words, part~(d) of \Cref{leg:gen} simply becomes the identity $\scr L \scr L f = f$ (which also entails that $\de g(\de f(X)) = X$) when $f \colon X \to \R$ is convex and continuous on $X$ convex and closed.
	
	 On the other hand, in general, the double transform $\scr L \scr L f$ provides the largest lower semicontinuous convex function $g \leq f$ (as it can be noted in the above example). The reader can consult, e.g., \cite{rock,vesely} for further details.
\end{remark}

Having defined and examined some basic properties of the Legendre transform of an arbitrary function, we will know show that much more can said about the Legendre transform of convex functions of the form 
\begin{equation} \label{fform}
f = ru + \tfrac 12 |\cdot |^2, \quad r>0, \ u \ \text{convex}.
\end{equation}
Recall that our final goal is to prove Alexandrov's theorem and note that $u$ is twice differentiable at some point if and only if $f$ is; this will be used later.

\begin{prop} \label{leg:f}
	Let $u\colon X \to \R$ be a convex function and let $r > 0$; define $f$ as in~\emph{(\ref{fform})}. Then the following hold:
	\begin{itemize}
		\item[(a)] the subdifferential $\de f \colon X \to \scr P(\R^n)$ is an expansive map; that is,
		\begin{equation} \label{leg:espan}
		|y_2 - y_1| \geq |x_2 - x_1| \quad \forall\, y_j \in \de f(x_j), \ j=1,2;
		\end{equation}
		\item[(b)] the subdifferential $G\defeq \de g$ of $g \defeq \scr L f$ is a single-valued function from $\de f(X)$ onto $X$, and it is the inverse of $\de f$; that is,
		\[
		G(\de f(x)) = x \quad \forall x \in X;
		\] 
		\item[(c)] $G$ is contractive ($1$-Lipschitz); that is,
		\begin{equation} \label{leg:contr}
		|G(y_2) - G(y_1)| \leq |y_2 - y_1| \quad \forall\, y_j \in \de f(x_j), \ j=1,2.
		\end{equation}	
	\end{itemize}
	If, in addition, $u$ is bounded, then, letting $M\defeq \sup_X |u|$, $\delta \defeq 2\sqrt{rM}$ and $X^\delta \defeq \{ x\in X: d(x,\de X) > \delta \}$, the following hold:
	\begin{itemize}
		\item[(d)]  $X^\delta\subset \de f(X)$ and $\de f(X^{\delta}) \subset X$;
		\item[(e)] $g\defeq \scr L f\in C^1(X^\delta)$, with $Dg = G$ on $X^\delta$.
	\end{itemize}
\end{prop}

\begin{proof}
	First of all, by \Cref{lem:subdiffsum} applied to $ru$ and $\frac12{|\cdot|^2}$ we have $\de f = r\de u + \I$. Therefore if $x\in X$, $y,p\in\R^n$ are related by $y = x + rp$, then
	\begin{equation}\label{subd_legendre}
	(x,y) \in \de f \iff (x,p) \in \de u.
	\end{equation}
	
	For the expansivity claim of part (a), consider  $(x_j, y_j) \in \de f$, with $j=1,2$ and $p_j \defeq (y_j - x_j)/r$. Using \eqref{subd_legendre}, we have $p_j \in \de u(x_j)$ and then by the monotonicity of $\de u$ (\Cref{subdmon}) we have
	\begin{equation*}
	\pair{p_2 - p_1}{x_2 - x_1} \geq 0, \quad \forall \, p_j \in \de u(x_j),\ j = 1,2,
	\end{equation*}
	from which it follows that
	\[
	|y_2-y_1||x_2-x_1| \geq \pair{y_2-y_1}{x_2-x_1} = |x_2-x_1|^2 + r\pair {p_2-p_1}{x_2-x_1} \geq |x_2 - x_1|^2,
	\]
	which yields \eqref{leg:espan}. 
	
	For the claims of part (b), let $x_1,x_2 \in G(y) \cap X$ with $y \in \de f(X)$. Then by \Cref{leg:gen}{(c)} one has $y \in \de f(x_1) \cap \de f(x_2)$, and thus $x_1 = x_2$ by \eqref{leg:espan}. Therefore $G(\de f(x)) = \{ x \}$ for each $x \in X$; indeed, this easily follows from the fact that $G(\de f(x))$ is convex (by \Cref{lem:intpropsubdiff}{(a)} and \Cref{leg:gen}{(b)}) and $X$ is open. This shows that $G$ is single-valued and it is the inverse of $\de f$.
	
	The contractivity of $G$ in part (c) follows easily. Indeed, in light of part (b), one has $x_j = G(y_j)$ in \eqref{leg:contr}; that is,
	\[
	|G(y_2) - G(y_1)| \leq |y_2 - y_1| \quad \forall y \in \de f(X).
	\]
	
	Now assuming that $u$ convex is also bounded, we verify the claims of parts (d) and (e). 
	First for $y\in X^\delta$ we will show that $y \in \de f(X)$. To that end, we first note that the continuous function defined by $h_y \defeq ru + \frac12|\cdot-\,y|^2$ satisfies
	\[
	\inf_X  h_y = \min_{\barr B_\delta(y)} h_y.
	\]
	Indeed, if $z\notin \barr B_\delta(y)$ we have (where by hypothesis $|u| \leq M$ on $X$):
	\[
	h_y(z) - h_y(y) = ru(z) + \tfrac12|z-y|^2 - ru(y) > -2rM + \tfrac12\delta^2 = 0;
	\]
	therefore,  $\min_{\barr B_\delta(y)} h_y \leq h_y(y) \leq \inf_{X\setminus \barr B_\delta(y)} h_y$. Now let $x \in \barr B_\delta(y)$ be a point which realizes the minimum of $h_y$ on $X$. Hence, $0\in \de h_y(x)$; that is, $0 = rp + (x-y)$ for some $p\in \de u(x)$. We deduce that $y\in\de f(X)$ by \eqref{subd_legendre}, which proves the first inclusion in part (d).
	
	Next, for $y\in \de f (X^{\delta})$ we will show that $x \in X$. Since $y\in \de f (X^{\delta})$, we have $y \in \de f (x)$ for some $x\in X^{\delta}$. We know that $f- \pair y\cdot $ has a minimum point on $X$ at $x$. This is equivalent to $h_y$ having a minimum point on $X$ at $x$, since by definition $h_y = f -\pair y\cdot + \frac12|y|^2$. Then, $h_y$ has a minimum point at $x$, and we have
	\[
	|x-y|^2 \leq 2r(u(z)-u(x)) + |z-y|^2 \quad \forall z\in X, 
	\]
	yielding
	\[
	|x-y|^2 \leq 4rM + |z-y|^2 \quad \forall z\in X
	\]
	and thus, taking the infimum over $z\in X$,
	\begin{equation} \label{usata:lemLf}
	|x-y| \leq \sqrt{\delta^2 + d(y, X)^2} \leq \delta + d(y, X).
	\end{equation}
	Suppose now that $y\notin X$ and consider the linear segment $I\defeq[x,y]$; since it is connected, there exists some point $w\in\de X \cap I$, otherwise $\{X,\barr X{}\compl\}$ would be a separation of $I$. We have
	\[
	d(x,\de X) \leq |x-w| = |x-y| - |y-w| \leq \delta,
	\]
	where we used \eqref{usata:lemLf} for the last inequality, since $|y-w| \geq d(y,\barr X) = d(y,X)$. This contradicts the hypothesis that $x\in X^{\delta}$, thus $y\in X$.
	
	Finally, the claims of part (e) follow easily. Since $G=\de g$ is single-valued, we know that $G=Dg$ on $X_\delta \subset \de f(X)$, and that $G$ is continuous; see \Cref{deusingv}.
\end{proof}

As anticipated, the proof of Alexandrov's theorem will depend on two ingredients (\Cref{mplt} and \Cref{sard} below). The former is a crucial property of the Legendre transform, which provides a useful sufficient condition for $f$ defined by~(\ref{fform}), with $u$ bounded, to be twice differentiable. It essentially states that the subdifferential $G$ of the Legendre transform $g$ of $f$ maps the complement of its critical points $k_G$ to a set of points where $f$ is twice differentiable; that is, $G(X^\delta \setminus k_G) \subset \mathrm{Diff}^2 f$, where $X^\delta$ is defined as in \Cref{leg:f}. The latter is a Lipschitz version of Sard's theorem. It allows one to handle the set of ``bad'' points (namely the set of critical values of $G$, which is a Lipschitz function by \Cref{leg:f}{(c)}), by showing that this set has zero Lebesgue measure.

Let us recall what critical points and critical values are, and state and prove the former ingredient (\Cref{mplt}). Then we will state the latter ingredient (\Cref{sard}) and use it in the proof of Alexandrov's theorem, while the reader is invited to have a look at the next Appendix~\ref{proofsard} for a technical measure-theoretic proof of \Cref{sard}, based on Besicovitch's covering theorem.

\begin{definition} \sid{critical point/value(s)} \syid{kG@$k_G$!the set of critical points of $G$}
	Let $\Omega \subset \R^n$ be open and let $G\colon \Omega \to \R^n$. We define the set of \emph{critical points} of $G$ to be
	\[
	k_G \defeq \big\{x\in \Omega:\ \text{either $DG(x)$ does not exists or $\det DG(x) = 0$} \big\},
	\]
	and we call $G(k_G)$ the set of \emph{critical values} of $G$.
\end{definition}

The first needed ingredient is the following.

\begin{lem}[Magical property of the Legendre transform]	\label{mplt} \sid{property!magical!of the Legendre transform}
	Let $f$ be defined by (\ref{fform}); that is, $f \defeq  ru + \frac{1}{2} | \cdot |^2$ with $r > 0$ a real number and $u$ a convex function which is bounded on $X$. Let $G$ and $X^\delta$ be defined as in \Cref{leg:f}. Assume the following:
	\begin{itemize}
		\item[(i)] $G$ is differentiable at $y_0 \in X_\delta$, with $B\defeq DG(y_0)$;
		\item[(ii)] $x_0 \defeq G(y_0)$ is not a critical value of $G$; that is, ${\rm det} \, DG(y_0) \neq 0$;
		\item[(iii)] $f$ is differentiable at $x_0$.
	\end{itemize}
	Then $f$ is twice differentiable at $x_0$, with $D^2f(x_0) = B^{-1}$.
\end{lem}

\begin{remark} \label{rmk:mplt}
	In short, \Cref{mplt} tells that, if $|u| \leq M$ on $X$, $f \defeq ru + \frac12|\cdot|^2$ with $r>0$, $G \defeq \de \scr Lf$, then, with $\delta \defeq 2\sqrt{rM}$,
	\[
	G\bigl( X^{\delta} \setminus k_{G} \bigr) \cap \mathrm{Diff}^1 f \subset \mathrm{Diff}^2 f
	\]
	and
	\[
	D^2 u \circ G = r^{-1}( (DG)^{-1} - I ) \quad \text{on $\big( X^{\delta} \setminus k_G \big) \cap G^{-1}(\mathrm{Diff}^1 f)$}.
	\]
	Clearly, one can also replace each $\mathrm{Diff}^k f$ by $\mathrm{Diff}^k u$, since they are the same set because $\frac12|\cdot|^2$ is smooth.
\end{remark}

\begin{proof}[Proof of \Cref{mplt}]
	First of all, we recall that $G = \de g \colon \de f(X) \to \R^n$ is single-valued (hence $G = Dg$) and contractive by \Cref{leg:f}(b)-(c). Now, with $x_0 = G(y_0)$ for $y_0 \in \de f(x_0)$  we have  $Df(x_0) = y_0$ since
	\[
	x_0 \in \de g(y_0) = \{ G(y_0)\} \ \iff \ y_0 \in \de f(x_0) = \{ Df(x_0)\} 
	\]
	by Proposition~\ref{leg:gen} and the hypothesis (iii) that $f$ is differentiable in $x_0$. 
	
	Assume now, without loss of generality, that $x_0 = 0$ and that $f(0) = Df(0) = 0$. This is possible since the function $\tilde f \defeq f(\cdot + x_0) -f(x_0) - \pair {y_0}\cdot$
	will satisfy $\tilde f (0) = 0$ and $D \tilde f(0) = Df(x_0) - y_0 = 0$ and will be twice differentiable in $0$ if and only $f$ is twice differentiable in $x_0$. 
	
	By the hypotheses (i) and (ii), $0$ is not a critical value of $G$ and hence $B\defeq  DG(0)$ is invertible. The proof of the lemma reduces to showing that
	\begin{equation}\label{D2f(0)}
	f(x) - Q_A(x) = o(|x|^2),  \quad \text{for $x \to 0$},
	\end{equation}
	where $A\defeq B^{-1}$ and $Q_A(x) \defeq \frac12\pair{Ax}x$.
	
	\smallskip
	\noindent\underline{Step 1.}
	By \Cref{leg:f}(d), given $x\in X^\delta$ there exists $y\in X$ such that $y \in \de f(x)$, then by \Cref{leg:gen}(c), $x \in \de g(y) = \{G(y)\}$; that is, $x=G(y)$. For this pair $(x,y) \in \de f$, by the definition of the Legendre transform one can write
	\begin{equation}\label{fg_identity}.
	f(x) - Q_A(x)  = Q_B(y) - g(y) + \tfrac12\pair{y-Ax}x + \tfrac12\pair{x-By}y.
	\end{equation}
	
	\smallskip	
	\noindent\underline{Step 2.}
	Since $G$ is differentiable in $y_0 = 0$ with $G(0) = 0$ and $B=DG(0)$, one has 
	\begin{equation}\label{taylor1}
	x =G(y) = By + o(|y|) \quad \text{for $y \to 0$},
	\end{equation}
	which also yields
	\begin{equation} \label{taylor4}
	\pair{y-Ax}y \leq \norm{A} \, |x-By| \, |x| = o(|y|^2).
	\end{equation}
	Indeed, since $A = B^{-1}$, $G$ is contractive, and $|x-By| = o(|y|)$ by \eqref{taylor1}, for $y \to 0$ one has
	\[
	\pair{y-Ax}y = \pair{-A(x-By)}{x} \leq \norm{A} \, |x - By| \, |x| \leq \norm{A} \, |xBy| \, |y| = o(|y|^2).
	\]
	
	\smallskip	
	\noindent\underline{Step 3.}
	The asymptotic expansion \eqref{taylor1} allows us to show that the Legendre transform $g$ of $f$ is twice differentiable with 
	\begin{equation}\label{leg:cl:der}
	D^2g(0) = B.
	\end{equation}
	Note that by definition $g(0) = 0$ (since $f(0) = 0$ and $(0,0) \in \de f$) and by hypothesis $G(0) = Dg(0) = 0$. Hence  the claim \eqref{leg:cl:der} is equivalent to 
	\begin{equation}\label{D2g(0)}
	g(y) - Q_B(y) = o(|y|^2) \quad \text{for $y \to 0$.}
	\end{equation}
	To see that \eqref{D2g(0)} holds, one first uses the mean value theorem for each $y$ fixed near $0$ to conclude that there exists $z\in[0,y]$ such that
	\begin{equation}\label{MVT}
	g(y) - Q_B(y) = \pair {Dg(z)- Bz}{y}.
	\end{equation}
	Then, using \eqref{taylor1} one has
	\[
	Dg(z) - Bz = G(z) - Bz = o(|z|) \quad \text{for $z \to 0$},
	\]
	which, since $|z(y)| \leq |y|$, together with \eqref{MVT} implies
	$$
	| g(y) - Q_B(y)| \leq |DG(z) - Bz| \, |y| = o(|y|^2) \quad \text{for $y \to 0$},
	$$
	as desired in \eqref{D2g(0)}. 
	
	\smallskip	
	\noindent\underline{Step 4.}
	Finally, we conclude the proof of the lemma by establishing \eqref{D2f(0)}. Using \eqref{taylor1}, \eqref{taylor4} and \eqref{D2g(0)} on the terms on the right-hand side of \eqref{fg_identity} yields
	\begin{equation*}
	f(x) - Q_A(x)  = o(|y|^2) \quad \text{for $y \to 0$},
	\end{equation*}
	which implies the needed \eqref{D2f(0)} since we can show that
	\begin{equation}\label{bigO}
	|y| = O(|x|) \quad \forall \, y \ \text{with $|y|$ small.}
	\end{equation}
	To see this, note that since $A = B^{-1}$, one has 
	\begin{equation*}
	|y| = |ABy| \leq || A|| \, (|x-By| + |x| ), \quad \forall\, x,y \in \R^n,
	\end{equation*}
	thus by \eqref{taylor1} for any $\epsilon > 0$ there exists $\delta > 0$ such that
	\[
	|y| \leq \norm{A} (\varepsilon |y| + |x|) \quad \forall \, y \in B_{\delta}(0);
	\]
	in particular, choosing $\varepsilon < 1/(2\norm{A})$ gives
	\[
	|y| \leq 2 \norm{A} \, |x| \quad \forall \, y \in B_{\delta}(0),
	\]
	while since $G$ is contractive, $|x| = |G(y)| \leq |y|$, and hence \eqref{bigO} is proved, and the proof complete.
\end{proof}

The second needed ingredient is the following, which will be proven in Appendix \ref{proofsard}.

\begin{thm}[Sard's theorem for Lipschitz functions] \label{sard} \sid{Theorem!Sard}
	Let $\Omega \subset \R^n$ be open. The set of critical values of a Lipschitz function $G: \Omega \to \R^n$ has Lebesgue measure zero.
\end{thm}

\begin{remark}
	As the reader might have noticed, our choice of defining the critical points of $G$ as those points at which $G$ either is not differentiable or has singular derivative (sometimes ones uses only the second condition) suggests that the proof of Sard's theorem will make use of Rademacher's theorem on the almost everywhere differentiability of Lipschitz functions.
\end{remark}

\begin{proof}[Proof of Alexandrov's \Cref{aleks}]
	Without loss of generality, we can suppose that $u$ is in fact bounded on $X$; indeed, if not, it suffices to consider an exhaustion by compact convex sets $K_j \uto X$ (for instance, one can take $K_j \defeq \big\{ x \in X :\ d(x,\de X) \geq 2^{-j}, \, |x| < 2^j \big\}$ for all $j\in\N$), prove the theorem for $u|_{\intr\! K_j}$ (which is bounded since $u$ is continuous on $X$ by \Cref{thm:convex_Lip}), and then note that $\mathrm{Diff}^2 u$ has full measure as well, since 
	\[
	X \setminus \mathrm{Diff}^2 u  \subset \bigcup_{j\in\N} \left( \intr K_j \setminus \mathrm{Diff}^2 u \right),
	\]
	which by the $\sigma$-subadditivity of the Lebesgue measure yields $|X \setminus \mathrm{Diff}^2 u| = 0$.
	
	Hence, we will prove the theorem for $u$ bounded on $X$. Fix $r>0$ and let $f \defeq ru + \frac12|\cdot|^2$. By \Cref{mplt} (see also \Cref{rmk:mplt}), we know that
	\begin{equation} \label{proofalex:incl}
	G(X^\delta \setminus k_G) \cap \mathrm{Diff}^1 u \subset \mathrm{Diff}^2 u, 
	\end{equation}
	where (see \Cref{leg:f})
	\begin{equation} \label{proofalex:delta}
	\delta = 2\sqrt{r\norm{u}_\infty}.
	\end{equation}
	By \Cref{rade:conv}, $\mathrm{Diff}^1 u$ has full measure, and, since $G$ is Lipschitz by \Cref{leg:f}{(c)}, $G(k_G)$ has measure zero by \Cref{sard}. Hence from \eqref{proofalex:incl} we have
	\begin{equation} \label{alexproof:ineq}
	\big|\mathrm{Diff}^2 u \big| \geq \big|G(X_\delta)\big| \geq \big|X^{2\delta}\big|,
	\end{equation}
	where the latter inequality comes from the fact that $X^{2\delta} \subset G(X^\delta)$. Indeed, by \Cref{leg:f}{(d)}, $\de f(X^{2\delta}) \subset X^\delta$, which by part {(b)} of the same proposition is equivalent to $X^{2\delta} \subset G(X^\delta)$, as we claimed.
	
	Notice now that $r > 0$ is arbitrary, thus we can let $r \dto 0$; that is, recalling \eqref{proofalex:delta}, we can let $\delta \dto 0$ in \eqref{alexproof:ineq}. We obtain $|\mathrm{Diff}^2 u | \geq | X |$, and thus $|\mathrm{Diff}^2 u | = | X |$, since the reverse inequality holds trivially. This proves that $\mathrm{Diff}^2 u$ has full measure (in $X$), which is the desired conclusion.
\end{proof}

\chapter{Proof of Sard's theorem for Lipschitz functions}\label{proofsard}

Since \Cref{sard} states that a set of ``bad'' points (the critical values of a Lipschitz function $G$) has zero measure, it is not surprising that one needs to appeal to some measure-theoretic covering lemma. We will make use of a theorem of Besicovitch~\cite{besic}, which we will prove in \Cref{app:bes} by adapteing the proof of a more general form contained in~\cite{fed:geo}. 

We will also need the following elementary estimate on controlling the {\em Lebesgue outer measure $m$} (see, e.g., \cite[Chapter~1]{stein} for its definition and basic properties) of images of small balls about points of differentiability of a vector field $G$ on $\R^n$.

\begin{lem} \label{app:ball}
	Suppose that $G\colon \Omega \to {\R^n}$ is differentiable at $a \in \Omega$ with $\Omega \subset \R^n$ open. Then for each $\epsilon >0$ there exists $\barr\rho = \barr\rho(a,\epsilon)>0$ such that 
	\[
	m\big(G\big(B_\rho(a)\big)\big) \leq \left(|{\det DG(a)}|+\epsilon\right)m(B_\rho(a)) \qquad \forall \rho\leq \barr\rho.
	\]
\end{lem}

\begin{proof}
	Since $G$ is differentiable at $a$, one has that 
	\[
	| G(x) - \ell(x) | \leq \omega(|x-a|)|x-a|, \quad \text{for all $x$ near $a$}, 
	\]
	where $\ell \defeq G(a) + \pair{DG(a)}{\cdot-a}$ is the linearization of $G$ at $a$ and $\omega\colon \R_+ \to \R_+$ is a modulus of continuity (so $\omega(0^+)=0$). This yields
	\[
	G\big(B_\rho(a)\big) \subset \ell\big(B_\rho(a)\big) + \omega(	\rho) \rho B_1(0),
	\]
	By the monotonicity of the outer measure $m$, one has 
	\[
	m(G(B_\rho(a))) \leq m(\ell(B_1(a)) + \omega(	\rho)  B_1(0))\rho^n,
	\]
	Now note that
	\[
	\lim_{\rho \dto 0} m(\ell(B_1(a)) + \omega(	\rho)  B_1(0)) = m(\ell(B_1(a))) = |\ell(B_1(a))| = |{\det DG(a)}|\, \alpha_n, 
	\]
	where $\alpha_n\defeq  m(B_1(0))$ and the latter equality comes from the change of variables formula. Therefore, there exists some $\barr\rho>0$ such that for $\rho\leq \barr\rho$ we have
	\[
	m\big(G\big(B_\rho(a)\big)\big) \leq m(\ell(B(a)) + \omega(	\rho)  B)\rho^n \leq (|{\det DG(a)}| + \epsilon)\,\alpha_n\rho^n
	\]
	and hence the result.
\end{proof}

Next is the needed covering lemma.

\begin{lem}[Besicovitch's covering theorem] \label{app:bes} \sid{Besicovitch's covering|seeonly{Theorem, Besicovitch's covering}} \sid{Theorem!Besicovitch's covering}
	Let $A\subset \R^n$ and $\rho\,\colon A \to \R_+$ be bounded. Then there exists $N\in\N$, depending only on $n$, and $N$ families of balls
	\[
	\call B_k = \big\{ B_j^k \defeq B_{\rho(a_j^k)}(a_j^k) :\ a_j^k\in A,\ j\in\N,\ B_j^k \cap B_i^k\ \forall\, i\neq j \big\}, \qquad k=1,\dots, N,
	\]
	such that 
	\[
	A \subset \bigcup_{k=1}^N \bigcup_{j\in\N} B_j^k.
	\]
\end{lem}

\begin{proof}
	We split the proof into four claims.
	
	\smallskip
	\noindent\underline{Claim 1}. \emph{Let $\tau >1$. Then
		\[
		A \subset \bigcup_{j\in \N} B_j,
		\]
		where $B_j = B_{\rho_j}(a_j)$ with $a_j\in A$ and $\rho_j=\rho(a_j)$ satisfying
		\[
		\text{either}\quad |a_i - a_j| > \rho_i > \rho_j/\tau \quad \text{or}\quad |a_i-a_j| > \rho_j > \rho_i/\tau.
		\]
	}
	Consider the class $\Lambda$ of subsets $\Q\subset A$ with the following two properties:
	\begin{enumerate}[left=\parindent,label=(\roman*),series=Bcond]
	\item either $|a - c| > \rho(a) > \rho(c)/\tau$ or $|a-c| > \rho(c) > \rho(a)/\tau$ for all $a,c\in \Q$;
	\item whenever $b\in A$, either $|a-b|\leq \rho(a)$ for some $a\in \Q$ or $|a-b| > \rho(a) > \rho(b)/\tau$ for all $a\in \Q$.
	\end{enumerate}
	By Hausdorff's maximal principle, $\Lambda$ has a maximal member $\mathfrak Q$. Then the family $\{B_{\rho(a)}(a) : a\in \mathfrak Q\}$ covers $A$. Indeed, otherwise 
	\[
	K\defeq \{ a\in A:\ |a-c| > \rho(a)\ \text{for all}\ a\in \mathfrak Q\} \neq \emptyset
	\]
	and thus we may choose $c\in K$ with $\tau\rho(c) > \sup_{a\in K} \rho(a)$. This implies $\mathfrak Q\subset \mathfrak Q \cup \{c\} \in \Lambda$, contrary to the maximality of $\mathfrak Q$. Also, since $A$ is second countable, there exists a countable subset of $\mathfrak Q$ which still covers $A$.
	
	\smallskip
	\noindent\underline{Claim 2}. \emph{Let $P\subset A$. Then there exists $R\subset P$ such that the following two properties hold:
	\begin{enumerate}[Bcond]
	\item $|a-b| > \rho(a)+\rho(b)$ for all $a,b\in \RR$;
	\item whenever $a\in P$, there exists $b\in R$ with $|a-b| \leq \rho(a) + \rho(b)$ and $\rho(b) > \rho(a)/\tau$.
	\end{enumerate}}
	Consider the class $\Lambda$ of those subsets $R\subset P$ satisfying (iii) and the following:
	\begin{enumerate}[Bcond]
	\item whenever $a\in P$, either $|a-b| > \rho(a) + \rho(b)$ for all $b\in R$ or (iv) holds.
	\end{enumerate}
	Again by Hausdorff's maximal principle, there exists a maximal member $\mathfrak R$ of $\Lambda$, which satisfies (iv); indeed, otherwise, arguing as in the previous step, property (v) would imply that there exists $a\in P\setminus \mathfrak R$ such that $\mathfrak R \cup \{ a \} \in \Lambda$, contrary to the maximality of $\mathfrak R$.
	
	\smallskip
	\noindent\underline{Claim 3}. \emph{Let $\epsilon \in (0,1)$ and $\tau \in (1,2-\epsilon)$ such that
		\begin{equation} \label{app:dis}
		\epsilon + \frac\tau{2-\epsilon} + \tau(\tau -1) < 1.
		\end{equation}
		Let $\Q\subset A$ satisfy (i) as well as the following property:
		\begin{enumerate}[Bcond]
		\item $|a-b| \leq \rho(a) + \rho(b)$ and $\rho(b) > \rho(a)/\tau$ for all $a,b\in \Q$.
		\end{enumerate}
		Then there exists $N \in \N$, depending only on $n$ and $\epsilon$, such that $\# \Q \leq N$.}\syid{card@$\#E$!the cardinality of the set $E$}
	
	First of all, note that such a $\tau$ exists (by continuity) since the left-hand side of \eqref{app:dis} for $\tau = 1$ and $\epsilon = 0$ is $\tfrac12$. Let now $\kappa = (2-\epsilon)/\tau$ and $a\in \Q$. Set
	\[
	\Q_1 \defeq\Q\cap \left(B_{\kappa\rho(a)}(a) \setminus \{a\}\right), \qquad \Q_2 \defeq\Q\cap B_{\kappa\rho(a)}(a)\compl.
	\]
	Let $b,c\in \Q_j$ for some $j\in\{1,2\}$ be distinct points with $|a-b| \geq |a-c|$ and $x\in\R^n$ such that
	\[
	|a-x| = |a-c|, \qquad |b-x| = |a-b| - |a-c|.
	\]
	Note that actually such an $x$ is uniquely determined by $a,b,c$ because of the strict convexity of the Euclidean norm, namely as
	\begin{equation} \label{app:sc}
	x = a + \frac{|a-c|}{|a-b|}(b-a).
	\end{equation}
	We see that 
	\[
	|x-c| \geq |a-c| + |b-c| - |a-b|
	\]
	and distinguish two cases. If $j=1$, since by properties (i) and (vi),
	\[
	|b-c| > \min\{\rho(b), \rho(c)\} > \rho(a)/\tau
	\]
	and similarly $|a-c| > \rho(a)/\tau$, while $|a-b| \leq \kappa\rho(a)$, with $\kappa\tau - 1 = 1-\epsilon >0$, we obtain
	\[
	|x-c| > |a-c| - (\rho(a)/\tau)(\kappa\tau-1) > \epsilon |a-c|.
	\]
	If $j=2$, note that we have the following inequalities:
	\[
	|a-b| \leq \rho(a) + \rho(b), \quad |a-c| > \max\{ \kappa\rho(a), \rho(c)/\tau \}, \quad
	|b-c|-\rho(b) > (1-\tau)\rho{(c)};
	\]
	then we obtain
	\[\begin{split}
	|x-c| &\geq |a-c| + |b-c| - \rho(a) -\rho(b) > |a-c| -\rho(a) - (\tau-1)\rho{(c)}\\
	&= |a-c|\biggl( 1- \frac{\rho(a)}{|a-c|} - (\tau-1) \frac{\rho{(c)}}{|a-c|} \biggr) > |a-c|\biggl( 1-\frac1\kappa - \tau(\tau-1) \biggr) \\
	&> \epsilon |a-c|,
	\end{split}\]
	where the last inequality comes from the condition~(\ref{app:dis}).
	
	At this point by~(\ref{app:sc}) one sees that 
	\[
	\epsilon < \frac{|x-c|}{|a-c|} = \left| \frac{b-a}{|b-a|} - \frac{c-a}{|c-a|} \right| = |\pi(b)-\pi{(c)}|,
	\]
	where $\pi(b)$ and $\pi{(c)}$ denote the projections of $b$ and $c$, respectively, onto the sphere $\de B_1(a)$. This implies that there is a one-to-one correspondence between points in $\Q_j$ and their projections and $\#\Q_j \leq K$ for some $K \in \N$. Indeed, we know that $\pi(\Q_j)$ is finite, because otherwise by the compactness of $\de B_1(a)$ it would have a limit point and thus there would exist points $\pi(b)$ and $\pi{(c)}$, with $|\pi(b) - \pi{(c)}| < \epsilon$. Clearly this argument is independent of $a$ and holds for every $\epsilon$-distanced subset of the unit sphere $\mathbb{S}^{n-1}$, hence $K$ depends only on $n$ and $\epsilon$.
	Putting all of this together, we have $\# \Q = \#(\Q_1\cup \Q_2 \cup \{a\}) = \#\Q_1 + \#\Q_2 + 1 \leq 2K + 1 =\vcentcolon N$.
	
	\smallskip
	\noindent\underline{Claim 4}. \emph{Let $\tau$ be as in Claim~3 with $\epsilon = \frac13$ and $P \defeq \{ B_j \}_{j\in\N}$ be as in Claim~1. Then $P$ is the union of $N$ families of pairwise disjoint balls.}
	
	Applying Claim~2, we define by induction subsets $P_k, R_k \subset P$, starting with $P_0 \defeq P$ and $R_0 \defeq \emptyset$ and letting, for $k\geq 1$, $P_k \defeq P_{k-1} \setminus \RR_{k-1}$, and $\RR_k \subset P_k$ satisfying (iii) and (iv). Note that condition (iii) implies that the family $\call B_k \defeq \{B_j^k \defeq B_{\rho_j^k}(a_j^k):\ a_j^k\in \RR_k\}$ is disjointed for each $k$ and by definition
	\[
	P_{N+1} = \bigcap_{k=1}^N\, (P \setminus \RR_k) = P \setminus \bigcup_{k=1}^N \RR_k.
	\]
	Hence we conclude the proof if we show that $P_{N+1} = \emptyset$. In fact, if $a\in P_{N+1}$, one uses property (iv) to select, for all $k=1,\dots, N$, some $b_j^k\in \RR_k$ with $|a-b_j^k| \leq \rho(a) + \rho_j^k$ and $\rho_j^k > \rho(a)/\tau$. Therefore $\Q\defeq \{a,b_1,\dots, b_N\}$ satisfies (i) and (iv), and also $\#\Q = N+1$ since the $\RR_j$'s are disjointed by definition; but this contradicts Claim~3, so it must be $P_{N+1} = \emptyset$ and the proof is concluded.
\end{proof}

Next, one can globalize the local measure estimate of Lemma \ref{app:ball} by way of the Besicovitch covering Lemma \ref{app:bes}.
\begin{prop}\label{prop:DG}
	Let $\Omega \subset \R^n$ be open, $G\colon \Omega\to {\R^n}$ and let $A\subset\Omega$ have $m(A)$ finite. Suppose $G$ is differentiable at every $a\in A$ and there exists a constant $M$ such that $\lvert\det DG\rvert\leq M$ in $A$. Then 
	\[
	m(G(A)) \leq cMm(A),
	\]
	where the constant $c$ depends only on $n$.
\end{prop}

\begin{proof}
	Fix $U\supset A$ open such that $m(U) \leq 2m(A)$; given $\epsilon >0$, by Lemma \ref{app:ball}, for each $a\in A$ there exists some ball $B_{\rho(a)}(a)$ such that 
	\[
	m(G(B_{\rho(a)}(a))) \leq (M+\epsilon) m(B_{\rho(a)}(a)).
	\]
	 Also, we may choose $\rho(a)$ in such a way that $B_{\rho(a)}(a) \subset U\cap \Omega$ for all $a\in A$; hence the map $\rho\,\colon A \to \R_+$ is bounded. Besicovitch's covering theorem (Lemma \ref{app:bes}) tells us that there exists a countable subset $\{a_i\}_{i\in\N} \subset A$ such that, letting $\rho_i \defeq \rho(a_i)$ and $B_i\defeq B_{\rho_i}(a_i)$, the family $\{B_i\}_{i\in\N}$ covers $A$ and every point $x\in\R^n$ is in at most $N$ of the balls $B_i$. Therefore $\sum_{i\in\N} \chi_{B_i} \leq N\chi_U$ and we have
	\[
	\sum_{i\in\N} m(B_i) \leq Nm(U) \leq 2N m(A)
	\]
	and thus
	\[\begin{split}
	m(G(A)) \leq m\biggl(G\biggl(\,\bigcup_{i\in\N} B_i \biggr)\biggr) \leq \sum_{i\in N} m(G(B_i)) \leq (M+\epsilon)\! \sum_{i\in\N} m(B_i) 
	\leq 2N(M+\epsilon)m(A).
	\end{split}\]
	As $\epsilon$ is arbitrary, we have the thesis, with $c=2N$.
\end{proof}

The treatment of the critical values corresponding to points of differentiability is now at hand.

\begin{cor}\label{cor:DG}
	Let $\Omega\subset\R^n$ be open and $G \colon \Omega \to \R^n$; set
	\[
	\call S_G\defeq \{x\in\Omega:\ G \ \text{is differentiable at $x$ and} \  \lvert\det DG(x)\rvert = 0\}.
	\]
	Then $G(\call S_G)$ has Lebesgue measure zero.
\end{cor}

\begin{proof}
	If $m(\call S_G)<\infty$, then the conclusion follows immediately from Proposition \ref{prop:DG}  with $A = \call{S}_G$ and $M = 0$. If $m(\call S_G) = \infty$, then consider the decomposition $\call S_G = \bigcup_{k \in N}A_k$ with $A_k\defeq  \call  S_G \cap B_k(0)$. One has $m(A_k) = 0$  for each $k$ by the previous step, and the conclusion follows from the subadditivity of $m$. 
\end{proof}

The proof of Sard's theorem follows easily from Corollary \ref{cor:DG} and Rademacher's theorem.

\begin{proof}[Proof of~\Cref{sard}]
	One needs to show that $m(G(k_G)) = 0$. The set of critical points can be decomposed into $k_G = \call N_G \cup \call S_G$, where $\call N_G$ denotes the set of points at which $G$ is not differentiable and $\call S_G$ is the set of points of differentiability where the Jacobian $DG$ is singular. Since $G$ is Lipschitz, $\call N_G$ has measure zero by Rademacher's Theorem~\ref{rade} and hence $G(\call N_G)$ has measure zero as well since $G$ is Lipschitz. By the above corollary, $G(\call S_G)$ has measure zero and we are done. \qedhere
\end{proof}

\backmatter


\let\indexspacestd\indexspace
\let\subitemstd\subitem
\renewcommand{\indexspace}{}
\renewcommand\subitem{$\cdot$~}
\printindex[symbol]
\renewcommand{\indexspace}{\indexspacestd}
\renewcommand{\subitem}{\subitemstd}
\printindex[subject]

\end{document}